\documentclass[10pt,a4paper]{article}
\usepackage[utf8]{inputenc}
\usepackage[T1]{fontenc}
\usepackage{geometry}
\usepackage{amsthm}
\usepackage[french,english]{babel}
\usepackage{amssymb}
\usepackage{lmodern}
\usepackage{amsmath,amsfonts}
\usepackage{mathrsfs} 
\usepackage{dsfont}
\usepackage{stmaryrd}
\DeclareMathOperator{\dive}{div} 
\usepackage{graphicx}
\usepackage{fullpage}
\usepackage[nottoc,notlot,notlof]{tocbibind}
\usepackage[colorlinks=true,urlcolor=black,linkcolor=black,citecolor=black]{hyperref}
\numberwithin{equation}{section}
\usepackage{mathtools,array}
\usepackage{verbatimbox}
\usepackage{color}
\usepackage{bm}
\usepackage{tikz-cd}
\usepackage{pst-node}
\usetikzlibrary{matrix}
\usepackage{mathtools}
\usepackage{verbatimbox}
\usepackage{diagbox}
\usepackage{array}
\newcolumntype{C}{>{$\displaystyle} c <{$}}
\usepackage{makecell}
\usepackage{float}
\usepackage{tabu}
\usepackage{cancel}
\usepackage{lscape}
\usepackage[babel=true]{csquotes}
\usepackage{bm}
\usepackage{xcolor}
\usepackage[new]{old-arrows}
\usepackage{stackengine}
\makeatletter
\def\env@dmatrix{\hskip -\arraycolsep
	\let\@ifnextchar\new@ifnextchar
	\def\arraystretch{2}%
	\array{*{\c@MaxMatrixCols}{>{\displaystyle}c}}%
}
\usepackage{soul}

\makeatother

\DeclareFontShape{OMX}{cmex}{m}{n}{
	<-7.5> cmex7
	<7.5-8.5> cmex8
	<8.5-9.5> cmex9
	<9.5-> cmex10
}{}
\SetSymbolFont{largesymbols}{normal}{OMX}{cmex}{m}{n}
\SetSymbolFont{largesymbols}{bold}  {OMX}{cmex}{m}{n}

\usepackage{enumitem}

\begin{document}

	\renewcommand{\thefootnote}{\fnsymbol{footnote}}
	
	\title{Morse Index Stability of Biharmonic Maps in Critical Dimension}

	\author{Alexis Michelat\footnote{Institute of Mathematics, EPFL B, Station 8, CH-1015 Lausanne, Switzerland.\hspace{.5em} \href{alexis.michelat@epfl.ch}{alexis.michelat@epfl.ch}}}
	\date{\today}
	
	\maketitle
	
	\vspace{-0.5em}
	
	\begin{abstract}
        Furthering the development of Da Lio-Gianocca-Rivière's Morse stability theory (\cite{riviere_morse_scs}) that was first applied to harmonic maps between manifolds and later extended to the case of Willmore immersions in \cite{morse_willmore_I,eigenvalue_annuli}, we generalise the method to the case of (intrinsic or extrinsic) biharmonic maps. In the course of the proof, we develop a novel method to prove strong energy quantization (in the space of squared-integrable functions that corresponds to the pre-dual of the Marcinkiewicz space of weakly squared-integrable functions) in a wide class of problems in geometric analysis, which allows us to recover some previous results in a unified fashion. 
	\end{abstract}

	\tableofcontents
	\vspace{0cm}
	\begin{center}
	\end{center}

	\theoremstyle{plain}
	\newtheorem*{theorem*}{Theorem}
	\newtheorem{theorem}{Theorem}[section]
	\newenvironment{theorembis}[1]
	{\renewcommand{\thetheorem}{\ref{#1}$'$}%
		\addtocounter{theorem}{-1}%
		\begin{theorem}}
		{\end{theorem}}
	\renewcommand*{\thetheorem}{\Alph{theorem}}
	\newtheorem{lemme}[theorem]{Lemma}
	\newtheorem*{lemme*}{Lemma}
	\newtheorem{propdef}[theorem]{Definition-Proposition}
	\newtheorem*{propdef*}{Definition-Proposition}
	\newtheorem{prop}[theorem]{Proposition}
	\newtheorem{cor}[theorem]{Corollary}
	\theoremstyle{definition}
	\newtheorem*{definition}{Definition}
	\newtheorem{defi}[theorem]{Definition}
	\newtheorem{rem}[theorem]{Remark}
	\newtheorem*{rem*}{Remark}
	\newtheorem{rems}[theorem]{Remarks}
	\newtheorem{remimp}[theorem]{Important Remark}
	\newtheorem{exemple}[theorem]{Example}
	\newtheorem{defi2}{Definition}
	\newtheorem{propdef2}[defi2]{Proposition-Definition}
	\newtheorem{remintro}[defi2]{Remark}
	\newtheorem{remsintro}[defi2]{Remarks}
	\newtheorem{conj}{Conjecture}
	\newtheorem{question}{Open Question}
	\renewcommand\hat[1]{%
		\savestack{\tmpbox}{\stretchto{%
				\scaleto{%
					\scalerel*[\widthof{\ensuremath{#1}}]{\kern-.6pt\bigwedge\kern-.6pt}%
					{\rule[-\textheight/2]{1ex}{\textheight}}
				}{\textheight}%
			}{0.5ex}}%
		\stackon[1pt]{#1}{\tmpbox}
	}
	\parskip 1ex
	\newcommand{\totimes}{\ensuremath{\,\dot{\otimes}\,}}
	\newcommand{\vc}[3]{\overset{#2}{\underset{#3}{#1}}}
	\newcommand{\conv}[1]{\ensuremath{\underset{#1}{\longrightarrow}}}
	\newcommand{\A}{\ensuremath{\vec{A}}}
	\newcommand{\B}{\ensuremath{\vec{B}}}
	\newcommand{\C}{\ensuremath{\mathbb{C}}}
	\newcommand{\D}{\ensuremath{\nabla}}
	\newcommand{\Disk}{\ensuremath{\mathbb{D}}}
	\newcommand{\E}{\ensuremath{\vec{E}}}
	\newcommand{\I}{\ensuremath{\mathbb{I}}}
	\newcommand{\Q}{\ensuremath{\vec{Q}}}
	\newcommand{\loc}{\ensuremath{\mathrm{loc}}}
	\newcommand{\z}{\ensuremath{\bar{z}}}
	\newcommand{\hh}{\ensuremath{\mathscr{H}}}
	\newcommand{\h}{\ensuremath{\vec{h}}}
	\newcommand{\vol}{\ensuremath{\mathrm{vol}}}
	\newcommand{\hs}[3]{\ensuremath{\left\Vert #1\right\Vert_{\mathrm{H}^{#2}(#3)}}}
	\newcommand{\R}{\ensuremath{\mathbb{R}}}
	\renewcommand{\P}{\ensuremath{\mathbb{P}}}
	\newcommand{\N}{\ensuremath{\mathbb{N}}}
	\newcommand{\Z}{\ensuremath{\mathbb{Z}}}
	\newcommand{\p}[1]{\ensuremath{\partial_{#1}}}
	\newcommand{\Res}{\ensuremath{\mathrm{Res}}}
	\newcommand{\lp}[2]{\ensuremath{\mathrm{L}^{#1}(#2)}}
	\renewcommand{\wp}[3]{\ensuremath{\left\Vert #1\right\Vert_{\mathrm{W}^{#2}(#3)}}}
	\newcommand{\wpn}[3]{\ensuremath{\Vert #1\Vert_{\mathrm{W}^{#2}(#3)}}}
	\newcommand{\np}[3]{\ensuremath{\left\Vert #1\right\Vert_{\mathrm{L}^{#2}(#3)}}}
    \newcommand{\snp}[3]{\ensuremath{\left| #1\right|_{\mathrm{L}^{#2}(#3)}}}
    \newcommand{\ntimeswp}[4]{\ensuremath{\left\Vert #1\right\Vert_{\mathrm{L}^{#2}\cdot\mathrm{W}^{#3}(#4)}}}
	\newcommand{\hp}[3]{\ensuremath{\left\Vert #1\right\Vert_{\mathrm{H}^{#2}(#3)}}}
	\newcommand{\ck}[3]{\ensuremath{\left\Vert #1\right\Vert_{\mathrm{C}^{#2}(#3)}}}
	\newcommand{\hardy}[2]{\ensuremath{\left\Vert #1\right\Vert_{\mathscr{H}^{1}(#2)}}}
	\newcommand{\lnp}[3]{\ensuremath{\left| #1\right|_{\mathrm{L}^{#2}(#3)}}}
    \newcommand{\serieslnp}[3]{\ensuremath{\left\Vert #1\right\Vert_{{l}^{#2}(#3)}}}
	\newcommand{\npn}[3]{\ensuremath{\Vert #1\Vert_{\mathrm{L}^{#2}(#3)}}}
	\newcommand{\nc}[3]{\ensuremath{\left\Vert #1\right\Vert_{C^{#2}(#3)}}}
	\renewcommand{\Re}{\ensuremath{\mathrm{Re}\,}}
	\renewcommand{\Im}{\ensuremath{\mathrm{Im}\,}}
	\newcommand{\dist}{\ensuremath{\mathrm{dist}}}
	\newcommand{\diam}{\ensuremath{\mathrm{diam}\,}}
	\newcommand{\leb}{\ensuremath{\mathscr{L}}}
	\newcommand{\supp}{\ensuremath{\mathrm{supp}\,}}
	\renewcommand{\phi}{\ensuremath{\vec{\Phi}}}
	\renewcommand{\H}{\ensuremath{\vec{H}}}
	\renewcommand{\L}{\ensuremath{\vec{L}}}
	\renewcommand{\lg}{\ensuremath{\mathscr{L}_g}}
	\renewcommand{\ker}{\ensuremath{\mathrm{Ker}}}
	\renewcommand{\epsilon}{\ensuremath{\varepsilon}}
	\renewcommand{\bar}{\ensuremath{\overline}}
	\newcommand{\s}[2]{\ensuremath{\langle #1,#2\rangle}}
	\newcommand{\pwedge}[2]{\ensuremath{\,#1\wedge#2\,}}
	\newcommand{\bs}[2]{\ensuremath{\left\langle #1,#2\right\rangle}}
	\newcommand{\scal}[2]{\ensuremath{\langle #1,#2\rangle}}
	\newcommand{\sg}[2]{\ensuremath{\left\langle #1,#2\right\rangle_{\mkern-3mu g}}}
	\newcommand{\n}{\ensuremath{\vec{n}}}
	\newcommand{\ens}[1]{\ensuremath{\left\{ #1\right\}}}
	\newcommand{\lie}[2]{\ensuremath{\left[#1,#2\right]}}
	\newcommand{\g}{\ensuremath{g}}
	\newcommand{\dzeta}{\ensuremath{\det\hphantom{}_{\kern-0.5mm\zeta}}}
	\newcommand{\e}{\ensuremath{\vec{e}}}
	\newcommand{\f}{\ensuremath{\vec{f}}}
	\newcommand{\ig}{\ensuremath{|\vec{\mathbb{I}}_{\phi}|}}
	\newcommand{\ik}{\ensuremath{\left|\mathbb{I}_{\phi_k}\right|}}
	\newcommand{\w}{\ensuremath{\vec{w}}}
	\newcommand{\hooklongrightarrow}{\lhook\joinrel\longrightarrow}
	\renewcommand{\tilde}{\ensuremath{\widetilde}}
	\newcommand{\vg}{\ensuremath{\mathrm{vol}_g}}
	\newcommand{\im}{\ensuremath{\mathrm{W}^{2,2}_{\iota}(\Sigma,N^n)}}
	\newcommand{\imm}{\ensuremath{\mathrm{W}^{2,2}_{\iota}(\Sigma,\R^3)}}
	\newcommand{\timm}[1]{\ensuremath{\mathrm{W}^{2,2}_{#1}(\Sigma,T\R^3)}}
	\newcommand{\tim}[1]{\ensuremath{\mathrm{W}^{2,2}_{#1}(\Sigma,TN^n)}}
	\renewcommand{\d}[1]{\ensuremath{\partial_{x_{#1}}}}
	\newcommand{\dg}{\ensuremath{\mathrm{div}_{g}}}
	\renewcommand{\Res}{\ensuremath{\mathrm{Res}}}
	\newcommand{\un}[2]{\ensuremath{\bigcup\limits_{#1}^{#2}}}
	\newcommand{\res}{\mathbin{\vrule height 1.6ex depth 0pt width
			0.13ex\vrule height 0.13ex depth 0pt width 1.3ex}}
    \newcommand{\antires}{\mathbin{\vrule height 0.13ex depth 0pt width 1.3ex\vrule height 1.6ex depth 0pt width 0.13ex}}
	\newcommand{\ala}[5]{\ensuremath{e^{-6\lambda}\left(e^{2\lambda_{#1}}\alpha_{#2}^{#3}-\mu\alpha_{#2}^{#1}\right)\left\langle \nabla_{\e_{#4}}\vec{w},\vec{\mathbb{I}}_{#5}\right\rangle}}
	\setlength\boxtopsep{1pt}
	\setlength\boxbottomsep{1pt}
	\newcommand\norm[1]{%
		\setbox1\hbox{$#1$}%
		\setbox2\hbox{\addvbuffer{\usebox1}}%
		\stretchrel{\lvert}{\usebox2}\stretchrel*{\lvert}{\usebox2}%
	}
	\allowdisplaybreaks
	\newcommand*\mcup{\mathbin{\mathpalette\mcapinn\relax}}
	\newcommand*\mcapinn[2]{\vcenter{\hbox{$\mathsurround=0pt
				\ifx\displaystyle#1\textstyle\else#1\fi\bigcup$}}}
	\def\Xint#1{\mathchoice
		{\XXint\displaystyle\textstyle{#1}}%
		{\XXint\textstyle\scriptstyle{#1}}%
		{\XXint\scriptstyle\scriptscriptstyle{#1}}%
		{\XXint\scriptscriptstyle\scriptscriptstyle{#1}}%
		\!\int}
	\def\XXint#1#2#3{{\setbox0=\hbox{$#1{#2#3}{\int}$ }
			\vcenter{\hbox{$#2#3$ }}\kern-.58\wd0}} 
	\def\ddashint{\Xint=}
	\newcommand{\dashint}[1]{\ensuremath{{\Xint-}_{\mkern-10mu #1}}}
	\newcommand\ccancel[1]{\renewcommand\CancelColor{\color{red}}\cancel{#1}}
	\newcommand\colorcancel[2]{\renewcommand\CancelColor{\color{#2}}\cancel{#1}}
	\newcommand{\abs}[1]{\left\lvert #1 \right \rvert}
	
	\renewcommand{\thetheorem}{\thesection.\arabic{theorem}}

 \section{Introduction}

 \subsection{Morse Stability in Geometric Analysis}

 Let $H$ be a Hilbert manifold modelled on a Hilbert space $H_0$, and let $E\in C^2(H,\R)$. Recall that $\varphi\in H$ is a critical point of $E$ if $DE(\varphi)=0$. Define a quadratic form $Q_{\varphi}$ on $T_{\varphi}H\simeq H_0$ such that for all $u\in T_{\varphi}H$, 
 \begin{align*}
     Q_{\varphi}(u)=D^2E(\varphi)(u,u).
 \end{align*}
 We define the Morse index of $\varphi$ by
 \begin{align*}
     \mathrm{Ind}_{E}(\varphi)=\dim\mathrm{Vec}_{\R}\left(T_{\varphi}\cap\ens{u:Q_{\varphi}(u)<0}\right)
 \end{align*}
 while the nullity of $\varphi$ is defined by
 \begin{align*}
     \mathrm{Null}_{E}(\varphi)=\dim\mathrm{Vec}_{\R}\left(T_{\varphi}\cap\ens{u:Q_{\varphi}(u)=0}\right).
 \end{align*}
 In favourable settings (harmonic maps between manifolds, conformally invariant problems in dimension $2$, Willmore immersions, biharmonic maps, etc), the index and nullity can be defined with respect to a differential operator. Assume that there exists a non-bounded linear operator $\mathscr{L}:\mathcal{D}\left(T_{\varphi}H\right)\rightarrow T_{\varphi}H$ such that
 \begin{align*}
     Q_{\varphi}(u)=\s{u}{\mathscr{L}u}_{H_0},
 \end{align*}
 we deduce that the index corresponds to the number of negative eigenvalues of $\mathscr{L}$, while the nullity is equal to the dimension of the Kernel of $\mathscr{L}$. Explicitly, for all $\lambda\in \R$, define
 \begin{align*}
     \mathscr{E}(\lambda)=\mathcal{D}\left(T_{\varphi}H\right)\cap\ens{u:\leb u=\lambda u},
 \end{align*}
 then
 \begin{align*}
     \mathrm{Ind}_{E}(\varphi)=\dim\bigoplus_{\lambda<0}\mathscr{E}(\lambda),
 \end{align*}
 while
 \begin{align*}
     \mathrm{Null}_{E}(\varphi)=\mathrm{dim}\,\mathscr{E}(0).
 \end{align*}
 If $\ens{\varphi_k}_{k\in \N}$ is a sequence of critical points of $E$ that converges strongly towards $\varphi\in E$, one can prove under suitable hypotheses that
 \begin{align*}
     \mathrm{Ind}_{E}(\varphi)\leq \liminf_{k\rightarrow \infty}\mathrm{Ind}_{E}(\varphi)\leq \limsup_{k\rightarrow \infty}\left(\mathrm{Ind}_{E}(\varphi_k)+\mathrm{Null}_{E}(\varphi_k)\right)\leq \mathrm{Ind}_{E}(\varphi)+\mathrm{Null}_{E}(\varphi).
 \end{align*}
 More generally, in the framework of bubbling convergence in geometric analysis, we also expect that this inequality holds. The lower semi-continuity of the Morse index is typically easier to prove and according to a general principle, the energy quantization suffices to establish the lower semi-continuity of the Morse index. In the case of classical minimal surfaces, it holds thanks to the well-known \emph{logarithmic cutoff trick}. In the setting of Almgren-Pitts theory, refer to \cite{marquesmorse} and in the case of the viscosity method for minimal surfaces, refer to \cite{lower}. On the other hand, upper semi-continuity results for the extended Morse index (equal to the sum of the Morse index and the nullity) require much more precise information on the convergence rate of critical points. See \cite{chodosh_mantoulidis} in the case of construction of minimal surfaces through the Allen-Cahn functional, \cite{zhou_xin,marques_lower_bound} in the case of the Almgren-Pitts functions, \cite{hao_yin1, hao_yin2}, and \cite{riviere_morse_scs} in the case of conformally invariant problems in dimension $2$ (see also \cite{lamm_hirsch}). \\

 Last year, F. Da Lio, M. Gianocca, and T. Rivière developed a new general approach to show upper semi-continuity results for the Morse index (\cite{riviere_morse_scs}). First developed in the case of conformally invariant problems in dimension $2$ (that include harmonic maps), together with Tristan Rivière, we extended the method to prove the upper semi-continuity of the Morse index of Willmore immersions (\cite{morse_willmore_I,eigenvalue_annuli}). In this article, we show how to generalise the theory to biharmonic maps dimension $4$. If the energy quantization implies (under suitable assumptions) the lower semi-continuity of the Morse index, Da Lio-Gianocca-Rivière showed that the improved energy quantization implies the upper semi-continuity of the extended Morse index. For more details on this method, refer to the introduction of \cite{morse_willmore_I} and to Section \ref{Section1.4} in the case of biharmonic maps. 

 \subsection{Morse Stability of Biharmonic Maps}

    Let $\Omega\subset \R^4$ be a bounded open subset of $\R^4$, and $M^m\subset \R^n$ be a $C^4$ submanifold of $\R^n$. We consider maps $u\in W^{2,2}(\Omega,M^m)$ and define the two biharmonic energies by 
    \begin{align}
        E(\varphi)=\frac{1}{2}\int_{\Omega}|\Delta u|^2dx,
    \end{align}
    and
    \begin{align}
        E_{0}(u)=\frac{1}{2}\int_{\Omega}\left|\left(\Delta u\right)^{\top}\right|^2dx.
    \end{align}
    A critical point of $E$ is called an (extrinsic) biharmonic map  and a critical point of $E_0$ an (intrinsic) biharmonic map. In this article, we are concerned with Morse stability results. 
    For all critical point $u\in W^{2,2}(B(0,1),M^m)$ of $E$, one easily shows that the second derivative $Q$ is given by
    \begin{align}\label{der2_biharmonique0}
        &Q_u(w)=\int_{\Omega}\bigg\{|\Delta w|^2+\s{A_u(\D u,\D u)}{(\Delta A_u)(w,w)+2(\D A_u)(\D w,w)+2\,A_u(\D w,\D w)+2\,A_u(w,\Delta w)}\nonumber\\
        &-2\s{\s{\Delta u}{\D P_u}}{(\D A_u)(w,w)+2\,A_u(\D w,w)}
        +\s{\s{\Delta(P_u)}{\Delta u}}{A_u(w,w)}\bigg\}dx.
    \end{align}
    Furthermore, one can show that there exists an elliptic, fourth-order differential operator $\mathscr{L}_u=\Delta^2+\mathrm{l.o.t.}$ such that 
    \begin{align}\label{d2_formula}
        Q_u(w)=\int_{\Omega}\bs{w}{\mathscr{L}_uw}dx.
    \end{align}
    In geometric application, one would typically consider biharmonic maps from a closed—\emph{i.e.} compact, without boundary—manifold $\Sigma^4$, in which case one define the index with respect to $W^{2,2}(\Sigma^4,\R^n)$ variations. However, since we restrict to biharmonic maps $u:B(0,1)\rightarrow M^m$, we will restrict to variations that vanish on the boundary. Explicitly, we define
    \begin{align*}
        W^{2,2}_0(B(0,1),\R^n)=\bar{\mathscr{D}(B(0,1),\R^n)}^{W^{2,2}}
    \end{align*}
    where for all $\Omega\subset \R^d$, the space of test function $\mathscr{D}(\Omega)=C^{\infty}_c(\Omega)$ is the space of smooth, compactly supported function, and 
    \begin{align*}
        \wp{v}{2,2}{B(0,1)}=\sqrt{\int_{B(0,1)}\left(|\D v|^2+|\D v|^2+|v|^2\right)dx}.
    \end{align*}
    A function $v\in W^{2,2}_0(B(0,1))$ can be extended by $0$ as a function $\bar{v}\in W^{2,2}(\R^4)=H^2(\R^4)$. Therefore, according to trace theory $H^{2}(\R^4)\hooklongrightarrow H^{\frac{3}{2}}(\partial B(0,1))=H^{\frac{3}{2}}(S^3)$, we have in the distributional sense $v=\partial_{\nu}v=0$ on $\partial B(0,1)$. Now, using \eqref{d2_formula}, we can define for all $\lambda\in \R$ the corresponding eigenspace as
    \begin{align*}
        \mathscr{E}(\lambda)=W^{2,2}_0(B(0,1))\cap\ens{u:\leb_uw=\lambda\,w}.
    \end{align*}
    Furthermore, an easy application of the spectral theorem shows that there exists an increasing sequence $\ens{\lambda_k}_{k\in \N}$ such that $\lambda_{k}\rightarrow \infty$ and $\mathscr{E}(\lambda)=\ens{0}$ for all $\lambda \in \R\setminus\ens{\lambda_k}_{k\in \N}$.
    The Morse index of the critical point $u$ is therefore defined as the number of negative eigenvalues (counted with multiplicity):
    \begin{align*}
        \mathrm{Ind}_E(u)=\dim\bigoplus_{\lambda\leq 0}\mathscr{E}(\lambda),
    \end{align*}
    while the nullity is defined by 
    \begin{align*}
        \mathrm{Null}_E(u)=\dim\mathscr{E}(0).
    \end{align*}
    Our main result is the following one.
    \begin{theorem}\label{th:Main}
        Let $\ens{u_k}_{k\in \N}:B(0,1)\rightarrow (M^m,h)$ be a sequence of biharmonic maps such that
        \begin{align*}
            \limsup_{k\rightarrow \infty}E(u_k)<\infty.
        \end{align*}
        Then, if $\ens{u_k}_{k\in\N}$ bubble converges towards $(u_{\infty},v_1,\cdots,v_N)$, we have
        \begin{align*}
            \mathrm{Ind}_E(u_{\infty})+\sum_{i=1}^N\mathrm{Ind}_E(v_i)\leq \liminf_{k\rightarrow \infty}\mathrm{Ind}_E(u_k)\leq \limsup_{k\rightarrow \infty}\mathrm{Ind}_E^0(u_k)\leq \mathrm{Ind}_E^0(u_{\infty})+\sum_{i=1}^N\mathrm{Ind}_E^0(v_i),
        \end{align*}
        where $\mathrm{Ind}_E^0=\mathrm{Ind}+\mathrm{Null}$.
    \end{theorem}
    Refer to \cite{sacks}, \cite{parker}, and the introduction of \cite{biharmonic_quanta} for the definition of bubble convergence. 
    We point out that as a by-product of our analysis, we show that a strong energy quantization for biharmonic maps holds. See Corollary \ref{cor_quanta} for more details. The proof is only made in the case of extrinsic biharmonic maps, but as one can see in \cite{riviere_lamm_biharmonic,biharmonic_quanta}, the analysis is trivially modified in the case of intrinsic biharmonic maps. For example, the Pohozaev identity is unchanged, and the second derivative only changes by sub-critical terms that can be absorbed to get inequality \eqref{elementary_d2_lower_bound}.

    \subsection{Description of the Proof}

    The main core of the proof is to show that negative variations cannot be localised in neck regions, lest negative variations vanish in the limit. Using the $\epsilon$-regularity, if $\Omega_k(\alpha)=B_{\alpha}\setminus\bar{B}_{\alpha^{-1}\rho_k}(0)$ is a neck region, we deduce that for some universal constant $C<\infty$, we have
 \begin{align*}		
 	|x|^2|\D^2\varphi_k(x)|+|x||\D \varphi_k(x)|\leq C\left(\np{\D\varphi_k}{2}{\Omega_k(\alpha)}+\np{\D\varphi_k}{4}{\Omega_k(\alpha)}\right).
 \end{align*}
 On the other hand, one easily shows that for all $u\in W^{2,2}_0(\Omega_k(\alpha))$, we have
 \begin{align}\label{elementary_d2_lower_bound}
 	&Q_{\varphi_k}(u)\geq \int_{\Omega_k(\alpha)}|\Delta u|^2dx-C\int_{\Omega_k(\alpha)}|\D \varphi_k|^2|u||\Delta u|dx-C\int_{\Omega_k(\alpha)}\left(|\D \varphi_k|^2+|\Delta \varphi_k|\right)|\D u|^2dx\nonumber\\
 	&-C\int_{\Omega_k(\alpha)}\left(|\Delta \varphi_k|^2+|\D \varphi_k|^2\right)|u|^2dx\nonumber\\
 	&\geq (1-\epsilon)\int_{\Omega_k(\alpha)}|\Delta u|^2dx-C\int_{\Omega}\left(|\D \varphi_k|^2+|\Delta \varphi_k|\right)|\D u|^2dx-C\int_{\Omega_k(\alpha)}|\Delta \varphi_k|^2|u|^2dx\nonumber\\
 	&-\frac{C}{\epsilon}\int_{\Omega_k(\alpha)}|\D \varphi_k|^4|u|^2dx.
 \end{align}
 Combining the two results, we deduce (taking $\epsilon=1/2$)
 \begin{align}
 	Q_{\varphi_k}(u)\geq \frac{1}{2}\int_{\Omega_k(\alpha)}|\Delta u|^2dx-C\int_{\Omega_k(\alpha)}\frac{|\D u|^2}{|x|^2}dx-C\int_{\Omega_k(\alpha)}\frac{u^2}{|x|^4}dx
 \end{align}
 that we will have to improve after obtaining more precise estimates on $\varphi_k$. 
 Now, recall the two inequalities from \cite{eigenvalue_annuli}.
 
 \begin{theorem}[Theorem $5.3$ p. $94$, \cite{eigenvalue_annuli}]\label{th:bound_biharmonic0}
 	Let $0<a<b<\infty$, let $\Omega=B_b\setminus\bar{B}_a(0)\subset\R^4$, and assume that
 	\begin{align}\label{large_conf_class}
 		\log\left(\frac{b}{a}\right)\geq \frac{15\sqrt{4+3\pi(\pi+1)}}{2}.
 	\end{align}
 	Then, for all $u\in W^{2,2}_0(\Omega)$, we have
 	\begin{align}\label{th:bound_biharmonic0:ineq}
 		\int_{\Omega}(\Delta u)^2dx\geq \left(4+\frac{\pi^2}{\log^2\left(\frac{b}{a}\right)}\right)\frac{\pi^2}{\log^2\left(\frac{b}{a}\right)}\int_{\Omega}\frac{u^2}{|x|^4}dx.
 	\end{align}
 \end{theorem}
 
 \begin{theorem}\label{th:bound_biharmonic1}
 	Let $0<a<b<\infty$, $\Omega=B_b\setminus\bar{B}_a(0)$ and let $u\in W^{2,2}(\Omega)$. Then, provided that
 	\begin{align}
 		&\log\left(\frac{b}{a}\right)\geq
 		\sup_{n\geq 1}\frac{\pi\sqrt{\big(8+204\,n(n+2)\big)+\sqrt{\big(8+204\,n(n+2)\big)^2+4\,n(n+2)+16\,n^2(n+2)^2}}}{\sqrt{32\,n(n+2)+16\,n^2(n+2)^2}}
 	\end{align}
 	we have
 	\begin{align}\label{th:bound_biharmonic1:ineq}
 		\int_{\Omega}(\Delta u)^2dx\geq \frac{9+\dfrac{10\pi^2}{\log^2\left(\frac{b}{a}\right)}+\dfrac{\pi^4}{\log^2\left(\frac{b}{a}\right)}}{3+\dfrac{\pi^2}{\log^2\left(\frac{b}{a}\right)}}\int_{\Omega}\frac{|\D u|^2}{|x|^2}dx.
 	\end{align}
 \end{theorem}
 Therefore, applying the pointwise estimate, for $k$ large enough, we get
 \begin{align}\label{ineq_neck_naive}
 	&Q_{\varphi_k(u)}\geq \left(\frac{1}{2}\left(4+\frac{\pi^2	}{\log^2\left(\frac{\alpha^2
}{\rho_k}\right)}\right)\frac{\pi^2}{\log^2\left(\frac{\alpha^2	}{\rho_k}\right)}-C\left(\np{\D^2\varphi_k}{2}{\Omega_k(\alpha)}+\np{\D \varphi_k}{4}{\Omega_k(\alpha)}\right)\right)\int_{\Omega_k(\alpha)}\frac{|u|^2}{|x|^4}dx\nonumber\\
&+\left(3-C\left(\np{\D^2\varphi_k}{2}{\Omega_k(\alpha)}+\np{\D \varphi_k}{4}{\Omega_k(\alpha)}\right)\right)\int_{\Omega_k(\alpha)}\frac{|\D u|^2}{|x|^2}dx.
 \end{align}
 Thanks to the energy quantization, we have
 \begin{align*}
 	\lim\limits_{\alpha\rightarrow 0}\limsup_{k\rightarrow \infty}\left(\np{\D^2\varphi_k}{2}{\Omega_k(\alpha)}+\np{\D \varphi_k}{4}{\Omega_k(\alpha)}\right)=0.
 \end{align*}
 Although this estimate is sufficient to obtain the positivity of the second member of the right-hand side of \eqref{ineq_neck_naive}, the sign of the first member is unclear. However, if we strengthen the estimate into the following Hölder-like inequality for some $0<\beta<1$
 \begin{align*}
 	|x|^2|\D^2\varphi_k(x)|+|x||\D\varphi_k(x)|\leq C\left(\left(\frac{|x|	}{b}\right)^{\beta}+\left(\frac{a}{|x|}\right)^{\beta}+\frac{\Lambda}{\log\left(\frac{\alpha}{\rho_k}\right)}\right),
 \end{align*}
 then, we can prove the neck regions contribute positively to the Morse index thanks to the following weighted Poincaré estimate (provided that $\Lambda$ is small enough).
 \begin{theorem}[Theorem $5.7$ p. $99$, \cite{eigenvalue_annuli}]\label{th:bound_biharmonic2}
 	Let $0<a<b<\infty$ and let $\Omega=B_b\setminus\bar{B}_a(0)\subset \R^4$. Then, for all $0<\beta<\infty$, there exists a constant $0<C_{\beta}<\infty$ such that for all $u\in W^{2,2}_0(\Omega)$, we have
 	\begin{align}\label{th:bound_biharmonic2:ineq}
 		\int_{\Omega}(\Delta u)^2dx\geq C_{\beta}\int_{\Omega}\left(\frac{u^2}{|x|^4}+\frac{|\D u|^2}{|x|^2}\right)\left(\left(\frac{|x|}{b}\right)^{\beta}+\left(\frac{a}{|x|}\right)^{\beta}\right)dx.
 	\end{align}
 \end{theorem}
   In order to obtain this improved pointwise estimate, we adapt the dyadic analysis of \cite{riviere_morse_scs} and \cite{morse_willmore_I} to two different kinds of divergence equations in dimension $4$, which poses no major technical difficulty. However, showing the smallness of $\Lambda$ is equivalent to the following improved energy quantization
   \begin{align*}
   	\lim\limits_{\alpha\rightarrow 0}\limsup_{k\rightarrow \infty}\left(\np{\D^2\varphi_k}{2,1}{\Omega_k(\alpha)}+\np{\D\varphi_k}{4,1}{\Omega_k(\alpha)}\right)=0,
   \end{align*}
   which is much more technical than the one obtained in \cite{riviere_morse_scs} or \cite{pointwise}. In the former case, thanks to the angular energy quantization, the holomorphy of the Hopf differential furnishes the missing estimate for the radial part of the gradient, while the specific structure of the Willmore equation in divergence form (reformulated as a Jacobian system in higher dimension) allows one to obtain the improved energy quantization more directly (\cite{pointwise}). However, in the case of biharmonic maps, we have no holomorphic quantity at hand and contrary to Willmore surfaces, it cannot be reformulated into a purely second-order system. Our new argument is general in nature and allows us to recover energy quantization results in the case of conformally invariant Lagrangians in dimension $2$ and the Yang-Mills functional. Furthermore, it was recently use in \cite{p_harmonic_da_lio_riviere}, and we believe that it should be applicable to various new settings, especially when one uses Sacks-Uhlenbeck approximations (or in the viscosity method), that break the special structure giving holomorphic objects. 
   
 \subsection{Da Lio-Gianocca-Rivière's Theory}\label{Section1.4}

 This pioneering analysis was first applied to harmonic maps, or more generally, conformally invariant problems (\cite{riviere_morse_scs}). The proof is based on four main steps:
 \begin{enumerate}
     \item \textbf{A strong energy quantization in the Lorentz space $L^{2,1}$.} An earlier principle discovered by F. Lin and T. Rivière (\cite{linriv,linriv_GL1,linriv_GL2}) shows that a bound on the $L^{2,1}$ norm grants under standard assumptions an energy quantization result for a sequence of critical points of uniformly bounded energy. Indeed, the energy quantization is equivalent to the no-neck energy property. Necks are annular region that link the macroscopic map to its bubbles. The limit energy is the energy of the limit map and the bubbles if and only if the energy in neck region vanish. Explicitly, a neck region can be defined as $\Omega_k(\alpha)=B_{\alpha}\setminus\bar{B}_{\alpha^{-1}\rho_k}(0)$ where $\ens{\rho_k}_{k\in\N}\subset (0,\infty)$ and $\rho_k\conv{k\rightarrow\infty}$. In the case of extrinsic biharmonic maps, the no-neck energy is equivalent to 
     \begin{align}\label{eq:energy_quantization}
         \lim_{\alpha\rightarrow 0}\limsup_{k\rightarrow \infty}\np{\Delta u_k}{2}{\Omega_k(\alpha)}=0.
     \end{align}
     Since the energy in neck region does not concentrate on dyadic annuli of neck regions, it is not difficult to show the weaker estimate
     \begin{align}\label{eq:weak_energy_quantization}
         \lim_{\alpha\rightarrow 0}\limsup_{k\rightarrow\infty}\np{\Delta u_k}{2,\infty}{\Omega_k(\alpha)}=0.
     \end{align}
     If $\np{\Delta u_k}{2,1}{\Omega_k(\alpha)}$ is uniformly bounded, the $L^{2,1}/L^{2,\infty}$ duality shows that
     \begin{align}
         \int_{\Omega_k(\alpha)}|\Delta u_k|^2dx\leq \np{\Delta u_k}{2,1}{\Omega_k(\alpha)}\np{\Delta u_k}{2,\infty}{\Omega_k(\alpha)}\leq C\np{\Delta u_k}{2,\infty}{\Omega_k(\alpha)}
     \end{align}
     which implies the no-neck energy property \eqref{eq:energy_quantization} by virtue of the energy quantization in the weak Lebesgue space $L^{2,\infty}$ (for a precise definition of those spaces together with $L^{2,1}$, refer to the appendix) given in \eqref{eq:weak_energy_quantization}. The proof of the $L^{2,1}$ energy quantization is the most analytically relevant part of the article as it applies to a variety of geometric equations (conformally invariant problems that include harmonic maps, biharmonic maps, Yang-Mills connections) as exemplified in the recent article of F. Da Lio and T. Rivière on $p$-harmonic maps (\cite{p_harmonic_da_lio_riviere}).
     \item \textbf{A Hölder-type estimate in neck regions.} As we saw above, a trivial estimate
     \begin{align*}
         |x|^2\left(|\Delta u|+|\D u|^2\right)\leq C\qquad\text{for all}\;\, x\in \Omega_{\frac{1}{2}}=B_{\frac{b}{2}}\setminus\bar{B}_{2a}(0).
     \end{align*}
     that follows from the $\epsilon$-regularity is not enough to show that the variations localised in neck regions give a positive contribution to the second derivative. One needs to refine this estimate as 
     \begin{align*}
         |x|^2\left(|\Delta u|+|\D u|^2\right)\leq C\left(\left(\frac{|x|}{b}\right)^{\beta}+\left(\frac{a}{|x|}\right)^{\beta}+\frac{1}{\log\left(\frac{b}{a}\right)}\right)\qquad\text{for all}\;\, x\in \Omega_{\frac{1}{2}}
     \end{align*}
     for some $\beta>0$.
     \item \textbf{Rellich and Hardy-Rellich Inequalites in Neck Regions.} Those inequalities that we discussed above allow one, taking advantage of the improved estimate in the neck regions, to show that neck regions contribute positively to the second derivative. Explicitly, if $\Omega=B_b\setminus\bar{B}_a(0)\subset \R^4$, we show that for all biharmonic map $u:\Omega\rightarrow M^m$, and for all $v\in W^{2,2}_0(\Omega_{\frac{1}{2}})$, we have for some $\lambda_1<\infty$
     \begin{align}\label{neck_positive0}
 	     &Q_{u}(v)\geq \left(\frac{1}{2}-\lambda_1\left(\np{\D^2u}{2}{\Omega}+\np{\D u}{4}{\Omega}\right)\right)\int_{\Omega_{\frac{1}{2}}}\frac{|\D v|^2}{|x|^2}dx\nonumber\\
          &+\left(2\pi^2-\lambda_1\left(\np{\D^2u}{2,1}{\Omega}+\np{\D u}{4,1}{\Omega}\right)\right)\frac{1}{\log^2\left(\frac{b}{a}\right)}\int_{\Omega_{\frac{1}{2}}}\frac{|v|^2}{|x|^4}dx\nonumber\\
          &+\left(\frac{C_{2\beta}}{12}-\lambda_1\left(\np{\D^2u}{2,1}{\Omega}+\np{\D u}{4,1}{\Omega}\right)\right)\int_{\Omega_{\frac{1}{2}}}\frac{|v|^2}{|x|^4}\left(\left(\frac{|x|}{b}\right)^{2\beta}+\left(\frac{a}{|x|}\right)^{2\beta}\right)dx
     \end{align}
     In particular for a sequence $\ens{u_k}_{k\in \N}$ of biharmonic maps of bounded energy, if $\Omega_k(\alpha)=B_{\alpha}\setminus\bar{B}_{\alpha^{-1}\rho_k}(0)$ is a sequence of associated neck regions, there exists $\alpha_0>0$ such that for all $0<\alpha<\alpha_0$ and for all $k\in \N$, the following inequality holds for all variation $v\in W^{2,2}_0(\Omega_k(\alpha))$:
     \begin{align}\label{stability_inequality}
        Q_{u_k}(v)&\geq \frac{1}{4}\int_{\Omega_k(\alpha)}\frac{|\D v|^2}{|x|^2}dx
        +\frac{\pi^2}{\log^2\left(\frac{\alpha^2}{\rho_k}\right)}\int_{\Omega_k(\alpha)}\frac{|v|^2}{|x|^4}dx
        \nonumber\\
        &+\frac{C_{2\beta}}{12}\int_{\Omega_k(\alpha)}\frac{|v|^2}{|x|^4}\left(\left(\frac{|x|}{\alpha}\right)^{2\beta}+\left(\frac{\alpha^{-1}\rho_k}{|x|}\right)^{2\beta}\right)dx.
     \end{align}
     \item \textbf{Sylvester's Law of Inertia.} If 
     \begin{align*}
         \omega_{\alpha,k}(x)=\frac{1}{|x|^4}\left(\left(\frac{|x|}{\alpha}\right)^{2\beta}+\left(\frac{\alpha^{-1}\rho_k}{|x|}\right)^{2\beta}+\frac{1}{\log^2\left(\frac{\alpha^2}{\rho_k}\right)}\right)
     \end{align*}
     is continuously extended by constant functions on $B(0,1)$, \eqref{stability_inequality} can be written in the weaker form
     \begin{align}\label{stable}
         Q_{u_k}(v)\geq \frac{1}{4}\int_{\Omega_k(\alpha)}\frac{|\D v|^2}{|x|^2}dx+\lambda_0\int_{\Omega_k(\alpha)}|v|^2\,\omega_{\alpha,k}\,dx
     \end{align}
     for some constant $0<\lambda_0<\infty$. Using Sylvester's Law of Inertia, both operators $\leb_{u_k}$ and $\leb_{\alpha,k}=\omega_{\alpha,k}^{-1}\leb_{u_k}$ have the same Morse index. Finally, if $w_k\in W^{2,2}_0(B(0,1))$ is a negative eigenvalue of $\leb_{\alpha,k}$, it is easy to show that the normalisation 
     \begin{align*}
         \int_{B(0,1)}|w_k|^2\omega_{\alpha,k}\,dx=1
     \end{align*}
     together with a uniform lower bound of the smallest eigenvalue shows that $\ens{w_k}_{k\in \N}$ is bounded is $W^{2,2}_0(B(0,1))$, and converges weakly in $W^{2,2}_0(B(0,1))$ towards a negative or null eigenvalue. Finally, using the stability inequality \eqref{stable}, one can show that $w_k$ is not localised in the neck region $\Omega_k(\alpha)$, which finally implies that it can be attached to a negative or zero variation of the limiting map $u_{\infty}$ or the bubbles, which finally implies the stability inequality.
 \end{enumerate}

 \subsection{Previously Known Results}

 The analytic study of biharmonic maps, following the work of F. Hélein on harmonic maps (\cite{heleinsymetrique,helein_general,helein}), was initiated by S.-Y. A. Chang, L. Wang and P. C. Yang  (\cite{biharmic_chang_wang_yang}) in the case of the sphere-valued maps, for which one can take advantage of Noether theorem to rewrite the equation in divergence form as in the case of harmonic maps. Late work of C. Wang (\cite{biharmonic_wang1,biharmonic_wang2,biharmonic_wang3}) continued and generalised the previous analysis (that notably included an analysis of the biharmonic heat flow), showing an energy quantization for sphere-valued map, and a regularity theorem in special cases (see also the previous contribution of P. Stzrelecki \cite{strzelecki}). Finally, T. Lamm and T. Rivière (\cite{riviere_lamm_biharmonic}) managed to obtain a general regularity result using Rivière's idea on conformally invariant variational problems (\cite{riviereconservation}) that allowed them to rewrite every general biharmonic-type equation into divergence form. 
 
 Wang's energy quantization result was later generalised later in the case of intrinsic biharmonic maps by P. Hornung and R. Moser (\cite{biharmonic_hornung_moser}) and finally treated in general by P. Laurain and T. Rivière (\cite{biharmonic_quanta}; see also \cite{biharmonic_wang4}). This is the latter approach that we adopt in this article. 

 \subsection{New Developments in the Morse Index Stability and Future Work}

 It is likely that our approach, especially the general method to obtain strong energy quantization, applies to more general problem, as it was already shown in \cite{p_harmonic_da_lio_riviere} in the case of $p$-harmonic maps. Furthermore, it is likely that it applies to polyharmonic maps in critical dimension too (see \cite{polyharmonic,moser_polyharmonic}), since one need only a Pohoazev identity, that always holds for such problems.  

    \textbf{Acknowledgments.} I thank Tristan Rivière for his unwavering support and many fruitful discussions.

    \section{Second Variation of the Biharmonic Energies}

    Let $\Omega\subset \R^4$ be a bounded open subset of $\R^4$, and $M^m\subset \R^n$ be a $C^4$ submanifold of $\R^n$. We consider maps $u\in W^{2,2}(\Omega,M^m)$ and define the two biharmonic energies by 
    \begin{align}
        E(\varphi)=\frac{1}{2}\int_{\Omega}|\Delta u|^2dx,
    \end{align}
    and
    \begin{align}
        E_{0}(u)=\frac{1}{2}\int_{\Omega}\left|\left(\Delta u\right)^{\top}\right|^2dx.
    \end{align}
    For all $\delta>0$, define
    \begin{align*}
        M_{\delta}=\R^n\cap\ens{x:\mathrm{dist}(x,M^m)<\delta}.
    \end{align*}
    For $\delta>0$ small enough, the nearest point projection $\Pi:M_{\delta}\rightarrow M$ is well-defined and a $C^4$ map (\textbf{3.1.20} \cite{federer}). For all $x\in M$, let $P(x)=\D\Pi(x):\R^d\rightarrow T_xM$ be the orthogonal projection, and $P^{N}(x)=\mathrm{Id}-\D\Pi_N(x):\R^d\rightarrow (T_xM)^{N}$. Those two maps are of class $C^3$, and if $A$ is the second fundamental form of the isometric immersion $\iota: M^m\rightarrow \R^n$, it is a $C^2$ map and we have
    \begin{align*}
        A_p(X,Y)=D_{X}P(x)(Y)\quad \text{for all}\;\, p\in M^m,\;\;\text{for all}\;\, X,Y\in T_xM^m,
    \end{align*}
    where $D$ is the pull-back of the Levi-Civita connection $\D$ on $M^m$. Let us recall that this is a symmetric map of $X,Y$ (this fact follows easily once one defines the normal connection on $M^m$). Let $v\in W^{2,2}(\Omega,\R^n)$, and consider the variation 
    \begin{align*}
        u_t=\pi_N(u+t\,v).
    \end{align*}
    Then, as $P$ is a $C^2$ map, we have by the Taylor formula
    \begin{align*}
        u_t=u+t\,\D\Pi(u)(v)+\frac{1}{2}t^2\,\D^2\Pi(u)(v,v)+o(t^2)=u+t\,w+\frac{1}{2}t^2d_wP_u(v)+o(t^2)
    \end{align*}
    where $w=P_u(v)=d\,\Pi_u(v)$.
    Therefore, we have
    \begin{align*}
        E(u_t)&=\frac{1}{2}\int_{\Omega}\left|\Delta u+t\,\Delta w+\frac{1}{2}t^2d_wP_u(v)\right|^2dx\\
        &=\frac{1}{2}\int_{\Omega}\left|\Delta u\right|^2dx+t\int_{\Omega}\s{\Delta u}{\Delta w}dx+\frac{1}{2}t^2\int_{\Omega}\left(|\Delta w|^2+\s{\Delta u}{\Delta\left(d_wP_u(v)\right)}\right)dx.
    \end{align*}
    In particular, we deduce that 
    \begin{align*}
        DE(u)(w)=\int_{\Omega}\s{\Delta u}{\Delta \left(P_u(v)\right)}dx=\int_{\Omega}\s{\Delta^2u}{P_u(v)}dx,
    \end{align*}
    which shows that $u$ is biharmonic if and only if $\left(\Delta^2u\right)^{\top}=0$. Likewise, we have
    \begin{align*}
        D^2E(u)(w,w)=\int_{\Omega}\left(|\Delta w|^2+\s{\Delta^2u}{d_wP_u(v)}\right)dx.
    \end{align*}
    Now, using the classical equation
    \begin{align*}
        \Delta^2u=\Delta\left(A_u(\D u,\D u)\right)+2\,\dive\left(\s{\Delta u}{\D P_u}\right)-\s{\Delta(P_u)}{\Delta u},
    \end{align*}
    we deduce that
    \begin{align*}
        Q_u(w)=\int_{\Omega}\left(|\Delta w|^2+\s{\Delta\left(A_u(\D u,\D u)\right)+2\,\dive\left(\s{\Delta u}{\D P_u}\right)-\s{\Delta(P_u)}{\Delta u}}{A_u(w,w)}\right)dx.
    \end{align*}
    If $\Omega$ is a closed Riemannian manifold or $w\in W^{2,2}_0(\Omega)$, then
    \begin{align*}
        &\int_{\Omega}\s{\Delta\left(A_u(\D u,\D u)\right)}{A_u(w,w)}dx\\
        &=\int_{\Omega}\s{A_u(\D u,\D u)}{(\Delta A_u)(w,w)+2(\D A_u)(\D w,w)+2\,A_u(dw,dw)+2\,A_u(w,\Delta w)}dx.
    \end{align*}
    Then, we have
    \begin{align*}
        \int_{\Omega}\s{\dive\left(\s{\Delta u}{\D P_u}\right)}{A_u(w,w)}dx=-\int_{\Omega}\s{\s{\Delta u}{\D P_u}}{(\D A_u)(w,w)+2\,A_u(\D w,w)}dx.
    \end{align*}
    Therefore, we finally get
    \begin{align}\label{der2_biharmonique}
        &Q_u(w)=\int_{\Omega}\bigg\{|\Delta w|^2+\s{A_u(\D u,\D u)}{(\Delta A_u)(w,w)+2(\D A_u)(\D w,w)+2\,A_u(\D w,\D w)+2\,A_u(w,\Delta w)}\nonumber\\
        &-2\s{\s{\Delta u}{\D P_u}}{(\D A_u)(w,w)+2\,A_u(\D w,w)}
        +\s{\s{\Delta(P_u)}{\Delta u}}{A_u(w,w)}\bigg\}dx.
    \end{align}
    Since $A\in C^2(N_{\delta})$ and $P\in C^3(N_{\delta})$, we deduce that there exists a universal constant $C<\infty$ (depending only on $N$) such that
    \begin{align}\label{pointwise_classical1}
    \left\{\begin{alignedat}{2}
        |\D A_u|+|\D P_u|&\leq C|\D u|\\
        |\Delta A_u|+|\Delta P_u|&\leq C\left(|\Delta u|+|\D u|^2\right)\\
        \end{alignedat}\right.
    \end{align}
    Therefore, we have
    \begin{align}\label{pointwise_classical2}
    \left\{\begin{alignedat}{1}
        &\left|\s{A_u(\D u,\D u)}{(\Delta A_u)(w,w)}\right|\leq C\left(|\D u|^4+|\D u|^2|\Delta u|\right)|w|^2\\
        &2\left|\s{A_u(\D u,\D u)}{(\D A_u)(\D w,w)}\right|\leq C|\D u|^3|\D w||w|\leq C|\D u|^4|w|^2+C|\D u|^2|\D w|^2\\
        &2\left|\s{A_u(\D u,\D u)}{A_u(\D w,\D w)}\right|\leq C|\D u|^2|\D w|^2\\
        &2\left|\s{A_u(\D u,\D u)}{A_u(w,\Delta w)}\right|\leq C|\D u|^2|w||\Delta w|\\
        &2\left|\s{\s{\Delta u}{\D P_u}}{(\D A_u)(w,w)}\right|\leq C|\D u|^2|\Delta u||w|^2\\
        &4\left|\s{\s{\Delta u}{\D P_u}}{A_u(\D w,w)}\right|\leq C|\D u||\Delta u||\D w||w|\leq C|\D u|^2|\Delta w||w|^2+|\Delta u||\D w|^2\\
        &2\left|\s{\s{\Delta(P_u)}{\Delta u}}{A_u(w,w)}\right|\leq C\left(|\Delta u|^2+|\D u|^2|\Delta u|\right)|w|^2
        \end{alignedat}\right.
    \end{align}
    and finally, by \eqref{der2_biharmonique}, \eqref{pointwise_classical1}, and \eqref{pointwise_classical2}, we deduce that
    \small
    \begin{align}\label{ineq_d2}
        &Q_u(w)\geq \int_{\Omega}|\Delta w|^2dx-C\int_{\Omega}|\D u|^2|w||\Delta w|dx-C\int_{\Omega}\left(|\D u|^2+|\Delta w|\right)|\D w|^2-C\int_{\Omega}\left(|\Delta u|^2+|\D u|^2\right)|w|^2dx\nonumber\\
        &\geq (1-\epsilon)\int_{\Omega}|\Delta w|^2dx-C\int_{\Omega}\left(|\D u|^2+|\Delta u|\right)|\D w|^2dx-C\int_{\Omega}|\Delta u|^2|w|^2dx
        -\frac{C}{\epsilon}\int_{\Omega}|\D u|^4|w|^2dx
    \end{align}
    \normalsize
    for all $0<\epsilon<1$.

     \section{Properties of Harmonic Functions in Higher Dimension}

     \subsection{Generalities and Lebesgue Space Estimates}

    Recall that thanks to \cite[p. $140$]{stein_weiss}, the dimension of the space $\mathscr{H}_n$ of spherical harmonics of degree $d$ on $\R^d$ is given by
    \begin{align}\label{dim_spherical_harmonics}
        \mathrm{dim}(\mathscr{H}_n)=\binom{n+d-1}{d-1}-\binom{n+d-3}{d-1}=\binom{n+d-1}{n}-\binom{n+d-3}{n-2}=N_d(n).
    \end{align}
    Furthermore, we have for all $n\in \N$
    \begin{align*}
        \Pi_{\lambda_n}(\Delta)=\p{r}^2+\frac{d-1}{r}\p{r}-\frac{\lambda_n}{r^2}
    \end{align*}
    where $\lambda_n=n(n+d-2)$ are the eigenvalues of $-\Delta_{S^{d-1}}$ on $S^{d-1}$. Now, if $f$ solve the equation
    \begin{align*}
        f''+\frac{d-1}{r}f'-\frac{\lambda_n}{r^2}f=0,
    \end{align*}
    then, making the change of variable $f(r)=Y(\log(r))$, we get
    \begin{align*}
        &f'=\frac{1}{r}Y'\\
        &f''=\frac{1}{r^2}\left(Y''-Y'\right),
    \end{align*}
    which implies that
    \begin{align*}
        0&=f''+\frac{d-1}{r}f'-\frac{\lambda_n}{r^2}f=\frac{1}{r^2}\left(Y''+(d-2)\,Y'-\lambda_nY\right).
    \end{align*}
    The characteristic polynomial is given by 
    \begin{align*}
        P(X)=X^2+(d-2)X-n(n+d-2)=(X-n)(x+n+d-2),
    \end{align*}
    which shows that a general harmonic function on an annulus $\Omega=B_b\setminus\bar{B}_a(0)$ admits the following expansion
    \begin{align*}
        u(r,\omega)=\sum_{n=0}^{\infty}\sum_{k=1}^{N_d(n)}\left(a_{n,k}\,r^n+b_{n,k}\,r^{-(n+d-2)}\right)Y_n^k(\omega).
    \end{align*}
    In particular, we get as the spherical harmonics are an orthogonal base of $L^2(S^{d-1})$ (for $d\geq 5$):
    \begin{align*}
        &\int_{\Omega}|u|^2dx=\beta(d)\sum_{n=0}^{\infty}\sum_{k=1}^{N_d(n)}\int_{a}^b\left(|a_{n,k}|^2r^{2n}+|b_{n,k}|^2r^{-2(n+d-2)}+2\,a_{n,k}\,b_{n,k}r^{-(d-2)}\right)r^{d-1}\,dr\\
        &=\beta(d)\sum_{n=0}^{\infty}\sum_{k=1}^{N_d(n)}\frac{|a_{n,k}|^2}{2n+d}b^{2n+d}\left(1-\left(\frac{a}{b}\right)^{2n+d}\right)+\beta(d)\sum_{n=0}^{\infty}\sum_{k=1}^{N_d(n)}\frac{|b_{n,k}|^2}{2n+d-4}\frac{1}{a^{2n+d-4}}\left(1-\left(\frac{a}{b}\right)^{2n+d-4}\right)\\
        &+\beta(d)\sum_{n=0}^{\infty}\sum_{k=1}^{N_d(n)}a_{n,k}\,b_{n,k}\,b^2\left(1-\left(\frac{a}{b}\right)^2\right),
    \end{align*}
    where $\beta(d)=\mathscr{H}^{d-1}(S^{d-1})=\dfrac{2\pi^{\frac{d}{2}}}{\Gamma\left(\frac{d}{2}\right)}$ is the volume of $S^{d-1}$, while for $d=4$, we get
    \begin{align*}
        &\int_{\Omega}|u|^2dx=\pi^2\sum_{n=0}^{\infty}\sum_{k=1}^{(n+1)^2}\frac{|a_{n,k}|^2}{n+2}b^{2(n+2)}\left(1-\left(\frac{a}{b}\right)^{2(n+2)}\right)+2\pi^2|b_{0,1}|^2\log\left(\frac{b}{a}\right)\\
        &+\pi^2\sum_{n=1}^{\infty}\sum_{k=1}^{(n+1)^2}\frac{|b_{n,k}|^2}{n}\frac{1}{a^{2n}}\left(1-\left(\frac{a}{b}\right)^{2n}\right)+2\pi^2\sum_{n=0}^{\infty}\sum_{k=1}^{N_d(n)}a_{n,k}\,b_{n,k}\,b^{2}\left(1-\left(\frac{a}{b}\right)^2\right),
    \end{align*}
    and for $d=3$, we have 
    \begin{align*}
        &\int_{\Omega}|u|^2dx=4\pi \sum_{n=0}^{\infty}\sum_{k=1}^{2n+1}\frac{|a_{n,k}|^2}{2n+3}b^{2n+3}\left(1-\left(\frac{a}{b}\right)^{2n+3}\right)+4\pi |b_{0,1}|^2\,b\left(1-\left(\frac{a}{b}\right)\right)\\
        &+4\pi\sum_{n=1}^{\infty}\sum_{k=1}^{2n+1}\frac{|b_{n,k}|^2}{2n-1}\frac{1}{a^{2n-1}}\left(1-\left(\frac{a}{b}\right)^{2n-1}\right)
        +4\pi\sum_{n=0}^{\infty}\sum_{k=1}^{2n+1}a_{n,k}\,b_{n,k}\,b^2\left(1-\left(\frac{a}{b}\right)^2\right).
    \end{align*}
    Now, let us find a lower estimate of each term in the first two series. For all $d\geq 3$, for all $n\geq 0$ (such that $(d,n)\notin\ens{(3,0),(4,0)}$), and for all $0<\epsilon<1$ we have
    \begin{align*}
        \left|a_{n,k}\,b_{n,k}\,b^2\right|\leq \frac{\epsilon}{2n+d}b^{2n+d}+\frac{2n+d}{4\epsilon}\frac{1}{b^{2n+d-4}}.
    \end{align*}
    Therefore, we get 
    \begin{align}\label{gen_ab_ineq_d}
        &\frac{|a_{n,k}|^2}{2n+d}b^{2n+d}\left(1-\left(\frac{a}{b}\right)^{2n+d}\right)+\frac{|b_{n,k}|^2}{2n+d-4}\left(1-\left(\frac{a}{b}\right)^{2n+d-4}\right)+a_{n,k}\,b_{n,k}\,b^2\left(1-\left(\frac{a}{b}\right)^2\right)\nonumber\\
        &\geq \frac{(1-\epsilon)|a_{n,k}|^2}{2n+d}b^{2n+d}\left(1-\left(\frac{a}{b}\right)^{2n+d}\right)\nonumber\\
        &+\frac{|b_{n,k}|^2}{2n+d-4}\left(1-\frac{(2n+d-4)(2n+d)}{4\epsilon}\left(\frac{a}{b}\right)^{2n+d-4}\right)\left(1-\left(\frac{a}{b}\right)^{2n+d-4}\right).
    \end{align}
    For $(d,n)=(3,0)$, for all $0<\delta<1$, we get
    \begin{align}\label{case_d_3_0_frequency}
        &\frac{|a_{0,1}|^2}{3}\left(1-\left(\frac{a}{b}\right)^3\right)+|b_{0,1}|^2\,b\left(1-\left(\frac{a}{b}\right)\right)+a_{0,1}\,b_{0,1}\,b^2\left(1-\left(\frac{a}{b}\right)^2\right)\geq \frac{1-\delta}{3}|a_{0,1}|^2b^3\left(1-\left(\frac{a}{b}\right)^3\right)\nonumber\\
        &+|b_{0,1}|^2\,b\left(1-\left(\frac{a}{b}\right)-\frac{3}{4\,\delta}\left(1-\left(\frac{a}{b}\right)^2\right)\right).
    \end{align}
    In particular, we need to choose $\dfrac{3}{4}<\delta<1$.
    Letting $X=\dfrac{a}{b}$ and $\alpha=\frac{3}{4\delta}$, we get a non-trivial estimate if and only if
    \begin{align*}
        P(X)=\alpha\,X-X+1-\alpha=\alpha\left(X-1\right)\left(X-\frac{1-\alpha}{\alpha}\right)>0,
    \end{align*}
    which is equivalent since $0<X<1$ to 
    \begin{align*}
        \frac{a}{b}=X<\frac{1-\alpha}{\alpha}=\frac{1-\frac{3}{4\delta}}{\frac{3}{4\delta}}=\frac{4\delta-3}{3}.
    \end{align*}
    In particular, we need to assume that $\dfrac{b}{a}>3$. For all fixed $\epsilon_0>0$, choose $0<\delta<1$ such that
    \begin{align*}
        \frac{3}{4\delta-3}=3+\epsilon_0
    \end{align*}
    which yields
    \begin{align*}
        \delta=\frac{3}{4}\frac{4+\epsilon_0}{3+\epsilon_0}.
    \end{align*}
    Therefore, \eqref{case_d_3_0_frequency} becomes
    \begin{align}\label{case_d_3_0_frequency2}
        &\frac{|a_{0,1}|^2}{3}\left(1-\left(\frac{a}{b}\right)^3\right)+|b_{0,1}|^2\,b\left(1-\left(\frac{a}{b}\right)\right)+a_{0,1}\,b_{0,1}b^2\left(1-\left(\frac{a}{b}\right)^2\right)\geq \frac{\epsilon_0}{12(3+\epsilon_0)}|a_{0,1}|^2b^3\left(1-\left(\frac{a}{b}\right)^3\right)\nonumber\\
        &+|b_{0,1}|^2\,b\left(\frac{1}{4+\epsilon_0}+\frac{3+\epsilon_0}{4+\epsilon_0}\left(\frac{a}{b}\right)^2-\left(\frac{a}{b}\right)\right),
    \end{align}
    which yields a non-trivial estimate provided that
    \begin{align}\label{case_d_3_conformal_class}
        \frac{b}{a}>3+\epsilon_0.
    \end{align}
    Finally, for $(d,n)=(4,0)$, we trivially have for all $0<\epsilon<1$
    \begin{align*}
        &\frac{|a_{0,1}|^2}{2}b^4\left(1-\left(\frac{a}{b}\right)^4\right)+2|b_{0,1}|^2\log\left(\frac{b}{a}\right)+2\,a_{0,1}\,b_{0,1}\,b^2\left(1-\left(\frac{a}{b}\right)^2\right)
        \geq\\
        &\frac{(1-\epsilon)}{2}b^4\left(1-\left(\frac{a}{b}\right)^4\right)+2|b_{0,1}|^2\log\left(\frac{b}{a}\right)\left(1-\frac{1}{\epsilon\,\log\left(\frac{b}{a}\right)}\left(1-\left(\frac{a}{b}\right)^2\right)\right).
    \end{align*}
    Therefore, we get the condition
    \begin{align*}
        \log\left(\frac{b}{a}\right)>\frac{1}{\epsilon}.
    \end{align*}
    Now, for general $d\geq 3$, $n\geq 0$ such that $(d,n)\notin\ens{(3,0),(4,0)}$, let us find a condition on the conformal class so that we do not get a trivial estimate. Let $\alpha=\log\left(\frac{b}{a}\right)$ and
    \begin{align*}
        f(x)=4\epsilon\,e^{\alpha x}-x(x+4)
    \end{align*}
    where $x=2n+d-4$. Notice that in all cases of interest, i.e. either $d=3,4$ and $n\geq 1$ or $d\geq 5$ and $n\geq 0$, we have $x\geq 1$, so we need only study $f$ on $[1,\infty[$. We have
    \begin{align*}
        &f'(x)=4\epsilon\alpha\,e^{\alpha x}-2(x+2)\\
        &f''(x)=4\epsilon\alpha^2\,e^{\alpha x}-2\geq 2\left(2\epsilon\,\alpha^2e^{\alpha}-1\right)\quad \text{for all}\;\, x\geq 1.
    \end{align*}
    Let us find an inverse of $g:\R_+\rightarrow\R, x\mapsto x^ne^x$ (where $n>0$) in the form $h(x)=\alpha\,W(\beta\,x^{\gamma})$, where $W:]-e^{-1},\infty[$ is the Lambert function. Notice that for $n=1$, we have $h=W$. In general, we have
    \begin{align*}
        g(h(x))=\alpha^nW(\beta\,x^{\gamma})^ne^{\alpha W(\beta\,x^{\gamma})}=\alpha^nW(\beta\,x^{\gamma})^{n-\alpha}\beta^{\alpha}x^{\alpha\gamma},
    \end{align*}
    where we used the identity $W(x)e^{W(x)}=x$. Therefore, we take $\alpha=n$, $\beta=\dfrac{1}{\alpha}=\dfrac{1}{n}$ and $\gamma=\dfrac{1}{n}$, which yields 
    \begin{align*}
        g^{-1}(x)=h(x)=2\,W\left(\frac{\sqrt[n]{x}}{n}\right).
    \end{align*}
    In particular, we have
    \begin{align}\label{equiv}
        \left(2\epsilon\,\alpha^2e^{\alpha}-1\geq 0\right)\Longleftrightarrow \left(\alpha\geq 2\,W\left(\frac{1}{2\sqrt{\epsilon}}\right)\right).
    \end{align}
    Notice that we can find a non-trivial solution $0<\epsilon<1$ if and only if $2\alpha^2e^{\alpha}\geq 1$, if
    \begin{align}\label{gen_d_conformal_cond1}
        \log\left(\frac{b}{a}\right)\geq 2\,W\left(\frac{1}{2\sqrt{2}}\right)=0.5398\cdots
    \end{align}
    Provided that \eqref{gen_d_conformal_cond1} is satisfied, we deduce that there exists $0<\epsilon<1$ such that \eqref{equiv} is satisfied. In particular $f'$ is increasing on $[1,\infty[$, and 
    \begin{align*}
        f'(1)=2\left(2\epsilon\alpha\,e^{\alpha}-3\right)\geq 0
    \end{align*}
    if and only if $\alpha\geq W\left(\frac{3}{2\epsilon}\right)$. Therefore, assuming that 
    \begin{align*}
        \log\left(\frac{b}{a}\right)\geq W\left(\frac{3}{2}\right)=0.72586\cdots
    \end{align*}
    there exists $0<\epsilon<1$ such that $2\epsilon\alpha\,e^{\alpha}-3\geq 0$, which implies that for all $x\geq 1$
    we have $f(x)\geq f(1)=4\epsilon\,e^{\alpha}-5>0$, provided that 
    \begin{align*}
        \frac{b}{a}>\frac{5}{4\epsilon}.
    \end{align*}
    Notice that 
    \begin{align*}
        e^{W\left(\frac{3}{2}\right)}=\frac{3}{2\,W\left(\frac{3}{2}\right)}=2.06651\cdots
    \end{align*}
    so assuming that 
    \begin{align*}
        \frac{b}{a}\geq \frac{9}{4},
    \end{align*}
    the condition on $\epsilon$ becomes
    \begin{align*}
        \epsilon\geq \epsilon_0=\frac{3}{\frac{9}{2}\,W\left(\frac{9}{4}\right)}=\frac{2}{3\,W\left(\frac{9}{4}\right)}=0.7344\cdots.
    \end{align*}
    Therefore, we have for all $\epsilon\geq \epsilon_0$
    \begin{align*}
        4\epsilon\,e^{\alpha x}-x(x+4)\geq 4\epsilon e^{\alpha}-5\geq 3\geq 0\qquad  \text{for all}\;\, x\geq 1,
    \end{align*}
    that implies in particular with $\epsilon=\frac{3}{4}$ that
    \begin{align*}
        1-\frac{x(x+4)}{3}e^{-\alpha x}\geq 1-\frac{4\epsilon_0}{3}=1-\frac{8}{9\,W\left(\frac{9}{4}\right)}=0.020760\cdots>\frac{1}{50}
    \end{align*}
    and finally, that
    \begin{align*}
        &\frac{|a_{n,k}|^2}{2n+d}b^{2n+d}\left(1-\left(\frac{a}{b}\right)^{2n+d}\right)+\frac{|b_{n,k}|^2}{2n+d-4}\left(1-\left(\frac{a}{b}\right)^{2n+d-4}\right)+a_{n,k}\,b_{n,k}\,b^2\left(1-\left(\frac{a}{b}\right)^2\right)\nonumber\\
        &\geq \frac{|a_{n,k}|^2}{4(2n+d)}b^{2n+d}\left(1-\left(\frac{a}{b}\right)^{2n+d}\right)  +\frac{|b_{n,k}|^2}{50(2n+d-4)}\left(1-\left(\frac{a}{b}\right)^{2n+d-4}\right).
    \end{align*}
    Applying this inequality again for $\epsilon=1$ yields
    \begin{align*}
        &\frac{|a_{n,k}|^2}{2n+d}b^{2n+d}\left(1-\left(\frac{a}{b}\right)^{2n+d}\right)+\frac{|b_{n,k}|^2}{2n+d-4}\left(1-\left(\frac{a}{b}\right)^{2n+d-4}\right)+a_{n,k}\,b_{n,k}\,b^2\left(1-\left(\frac{a}{b}\right)^2\right)\nonumber\\
        &\geq \left(1-\epsilon_0\right)\frac{|b_{n,k}|^2}{2n+d-4}\left(1-\left(\frac{a}{b}\right)^{2n+d-4}\right)
        \geq \frac{|b_{n,k}|^2}{4(2n+d-4)}\left(1-\left(\frac{a}{b}\right)^{2n+d-4}\right).
    \end{align*}
    Gathering the two inequalities finally yields the more symmetric
    \begin{align*}
        &\frac{|a_{n,k}|^2}{2n+d}b^{2n+d}\left(1-\left(\frac{a}{b}\right)^{2n+d}\right)+\frac{|b_{n,k}|^2}{2n+d-4}\left(1-\left(\frac{a}{b}\right)^{2n+d-4}\right)+a_{n,k}\,b_{n,k}\,b^2\left(1-\left(\frac{a}{b}\right)^2\right)\nonumber\\
        &\geq \frac{|a_{n,k}|^2}{8(2n+d)}b^{2n+d}\left(1-\left(\frac{a}{b}\right)^{2n+d}\right)
        +\frac{|b_{n,k}|^2}{8(2n+d-4)}\left(1-\left(\frac{a}{b}\right)^{2n+d-4}\right).
    \end{align*}

    Therefore, $f$ is strictly decreasing on $\left[0,x(\epsilon)=2\,W\left(\frac{1}{2}\sqrt{\frac{3}{\epsilon}}\right)\right]$ and strictly increasing on $\left[x(\epsilon),\infty\right[$. Therefore, we get the condition
    \begin{align*}
        4\epsilon e^{\alpha x(\epsilon)}-x(\epsilon)(x(\epsilon)+4)>0,
    \end{align*}
    or
    \begin{align}\label{gen_d_conformal_cond2}
        \log\left(\frac{b}{a}\right)\geq \frac{1}{x(\epsilon)}\log\left(\frac{x(\epsilon)(x(\epsilon)+4)}{4\epsilon}\right).
    \end{align}
    For example, taking $\epsilon=\dfrac{3}{4}$, we have $x(\epsilon)=2\,W(1)=1.134\cdots$, the condition becomes
    \begin{align}\label{gen_d_conformal_cond3}
        \log\left(\frac{b}{a}\right)\geq \frac{1}{2W(1)}\log\left(\frac{4W(1)(W(1)+2)}{3}\right)=0.5848...
    \end{align}
    or
    \begin{align}\label{gen_d_conformal_cond4}
        \dfrac{b}{a}\geq \exp\left(\frac{1}{2W(1)}\log\left(\frac{4W(1)(W(1)+2)}{3}\right)\right)=1.7946\cdots
    \end{align}
    which implies by \eqref{gen_ab_ineq_d} that 
    \begin{align*}
        &\frac{|a_{n,k}|^2}{2n+d}b^{2n+d}\left(1-\left(\frac{a}{b}\right)^{2n+d}\right)+\frac{|b_{n,k}|^2}{2n+d-4}\frac{1}{a^{2n+d-4}}\left(1-\left(\frac{a}{b}\right)^{2n+d-4}\right)+a_{n,k}\,b_{n,k}\,b^2\left(1-\left(\frac{a}{b}\right)^2\right)\\
        &\geq \frac{|a_{n,k}|^2}{4(2n+d)}b^{2n+d}\left(1-\left(\frac{a}{b}\right)^{2n+d}\right)     \\
        &+\frac{|b_{n,k}|^2}{2n+d-4}\frac{1}{a^{2n+d-4}}\left(1-\frac{(2n+d-4)(2n+d)}{3}\left(\frac{a}{b}\right)^{2n+d-4}\right)\left(1-\left(\frac{a}{b}\right)^{2n+d-4}\right).
    \end{align*}
    Gathering the above results, we deduce the following theorem.
    \begin{theorem}
        Let $d\geq 3$, $0<a<b<\infty$ and $\Omega=B_b\setminus\bar{B}_a(0)\subset \R^d$. Assume that $u\in L^2(\Omega)$ is a harmonic function, and let $\ens{a_{n,k},b_{n,k}}_{n\in \N, 1\leq k\leq N_d(n)}\subset \R$ be such that
        \begin{align*}
            u(r,\omega)=\sum_{n=0}^{\infty}\sum_{k=1}^{N_d(n)}\left(a_{n,k}\,r^n+b_{n,k}\,r^{-(n+d-2)}\right)Y_n^k(\omega),
        \end{align*}
        where $Y_n^k$ are spherical harmonics that give an orthogonal base of $L^2(S^{d-1})$ such that
        \begin{align*}
            \dashint{S^{d-1}}\s{Y_n^k}{Y_m^l}\,d\mathscr{H}^{d-1}=\delta_{n,m}\delta_{k,l}\qquad \text{for all}\;\, n,m,k,l\in \N. 
        \end{align*}
        Then, for $d=3$, for all $\epsilon_0>0$, provided that 
        \begin{align*}
            \frac{b}{a}\geq 3+\epsilon_0,
        \end{align*}
        we have
        \begin{align*}
            &\int_{\Omega}|u|^2dx\geq \frac{\pi\epsilon_0}{3(3+\epsilon_0)}|a_{0,1}|^2b^3\left(1-\left(\frac{a}{b}\right)^3\right)+|b_{0,1}|^2\,b\left(\frac{1}{4+\epsilon_0}+\frac{3+\epsilon_0}{4+\epsilon_0}\left(\frac{a}{b}\right)^2-\left(\frac{a}{b}\right)\right)\\
            &+\frac{\pi}{2}\sum_{n=1}^{\infty}\sum_{k=1}^{2n+1}\frac{|a_{n,k}|^2}{2n+3}b^{2n+3}\left(1-\left(\frac{a}{b}\right)^{2n+3}\right)
            +\frac{\pi}{2}\sum_{n=1}^{\infty}\sum_{k=1}^{2n+1}\frac{|b_{n,k}|^2}{2n-1}\frac{1}{a^{2n+d-4}}\left(1-\left(\frac{a}{b}\right)^{2n-1}\right).
        \end{align*}
        If $d=4$, for all 
        \begin{align*}
            \frac{b}{a}\geq 4,
        \end{align*}
        we have
        \begin{align*}
            &\int_{\Omega}|u|^2dx\geq \frac{\log(4)-1}{2\log(4)}b^4\left(1-\left(\frac{a}{b}\right)^4\right)+2|b_{0,1}|^2\log\left(\frac{b}{a}\right)\left(1-\frac{\log(4)}{\log\left(\frac{b}{a}\right)}\left(1-\left(\frac{a}{b}\right)^2\right)\right)\\
            &+\frac{\pi^2}{4}\sum_{n=1}^{\infty}\sum_{k=1}^{(n+1)^2}\frac{|a_{n,k}|^2}{n+2}b^{2(n+2)}\left(1-\left(\frac{a}{b}\right)^{2(n+2)}\right)+\frac{\pi^2}{4}\sum_{n=1}^{\infty}\sum_{k=1}^{(n+1)^2}\frac{|b_{n,k}|^2}{n}\frac{1}{a^{2n}}\left(1-\left(\frac{a}{b}\right)^{2n}\right).
        \end{align*}
        For all $d\geq 5$, provided that the following condition on the conformal class holds
        \begin{align}\label{class_conf_lemma_l2}
            \frac{b}{a}\geq \frac{9}{4},
        \end{align}
        we have 
        \begin{align}\label{l2_dim5}
            &\int_{\Omega}|u|^2dx\geq \frac{\beta(d)}{8}\sum_{n=0}^{\infty}\sum_{k=1}^{N_d(n)}\frac{|a_{n,k}|^2}{2n+d}b^{2n+d}\left(1-\left(\frac{a}{b}\right)^{2n+d}\right)\nonumber\\
            &+\frac{\beta(d)}{8}\sum_{n=0}^{\infty}\sum_{k=1}^{N_d(n)}\frac{|b_{n,k}|^2}{2n+d-4}\frac{1}{a^{2n+d-4}}\left(1-\left(\frac{a}{b}\right)^{2n+d-4}\right).
        \end{align}
        Furthermore, assuming that for some $a\leq r\leq b$ 
        \begin{align}
            \int_{\partial B(0,r)}\partial_{\nu}u\,d\mathscr{H}^1=0,
        \end{align}
        then, for all $d\geq 3$, if \eqref{class_conf_lemma_l2} holds, we get
        \begin{align}\label{no_flux_ineq}
            &\int_{\Omega}|u|^2dx\geq \frac{\beta(d)}{8}\sum_{n=0}^{\infty}\sum_{k=1}^{N_d(n)}\frac{|a_{n,k}|^2}{2n+d}b^{2n+d}\left(1-\left(\frac{a}{b}\right)^{2n+d}\right)\nonumber\\
            &+\frac{\beta(d)}{8}\sum_{n=0}^{\infty}\sum_{k=1}^{N_d(n)}\frac{|b_{n,k}|^2}{2n+d-4}\frac{1}{a^{2n+d-4}}\left(1-\left(\frac{a}{b}\right)^{2n+d-4}\right).
        \end{align}
    \end{theorem}

    Furthermore, using polar coordinates, we see that
    \begin{align*}
        |\D u|^2=|\p{r}u|^2+\frac{1}{r^2}|\D_{S^{d-1}}u|^2=|\p{r}u|^2+\frac{1}{r^2}|\D_{\omega}u|^2.
    \end{align*}
    Then, observe that for all $n,m\in \N$ and $k\in \ens{1,\cdots,N_d(n)}$, $l\in \ens{1,\cdots,N_d(m)}$, we have
    \begin{align}\label{gradient_orthogonality_spherical_harmonics}
        \int_{S^{d-1}}\s{\D_{S^{d-1}}Y_n^k}{\D_{S^{d-1}}Y_m^l}d\mathscr{H}^{d-1}&=-\int_{S^{d-1}}\s{\Delta_{S^{d-1}}Y_n^k}{Y_{m}^l}d\mathscr{H}^{d-1}\nonumber\\
        &=n(n+d-2)\int_{S^{d-1}}\s{Y_n^k}{Y_{m}^l}d\mathscr{H}^{d-1}
        =\beta(d)\,\lambda_n\,\delta_{n,m}\delta_{l,k}.
    \end{align}
    Therefore, since
    \begin{align*}
        \p{r}u=\sum_{n=0}^{\infty}\sum_{k=1}^{N_d(n)}\left(n\,a_{n,k}\,r^{n-1}-(n+d-2)b_{n,k}\,r^{-(n+d-1)}\right)Y_n^k(\omega),
    \end{align*}
    and
    \begin{align}\label{dirichlet_annulus}
        &\int_{\Omega}|\D u|^2dx=\beta(d)\sum_{n=0}^{\infty}\sum_{k=1}^{N_d(n)}\Big(n^2|a_{n,k}|^2r^{2(n-1)}+(n+d-2)^2r^{-2(n+d-1)}|b_{n,k}|^2\\
        &-2n(n+d-2)a_{n,k}\,b_{n,k}\,r^{-d}\Big)r^{d-1}\,dr\nonumber\\
        &+\beta(d)\sum_{n=0}^{\infty}\sum_{k=1}^{N_d(n)}\int_{a}^bn(n+d-2)\left(|a_{n,k}|^2r^{2(n-1)}+|b_{n,k}|^2r^{-2(n+d-1)}+2\,a_{n,k}\,b_{n,k}r^{-d}\right)r^{d-1}\,dr\nonumber\\
        &=\beta(d)\sum_{n=0}^{\infty}\sum_{k=1}^{N_d(n)}\left(n|a_{n,k}|^2b^{2n+d-2}+(n+d-2)|b_{n,k}|^2\frac{1}{a^{2n+d-2}}\right)\left(1-\left(\frac{a}{b}\right)^{2n+d-2}\right),
    \end{align}
    where we used that $n^2+n(n+d-2)=n(2n+d-2)$ and $(n+d-2)^2+n(n+d-2)=(n+d-2)(2n+d-2)$.
    Taking $d=4$ yields
    \begin{align}\label{dirichlet_dim4}
        \int_{\Omega}|\D u|^2dx&=2\pi^2\sum_{n=0}^{\infty}\sum_{k=1}^{(n+1)^2}\left(n|a_{n,k}|^2b^{2(n+1)}+(n+2)|b_{n,k}|^2\frac{1}{a^{2(n+1)}}\right)\left(1-\left(\frac{a}{b}\right)^{2(n+1)}\right).
    \end{align}
    If $d\geq 3$ is arbitrary, we also get
    \begin{align}\label{dirichlet_weighted}
        \int_{\Omega}\frac{|\D u|^2}{|x|^2}dx&=\beta(d)\sum_{n=0}^{\infty}\sum_{k=1}^{N_d(n)}\int_{a}^b\left(n(2n+d-2)|a_{n,k}|^2r^{2(n-1)}\right.\nonumber\\
        &\left.+(n+d-2)(2n+d-2)|b_{n,k}|^2r^{-2(n+d-1)}\right)r^{d-3}\,dr\nonumber\\
        &=\beta(d)\sum_{n=1}^{\infty}\sum_{k=1}^{N_d(n)}\frac{n(2n+d-2)}{2n+d-4}|a_{n,k}|^2b^{2n+d-4}\left(1-\left(\frac{a}{b}\right)^{2n+d-4}\right)\nonumber\\
        &+\beta(d)\sum_{n=0}^{\infty}\sum_{k=1}^{N_d(n)}\frac{(n+d-2)(2n+d-4)}{2n+d}|b_{n,k}|^2\frac{1}{a^{2n+d}}\left(1-\left(\frac{a}{b}\right)^{2n+d}\right).
    \end{align}
    \begin{lemme}\label{gen_d_dirichlet_comparison}
        Let $d\geq 3$, $0<a<b<\infty$ and $\Omega=B_b\setminus\bar{B}_a(0)\subset \R^d$. Assume that $u:\Omega\rightarrow \R$ is harmonic on $\Omega$. Then, for all $a\leq r<s\leq b$, we have
        \begin{align}\label{dirichlet_comp}
            \int_{B_{s}\setminus\bar{B}_r(0)}|\D u|^2dx\leq \left(\bigg(\frac{s}{b}\bigg)^d+\bigg(\frac{a}{r}\bigg)^{d-2}\right)\int_{\Omega}|\D u|^2dx
        \end{align}
        and
        \begin{align}\label{dirichlet_comp_weighted0}
            \int_{B_{s}\setminus\bar{B}_r(0)}\frac{|\D u|^2}{|x|^2}dx\leq \left(\bigg(\frac{s}{b}\bigg)^{d-2}+\bigg(\frac{a}{r}\bigg)^{d}\right)\int_{\Omega}\frac{|\D u|^2}{|x|^2}dx.
        \end{align}
    \end{lemme}
    \begin{proof}
        We have by \eqref{dirichlet_annulus}
        \begin{align*}
            &\int_{B_{s}\setminus\bar{B}_r(0)}|\D u|^2dx=\beta(d)\sum_{n=0}^{\infty}\sum_{k=1}^{N_d(n)}\left(n|a_{n,k}|^2s^{2n+d-2}+(n+d-2)|b_{n,k}|^2\frac{1}{r^{2n+d-2}}\right)\left(1-\left(\frac{r}{s}\right)^{2n+d-2}\right)\\
            &\leq \beta(d)\bigg(\frac{s}{b}\bigg)^d\sum_{n=1}^{\infty}\sum_{k=1}^{N_d(n)}n|a_{n,k}|^2b^{2n+d-2}\left(1-\left(\frac{a}{b}\right)^{2n+d-2}\right)\\
            &+\beta(d)\bigg(\frac{a}{r}\bigg)^{d-2}\sum_{n=0}^{\infty}\sum_{k=1}^{N_d(n)}(n+d-2)|b_{n,k}|^2\frac{1}{a^{2n+d-2}}\left(1-\left(\frac{a}{b}\right)^{2n+d-2}\right)\\
            &\leq \left(\bigg(\frac{s}{b}\bigg)^d+\bigg(\frac{a}{r}\bigg)^{d-2}\right)\int_{\Omega}|\D u|^2dx.
        \end{align*}
        By the exact same proof, using instead \eqref{dirichlet_weighted}, we get
        \begin{align*}
            \int_{B_{s}\setminus\bar{B}_r(0)}\frac{|\D u|^2}{|x|^2}dx\leq \left(\bigg(\frac{s}{b}\bigg)^{d-2}+\bigg(\frac{a}{r}\bigg)^{d}\right)\int_{\Omega}\frac{|\D u|^2}{|x|^2}dx,
        \end{align*}
        which concludes the proof of the lemma.
    \end{proof}
    This inequality can be strengthened if we impose an additional partial Dirichlet condition.
    \begin{lemme}\label{gen_d_dirichlet_comp2}
        Let $d\geq 3$, $0<a<b<\infty$ and $\Omega=B_b\setminus\bar{B}_a(0)\subset \R^d$. Assume that $u:\Omega\rightarrow \R$ is harmonic on $\Omega$ and that $u=0$ on $\partial B(0,b)$. Then, for all $a\leq r<s\leq b$, we have
        \begin{align}\label{dirichlet_comp2}
            \int_{B_{s}\setminus\bar{B}_r(0)}|\D u|^2dx\leq 2\left(\frac{a}{r}\right)^{d-2}\int_{\Omega}|\D u|^2dx
        \end{align}
        and for all $d\geq 5$
        \begin{align}\label{dirichlet_comp_weighted2}
            \int_{B_{s}\setminus\bar{B}_r(0)}\frac{|\D u|^2}{|x|^2}dx\leq \left(2+\frac{6}{d-2}+\frac{8}{(d-2)^2}\right)\bigg(\frac{a}{r}\bigg)^{d}\int_{\Omega}\frac{|\D u|^2}{|x|^2}dx.
        \end{align}
        If $d=3,4$, assuming that there exists $a\leq r\leq b$ such that
        \begin{align}\label{flux_weighted_l2}
            \int_{\partial B(0,r)}\partial_{\nu}u\,d\mathscr{H}^{d-1}=0,
        \end{align}
        we deduce that 
        \begin{align*}
            \int_{B_{s}\setminus\bar{B}_r(0)}\frac{|\D u|^2}{|x|^2}dx\leq \left(2+\frac{6}{d-2}+\frac{8}{(d-2)^2}\right)\bigg(\frac{a}{r}\bigg)^{d+2}\int_{\Omega}\frac{|\D u|^2}{|x|^2}dx
        \end{align*}
    \end{lemme}
    \begin{proof}
        Since $u=0$ on $\partial B(0,b)$, we deduce that for all $n\in \N$ and all $1\leq k\leq N_d(n)$
        \begin{align*}
            0=\int_{\partial B(0,b)}\s{u}{Y_n^k}\,d\mathscr{H}^{d-1}=\beta(d)\left(a_{n,k}b^n+b_{n,k}\,b^{-(n+d-2)}\right).
        \end{align*}
        Therefore, for all $n\in \N$ and $1\leq k\leq N_d(n)$, we have
        \begin{align*}
            a_{n,k}=-b^{-(2n+d-2)}\,b_{n,k},
        \end{align*}
        which implies that for all $a\leq r<s\leq b$ 
        \begin{align*}
            &\int_{B_s\setminus\bar{B}_r(0)}|\D u|^2dx=\beta(d)\sum_{n=0}^{\infty}\sum_{k=1}^{N_d(n)}\left(n|a_{n,k}|^2s^{2n+d-2}+(n+d-2)|b_{n,k}|^2\frac{1}{r^{2n+d-2}}\right)\left(1-\left(\frac{r}{s}\right)^{2n+d-2}\right)\\
            &=\beta(d)\sum_{n=0}^{\infty}\sum_{k=1}^{N_d(n)}\left(n|b_{n,k}|^2\left(\frac{s}{b^2}\right)^{2n+d-2}+(n+d-2)|b_{n,k}|^2\frac{1}{r^{2n+d-2}}\right)\left(1-\left(\frac{r}{s}\right)^{2n+d-2}\right)\\
            &=\beta(d)\sum_{n=0}^{\infty}\sum_{k=1}^{N_d(n)}|b_{n,k}|^2\frac{1}{r^{2n+d-2}}\left(n+d-2+n\left(\frac{r}{b}\right)^{2n+d-2}\left(\frac{s}{b}\right)^{2n+d-2}\right)\left(1-\left(\frac{r}{s}\right)^{2n+d-2}\right)\\
            &\leq 2\beta(d)\sum_{n=0}^{\infty}\sum_{k=1}^{N_d(n)}(n+d-2)|b_{n,k}|^2\frac{1}{r^{2n+d-2}}\left(1-\left(\frac{r}{s}\right)^{2n+d-2}\right)\\
            &\leq 2\beta(d)\left(\frac{a}{r}\right)^{d-2}\sum_{n=0}^{\infty}\sum_{k=1}^{N_d(n)}(n+d-2)|b_{n,k}|^2\frac{1}{a^{2n+d-2}}\left(1-\left(\frac{a}{b}\right)^{2n+d-2}\right)\\
            &\leq 2\left(\frac{a}{r}\right)^{d-2}\int_{\Omega}|\D u|^2dx.
        \end{align*}
        Likewise, we have for all $d\geq 5$
        \begin{align*}
            &\int_{B_s\setminus\bar{B}_r(0)}\frac{|\D u|^2}{|x|^2}dx=\beta(d)\sum_{n=1}^{\infty}\sum_{k=1}^{N_d(n)}\frac{n(2n+d-2)}{2n+d-4}|a_{n,k}|^2s^{2n+d-4}\left(1-\left(\frac{r}{s}\right)^{2n+d-4}\right)\nonumber\\
        &+\beta(d)\sum_{n=0}^{\infty}\sum_{k=1}^{N_d(n)}\frac{(n+d-2)(2n+d-4)}{2n+d}|b_{n,k}|^2\frac{1}{r^{2n+d}}\left(1-\left(\frac{r}{s}\right)^{2n+d}\right)\\
        &=\beta(d)\sum_{n=1}^{\infty}\sum_{k=1}^{N_d(n)}\frac{n(2n+d-2)}{2n+d-4}|b_{n,k}|^2\frac{s^{2n+d}}{b^{4n+2d}}\left(1-\left(\frac{r}{s}\right)^{2n+d-2}\right)\\
        &+\beta(d)\sum_{n=0}^{\infty}\sum_{k=1}^{N_d(n)}\frac{(n+d-2)(2n+d-4)}{2n+d}|b_{n,k}|^2\frac{1}{r^{2n+d}}\left(1-\left(\frac{r}{s}\right)^{2n+d}\right)\\
        &=\beta(d)\sum_{n=0}^{\infty}\sum_{k=1}^{N_d(n)}|b_{n,k}|^2\frac{1}{r^{2n+d}}\left(\frac{(n+d-2)(2n+d-4)}{2n+d}\left(1-\left(\frac{r}{s}\right)^{2n+d-4}\right)\right.\\
        &\left.+\frac{n(2n+d-2)}{2n+d-4}\left(\frac{r}{b}\right)^{2n+d}\left(\frac{s}{b}\right)^{2n+d}\left(1-\left(\frac{r}{s}\right)^{2n+d-2}\right)\right)\\
        &\leq \left(2+\frac{6}{d-2}+\frac{8}{(d-2)^2}\right)\beta(d)\sum_{n=0}^{\infty}\sum_{k=1}^{N_d(n)}\frac{(n+d-2)(2n+d-4)}{2n+d}|b_{n,k}|^2\frac{1}{r^{2n+d}}\left(1-\left(\frac{r}{s}\right)^{2n+d-4}\right)\\
        &\leq \left(2+\frac{6}{d-2}+\frac{8}{(d-2)^2}\right)\left(\frac{a}{r}\right)^{d}\int_{\Omega}\frac{|\D u|^2}{|x|^2}dx,
        \end{align*}
        where we used that for all $n\geq 1$, the following inequality holds
        \begin{align*}
            &\frac{n(2n+d-2)}{2n+d-4}\left(\frac{(n+d-2)(2n+d-4}{2n+d-4}\right)^{-1}=\frac{n}{n+d-2}\frac{2n+d-2}{2n+d-4}\frac{2n+d}{2n+d-4}\\
            &=\frac{n}{n+d-2}\left(1+\frac{2}{2n+d-4}\right)\left(1+\frac{4}{2n+d-4}\right)
            \leq \left(1+\frac{2}{d-2}\right)\left(1+\frac{4}{d-2}\right).
        \end{align*}
        The final inequality for $d=3,4$ follows identically. 
        Therefore, the argument is complete thanks to the proof of Lemma \ref{gen_d_dirichlet_comparison}.
    \end{proof}
    In particular, we get
    \begin{align*}
        \int_{B_{2r}\setminus\bar{B}_r(0)}|\D u|^2dx\leq 2\left(\frac{a}{r}\right)^{d-2}\int_{\Omega}|\D u|^2dx
    \end{align*}
    which shows that for all $d\geq 3$ (without the extra condition \eqref{flux_weighted_l2} on $u$)
    \begin{align}\label{dirichlet_weighted_typeI}
        \int_{B_{2r}\setminus\bar{B}_r(0)}\frac{|\D u|^2}{|x|^2}dx&\leq \frac{1}{r^2} \int_{B_{2r}\setminus\bar{B}_r(0)}|\D u|^2dx\leq 2\frac{a^{d-2}}{r^d}\int_{\Omega}|\D u|^2dx\nonumber\\
        &\leq 2\left(\frac{a}{r}\right)^d\int_{\Omega}\frac{|\D u|^2}{|x|^2}dx.
    \end{align}
    The final lemma is elementary.
    \begin{lemme}\label{dirichlet_comp3}
         Let $d\geq 3$, $0<b<\infty$ and $u:B(0,b)\rightarrow \R$ be a harmonic function function such that $\D u\in L^2(B(0,b))$. Then, for all $0\leq r\leq b$, we have
         \begin{align*}
             \int_{B(0,r)}|\D u|^2dx\leq \left(\frac{r}{b}\right)^d\int_{B(0,b)}|\D u|^2dx.
         \end{align*}
    \end{lemme}
    \begin{proof}
        Since $2d-2>d-1$, for all $n\in \N$, we have $\D(x\mapsto |x|^{-(n+d-2)})\notin L^2(B(0,b))$, and we deduce that $b_{n,k}=0$ for all $n\in \N$ and $1\leq k\leq N_d(n)$. Therefore, $u$ admits the expansion
        \begin{align*}
            u(r,\omega)=\sum_{n=0}^{\infty}\sum_{k=1}^{N_d(n)}a_{n,k}\,r^n\,Y_n^k(\omega).
        \end{align*}
        We deduce by \eqref{dirichlet_annulus} that 
        \begin{align*}
            \int_{B(0,r)}|\D u|^2dx&=\beta(d)\sum_{n=1}^{\infty}\sum_{k=1}^{N_d(n)}n|a_{n,k}|^2r^{2n+d-2}\leq \left(\frac{r}{b}\right)^d\,\beta(d)\sum_{n=1}^{\infty}\sum_{k=1}^{N_d(n)}n|a_{n,k}|^2b^{2n+d-2}\\
            &=\left(\frac{r}{b}\right)^{d}\int_{B(0,b)}|\D u|^2dx,
        \end{align*}
        which can also be seen directly with the mean-value formula.
    \end{proof}
    \begin{lemme}\label{dirichlet_comp_weighted}
        Let $d\geq 3$, $0<b<\infty$ and $u:B(0,b)\rightarrow \R$ be a harmonic function such that $\D u\in L^2(B(0,b))$. Then, for all $0\leq r\leq b$, we have
        \begin{align*}
            \int_{B(0,r)}\frac{|\D u|^2}{|x|^2}dx\leq \left(\frac{r}{b}\right)^{d-2}\int_{B(0,b)}\frac{|\D u|^2}{|x|^2}dx.
        \end{align*}
    \end{lemme}
    \begin{proof}
        We have
        \begin{align*}
            \int_{B(0,r)}\frac{|\D u|^2}{|x|^2}dx&=\beta(d)\sum_{n=1}^{\infty}\sum_{k=1}^{N_d(n)}\frac{n(2n+d-2)}{2n+d-4}|a_{n,k}|^2r^{2n+d-4}\\
            &\leq \beta(d)\left(\frac{r}{b}\right)^{d-2}\sum_{n=1}^{\infty}\sum_{k=1}^{N_d(n)}\frac{n(2n+d-2)}{2n+d-4}|a_{n,k}|^2b^{2n+d-4}
            =\left(\frac{r}{b}\right)^{d-2}\int_{B(0,b)}\frac{|\D u|^2}{|x|^2}dx.
        \end{align*}
    \end{proof}
    We also need to get additional estimates for $u$.
    \begin{lemme}\label{function_comp_dyadic}
        Let $d\geq 3$, $0<a<b<\infty$ and $\Omega=B_b\setminus\bar{B}_a(0)\subset \R^d$. Assume that $u\in L^2(\Omega)$ and that $u=0$ on $\partial B(0,b)$, and provided that $d=3,4$, assume that for some $a\leq r\leq b$, there holds
        \begin{align}\label{no_flux_predyadic}
            \int_{\partial B(0,r)}\partial_{\nu}u\,d\mathscr{H}^{d-1}=0.
        \end{align}
        Then, for all $a\leq r<s\leq b$, if $d=3,4$, we have 
        \begin{align}\label{ineq:function_comp_dyadic1}
            \int_{B_s\setminus\bar{B}_r(0)}|u|^2dx\leq 2\frac{1-\left(\frac{r}{s}\right)^{d+2}}{1-\left(\frac{a}{b}\right)^{d-2}}\left(\frac{a}{r}\right)^{d-2}\int_{\Omega}|u|^2dx,
        \end{align}
        while for $d\geq 5$
        \begin{align}\label{ineq:function_comp_dyadic1bis}
            \int_{B_s\setminus\bar{B}_r(0)}|u|^2dx\leq 2\frac{1-\left(\frac{r}{s}\right)^{d}}{1-\left(\frac{a}{b}\right)^{d-4}}\left(\frac{a}{r}\right)^{d-4}\int_{\Omega}|u|^2dx,
        \end{align}
    \end{lemme}
    \begin{proof}
    As previously, $u$ admits the following expansion in $L^2$
    \begin{align}\label{new_expansion1}
        u(x)=\sum_{n=0}^{\infty}\sum_{k=1}^{N_d(n)}\left(a_{n,k}|x|^n+b_{n,k}|x|^{-(n+d-2)}\right)Y_n^k\left(\frac{x}{|x|}\right).
    \end{align}
        The hypothesis $u=0$ on $\partial B(0,b)$ implies that for all $n\in \N$ and for all $1\leq k\leq N_d(n)$
        \begin{align*}
            0=\int_{\partial B(0,b)}\s{u}{Y_n^k}d\mathscr{H}^{d-1}=\beta(d)\left(a_{n,k}\,b^n+b_{n,k}\,b^{-(n+d-2)}\right).
        \end{align*}
        Therefore, for all $n\in \N$ and $1\leq k\leq N_d(n)$, we have
        \begin{align*}
            a_{n,k}=-b^{-(2n+d-2)}\,b_{n,k}.
        \end{align*}
        Therefore, we have for all $a\leq r<s\leq b$
        \begin{align}\label{eq:expansion_u_square_pre_wente}
            &\int_{B_s\setminus\bar{B}_r(0)}|u|^2dx=\beta(d)\sum_{n=0}^{\infty}\sum_{k=1}^{N_d(n)}\frac{|a_{n,k}|^2}{2n+d}s^{2n+d}\left(1-\left(\frac{r}{s}\right)^{2n+d}\right)\nonumber\\
            &+\beta(d)\sum_{n=0}^{\infty}\sum_{k=1}^{N_d(n)}\frac{|b_{n,k}|^2}{2n+d-4}\frac{1}{r^{2n+d-4}}\left(1-\left(\frac{r}{s}\right)^{2n+d-4}\right)
            +\beta(d)\sum_{n=0}^{\infty}\sum_{k=1}^{N_d(n)}a_{n,k}\,b_{n,k}\,s^2\left(1-\left(\frac{r}{s}\right)^2\right)\nonumber\\
            &=\beta(d)\sum_{n=0}^{\infty}\sum_{k=1}^{N_d(n)}|b_{n,k}|^2\frac{1}{r^{2n+d-4}}\left(1-\left(\frac{r}{s}\right)^{2n+d-4}+\left(\frac{r}{b}\right)^{2n+d-4}\left(\frac{s}{b}\right)^{2n+d}\left(1-\left(\frac{r}{s}\right)^{2n+d}\right)\right)\nonumber\\
            &-\beta(d)\sum_{n=0}^{\infty}\sum_{k=1}^{N_d(n)}|b_{n,k}|^2\frac{s^2}{b^{2n+d-2}}\left(1-\left(\frac{r}{s}\right)^2\right)\nonumber\\
            &=\beta(d)\sum_{n=0}^{\infty}\sum_{k=1}^{N_d(n)}|b_{n,k}|^2\frac{1}{r^{2n+d-4}}\left(1-\left(\frac{r}{s}\right)^{2n+d-4}+\left(\frac{r}{b}\right)^{2n+d-4}\left(\frac{s}{b}\right)^{2n+d}\left(1-\left(\frac{r}{s}\right)^{2n+d}\right)\right.\nonumber\\
            &\left.-\left(\frac{s}{b}\right)^2\left(\frac{r}{b}\right)^{2n+d-4}\left(1-\left(\frac{r}{s}\right)^2\right)\right),
        \end{align}
        where for $d=3,4$, $a_{0,1}=b_{0,1}=0$ due to the extra condition \eqref{no_flux_predyadic}, so that the second series is well-defined. Taking $s=b$ and $r=a$, we get
        \begin{align}\label{eq:expansion_u_square_pre_wente2}
            \int_{\Omega}|u|^2dx&=\beta(d)\sum_{n=0}^{\infty}\sum_{k=1}^{N_d(n)}|b_{n,k}|^2\frac{1}{a^{2n+d-4}}\left(1-\left(\frac{a}{b}\right)^{2n+d-4}+\left(\frac{a}{b}\right)^{2n+d-2}\left(1-\left(\frac{a}{b}\right)^{2n+d-2}\right)\right)\nonumber\\
            &\geq \beta(d)\sum_{n=0}^{\infty}\sum_{k=1}^{N_d(n)}|b_{n,k}|^2\frac{1}{a^{2n+d-4}}\left(1-\left(\frac{a}{b}\right)^{2n+d-4}\right).
        \end{align}
        On the other hand, by \eqref{eq:expansion_u_square_pre_wente}, we have the elementary estimate
        \begin{align}\label{eq:expansion_u_square_pre_wente3}
            \int_{B_s\setminus\bar{B}_{r}(0)}|u|^2dx\leq 2\,\beta(d)\sum_{n=0}^{\infty}\sum_{k=1}^{N_d(n)}|b_{n,k}|^2\frac{1}{r^{2n+d-4}}\left(1-\left(\frac{r}{s}\right)^{2n+d}\right).
        \end{align}
        Therefore, by \eqref{eq:expansion_u_square_pre_wente2} and \eqref{eq:expansion_u_square_pre_wente3}, we deduce that for $d=3,4$ (recall that $b_{0,1}=0$ in this case)
        \begin{align*}
            \int_{B_s\setminus\bar{B}_r(0)}|u|^2dx\leq 2\frac{1-\left(\frac{r}{s}\right)^{d+2}}{1-\left(\frac{a}{b}\right)^{d-2}}\left(\frac{a}{r}\right)^{d-2}\int_{\Omega}|u|^2dx,
        \end{align*}
        while for $d\geq 5$
        \begin{align*}
            \int_{B_s\setminus\bar{B}_r(0)}|u|^2dx\leq 2\frac{1-\left(\frac{r}{s}\right)^{d}}{1-\left(\frac{a}{b}\right)^{d-4}}\left(\frac{a}{r}\right)^{d-4}\int_{\Omega}|u|^2dx,
        \end{align*}
        which concludes the proof of the lemma.        
    \end{proof}
    \begin{lemme}\label{function_comp_dyadic2}
         Let $d\geq 3$, $0<b<\infty$ and $u\in L^2(B(0,b))$ be a harmonic function. Then, for all $0\leq r\leq b$, we have
         \begin{align*}
             \int_{B(0,r)}|u|^2dx\leq \left(\frac{r}{b}\right)^d\int_{B(0,b)}|u|^2dx.
         \end{align*}
    \end{lemme}
    \begin{proof}
        Since $u$ is smooth in $B(0,b)$, \eqref{new_expansion1} reduce to
        \begin{align*}
             u(x)=\sum_{n=0}^{\infty}\sum_{k=1}^{N_d(n)}a_{n,k}|x|^n\,Y_n^k\left(\frac{x}{|x|}\right).
        \end{align*}
        Therefore, we have
        \begin{align*}
            \int_{B(0,r)}|u|^2dx=\beta(d)\sum_{n=0}^{\infty}\sum_{k=1}^{N_d(n)}\frac{|a_{n,k}|^2}{2n+d}r^{2n+d}\leq \left(\frac{r}{b}\right)^d\,\beta(d)\sum_{n=0}^{\infty}\sum_{k=1}^{N_d(n)}\frac{|a_{n,k}|^2}{2n+d}b^{2n+d}=\left(\frac{r}{b}\right)^d\int_{B(0,b)}|u|^2dx.
        \end{align*}
    \end{proof}
    Then, recall that by the Bochner formula, for all Riemannian manifold $(\Sigma,g)$, and for all $\varphi\in C^{\infty}(\Sigma)$ we have 
    \begin{align*}
        \frac{1}{2}\Delta_g|d\varphi|_g^2=|\D^2\varphi|_g^2+\mathrm{Ric}_g(du,du)+\s{du}{d\,\Delta_gu}_g.
    \end{align*}
    Assuming that $\Sigma$ is compact and has no boundary, we get by Stokes theorem
    \begin{align*}
        0&=\frac{1}{2}\int_{\Sigma}\Delta_{g}|d\varphi|_g^2\,d\vg=\int_{\Sigma}\left(|\D^2\varphi|_g^2+\mathrm{Ric}_g(du,du)+\s{du}{d\,\Delta_gu}_g\right)d\vg\\
        &=\int_{\Sigma}\Big(|\D^{2}\varphi|^2+\mathrm{Ric}_g(du,du)-(\Delta_gu)^2\Big)d\vg,
    \end{align*}
    which yields the identity
    \begin{align}\label{ipp_bochner}
        \int_{\Sigma}|\D^2\varphi|_g^2\,d\vg=\int_{\Sigma}\Big((\Delta_gu)^2-\mathrm{Ric}_g(du,du)\Big)d\vg.
    \end{align}
    In particular, if $(\Sigma,g)$ has positive Ricci curvature, we get
    \begin{align*}
        \int_{\Sigma}|\D^2\varphi|_g^2\,d\vg\leq \int_{\Sigma}(\Delta_gu)^2d\vg.
    \end{align*}
    In the special case $(\Sigma,g)=(S^{d-1},g_0)$ of the sphere equipped with the round metric, we get
    \begin{align}\label{ipp_hessien_sphere}
        \int_{S^{d-1}}|\D^2_{S^{d-1}}\varphi|^2\,d\mathscr{H}^{d-1}=\int_{S^{d-1}}\left(\left(\Delta_{S^{d-1}}\varphi\right)^2-(d-2)|\D_{S^{d-1}}\varphi|^2\right)d\mathscr{H}^{d-1}.
    \end{align}
    Now, if $u$ is a harmonic function as above, \emph{i.e.} if
    \begin{align*}
        u(r,\omega)=\sum_{n=0}^{\infty}\sum_{k=1}^{N_d(n)}\left(a_{n,k}\,r^n+b_{n,k}\,r^{-(n+d-2)}\right)Y_n^k(\omega)=\sum_{n=0}^{\infty}\sum_{k=1}^{N_d(n)}u_{n,k}(r)\,Y_n^k(\omega)
    \end{align*}
    we deduce by orthogonality of the $Y_n^k$ family and the identity
    \begin{align*}
        -\Delta_{S^{d-1}}Y_n^k=n(n+d-2)\,Y_{n}^k=\lambda_n\,Y_n^k
    \end{align*}
    that
    \begin{align*}
        &\int_{S^{d-1}}|\D^2_{\omega}u(r,\omega)|^2d\,\mathscr{H}^{d-1}(\omega)=\int_{S^{d-1}}\left(\left(\Delta_{S^{d-1}}u\right)^2+(d-2)\s{u}{\Delta_{S^{d-1}}u}\right)d\mathscr{H}^{d-1}\\
        &=\beta(d)\sum_{n=0}^{\infty}\sum_{k=1}^{N_d(n)}|u_n^k(r)|^2\left(\lambda_n^2-(d-2)\lambda_n\right)=\beta(d)\sum_{n=1}^{\infty}\sum_{k=1}^{N_d(n)}n(n+d-2)\left(n^2+(d-2)(n-1)\right)|u_n^k(r)|^2.
    \end{align*}
    Now, we have
    \begin{align*}
        |\D^2u|^2=|\p{r}^2u|^2+\frac{1}{r^2}\left|\D_{S^{d-1}}\left(\p{r}u\right)\right|+\left|\p{r}\left(\frac{1}{r}\D_{\omega}u\right)\right|^2+\frac{1}{r^4}|\D_{S^{d-1}}^2u|^2.
    \end{align*}
    We have by \eqref{gradient_orthogonality_spherical_harmonics}
    \begin{align*}
        \int_{S^{d-1}}\frac{1}{r^2}\left|\D_{S^{d-1}}\left(\p{r}u\right)\right|^2\,d\mathscr{H}^{d-1}=\beta(d)\sum_{n=0}^{\infty}\sum_{k=1}^{N_d(n)}n(n+d-2)\left(n\,a_{n,k}\,r^{n-2}-(n+d-2)\,b_{n,k}\,r^{-(n+d)}\right)^2.
    \end{align*}
    Likewise, 
    \begin{align*}
        \int_{S^{d-1}}\left|\p{r}\left(\frac{1}{r}\D_{\omega}u\right)\right|^2d\mathscr{H}^{d-1}=\beta(d)\sum_{n=0}^{\infty}\sum_{k=1}^{N_d(n)}n(n+d-2)\left((n-1)a_{n,k}r^{n-2}-(n+d-1)b_{n,k}r^{-(n+d)}\right)^2.
    \end{align*}
    Finally, we have
    \begin{align*}
        \int_{S^{d-1}}|\p{r}^2u|^2d\mathscr{H}^{d-1}=\beta(d)\sum_{n=0}^{\infty}\sum_{k=1}^{N_d(n)}\left(n(n-1)\,a_{n,k}\,r^{n-2}+(n+d-1)(n+d-2)\,b_{n,k}\,r^{-(n+d)}\right)^2,
    \end{align*}
    and
    \begin{align*}
        &\int_{S^{d-1}}|\D^2u|^2d\mathscr{H}^{d-1}=\beta(d)\sum_{n=0}^{\infty}\sum_{k=1}^{N_d(n)}\left(n(n+d-2)(n^2+(d-2)(n-1))\right.\\
        &\left.+n^3(n+d-2)+n(n-1)^2(n+d-2)+n^2(n-1)^2\right)|a_{n,k}|^2r^{2(n-2)}\\
        &+\beta(d)\sum_{n=0}^{\infty}\sum_{k=1}^{N_d(n)}\left(n(n+d-2)(n^2+(d-2)(n-1))+n(n+d-2)^3\right.\\
        &\left.+n(n+d-1)^2(n+d-2)+(n+d-1)^2(n+d-2)^2\right)|b_{n,k}|^2r^{-2(n+d)}\\
        &+2\beta(d)\sum_{n=0}^{\infty}\sum_{k=1}^{N_d(n)}\left(n(n+d-2)\left(n^2+(d-2)(n-1)\right)-n^2(n+d-2)^2\right.\\
        &\left.-n(n-1)(n+d-1)(n+d-2)+n(n-1)(n+d-1)(n+d-2)\right)\,a_{n,k}\,b_{n,k}\,r^{-(d+2)}\\
        &=\beta(d)\sum_{n=1}^{\infty}\sum_{k=1}^{N_d(n)}n\left(4n^3+4(d-3)n^2+(d^2-7d+12)n-(d-2)(d-3)\right)|a_{n,k}|^2r^{2(n-2)}\\
        &+\beta(d)\sum_{n=0}^{\infty}\sum_{k=1}^{N_d(n)}(n+d-2)\left(4n^3+4(2d-3)n^2+(5d^2-15d+12)n+(d-1)^2(d-2)\right)|b_{n,k}|^2r^{-2(n+d)}\\
        &-2(d-2)\beta(d)\sum_{n=0}^{\infty}\sum_{k=1}^{N_d(n)}n(n+d-2)\,a_{n,k}\,b_{n,k}\,r^{-(d+2)}.
    \end{align*}
    Therefore, we deduce if $\Omega=B_b\setminus\bar{B}_a(0)$ that
    \small
    \begin{align*}
        &\int_{\Omega}|\D^2u|^2dx=\beta(d)\sum_{n=1}^{\infty}\sum_{k=1}^{N_d(n)}\int_{a}^bn\left(4n^3+4(d-3)n^2+(d-3)(d-4)n-(d-2)(d-3)\right)|a_{n,k}|^2r^{2(n-2)}\,r^{d-1}\,dr\\
        &+\beta(d)\sum_{n=0}^{\infty}\sum_{k=1}^{N_d(n)}\int_{a}^b(n+d-2)\left(4n^3+4(2d-3)n^2+(5d^2-15d+12)n+(d-1)^2(d-2)\right)|b_{n,k}|^2r^{-2(n+d)}\,r^{d-1}\,dr\\
        &-2(d-2)\beta(d)\sum_{n=1}^{\infty}\sum_{k=1}^{N_d(n)}n(n+d-2)a_{n,k}\,b_{n,k}\,r^{-3}\,dr\\
        &=\beta(d)\left\{\sum_{n=1}^{\infty}\sum_{k=1}^{N_d(n)}\frac{n\left(4n^3+4(d-3)n^2+(d-3)(d-4)n-(d-2)(d-3)\right)}{2n+d-4}[a_{n,k}|^2b^{2n+d-4}\left(1-\left(\frac{a}{b}\right)^{2n+d-4}\right)\right.\\
        &+\sum_{n=0}^{\infty}\sum_{k=1}^{N_d(n)}\frac{(n+d-2)\left(4n^3+4(2d-3)n^2+(5d^2-15d+12)n+(d-1)^2(d-2)\right)}{2n+d}|b_{n,k}|^2\frac{1}{a^{2n+d}}\left(1-\left(\frac{a}{b}\right)^{2n+d}\right)\\
        &\left.-(d-2)\sum_{n=1}^{\infty}\sum_{k=1}^{N_d(n)}n(n+d-2)a_{n,k}\,b_{n,k}\frac{1}{a^2}\left(1-\left(\frac{a}{b}\right)^2\right)\right\}.
    \end{align*}
    \normalsize
    We estimate
    \begin{align*}
        \left|(d-2)n(n+d-2)\,a_{n,k}\,b_{n,k}\,\frac{1}{a^2}\right|\leq \frac{1}{2}(d-2)n^2|a_{n,k}|^2a^{2n+d-4}+\frac{1}{2}(d-2)(n+d-2)^2|b_{n,k}|^2\frac{1}{a^{2n+d}}. 
    \end{align*}
    Then, notice that 
    \begin{align*}
        &\frac{(n+d-2)\left(4n^3+4(2d-3)n^2+(5d^2-15d+12)n+(d-1)^2(d-2)\right)}{2n+d}-\frac{1}{2}(d-2)(n+d-2)^2\\
        &=\frac{(n+d-2)}{2(2n+d)}\left(8n^3+2(7d-10)n^2+(7d^2-20d+16)n+(d-2)\left((d-1)^2+1\right)\right)
    \end{align*}
    Since the function in the numerator is strictly increasing in $n$ (notice that $4d^2-17d+18\geq 3>0$ for all $d\geq 3$), we deduce that 
    \begin{align*}
        &\inf_{n\geq 1}\left(8n^3+2(7d-10)n^2+(7d^2-20d+16)n+(d-2)\left((d-1)^2+1\right)\right)\\
        &=8+2(7d-10)+(7d^2-20d+16)+(d-2)\left((d-1)^2+1\right)=d^2(d+3).
    \end{align*}
    Therefore, we finally obtain the estimate for an arbitrary conformal class $0<a<b<\infty$
    \begin{align}\label{est_below_hessian_harmonic}
        &\int_{\Omega}|\D^2u|^2dx\geq \beta(d)\sum_{n=1}^{\infty}\sum_{k=1}^{N_d(n)}\frac{n}{2(2n+d-4)}|a_{n,k}|^2b^{2n+d-4}\bigg(2\left(4n^3+4(d-3)n^2+(d-3)(d-4)n\right.\nonumber\\
        &\left.-(d-2)(d-3)\right)\left(1-\left(\frac{a}{b}\right)^{2n+d-4}\right)-(d-2)n(2n+d-4)\left(1-\left(\frac{a}{b}\right)^2\right)\left(\frac{a}{b}\right)^{2n+d-4}\bigg)\nonumber\\
        &+\beta(d)\sum_{n=0}^{\infty}\sum_{k=1}^{N_d(n)}\frac{(n+d-2)}{2(2n+d)}\left(8n^3+2(7d-10)n^2+(7d^2-20d+16)n+(d-2)\left((d-1)^2+1\right)\right)\nonumber\\
        &\times |b_{n,k}|^2\frac{1}{a^{2n+d}}\left(1-\left(\frac{a}{b}\right)^{2n+d}\right).
    \end{align}

    \subsection{Lorentz Space estimate}

    \begin{theorem}\label{lorentz_l2_gen_d}
        Let $d\geq 3$. Then, there exists a universal constant $C_d<\infty$ with the following property. For all $0<2\, a\leq b<\infty$, define $\Omega=B_b\setminus\bar{B}_a(0)\subset \R^d$. For all $0<\alpha<1$, define $\Omega_{\alpha}=B_{\alpha\,b}\setminus\bar{B}_{\alpha^{-1}a}(0)$, and let $u\in L^2(\Omega)$ be a harmonic function. If $d=3,4$, assume that for some $a\leq r\leq b$, we have
        \begin{align}\label{flux_condition_gen_d}
            \int_{\partial B(0,r)}\partial_{\nu}u\,d\mathscr{H}^1=0.
        \end{align}
        Then, assuming that 
        \begin{align*}
            \frac{b}{a}\geq \frac{9}{4},
        \end{align*}
        we have for all $0<\alpha<1$
        \begin{align}\label{lorentz_l2_gen_d_ineq}
            \np{u}{2,1}{\Omega_{\alpha}}\leq \left\{\begin{alignedat}{1}
                &\frac{C_d}{\sqrt{1-\left(\frac{a}{b}\right)^{d-2}}}\frac{\alpha^{\frac{d-2}{2}}}{(1-\alpha^2)^{d-1}}\np{u}{2}{\Omega}\qquad \text{for all}\;\, d=3,4\\
                &\frac{C_d}{\sqrt{1-\left(\frac{a}{b}\right)^{d-4}}}\frac{\alpha^{\frac{d-4}{2}}}{(1-\alpha^2)^{d-1}}\np{u}{2}{\Omega}\qquad \text{for all}\;\, d\geq 5.
            \end{alignedat}\right.
        \end{align}
    \end{theorem}
    \begin{proof}
        Recall that for all $f\in L^{2,1}(\Omega)$ (\cite[Appendix]{pointwise}), we have
    \begin{align*}
        \np{f}{2,1}{\Omega}=4\int_{0}^{\infty}\Big(\leb^d\left(\Omega\cap\ens{x:|f(x)|>t}\right)\Big)^{\frac{1}{2}}dt.
    \end{align*}
    For all $n\geq 1$, and for all $t>0$, we have
    \begin{align}\label{lorentz_pos_frequency}
        \leb^d\left(\Omega\cap\ens{x:|x|^n>t}\right)=\left\{\begin{alignedat}{2}
            &\beta(d)\left(b^d-a^d\right)\qquad&& \text{for all}\;\, t\leq a^n\\
            &\beta(d)\left(b^d-t^{\frac{d}{n}}\right)\qquad&& \text{for all}\;\, a^n<t<b^{n}\\
            &0\qquad&& \text{for all}\;\, t\geq b^{n}.
        \end{alignedat}\right.
    \end{align}
    which implies that
    \begin{align*}
        \np{|x|^{n}}{2,1}{\Omega}&=4\sqrt{\beta(d)}a^n\sqrt{b^d-a^d}+4\sqrt{\beta(d)}\int_{a^n}^{b^n}\sqrt{b^d-t^{\frac{d}{n}}}\,dt\leq 4\sqrt{\beta(d)}\left(a^nb^{\frac{d}{2}}+b^{\frac{d}{2}}(b^n-a^n)\right)\\
        &=4\sqrt{\beta(d)}\,b^{n+\frac{d}{2}}
    \end{align*}
    and this estimate also trivially holds for $n=0$. On the other hand, we have for all $n\geq 1$
    \begin{align}\label{lorentz_neg_frequency}
        \leb^d\left(\Omega\cap\ens{x:\frac{1}{|x|^n}>t}\right)=\left\{\begin{alignedat}{2}
            &\beta(d)\left(b^d-a^d\right)\qquad&& \text{for all}\;\, t\leq \frac{1}{b^n}\\
            &\beta(d)\left(\frac{1}{t^{\frac{d}{n}}}-a^d\right)\qquad&& \text{for all}\;\, \frac{1}{b^n}<t<\frac{1}{a^n}\\
            &0\qquad&& \text{for all}\;\, t\geq \frac{1}{a^n}
        \end{alignedat}\right.,
    \end{align}
    which implies that for $n>\dfrac{d}{2}$
    \begin{align*}
        \np{\frac{1}{|x|^n}}{2,1}{\Omega}&=4\sqrt{\beta(d)}\frac{1}{b^n}\sqrt{b^d-a^d}+4\sqrt{\beta(d)}\int_{\frac{1}{b^n}}^{\frac{1}{a^n}}\sqrt{\frac{1}{t^{\frac{d}{n}}}-a^d}dt\leq 4\pi\sqrt{2}\frac{1}{b^{n-2}}+4\sqrt{\beta(d)}\int_{\frac{1}{b^n}}^{\frac{1}{a^n}}\frac{dt}{t^{\frac{d}{2n}}}dt\\
        &=4\sqrt{\beta(d)}\frac{1}{b^{n-\frac{d}{2}}}+4\sqrt{\beta(d)}\frac{2n}{2n-d}\left(\frac{1}{a^{n-\frac{d}{2}}}-\frac{1}{b^{n-\frac{d}{2}}}\right)\leq 4\sqrt{\beta(d)}\frac{2n}{2n-d}\frac{1}{a^{n-\frac{d}{2}}}.
    \end{align*}
    Notice that for $n+d-2>\dfrac{d}{2}$ if and only if $n>\dfrac{4-d}{2}$. Therefore, there are no condition for $d\geq 5$, and for $d=3,4$, the condition is equivalent to $n\geq 1$ (this is why we have made the hypothesis \eqref{flux_condition_gen_d}).
    For $n=\dfrac{d}{2}$, we have
    \begin{align*}
        &\np{\frac{1}{|x|^{\frac{d}{2}}}}{2,1}{\Omega}=4\sqrt{\beta(d)}\frac{1}{b^{\frac{d}{2}}}\sqrt{b^d-a^d}+4\sqrt{\beta(d)}\int_{\frac{1}{b^{\frac{d}{2}}}}^{\frac{1}{a^{\frac{d}{2}}}}\frac{1}{t}\sqrt{1-\left(a^dt\right)^2}\,dt=4\sqrt{\beta(d)}\frac{1}{b^{\frac{d}{2}}}\sqrt{b^4-a^d}\\
        &+4\sqrt{\beta(d)}\left[\log(t)-\log\left(1+\sqrt{1-t^2}\right)+\sqrt{1-t^2}\right]^1_{\left(\frac{a}{b}\right)^{\frac{d}{2}}}\\
        &=4\sqrt{\beta(d)}\left(\log\left(\frac{b}{a}\right)+\log\left(1+\sqrt{1-\left(\frac{a}{b}\right)^d}\right)\right)
    \end{align*}
    which cannot be bounded by 
    \begin{align*}
        \np{\frac{1}{|x|^{\frac{d}{2}}}}{2}{\Omega}=\sqrt{\beta(d)\log\left(\frac{b}{a}\right)}
    \end{align*}
    independently of the conformal class of the annulus. Recalling that $u$ admits the expansion 
    \begin{align*}
        u(r,\omega)=\sum_{n=0}^{\infty}\sum_{k=1}^{N_d(n)}\left(a_{n,k}\,r^n+b_{n,k}\,r^{-(n+d-2)}\right)Y_n^k(\omega).
    \end{align*}
    Then, recall that by \cite[Corollary $2.9$ p. $144$]{stein_weiss} (notice the $\beta(d)$ shift)
    \begin{align}\label{id_spherical_harmonics}
        \sum_{k=1}^{N_d(n)}|Y_n^k(\omega)|^2=N_d(n).
    \end{align}
    Therefore, we have (recall that $b_{0,1}=0$ for $d=3,4$)
    \small
    \begin{align*}
        \np{u}{2,1}{\Omega_{\alpha}}&\leq 4\sqrt{\beta(d)}\sum_{n=0}^{\infty}\sum_{k=1}^{N_d(n)}\sqrt{N_d(n)}|a_{n,k}|\left(\alpha\,b\right)^{n+\frac{d}{2}}+4\sqrt{\beta(d)}\sum_{n=0}^{\infty}\sum_{k=1}^{N_d(n)}\sqrt{N_d(n)}\frac{2(n+d-2)}{2n+d-4}|b_{n,k}|\frac{\alpha^{n+d-2-\frac{d}{2}}}{a^{n+d-2-\frac{d}{2}}}\\
        &=4\sqrt{\beta(d)}\sum_{n=0}^{\infty}\sum_{k=1}^{N_d(n)}\sqrt{N_d(n)}|a_{n,k}|\left(\alpha\,b\right)^{n+\frac{d}{2}}+4\sqrt{\beta(d)}\sum_{n=0}^{\infty}\sum_{k=1}^{N_d(n)}\sqrt{N_d(n)}\frac{2(n+d-2)}{2n+d-4}|b_{n,k}|\frac{\alpha^{n+\frac{d}{2}-2}}{a^{n+\frac{d}{2}-2}}.
    \end{align*}
    \normalsize
    Therefore, we get by Cauchy-Schwarz inequality
    \begin{align*}
        &\np{u}{2,1}{\Omega_{\alpha}}\leq 4\sqrt{\beta(d)}\left(\sum_{n=0}^{\infty}\frac{|a_{n,k}|^2}{2n+d}b^{2n+d}\right)^{\frac{1}{2}}\left(\sum_{n=0}^{\infty}\sum_{k=1}^{N_d(n)}(2n+d)N_d(n)\alpha^{2n+d}\right)^{\frac{1}{2}}\\
        &+4\sqrt{\beta(d)}\left(\sum_{n=\mathbf{1}_{\ens{d\leq 4}}}^{\infty}\sum_{k=1}^{N_d(n)}\frac{|b_{n,k}|^2}{2n+d-4}\frac{1}{a^{2n+d-4}}\right)^{\frac{1}{2}}\left(\sum_{n=\mathbf{1}_{\ens{d\leq 4}}}^{\infty}\sum_{k=1}^{N_d(n)}\frac{4(n+d-2)^2}{2n+d-4}N_d(n)\alpha^{2n+d-4}\right)^{\frac{1}{2}}.
    \end{align*}
    Now, recall that
    \begin{align*}
        N_d(n)&=\binom{n+d-1}{d-1}-\binom{n+d-3}{d-1}=\frac{(n+d-1)!}{(d-1)!\,n!}-\frac{(n+d-3)!}{(d-1)!(n-2)!}\\
        &=\frac{(n+d-3)!}{(d-1)!\,n!}\left((n+d-1)(n+d-2)-n(n-1)\right)=\frac{(n+d-3)!}{(d-2)!\,n!}\left(2n+d-2\right)\\
        &\leq \frac{(n+d-3)^{d-3}}{(d-2)!}\left(2n+d-2\right)
    \end{align*}
    which is sharp for $d=3,4$ since $N_3(n)=2n+1$ and $N_4(n)=(n+1)^2$. In particular, $N_d(n)$ is a polynomial of degree $d-2$. Now, notice that for all $k\in \N$, there exists a universal constant $C_k<\infty$ such that for all $0<x<1$,
    \begin{align*}
        \sum_{n=0}^{\infty}(n+1)^kx^k\leq \frac{C_k}{(1-x)^{k+1}}.
    \end{align*}
    Therefore, as the polynomial $P_d(n)=(2n+d)N_d^2(n)$ has degree $\deg(P_d(n))=1+2(d-2)=2d-3$, we deduce that there exists a universal constant $0<\Gamma_1(d)<\infty$ such that
    \begin{align*}
        \sum_{n=0}^{\infty}\sum_{k=1}^{N_d(n)}(2n+d)N_d(n)\alpha^{2n}=\sum_{n=0}^{\infty}(2n+d)N_d^2(n)\alpha^{2n}\leq \frac{\Gamma_1(d)}{(1-\alpha^2)^{2d-2}}
    \end{align*}
    and
    \begin{align*}
        \sum_{n=\mathbf{1}_{\ens{d\leq 4}}}^{\infty}\sum_{k=1}^{N_d(n)}\frac{4(n+d-2)^2}{(2n+d-4)}N_d(n)\alpha^{2n+d-4}\leq \frac{\Gamma_1(d)}{(1-\alpha^2)^{2d-2}}\alpha^{2(\delta_{\ens{d=3}}+\delta_{\ens{d=4}})},
    \end{align*}
    where $\delta$ is the Dirac indicator. Recalling inequality \eqref{no_flux_ineq}, we get
    \begin{align}\label{u_bound_l2_gen_d}
        &\int_{\Omega}|u|^2dx\geq \frac{\beta(d)}{8}\sum_{n=0}^{\infty}\sum_{k=1}^{N_d(n)}\frac{|a_{n,k}|^2}{2n+d}b^{2n+d}\left(1-\left(\frac{a}{b}\right)^{2n+d}\right)\nonumber\\
        &+\frac{\beta(d)}{8}\sum_{n=0}^{\infty}\sum_{k=1}^{N_d(n)}\frac{|b_{n,k}|^2}{2n+d-4}\frac{1}{a^{2n+d-4}}\left(1-\left(\frac{a}{b}\right)^{2n+d}\right)\nonumber\\
        &\geq  \frac{\beta(d)}{8}\left(\sum_{n=0}^{\infty}\sum_{k=1}^{N_d(n)}\frac{|a_{n,k}|^2}{2n+d}b^{2n+d}+\sum_{n=0}^{\infty}\sum_{k=1}^{N_d(n)}\frac{|b_{n,k}|^2}{2n+d-4}\frac{1}{a^{2n+d-4}}\right)\left\{\begin{alignedat}{2}
            &\left(1-\left(\frac{a}{b}\right)^{d-2}\right)&&\qquad\text{for}\;\, d=3,4\\
            &\left(1-\left(\frac{a}{b}\right)^{d-4}\right)&&\qquad\text{for}\;\, d\geq 5,
        \end{alignedat}\right.
    \end{align}   
    where we used that $b_{0,1}=0$ for $d=3,4$ by assumption.
    
    Therefore, we finally get 
    \begin{align*}
        \np{u}{2,1}{\Omega_{\alpha}}\leq \left\{\begin{alignedat}{2}
            &\frac{8\sqrt{2\,\Gamma_1(d)}}{\sqrt{1-\left(\frac{a}{b}\right)^{d-2}}}\frac{\alpha^{\frac{d-2}{2}}}{\left(1-\alpha^2\right)^{d-1}}\np{u}{2}{\Omega}&&\qquad\text{for}\;\, d=3,4\\
            &\frac{8\sqrt{2\,\Gamma_1(d)}}{\sqrt{1-\left(\frac{a}{b}\right)^{d-4}}}\frac{\alpha^{\frac{d-4}{2}}}{\left(1-\alpha^2\right)^{d-1}}\np{u}{2}{\Omega}&&\qquad\text{for}\;\, d\geq 5,
        \end{alignedat}\right.
    \end{align*}
    which concludes the proof of the theorem.
    \end{proof}
    Let us now prove this estimate for the Dirichlet energy. 

    \begin{theorem}\label{dirichlet_dim_arbitraire}
        Let $d\geq 3$. Then, there exists a universal constant $C_d'<\infty$ with the following property. For all $0<a<b<\infty$, define $\Omega=B_b\setminus\bar{B}_a(0)\subset \R^d$. For all $0<\alpha<1$, define $\Omega_{\alpha}=B_{\alpha\,b}\setminus\bar{B}_{\alpha^{-1}a}(0)$, and let $u\in W^{1,2}(\Omega)$ be a harmonic function. Then, for all $0<\alpha<1$, we have the estimate
        \begin{align}\label{dirichlet_dim_arbitraire_ineq}
            \np{\D u}{2,1}{\Omega_{\alpha}}\leq \frac{C_d'}{\sqrt{1-\left(\frac{a}{b}\right)^{d-2}}}\frac{\alpha^{\frac{d-2}{2}}}{\left(1-\alpha^2\right)^{d-1}}\np{\D u}{2}{\Omega}.
        \end{align}
    \end{theorem}
    \begin{proof}
        Recall that by \eqref{dirichlet_annulus}, we have
        \begin{align*}
            \int_{\Omega}|\D u|^2dx=\beta(d)\sum_{n=0}^{\infty}\sum_{k=1}^{N_d(n)}\left(n|a_{n,k}|^2b^{2n+d-2}+(n+d-2)|b_{n,k}|^2\frac{1}{a^{2n+d-2}}\right)\left(1-\left(\frac{a}{b}\right)^{2n+d-2}\right).
        \end{align*}
        Now, we need to estimate the gradient of spherical harmonics. It can be be done in multiple ways, but the easiest one is to use the representation formula for spherical harmonics of degree $n$ (\cite[p. $240$]{transcendental})
        \begin{align*}
            \Re\left (i^k\,e^{\pm m_{d-2}\varphi}\prod_{j=1}^{d-2}\sin^{m_j}(\theta_j)G_{m_{j-1}-m_j}^{m_j+\frac{d-2}{2}-\frac{j-1}{2}}\left(\cos(\theta_j)\right)\right)
        \end{align*}
        where $k=0,1$, and $0\leq m_{d-2}\leq \cdots\leq m_1\leq m_0\leq n$, while
        $(\theta_1,\cdots,\theta_{d-2},\varphi)\in [0,\pi]^{d-2}\times [0,2\pi]$ are the usual spherical coordinates on the sphere, and $G_{\alpha}^{\beta}$ are the Gegenbauer polynomials or ultra-spherical polynomials of degree $\alpha$ and order $\beta$. Thanks to the recurrence relation 
        \begin{align*}
            \frac{d}{dx}G_{\alpha}^{\beta}=2\beta\,C_{\alpha-1}^{\beta+1},
        \end{align*}
        and using the elementary estimate $|Y_n^k(\omega)|\leq N_d(n)$, we deduce that for all $l\in \N$, there exists a universal constant $\Gamma_d(l)<\infty$ such that
        \begin{align}
            |\D_{\omega}^l Y_n^k(\omega)|\leq \Gamma_d(l)n^l\sqrt{N_d(n)}.
        \end{align}
        In particular, we deduce that
        \begin{align}\label{gradient_pointwise_gen_d}
            &|\D u|\leq |\p{r}u|+\frac{1}{r}|\D_{\omega}u|\nonumber\\
            &\leq \left|\sum_{n=0}^{\infty}\sum_{k=1}^{N_d(n)}\left(n\,a_{n,k}\,r^{n-1}-(n+d-2)b_{n,k}\,r^{-(n+d-1)}\right)Y_n^k(\omega)\right|\nonumber\\
            &+\left|\sum_{n=1}^{\infty}\sum_{k=1}\left(a_{n,k}\,r^{n-1}+b_{n,k}\,r^{-(n+d-1)}\right)\D_{\omega}Y_n^k(\omega)\right|\nonumber\\
            &\leq \left(1+\Gamma_d(1)\right)\sum_{n=1}^{\infty}\sum_{k=1}^{N_d(n)}n\,\sqrt{N_d(n)}|a_{n,k}|r^{n-1}+\sum_{n=0}^{\infty}\sum_{k=1}^{N_d(n)}\Big(n+d-2+\Gamma_d(1)n\Big)\sqrt{N_d(n)}|b_{n,k}|r^{-(n+d-1)}\nonumber\\
            &=\left(1+\Gamma_d(1)\right)\sum_{n=1}^{\infty}\sum_{k=1}^{N_d(n)}n\sqrt{N_d(n)}|a_{n,k}|r^{n-1}+\left(1+\Gamma_d(1)\right)\sum_{n=0}^{\infty}\sum_{k=1}^{N_d(n)}(n+d-2)\sqrt{N_d(n)}|b_{n,k}|r^{-(n+d-1)}.
        \end{align}
        Therefore, we have by \eqref{lorentz_pos_frequency} and \eqref{lorentz_neg_frequency}
        \begin{align*}
            &\np{\D u}{2,1}{\Omega_{\alpha}}\leq 4\sqrt{\beta(d)}\left(1+\Gamma_1(d)\right)\left(\sum_{n=1}^{\infty}\sum_{k=1}^{N_d(n)}n\sqrt{N_d(n)}|a_{n,k}|\left(\alpha b\right)^{n-1+\frac{d}{2}}\right.\\
            &\left.+\sum_{n=0}^{\infty}\sum_{k=1}^{N_d(n)}\left(n+d-2\right)\sqrt{N_d(n)}\frac{2(n+d-1)}{2n+d-2}|b_{n,k}|\frac{\alpha^{n+\frac{d}{2}-1}}{a^{n+\frac{d}{2}-1}}\right)\\
            &\leq 4\sqrt{\beta(d)}\left(1+\Gamma_d(1)\right)\left(\left(\sum_{n=1}^{\infty}\sum_{k=1}^{N_d(n)}nN_d(n)\alpha^{2n+d-2}\right)^{\frac{1}{2}}\left(\sum_{n=1}^{\infty}\sum_{k=1}^{N_d(n)}n|a_{n,k}|^2b^{2n+d-2}\right)^{\frac{1}{2}}\right.\\
            &\left.+4\sqrt{\beta(d)}\left(\sum_{n=0}^{\infty}\sum_{k=1}^{N_d(n)}\frac{4(n+d-1)^2}{(2n+d-2)^2}\left(n+d-2\right)\alpha^{2n+d-2}\right)^{\frac{1}{2}}\left(\sum_{n=0}^{\infty}(n+d-2)|b_{n,k}|^2\frac{1}{a^{2n+d-2}}\right)^{\frac{1}{2}}\right)\\
            &\leq \frac{C_d'}{\sqrt{1-\left(\frac{a}{b}\right)^{d-2}}}\left(\frac{\alpha^{\frac{d-2}{2}}}{\left(1-\alpha^2\right)^{d-1}}\right)^{\frac{d-2}{2}}\np{\D u}{2}{\Omega}
        \end{align*}
        where we used Cauchy-Schwarz inequality, and the identity \eqref{dirichlet_annulus}
        \begin{align*}
            \int_{\Omega}|\D u|^2dx=\beta(d)\sum_{n=0}^{\infty}\sum_{k=1}^{N_d(n)}\left(n|a_{n,k}|^2b^{2n+d-2}+(n+d-2)|b_{n,k}|^2\frac{1}{a^{2n+d-2}}\right)\left(1-\left(\frac{a}{b}\right)^{2n+d-2}\right)
        \end{align*}
        which implies the sharp inequality
        \begin{align*}
            \beta(d)\sum_{n=0}^{\infty}\sum_{k=1}^{N_d(n)}\left(n|a_{n,k}|^2b^{2n+d-2}+(n+d-2)|b_{n,k}|^2\frac{1}{a^{2n+d-2}}\right)\leq \frac{1}{1-\left(\frac{a}{b}\right)^{d-2}}\int_{\Omega}|\D u|^2dx.
        \end{align*}
    \end{proof}

    Now, we get an estimate on the Hessian matrix of harmonic maps in annular regions.

    \begin{theorem}\label{lorentz_l2_hessian}
        Let $d\geq 3$. Then, there exists a universal constant $C_d'<\infty$ with the following property. For all $0<a<b<\infty$, define $\Omega=B_b\setminus\bar{B}_a(0)\subset \R^d$, and for all $\mathbf{d\geq 7}$, assume that  
        \begin{align}\label{conf_hessian0}
            \log\left(\frac{b}{a}\right)\geq \frac{1}{d-2}\log\left(\frac{d}{4}\right).
        \end{align}
        For all $0<\alpha<1$, define $\Omega_{\alpha}=B_{\alpha\,b}\setminus\bar{B}_{\alpha^{-1}a}(0)$, and let $u\in W^{2,2}(\Omega)$ be a harmonic function. Then, for all $0<\alpha<1$, we have the estimate
        \begin{align}\label{lorentz_l2_hessian_ineq}
            \np{\D^2u}{2,1}{\Omega_{\alpha}}\leq \frac{C_d''}{\sqrt{1-\left(\frac{a}{b}\right)^{d-2}}}\frac{\alpha^{\frac{d-2}{2}}}{(1-\alpha^2)^{d-1}}\np{\D^2u}{2}{\Omega}.
        \end{align}
    \end{theorem}
    \begin{proof}
        We have
        \begin{align*}
            \D^2\left(\sum_{n=0}^{1}\sum_{k=1}^{N_d(n)}a_{n,k}r^nY_n^k(\omega)\right)=\D^2\left(a_{0,1}+\sum_{k=1}^{d}a_{1,k}x_k\right)=0.
        \end{align*}
        Therefore, we have
        \begin{align}\label{pointwise_sec_der_gen_d}
            &|\D^2u|\leq |\p{r}^2u|+\frac{1}{r}|\D_{\omega}\left(\p{r}u\right)|+\left|\p{r}\left(\frac{1}{r}\D_{\omega}u\right)\right|+\frac{1}{r}\left|\D_{\omega}(\p{r}u)\right|+\frac{1}{r^2}|\D^2_{\omega}u|\nonumber\\
            &\leq \sum_{n=2}^{\infty}\sum_{k=1}^{N_d(n)}n(n-1)\sqrt{N_d(n)}\,|a_{n,k}|r^{n-2}+\sum_{n=2}^{\infty}\sum_{k=1}^{N_d(n)}(n+d-1)(n+d-2)\sqrt{N_d(n)}\,|b_{n,k}|r^{-(n+d)}\nonumber\\
            &+\Gamma_d(1)\sum_{n=2}^{\infty}\sum_{k=1}^{N_d(n)}n(2n-1)\sqrt{N_d(n)}\,|a_{n,k}|r^{n-2}+\Gamma_d(1)\sum_{n=2}^{\infty}\sum_{k=1}^{N_d(n)}n(2n+2d-3)\sqrt{N_d(n)}\,|b_{n,k}|r^{-(n+d)}\nonumber\\
            &+\Gamma_d(2)\sum_{n=2}^{\infty}\sum_{k=1}^{N_d(n)}n^2\sqrt{N_d(n)}\,|a_{n,k}|r^{n-2}+\Gamma_d(2)\sum_{n=2}^{\infty}\sum_{k=1}^{N_d(n)}n^2\sqrt{N_d(n)}\,|b_{n,k}|r^{-(n+d)}\nonumber\\
            &\leq \left(1+2\,\Gamma_1(d)+\Gamma_2(d)\right)\sum_{n=2}^{\infty}\sum_{k=1}^{N_d(n)}n^2\sqrt{N_d(n)}|a_{n,k}|r^{n-2}\nonumber\\
            &+\left(1+2\,\Gamma_1(d)+\Gamma_2(d)\right)\sum_{n=0}^{\infty}\sum_{k=1}^{N_d(n)}(n+d-1)(n+d-2)\sqrt{N_d(n)}\,|b_{n,k}|r^{-(n+d)}.
        \end{align}
        Therefore, we get for all $0<\alpha<1$
        \begin{align*}
            &\np{\D^2u}{2,1}{\Omega_{\alpha}}\leq \left(1+2\,\Gamma_d(1)+\Gamma_d(2)\right)\sum_{n=2}^{\infty}\sum_{k=1}^{N_d(n)}n^2\sqrt{N_d(n)}\,|a_{n,k}|\left(\alpha\, b\right)^{n-2+\frac{d}{2}}\\
            &+\left(1+2\,\Gamma_d(1)+\Gamma_d(2)\right)\sum_{n=0}^{\infty}\sum_{k=1}^{N_d(n)}(n+d-1)(n+d-2)\sqrt{N_d(n)}\,|b_{n,k}|\frac{\alpha^{n+\frac{d}{2}}}{a^{n+\frac{d}{2}}}.
        \end{align*}
        Now, we need to estimate
        \begin{align*}
            &2\left(4n^3+4(d-3)n^2+(d^2-7d+12)n-(d-2)(d-3)\right)\left(1-\left(\frac{a}{b}\right)^{2n+d-4}\right)\\&-(d-2)n(2n+d-4)\left(1-\left(\frac{a}{b}\right)^2\right)\left(\frac{a}{b}\right)^{2n+d-4}\\
            &\geq \left(2\left(4n^3+4(d-3)n^2+(d^2-7d+12)n-(d-2)(d-3)\right)-(d-2)n(2n+d-4)\left(\frac{a}{b}\right)^{2n+d-4}\right)\\
            &\times\left(1-\left(\frac{a}{b}\right)^{2n+d-4}\right)\\
            &\geq \left(2\left(4n^3+4(d-3)n^2+(d^2-7d+12)n-(d-2)(d-3)\right)-(d-2)n(2n+d-4)\right)\left(1-\left(\frac{a}{b}\right)^{2n+d-4}\right).
        \end{align*}
        For $d=3$, we get
        \begin{align*}
            &2\left(4n^3+4(d-3)n^2+(d-3)(d-4)n-(d-2)(d-3)\right)-(d-2)n(2n+d-4)=8n^3-n(2n-1)\\
            &=8n^3-2n^2+n\geq 6n^3+n.
        \end{align*}
        For $d=4$, we have
        \begin{align*}
            &2\left(4n^3+4(d-3)n^2+(d-3)(d-4)n-(d-2)(d-3)\right)-(d-2)n(2n+d-4)
            =8n^3+8n^2-4-4n^2\\
            &=8n^3+4n^2-4\geq 8n^3.
        \end{align*}
        For $d=5$, we have 
        \begin{align*}
            &2\left(4n^3+4(d-3)n^2+(d-3)(d-4)n-(d-2)(d-3)\right)-(d-2)n(2n+d-4)\\
            &=8n^3+16n^2+4n-12-6n^2-3n=8n^3+10n^2+n-12\geq 7n^3.
        \end{align*}
        Finally, for $d=6$, we have
        \begin{align*}
            &2\left(4n^3+4(d-3)n^2+(d-3)(d-4)n-(d-2)(d-3)\right)-(d-2)n(2n+d-4)\\
            &=8n^3+24n^2+12n-24-8n^2-8n=8n^3+16n^2+4n-24\geq 4n^3.
        \end{align*}
        However, for $d=7$, we have
        \begin{align*}
            &2\left(4n^3+4(d-3)n^2+(d-3)(d-4)n-(d-2)(d-3)\right)-(d-2)n(2n+d-4)\\
            &=8n^3+32n^2+24n-40-10n^2+15n=8n^3+22n^2+9n-40
        \end{align*}
        which is equal to $-1<0$ for $n=1$. Therefore, we need to find a condition on the conformal class for $d\geq 7$ so that our estimate is non-trivial and capture the cubic growth of the coefficient in \eqref{est_below_hessian_harmonic}. Let $\epsilon>0$ and let us find a condition on the conformal class such that 
        \begin{align}\label{lower_bound_hessian_a_nk}
            2\left(4n^3+4(d-3)n^2+(d-3)(d-4)n-(d-2)(d-3)\right)-(d-2)n(2n+d-4)\left(\frac{a}{b}\right)^{2n+d-4}\geq 2\epsilon\,n^3
        \end{align}
        if and only if
        \begin{align*}
            \left(\frac{a}{b}\right)^{2n+d-4}\leq \frac{2\left((4-\epsilon)n^3+4(d-3)n^2+(d-3)(d-4)n-(d-2)(d-3)\right)}{(d-2)n(2n+d-2)}.
        \end{align*}
        For simplicity, let
        \begin{align*}
        \left\{\begin{alignedat}{1}
            \alpha&=\frac{4(d-3)}{4-\epsilon}\\
            \beta&=\frac{(d-3)(d-4)}{4-\epsilon}\\
            \gamma&=\frac{d-2}{2}\\
            \delta&=\frac{(d-2)(d-3)}{4-\epsilon},
        \end{alignedat}\right.
        \end{align*}
        and consider the function
        \begin{align*}
            f(x)=\frac{x^3+\alpha x^2+\beta x-\delta}{x(x+\gamma)}.
        \end{align*}
        We have
        \begin{align*}
            f'(x)&=\frac{3x^2+2\alpha\, x+\beta}{x(x+\gamma)}-\frac{(2\,x+\gamma)\left(x^3+\alpha\, x^2+\beta x-\delta\right)}{x^2(x+\gamma)^2}\\
            &=\frac{x^4+2\gamma\,x^3+\left(\alpha\gamma-\beta\right)x^2+2\delta\,x+\gamma\delta}{x^2(x+\gamma)^2}>0
        \end{align*}
        provided that $\alpha\gamma-\beta\geq 0$. We have
        \begin{align*}
            \alpha\gamma-\beta=\frac{1}{4-\epsilon}\left(2(d-3)(d-2)-(d-3)(d-4)\right)=\frac{d(d-3)}{4-\epsilon}\geq 0\qquad\text{for all}\;\, d\geq 3.
        \end{align*}
        In particular, $f$ is strictly increasing on $[1,\infty[$, which implies that
        \begin{align*}
            &\inf_{n\geq 1}\frac{2\left((4-\epsilon)n^3+4(d-3)n^2+(d-3)(d-4)n-(d-2)(d-3)\right)}{(d-2)n(2n+d-2)}\\
            &=\frac{2\left(4-\epsilon+4(d-3)+(d-3)(d-4)-(d-2)(d-3)\right)}{d(d-2)}\\
            &=\frac{2\left(4-\epsilon+2(d-3)\right)}{d(d-2)}=\frac{2(2(d-1)-\epsilon)}{d(d-2)}.
        \end{align*}
        Therefore, \eqref{lower_bound_hessian_a_nk} is satisfied provided that 
        \begin{align*}
            \left(\frac{a}{b}\right)^{d-2}\leq \frac{2(2(d-1)-\epsilon)}{d(d-2)},
        \end{align*}
        or
        \begin{align}\label{conf_hessian}
            \log\left(\frac{b}{a}\right)\geq \frac{1}{d-2}\log\left(\frac{d(d-2)}{2(2(d-1)-\epsilon)}\right).
        \end{align}
        Numerical estimation show that for $\epsilon=0$, we have
        \begin{align*}
            \sup_{d\geq 3}\frac{1}{d-2}\log\left(\frac{d(d-2)}{2(2(d-1)-\epsilon)}\right)<\frac{1}{10},
        \end{align*}
        which implies that there exists $\epsilon_0>0$ such that \eqref{conf_hessian} is satisfied provided that
        \begin{align*}
            \frac{b}{a}\geq e^{\frac{1}{10}}=1.105\cdots
        \end{align*}
        In general, choosing $\epsilon=2$, we get (without any conditions if $3\leq d\leq 6$)
        \begin{align*}
            &2\left(4n^3+4(d-3)n^2+(d^2-7d+12)n-(d-2)(d-3)\right)\left(1-\left(\frac{a}{b}\right)^{2n+d-4}\right)\\
            &-(d-2)n(2n+d-4)\left(1-\left(\frac{a}{b}\right)^2\right)\left(\frac{a}{b}\right)^{2n+d-4}
            \geq 4n^3,
        \end{align*}
        and finally \eqref{est_below_hessian_harmonic} becomes
        \begin{align}\label{est_below_hessian_harmonic2}
            &\int_{\Omega}|\D^2u|^2\geq \beta(d)\sum_{n=2}^{\infty}\sum_{k=1}^{N_d(n)}\frac{2n^4}{2n+d-4}|a_{n,k}|^2b^{2n+d-4}\left(1-\left(\frac{a}{b}\right)^{2n+d-4}\right)\nonumber\\
            &+\beta(d)\sum_{n=0}^{\infty}\sum_{k=1}^{N_d(n)}\frac{(n+d-2)}{2(2n+d)}\left(8n^3+2(7d-10)n^2+(7d^2-20d+16)n+(d-2)\left((d-1)^2+1\right)\right)\\
            &\times|b_{n,k}|^2\frac{1}{a^{2n+d}}\left(1-\left(\frac{a}{b}\right)^{2n+d}\right)\nonumber\\
            &\geq \beta(d)\sum_{n=1}^{\infty}\sum_{k=1}^{N_d(n)}\frac{2n^4}{2n+d-4}|a_{n,k}|^2b^{2n+d-4}\left(1-\left(\frac{a}{b}\right)^{2n+d-4}\right)\nonumber\\
            &+\beta(d)\sum_{n=0}^{\infty}\sum_{k=1}^{N_d(n)}\frac{(2n^3+1)(n+d-2)}{2n+d}|b_{n,k}|^2\frac{1}{a^{2n+d}}\left(1-\left(\frac{a}{b}\right)^{2n+d}\right)
        \end{align}
        and the condition on the conformal class for $d\geq 7$ is 
        \begin{align}\label{conf_hessian2}
            \log\left(\frac{b}{a}\right)\geq \frac{1}{d-2}\log\left(\frac{d}{4}\right).
        \end{align}
        Furthermore, \eqref{est_below_hessian_harmonic2} implies that
        \small
        \begin{align*}
            \beta(d)\left(\sum_{n=2}^{\infty}\sum_{k=1}^{N_d(n)}\frac{2n^4}{2n+d-4}|a_{n,k}|^2b^{2n+d-4}+\sum_{n=0}^{\infty}\sum_{k=1}^{N_d(n)}\frac{(2n^3+1)(n+d-2)}{2n+d}\frac{1}{a^{2n+d}}\right)\leq \frac{1}{1-\left(\frac{a}{b}\right)^{d-2}}\np{\D^2u}{2}{\Omega}.
        \end{align*}
        \normalsize
        Finally, we deduce that 
        \begin{align*}
            &\np{\D^2u}{2,1}{\Omega_{\alpha}}\leq \left(1+2\,\Gamma_d(1)+\Gamma_d(2)\right)\sum_{n=2}^{\infty}\sum_{k=1}^{N_d(n)}n^2\sqrt{N_d(n)}|a_{n,k}|\left(\alpha b\right)^{n-2+\frac{d}{2}}\\
            &+\left(1+2\,\Gamma_d(1)+\Gamma_d(2)\right)\sum_{n=0}^{\infty}\sum_{k=1}^{N_d(n)}\sqrt{N_d(n)}(n+d-1)(n+d-2)|b_{n,k}|\frac{\alpha^{n+\frac{d}{2}}}{a^{n+\frac{d}{2}}}\\
            &\leq  \widetilde{\Gamma}_d(1)\left(\sum_{n=2}^{\infty}\sum_{k=1}^{N_d(n)}\frac{n^4}{2n+d-4}|a_{n,k}|^2b^{2n+d-4}\right)^{\frac{1}{2}}\left(\sum_{n=1}^{\infty}\sum_{k=1}^{N_d(n)}(2n+d-4)\alpha^{2n+d-4}\right)^{\frac{1}{2}}\\
            &+\widetilde{\Gamma}_d(1)\left(\sum_{n=0}^{\infty}\sum_{k=1}^{N_d(n)}N_d(n)\frac{(2n^3+1)(n+d-2)}{2n+d}|b_{n,k}|^2\frac{1}{a^{2n+d}}\right)^{\frac{1}{2}}\\
            &\times \left(\sum_{n=0}^{\infty}\sum_{k=1}^{N_d(n)}N_d(n)\frac{(2n+d)(n+d-1)^2(n+d-2)}{(2n^3+1)}\alpha^{2n+d}\right)^{\frac{1}{2}}\\
            &\leq \frac{C_d''}{\sqrt{1-\left(\frac{a}{b}\right)^{d-2}}}\frac{\alpha^{\frac{d-2}{2}}}{(1-\alpha^2)^{d-1}}\np{\D^2u}{2}{\Omega}
        \end{align*}
        which concludes the proof of the theorem.
    \end{proof}
    Now, we move to inequalities involving various derivatives.
    \begin{theorem}\label{pre_dirichlet_arbitraire}
    Let $d\geq 3$. Then, there exists a universal constant $\Gamma_d<\infty$ with the following property. For all $0<a<b<\infty$, define $\Omega=B_b\setminus\bar{B}_a(0)\subset \R^4$. For all $0<\alpha<1$, define $\Omega_{\alpha}=B_{\alpha\,b}\setminus\bar{B}_{\alpha^{-1}a}(0)$, and let $u\in W^{1,2}(\Omega)$ be a harmonic function. Then, there exists $c\in \R$ such that for all $0<\alpha<1$, we have the estimate
        \begin{align}
            \np{u-c}{\frac{2d}{d-2},1}{\Omega_{\alpha}}\leq \frac{\Gamma_d}{\sqrt{1-\left(\frac{a}{b}\right)^{d-2}}}\frac{\alpha^{\frac{d-2}{2}}}{\left(1-\alpha^2\right)^{d-2}}\np{\D u}{2}{\Omega}.
        \end{align}
    \end{theorem}
    \begin{proof}
        We have
        \begin{align*}
            u(r,\omega)-a_{0,1}=\sum_{n=1}^{\infty}\sum_{k=1}^{N_d(n)}a_{n,k}\,r^n\,Y_n^k(\omega)+\sum_{n=0}^{\infty}\sum_{k=1}^{N_d(n)}b_{n,k}\,r^{-(n+d-2)}\,Y_n^k(\omega).
        \end{align*}
        Recall that for all $1<p<\infty$, we have
        \begin{align*}
            \np{f}{p,1}{\Omega}=\frac{p^2}{p-1}\int_{0}^{\infty}\left(\leb^d\left(\Omega\cap\ens{x:|f(x)|>t}\right)\right)^{\frac{1}{p}}dt.
        \end{align*}
        Recall that
        \begin{align}\label{lorentz_pos_frequency2}
        \leb^d\left(\Omega\cap\ens{x:|x|^n>t}\right)=\left\{\begin{alignedat}{2}
            &\beta(d)\left(b^d-a^d\right)\qquad&& \text{for all}\;\, t\leq a^n\\
            &\beta(d)\left(b^d-t^{\frac{d}{n}}\right)\qquad&& \text{for all}\;\, a^n<t<b^{n}\\
            &0\qquad&& \text{for all}\;\, t\geq b^{n}.
        \end{alignedat}\right.
    \end{align}
    and that for all $n\geq 1$
    \begin{align}\label{lorentz_neg_frequency1}
        \leb^d\left(\Omega\cap\ens{x:\frac{1}{|x|^n}>t}\right)=\left\{\begin{alignedat}{2}
            &\beta(d)\left(b^d-a^d\right)\qquad&& \text{for all}\;\, t\leq \frac{1}{b^n}\\
            &\beta(d)\left(\frac{1}{t^{\frac{d}{n}}}-a^d\right)\qquad&& \text{for all}\;\, \frac{1}{b^n}<t<\frac{1}{a^n}\\
            &0\qquad&& \text{for all}\;\, t\geq \frac{1}{a^n}
        \end{alignedat}\right.,
    \end{align}
        Therefore, for $p=\dfrac{2d}{d-2}>2$, we have
        \begin{align*}
            \np{|x|^n}{p,1}{\Omega}&=\frac{p^2}{p-1}\left(\left(\beta(d)\right)^{\frac{1}{p}}a^n\left(b^d-a^d\right)^{\frac{1}{p}}+\left(\beta(d)\right)^{\frac{1}{p}}\int_{a^n}^{b^{n}}\left(b^d-t^{\frac{d}{n}}\right)^{\frac{1}{p}}dt\right)\\
            &\leq \frac{p^2}{p-1}\left(\left(\beta(d)\right)^{\frac{1}{p}}a^n\left(b^d-a^d\right)^{\frac{1}{p}}+\left(\beta(d)\right)^{\frac{1}{p}}(b^n-a^n)b^{\frac{d}{p}}\right)\\
            &\leq \frac{p^2}{p-1}\left(\beta(d)\right)^{\frac{1}{p}}b^{n+\frac{d}{p}}=\frac{4d^2}{d^2-4}\left(\beta(d)\right)^{\frac{d-2}{2d}}b^{n+\frac{d-2}{2}}.
        \end{align*}
        Notice that $p>2$
        Likewise, we get for all $n\geq 1$
        \begin{align*}
            &\np{\frac{1}{|x|^n}}{p,1}{\Omega}=\frac{p^2}{p-1}\left(\left(\beta(d)\right)^{\frac{1}{p}}\frac{1}{b^n}\left(b^d-a^d\right)^{\frac{1}{p}}+\left(\beta(d)\right)^{\frac{1}{p}}\int_{\frac{1}{b^n}}^{\frac{1}{a^n}}\left(\frac{1}{t^{\frac{d}{n}}}-a^d\right)^{\frac{1}{p}}dt\right)\\
            &\leq \frac{p^2}{p-1}\left(\left(\beta(d)\right)^{\frac{1}{p}}\frac{1}{b^n}\left(b^d-a^d\right)^{\frac{1}{p}}+\left(\beta(d)\right)^{\frac{1}{p}}\int_{\frac{1}{b^n}}^{\frac{1}{a^n}}\frac{1}{t^{\frac{d}{np}}}dt\right)
        \end{align*}
        We have 
        \begin{align*}
            \frac{d}{np}=\frac{d}{\frac{2dn}{d-2}}=\frac{d-2}{2n},
        \end{align*}
        and for $n\geq d-2$, we have
        \begin{align*}
            \frac{d-2}{2n}\leq \frac{1}{2}<1,
        \end{align*}
        which shows that 
        \begin{align*}
            \int_{\frac{1}{b^n}}^{\frac{1}{a^n}}\frac{1}{t^{\frac{d}{np}}}dt=\left[\frac{1}{1-\frac{d-2}{2n}}t^{1-\frac{d-2}{2n}}\right]_{\frac{1}{b^n}}^{\frac{1}{a^n}}=\frac{2n}{2n-(d-2)}\left(\frac{1}{a^{n-\frac{d-2}{2}}}-\frac{1}{b^{n-\frac{d-2}{2}}}\right).
        \end{align*}
        Therefore, we finally get
        \begin{align*}
            \np{\frac{1}{|x|^n}}{p,1}{\Omega}\leq \frac{4d^2}{d^2-4}\left(\beta(d)\right)^{\frac{d-2}{2d}}\frac{2n}{2n-(d-2)}\frac{1}{a^{n-\frac{d-2}{2}}}.
        \end{align*}
        Therefore, we have by the elementary inequality $|Y_n^k(\omega)|\leq \sqrt{N_d(n)}$
        \begin{align*}
            &\np{u-a_{0,1}}{\frac{2d}{d-2},1}{\Omega_{\alpha}}\leq \frac{4d^2}{d^2-4}\left(\beta(d)\right)^{\frac{d-2}{2d}}\sum_{n=1}^{\infty}\sum_{k=1}^{N_d(n)}\sqrt{N_d(n)}|a_{n,k}|\left(\alpha\,b\right)^{n+\frac{d-2}{2}}\\
            &+\frac{4d^2}{d^2-4}\left(\beta(d)\right)^{\frac{d-2}{2d}}\sum_{n=0}^{\infty}\sum_{k=1}^{N_d(n)}\sqrt{N_d(n)}\frac{2(n+d-2)}{2n+d-2}|b_{n,k}|\frac{\alpha^{{n+\frac{d-2}{2}}}}{a^{n+\frac{d-2}{2}}}.
        \end{align*}
        On the other hand, we have
        \begin{align*}
            \int_{\Omega}|\D u|^2dx=\beta(d)\sum_{n=0}^{\infty}\sum_{k=1}^{N_d(n)}\left(n|a_{n,k}|^2b^{2n+d-2}+(n+d-2)|b_{n,k}|^2\frac{1}{a^{2n+d-2}}\right)\left(1-\left(\frac{a}{b}\right)^{2n+d-2}\right),
        \end{align*}
        which implies that
        \begin{align*}
            &\np{u-a_{0,1}}{\frac{2d}{d-2},1}{\Omega_{\alpha}}\leq \frac{4d^2}{d^2-4}\left(\beta(d)\right)^{\frac{d-2}{2d}}\left(\sum_{n=1  }^{\infty}\frac{N_d(n)^2}{n}\alpha^{2n+d-2}\right)^{\frac{1}{2}}\left(\sum_{n=1}^{\infty}\sum_{k=1}^{N_d(n)}n|a_{n,k}|^2b^{2n+d-2}\right)^{\frac{1}{2}}\\
            &+\frac{4d^2}{d^2-4}\left(\beta(d)\right)^{\frac{d-2}{2d}}\left(\sum_{n=0}^{\infty}N_d(n)^2\frac{4(n+d-2)}{(2n+d-2)^2}\alpha^{2n+d-2}\right)^{\frac{1}{2}}\left(\sum_{n=1}^{\infty}\sum_{k=1}^{N_d(n)}(n+d-2)|b_{n,k}|^2\frac{1}{a^{2n+d-2}}\right)^{\frac{1}{2}}\\
            &\leq \frac{\Gamma_d}{\sqrt{1-\left(\frac{a}{b}\right)^{d-2}}}\frac{\alpha^{\frac{d-2}{2}}}{\left(1-\alpha^2\right)^{d-2}}\np{\D u}{2}{\Omega},
        \end{align*}
        which concludes the proof of the theorem. 
    \end{proof}

    \begin{theorem}\label{lorentz_l2_grad_hessian}
    Let $d\geq 3$. There exists a constant $\Gamma_d'<\infty$ with the following property. For all $0<a<b<\infty$, let $\Omega=B_b\setminus\bar{B}_a(0)\subset \R^d$ and assume for $d\geq 7$ that
    \begin{align*}
        \log\left(\frac{b}{a}\right)\geq \frac{1}{d-2}\log\left(\frac{d}{4}\right),
    \end{align*}
    and fix $u\in W^{2,2}(\Omega)$. Then, for all $0<\alpha<1$, we have $\D u\in L^{\frac{2d}{d-2},1}(\Omega_{\alpha})$ and 
    \begin{align}\label{lorentz_l2_grad_hessian_ineq}
        \np{\D (u-\lambda\cdot x)}{\frac{2d}{d-2},1}{\Omega_{\alpha}}\frac{\Gamma_d'}{\sqrt{1-\left(\frac{a}{b}\right)^{d-2}}}\frac{\alpha^{\frac{d}{2}}}{\left(1-\alpha^2\right)^{d-2}}\np{\D^2u}{2}{\Omega}.
    \end{align}
    \end{theorem}
    \begin{proof}
        Recall that by \eqref{gradient_pointwise_gen_d}, we have 
        \begin{align*}
            \left|\D \left(u-\sum_{k=1}^{d}a_{1,k}x_k\right)\right|&\leq \left(1+\Gamma_d(1)\right)\sum_{n=2}^{\infty}\sum_{k=1}^{N_d(n)}n\sqrt{N_d(n)}|a_{n,k}|r^{n-1}\\
            &+\left(1+\Gamma_d(1)\right)\sum_{n=0}^{\infty}\sum_{k=1}^{N_d(n)}(n+d-2)\sqrt{N_d(n)}|b_{n,k}|r^{-(n+d-1)}.
        \end{align*}
        In general, we see that
        \begin{align*}
            |\D^lu|&\leq \left(\sum_{k=0}^l\binom{l}{k}\Gamma_d(k)\right)\left(\sum_{n=l}^{\infty}\sum_{k=1}^{N_d(n)}n^l\sqrt{N_d(n)}\,|a_{n,k}|r^{n-l}\right.\\
            &\left.+\sum_{n=0}^{\infty}\sum_{k=1}^{N_d(n)}\frac{(n+d-2)!}{(n+d-2-l)!}\sqrt{N_d(n)}\,|b_{n,k}|r^{-(n+d-2+l)}\right).
        \end{align*}
        Recall also that
        \small
        \begin{align*}
            \beta(d)\left(\sum_{n=2}^{\infty}\sum_{k=1}^{N_d(n)}\frac{2n^4}{2n+d-4}|a_{n,k}|^2b^{2n+d-4}+\sum_{n=0}^{\infty}\sum_{k=1}^{N_d(n)}\frac{(2n^3+1)(n+d-2)}{2n+d}\frac{1}{a^{2n+d}}\right)\leq \frac{1}{1-\left(\frac{a}{b}\right)^{d-2}}\np{\D^2u}{2}{\Omega}.
        \end{align*}
        \normalsize
        We deduce if $\lambda=\left(a_{1,1},\cdots,a_{1,d}\right)$ that
        \begin{align*}
            &\np{\D \left(u-\lambda\cdot x\right)}{\frac{2d}{d-2},1}{\Omega_{\alpha}}\leq \frac{4d^2}{d^2-4}\left(\beta(d)\right)^{\frac{d-2}{2d}}\left(1+\Gamma_d(1)\right)\left(\sum_{n=2}^{\infty}\sum_{k=1}^{N_d(n)}n\sqrt{N_d(n)}|a_{n,k}|(\alpha\,b)^{n-1+\frac{d-2}{2}}\right.\\
            &\left.+\sum_{n=0}^{\infty}\sum_{k=1}^{N_d(n)}(n+d-2)\sqrt{N_d(n)}|b_{n,k}|\frac{2(n+d-1)}{2(n+d-1)-(d-2)}\frac{\alpha^{{n+d-1-\frac{d-2}{2}}}}{a^{n+d-1-\frac{d-2}{2}}}\right)\\
            &=\frac{4d^2}{d^2-4}\left(\beta(d)\right)^{\frac{d-2}{2d}}\left(1+\Gamma_d(1)\right)\left(\sum_{n=2}^{\infty}\sum_{k=1}^{N_d(n)}n\sqrt{N_d(n)}|a_{n,k}|\left(\alpha\,b\right)^{n+\frac{d}{2}-2}\right.\\
            &\left.+\sum_{n=0}^{\infty}\sum_{k=1}^{N_d(n)}(n+d-2)\sqrt{N_d(n)}|b_{n,k}|\frac{2(n+d-1)}{2n+d}\frac{\alpha^{n+\frac{d}{2}}}{a^{n+\frac{d}{2}}}\right)\\
            &\leq \frac{4d^2}{d^2-4}\left(\beta(d)\right)^{\frac{d-2}{2d}}\left(1+\Gamma_d(1)\right)\left(\left(\sum_{n=2}^{\infty}\frac{2n+d-4}{n^2}N_d(n)^2\alpha^{2n+d-4}\right)^{\frac{1}{2}}\right.\\
            &\times\left(\sum_{n=2}^{\infty}\sum_{k=1}^{N_d(n)}\frac{n^4}{2n+d-4}|a_{n,k}|^2b^{2n+d-4}\right)^{\frac{1}{2}}
            +\left(\sum_{n=0}^{\infty}\frac{4(n+d-1)^2(n+d-2)}{(2n^3+1)(2n+d)}\alpha^{2n+d}\right)^{\frac{1}{2}}\\
            &\left.\times\left(\sum_{n=0}^{\infty}\sum_{k=1}^{N_d(n)}\frac{(2n^3+1)(n+d-2)}{2n+d}\frac{1}{a^{2n+d}}\right)^{\frac{1}{2}}\right)
            \leq \frac{\Gamma_d'}{\sqrt{1-\left(\frac{a}{b}\right)^{d-2}}}\frac{\alpha^{\frac{d}{2}}}{\left(1-\alpha^2\right)^{d-2}}\np{\D^2u}{2}{\Omega}.
        \end{align*}
    \end{proof}

    \begin{theorem}\label{d2_l21}
    Let $d\geq 3$. Then, there exists a universal constant $\Gamma_d^{\ast}<\infty$ with the following property. For all $0<a<b<\infty$, define $B_b\setminus\bar{B}_a(0)\subset \R^d$. For all $0<\alpha<1$, define $\Omega_{\alpha}=B_{\alpha\,b}\setminus\bar{B}_{\alpha^{-1}a}(0)
    $, and let $u\in L^2(\Omega)$ be a harmonic function. If $d=3,4$, assume that for some $a\leq r\leq b$, we have
    \begin{align*}
        \int_{\partial B(0,r)}\partial_{\nu}u\,d\mathscr{H}^{d-1}=0.
    \end{align*}
    Assume that 
    \begin{align}
        \dfrac{b}{a}\geq \frac{9}{4}.
    \end{align}
    Then, for all $0<\alpha<1$, we have 
    \begin{align}\label{d2_l21_ineq}
         \np{\D u}{\frac{2d}{d+2},1}{\Omega_{\alpha}}\leq \left\{\begin{alignedat}{2}
                &\frac{\Gamma_d^{\ast}}{\sqrt{1-\left(\frac{a}{b}\right)^{d-2}}}\frac{\alpha^{\frac{d-2}{2}}}{\left(1-\alpha^2\right)^d}\np{u}{2}{\Omega}\qquad&&\text{if}\;\, d=3,4\nonumber\\
                &\frac{\Gamma_d^{\ast}}{\sqrt{1-\left(\frac{a}{b}\right)^{d-4}}}\frac{\alpha^{\frac{d-4}{2}}}{\left(1-\alpha^2\right)^d}\np{u}{2}{\Omega}\qquad&&\text{if}\;\, d\geq 5.
            \end{alignedat}\right.
    \end{align}
    \end{theorem}
    \begin{proof}
        We have for all $n\geq 1$ and for $p=\dfrac{2d}{d+2}$
        \begin{align*}
            \np{|x|^{n-1}}{p,1}{\Omega}\leq \frac{p^2}{p-1}\left(\beta(d)\right)^{\frac{1}{p}}b^{n-1+\frac{d}{p}}=\frac{4d^2}{d^2-4}\left(\beta(d)\right)^{\frac{d+2}{2d}}b^{n-1+\frac{d+2}{2}}=\frac{4d^2}{d^2-4}\left(\beta(d)\right)^{\frac{d+2}{2d}}b^{n+\frac{d}{2}},
        \end{align*}
        and for all $n\geq 1$, using that
        \begin{align*}
            \frac{d}{(n+d-1)p}=\frac{d}{(n+d-1)\frac{2d}{d+2}}=\frac{d+2}{2(n+d-1)}<1
        \end{align*}
        for all $n\geq 0$ and $d\geq 3$.
        Therefore, we have
        \begin{align*}
            \np{\frac{1}{|x|^{n+d-1}}}{p,1}{\Omega}&\leq \frac{4d^2}{d^2-4}\left(\beta(d)\right)^{\frac{d+2}{2d}}\frac{2(n+d-1)}{2(n+d-1)-(d+2)}\frac{1}{a^{n+d-1-\frac{d+2}{2}}}\\
            &=\frac{4d^2}{d^2-4}\left(\beta(d)\right)^{\frac{d+2}{2d}}\frac{2(n+d-1)}{2(n+d-1)-(d+2)}\frac{1}{a^{n+\frac{d}{2}-2}},
        \end{align*}
        and using \eqref{gradient_pointwise_gen_d}, we get
        \begin{align*}
            &\np{\D u}{\frac{2d}{d+2},1}{\Omega_{\alpha}}\leq \frac{4d^2}{d^2-4}\left(\beta(d)\right)^{\frac{d+2}{2d}}\left(1+\Gamma_d(1)\right)\sum_{n=1}^{\infty}\sum_{k=1}^{N_d(n)}n\,\sqrt{N_d(n)}|a_{n,k}|\left(\alpha\,b\right)^{n+\frac{d}{2}}\\
            &+\frac{4d^2}{d^2-4}\left(\beta(d)\right)^{\frac{d+2}{2d}}\sum_{n=0}^{\infty}\sum_{k=1}^{N_d(n)}\Big(n+d-2+\Gamma_d(1)\Big)\sqrt{N_d(n)}\frac{2(n+d-1)}{2(n+d-1)-(d+2)}|b_{n,k}|\frac{\alpha^{n+\frac{d}{2}-2}}{a^{n+\frac{d}{2}-2}}.
        \end{align*}
        Assuming for $d=3,4$ that
        \begin{align*}
            \int_{\partial B(0,r)}\partial_{\nu}u\,d\mathscr{H}^{d-1}=0,
        \end{align*}
        we deduce that 
        \begin{align}
            &\np{\D u}{\frac{2d}{d+2},1}{\Omega_{\alpha}}\leq \frac{4d^2}{d^2-4}\left(\beta(d)\right)^{\frac{d+2}{2d}}\left(1+\Gamma_d(1)\right)\left\{\left(\sum_{n=1}^{\infty}n^2(2n+d)\,N_d(n)^2\alpha^{n+\frac{d}{2}}\right)^{\frac{1}{2}}\right.\\
            &\times \left(\sum_{n=1}^{\infty}\sum_{k=1}^{N_d(n)}\frac{|b_{n,k}|^2}{2n+d}b^{2n+d}\right)^{\frac{1}{2}}
            +\left(\sum_{n=\mathbf{1}_{\ens{d\leq 4}}}^{\infty}\left(n+d-2\right)N_d(n)^2\frac{4(n+d-1)^2}{2n+d-4}\alpha^{2n+d-4}\right)^{\frac{1}{2}}\nonumber\\
            &\left.\times \left(\sum_{n=0}^{\infty}\sum_{k=1}^{N_d(n)}\frac{|b_{n,k}|^2}{2n+d-4}\frac{1}{a^{2n+d-4}}\right)^{\frac{1}{2}}\right\}
            \leq \left\{\begin{alignedat}{2}
                &\frac{\Gamma_d^{\ast}}{\sqrt{1-\left(\frac{a}{b}\right)^{d-2}}}\frac{\alpha^{\frac{d-2}{2}}}{\left(1-\alpha^2\right)^d}\np{u}{2}{\Omega}\qquad&&\text{if}\;\, d=3,4\nonumber\\
                &\frac{\Gamma_d^{\ast}}{\sqrt{1-\left(\frac{a}{b}\right)^{d-2}}}\frac{\alpha^{\frac{d-4}{2}}}{\left(1-\alpha^2\right)^d}\np{u}{2}{\Omega}\qquad&&\text{if}\;\, d\geq 5
            \end{alignedat}\right.
        \end{align}
        where we used that $\mathrm{deg}(n^2(2n+d)\,N_d(n)^2)=3+2(d-2)=2d-1$ and \eqref{u_bound_l2_gen_d}. 
    \end{proof}

    \begin{theorem}
    Let $d\geq 3$. There exists a constant $\Gamma_d^{\ast}<\infty$ with the following property. For all $0<a<b<\infty$, let $\Omega=B_b\setminus\bar{B}_a(0)\subset \R^d$ and assume that $u\in W^{1,2}(\Omega)$. Then, for all $0<\alpha<1$, we have $\D^2u\in L^{\frac{2d}{d+2},1}(\Omega_{\alpha})$ and 
    \begin{align}
        \np{\D^2u}{\frac{2d}{d+2},1}{\Omega_{\alpha}}\leq \frac{\Gamma_d^{\ast\ast}}{\sqrt{1-\left(\frac{a}{b}\right)^{d-2}}}\frac{\alpha^{\frac{d-2}{2}}}{\left(1-\alpha^2\right)^{d}}\np{\D u}{2}{\Omega}.
    \end{align}
    \end{theorem}
    \begin{proof}
        Recall the explicit upper bound:
        \begin{align*}
            &|\D^2u|\leq \left(1+2\,\Gamma_1(d)+\Gamma_2(d)\right)\sum_{n=2}^{\infty}\sum_{k=1}^{N_d(n)}n^2\sqrt{N_d(n)}|a_{n,k}|r^{n-2}\nonumber\\
            &+\left(1+2\,\Gamma_1(d)+\Gamma_2(d)\right)\sum_{n=0}^{\infty}\sum_{k=1}^{N_d(n)}(n+d-1)(n+d-2)\sqrt{N_d(n)}\,|b_{n,k}|r^{-(n+d)}
        \end{align*}
        and
        \begin{align*}
            \int_{\Omega}|\D u|^2dx=\beta(d)\sum_{n=0}^{\infty}\sum_{k=1}^{N_d(n)}\left(n|a_{n,k}|^2b^{2n+d-2}+(n+d-2)|b_{n,k}|^2\frac{1}{a^{2n+d-2}}\right)\left(1-\left(\frac{a}{b}\right)^{2n+d-2}\right). 
        \end{align*}
        Therefore, we get
        \begin{align*}
            &\np{\D^2u}{\frac{2d}{d+2},1}{\Omega_{\alpha}}\leq \frac{4d^2}{d^2-4}\left(\beta(d)\right)^{\frac{d+2}{2d}}\left(1+2\,\Gamma_1(d)+\Gamma_2(d)\right)\left\{\sum_{n=2}^{\infty}\sum_{k=1}^{N_d(n)}n^2\sqrt{N_d(n)}|a_{n,k}|\left(\alpha\,b\right)^{n-1+\frac{d}{2}}\right.\\
            &\left.+\sum_{n=0}^{\infty}\sum_{k=1}^{N_d(n)}(n+d-1)(n+d-2)\sqrt{N_d(n)}\frac{2(n+d)}{2(n+d)-(d+2)}|b_{n,k}|\frac{\alpha^{n+\frac{d}{2}-1}}{a^{n+\frac{d}{2}-1}}\right\}\\
            &\leq \frac{4d^2}{d^2-4}\left(\beta(d)\right)^{\frac{d+2}{2d}}\left(1+2\,\Gamma_1(d)+\Gamma_2(d)\right)\\
            &\times\left\{\left(\sum_{n=2}^{\infty}n^3N_d(n)^2|a_{n,k}|\alpha^{2n+d-2}\right)^{\frac{1}{2}}\left(\sum_{n=1}^{\infty}\sum_{k=1}^{N_d(n)}n|a_{n,k}|^2b^{2n+d-2}\right)^{\frac{1}{2}}\right.\\
            &\left.+
            \left(\sum_{n=0}^{\infty}(n+d-1)^2N_d(n)^2\frac{4(n+d)^2}{n+d-2}\alpha^{2n+d-2}\right)^{\frac{1}{2}}\left(\sum_{n=0}^{\infty}\sum_{k=1}^{N_d(n)}(n+d-2)|b_{n,k}|^2\frac{1}{a^{2n+d-2}}\right)^{\frac{1}{2}}\right\}\\
            &\leq \frac{\Gamma_d^{\ast\ast}}{\sqrt{1-\left(\frac{a}{b}\right)^{d-2}}}\frac{\alpha^{\frac{d-2}{2}}}{\left(1-\alpha^2\right)^{d}}\np{\D u}{2}{\Omega}.
        \end{align*}
    \end{proof}

    \begin{theorem}\label{d3_l21}
    Let $d\geq 3$. There exists a constant $\Gamma_d^{\ast\ast\ast}<\infty$ with the following property. For all $0<a<b<\infty$, let $\Omega=B_b\setminus\bar{B}_a(0)\subset \R^d$ and assume for $d\geq 7$ that
    \begin{align*}
        \log\left(\frac{b}{a}\right)\geq \frac{1}{d-2}\log\left(\frac{d}{4}\right),
    \end{align*}
    and fix $u\in W^{2,2}(\Omega)$. Then, for all $0<\alpha<1$, we have $\D^3u\in L^{\frac{2d}{d+2},1}(\Omega_{\alpha})$ and 
    \begin{align}\label{d3_l21_ineq}
        \np{\D^3u}{\frac{2d}{d+2},1}{\Omega_{\alpha}}\leq \frac{\Gamma_d^{\ast\ast\ast}}{\sqrt{1-\left(\frac{a}{b}\right)^{d-2}}}\frac{\alpha^{\frac{d}{2}}}{\left(1-\alpha^2\right)^{d}}\np{\D^2u}{2}{\Omega}
    \end{align}
    \end{theorem}
    \begin{proof}
        Recalling that
        \begin{align*}
            u(r,\omega)=\sum_{n=0}^{\infty}\sum_{k=1}^{N_d(n)}\left(a_{n,k}\,r^n+b_{n,k}r^{-(n+d-2)}\right)Y_n^k(\omega).
        \end{align*}
        In particular, we have
        \begin{align*}
            \D^3P=\D^3\left(\sum_{n=0}^{2}\sum_{k=1}^{N_d(n)}a_{n,k}r^nY_n^k(\omega)\right)=0
        \end{align*}
        since $P$ is a is a polynomial function of order at most $2$. Recall also that
        \begin{align*}
            |\D^l_{\omega}Y_n^k|\leq \Gamma_d(l)n^l\sqrt{N_k(n)}
        \end{align*}
        and that $\Gamma_d(0)=1$.
        We deduce that 
        \begin{align*}
            &|\D^3u(r,\omega)|\leq \left|\p{r}^3\left(u-P\right)\right|+\frac{1}{r}\left|\D_{\omega}(\p{r}^2\left(u-P\right))\right|+\left|\p{r}\left(\frac{1}{r}\D_{\omega}\p{r}\left(u-P\right)\right)\right|+\frac{1}{r^2}\left|\D_{\omega}^2\p{r}\left(u-P\right)\right|\\
            &+\left|\p{r}\left(\p{r}\left(\frac{1}{r}\D_{\omega}\left(u-P\right)\right)\right)\right|
            +\frac{1}{r}\left|\D_{\omega}\p{r}\left(\frac{1}{r}\D_{\omega}\left(u-P\right)\right)\right|\\
            &+\left|\p{r}\left(\frac{1}{r^2}\D_{\omega}^2\left(u-P\right)\right)\right|+\frac{1}{r^3}\left|\D_{\omega}^3\left(u-P\right)\right|\\
            &\leq \left(1+3\,\Gamma_d(1)+3\,\Gamma_d(2)+\Gamma_d(3)\right)\left(\sum_{n=3}^{\infty}\sum_{k=1}^{N_d(n)}n^3\sqrt{N_d(n)}\,|a_{n,k}|r^{n-3}\right.\\
            &\left.+\sum_{n=0}^{\infty}\sum_{k=1}^{N_d(n)}(n+d)(n+d-1)(n+d-2)\sqrt{N_d(n)}\,|b_{n,k}|r^{-(n+d+1)}\right).
        \end{align*}
        Then, recall that 
        \small
        \begin{align*}
            \beta(d)\left(\sum_{n=2}^{\infty}\sum_{k=1}^{N_d(n)}\frac{2n^4}{2n+d-4}|a_{n,k}|^2b^{2n+d-4}+\sum_{n=0}^{\infty}\sum_{k=1}^{N_d(n)}\frac{(2n^3+1)(n+d-2)}{2n+d}\frac{1}{a^{2n+d}}\right)\leq \frac{1}{1-\left(\frac{a}{b}\right)^{d-2}}\np{\D^2u}{2}{\Omega}.
        \end{align*}
        \normalsize
        provided if $d\geq 7$ that 
        \begin{align}\label{conf_hessian2_bis}
            \log\left(\frac{b}{a}\right)\geq \frac{1}{d-2}\log\left(\frac{d}{4}\right).
        \end{align}
        Therefore, we have
        \begin{align*}
            &\np{\D^3u}{\frac{2d}{d+2},1}{\Omega_{\alpha}}\leq \frac{4d^2}{d^2-4}\left(\beta(d)\right)^{\frac{d+2}{2d}}\left(1+3\,\Gamma_d(1)+3\,\Gamma_d(2)+\Gamma_d(3)\right)\left(\sum_{n=3}^{\infty}\sum_{k=1}^{N_d(n)}n^3\sqrt{N_d(n)}\,|a_{n,k}|b^{n-2+\frac{d}{2}}\right.\\
            &+\left.\sum_{n=0}^{\infty}\sum_{k=1}^{N_d(n)}(n+d)(n+d-1)(n+d-2)\sqrt{N_d(n)}\frac{2(n+d+1)}{2(n+d+1)-(d+2)}|b_{n,k}|\frac{\alpha^{n+\frac{d}{2}}}{a^{n+\frac{d}{2}}}\right)\\
            &=\frac{4d^2}{d^2-4}\left(\beta(d)\right)^{\frac{d+2}{2d}}\left(1+3\,\Gamma_d(1)+3\,\Gamma_d(2)+\Gamma_d(3)\right)\left(\sum_{n=3}^{\infty}\sum_{k=1}^{N_d(n)}n^3\sqrt{N_d(n)}\,|a_{n,k}|b^{n-2+\frac{d}{2}}\right.\\
            &+\left.\sum_{n=0}^{\infty}\sum_{k=1}^{N_d(n)}(n+d)(n+d-1)(n+d-2)\sqrt{N_d(n)}\frac{2(n+d+1)}{2n+d}|b_{n,k}|\frac{\alpha^{n+\frac{d}{2}}}{a^{n+\frac{d}{2}}}\right)\\
            &\leq \frac{4d^2}{d^2-4}\left(\beta(d)\right)^{\frac{d+2}{2d}}\widetilde{\Gamma_{d}}\left(\left(\sum_{n=3}^{\infty}n^2(2n+d-4)N_d(n)^2\alpha^{2n+d-4}\right)^{\frac{1}{2}}\left(\sum_{n=3}^{\infty}\sum_{k=1}^{N_d(n)}\frac{2n^4}{2n+d-4}|a_{n,k}|^2b^{2n+d-4}\right)^{\frac{1}{2}}\right.\\
            &\left.+\left(\sum_{n=0}^{\infty}\frac{4(n+d+1)^2}{2n+d}\frac{(n+d)^2(n+d-1)(n+d-2)}{2n^3+1}N_d(n)^2\alpha^{2n+d}\right)^{\frac{1}{2}}\right.\\
            &\left.\times \left(\sum_{n=0}^{\infty}\sum_{k=1}^{N_d(n)}\frac{(2n^2+1)(n+d-2)}{2n+d}\frac{1}{a^{2n+d}}\right)^{\frac{1}{2}}\right)\leq \frac{\Gamma_d^{\ast\ast\ast}}{\sqrt{1-\left(\frac{a}{b}\right)^{d-2}}}\frac{\alpha^{\frac{d}{2}}}{\left(1-\alpha^2\right)^{d}}\np{\D^2u}{2}{\Omega},
        \end{align*}
        which concludes the proof of the theorem.
    \end{proof}

    \subsection{Pointwise estimates}

    \begin{theorem}\label{pointwise_harmonic_u}
        Let $d\geq 3$. There exists a universal constant $\Lambda_d<\infty$ with the following property. For all $0<a<b<\infty$, let $\Omega=B_b\setminus\bar{B}_a(0)\subset \R^d$ and assume that
        \begin{align*}
            \frac{b}{a}\geq \frac{9}{4}.
        \end{align*}
        Let $u\in L^2(\Omega)$ be a harmonic function. Assume that for $d=3,4$, we have for some $a\leq r\leq b$
        \begin{align*}
            \int_{\partial B(0,r)}\partial_{\nu}u\,d\mathscr{H}^{d-1}=0.
        \end{align*}
        Then, $u\in L^{\infty}_{\mathrm{loc}}(\Omega)$, and for all $x\in \Omega$, provided that $d=3,4$, we have
        \begin{align*}
            |u(x)|\leq 
                \frac{\Lambda_d}{\sqrt{1-\left(\frac{a}{b}\right)^{d-2}}}\frac{1}{|x|^{\frac{d}{2}}}\left(\frac{1}{\left(1-\left(\frac{|x|}{b}\right)^2\right)^{d-1}}\left(\frac{|x|}{b}\right)^{\frac{d}{2}}+\frac{1}{\left(1-\left(\frac{a}{|x|}\right)^2\right)^{d-1}}\left(\frac{a}{|x|}\right)^{\frac{d-2}{2}}\right)\np{u}{2}{\Omega},
        \end{align*}
        while for $d\geq 5$, we have
        \begin{align*}
            |u(x)|\leq \frac{\Lambda_d}{\sqrt{1-\left(\frac{a}{b}\right)^{d-4}}}\frac{1}{|x|^{\frac{d}{2}}}\left(\frac{1}{\left(1-\left(\frac{|x|}{b}\right)^2\right)^{d-1}}\left(\frac{|x|}{b}\right)^{\frac{d}{2}}+\frac{1}{\left(1-\left(\frac{a}{|x|}\right)^2\right)^{d-1}}\left(\frac{a}{|x|}\right)^{\frac{d-4}{2}}\right)\np{u}{2}{\Omega}.
        \end{align*}
    \end{theorem}
    \begin{proof}
    Recall that by \eqref{no_flux_ineq}, we have
    \begin{align}\label{no_flux_ineq_bis}
            &\int_{\Omega}|u|^2dx\geq \frac{\beta(d)}{8}\sum_{n=0}^{\infty}\sum_{k=1}^{N_d(n)}\frac{|a_{n,k}|^2}{2n+d}b^{2n+d}\left(1-\left(\frac{a}{b}\right)^{2n+d}\right)\nonumber\\
            &+\frac{\beta(d)}{8}\sum_{n=0}^{\infty}\sum_{k=1}^{N_d(n)}\frac{|b_{n,k}|^2}{2n+d-4}\frac{1}{a^{2n+d-4}}\left(1-\left(\frac{a}{b}\right)^{2n+d-4}\right),
        \end{align}
        where $b_{0,1}$ for $d=3,4$, so that the sum is well-defined for all $d\geq 3,4$, and each term in the series is positive. 
        We have for all $x\in \Omega$
        \small
        \begin{align*}
            &|u(x)|\leq \sum_{n=0}^{\infty}\sum_{k=1}^{N_d(n)}\sqrt{N_d(n)}\,|a_{n,k}||x|^n+\sum_{n=0}^{\infty}\sum_{k=1}^{N_d(n)}\sqrt{N_d(n)}\,|b_{n,k}||x|^{-(n+d-2)}\\
            &=\frac{1}{|x|^{\frac{d}{2}}}\left(\sum_{n=0}^{\infty}\sum_{k=1}^{N_d(n)}\sqrt{N_d(n)}\,|a_{n,k}||x|^{n+\frac{d}{2}}+\sum_{n=0}^{\infty}\sum_{k=1}^{N_d(n)}\sqrt{N_d(n)}\,|b_{n,k}||x|^{-\left(n+\frac{d}{2}-2\right)}\right)\\
            &\leq \frac{1}{|x|^{\frac{d}{2}}}\left(\sum_{n=0}^{\infty}(2n+d)N_d(n)^2\left(\frac{x}{b}\right)^{2n+d}\right)^{\frac{1}{2}}\left(\sum_{n=0}^{\infty}\sum_{k=1}^{N_d(n)}\frac{|a_{n,k}|^2}{2n+d}b^{2n+d}\right)^{\frac{1}{2}}\\
            &+\frac{1}{|x|^{\frac{d}{2}}}\left(\sum_{n=0}^{\infty}(2n+d-4)N_d(n)^2\left(\frac{a}{|x|}\right)^{2n+d-4}\right)^{\frac{1}{2}}\left(\sum_{n=0}^{\infty}\sum_{k=1}^{N_d(n)}\frac{|b_{n,k}|^2}{2n+d-4}\frac{1}{a^{2n+d-4}}\right)^{\frac{1}{2}}\\
            &\leq \left\{\begin{alignedat}{2}
                &\frac{\Lambda_d}{\sqrt{1-\left(\frac{a}{b}\right)^{d-2}}}\frac{1}{|x|^{\frac{d}{2}}}\left(\frac{1}{\left(1-\left(\frac{|x|}{b}\right)^2\right)^{d-1}}\left(\frac{|x|}{b}\right)^{\frac{d}{2}}+\frac{1}{\left(1-\left(\frac{a}{|x|}\right)^2\right)^{d-1}}\left(\frac{a}{|x|}\right)^{\frac{d-2}{2}}\right)\np{u}{2}{\Omega}\qquad&&\text{if}\;\, d=3,4\\
                &\frac{\Lambda_d}{\sqrt{1-\left(\frac{a}{b}\right)^{d-4}}}\frac{1}{|x|^{\frac{d}{2}}}\left(\frac{1}{\left(1-\left(\frac{|x|}{b}\right)^2\right)^{d-1}}\left(\frac{|x|}{b}\right)^{\frac{d}{2}}+\frac{1}{\left(1-\left(\frac{a}{|x|}\right)^2\right)^{d-1}}\left(\frac{a}{|x|}\right)^{\frac{d-4}{2}}\right)\np{u}{2}{\Omega}\qquad&&\text{if}\;\, d\geq 5
            \end{alignedat}\right.
        \end{align*}
        \normalsize
        since $\deg((2n+d)N_d(n)^2)=1+2(d-2)=2d-3$. Here, $\Lambda_d$ is explicitly given by
        \begin{align*}
            \Lambda_d=2\sqrt{\frac{2\,C_d}{\beta(d)}},
        \end{align*}
        where $C_d<\infty$ is the minimal constant such that the following inequality holds true
        \begin{align*}
            \sum_{n=0}^{\infty}(2n+d)N_d(n)^2\beta^n\leq \frac{C_d}{(1-\beta)^{2(d-1)}}.
        \end{align*}
        For example, with $d=3$, we have $N_d(n)=2n+1$, and we have 
        \begin{align*}
            &(2n+3)(2n+1)^2=(2n+3)(4n^2+4n+1)=(2n+3)(4(n-1)(n-2)+12n-7)\\
            &=8n(n-1)(n-2)+2n(12(n-1)+5)+3(4n(n-1)+8n+1)\\
            &=8n(n-1)(n-2)+24n(n-1)+10n+12n(n-1)+24n+3\\
            &=8n(n-1)(n-2)+36n(n-1)+10n+3
        \end{align*}
        which implies as
        \begin{align*}
            &\sum_{n=0}^{\infty}\beta^n=\frac{1}{1-\beta}\\
            &\sum_{n=0}^{\infty}n\beta^n=\frac{\beta}{(1-\beta)^2}\\
            &\sum_{n=0}^{\infty}n(n-1)\beta^n=\frac{2\beta^2}{(1-\beta)^3}\\
            &\sum_{n=0}^{\infty}n(n-1)(n-2)\beta^n=\frac{6\beta^3}{(1-\beta)^4}
        \end{align*}
        that
        \begin{align*}
            &\sum_{n=0}^{\infty}(2n+3)(2n+1)^2\beta^n=\frac{48\beta^2}{(1-\beta)^4}+\frac{72\beta}{(1-\beta)^3}+\frac{10\beta}{(1-\beta)^2}+\frac{3}{(1-\beta)}\\
            &=\frac{48\beta^2+72\beta(1-\beta)+10\beta(1-2\beta+\beta^2)+3(1-3\beta+3\beta^2-\beta^3)}{(1-\beta)^4}\\
            &=\frac{3+73\beta-35\beta^2+7\beta^3}{(1-\beta)^4}\leq \frac{76}{(1-\beta)^4},
        \end{align*}
        which yields $C_3=76$ and
        \begin{align*}
            \Lambda_3=2\sqrt{\frac{2C_2}{\beta(3)}}=2\sqrt{\frac{38}{\pi}}.
        \end{align*}
        For $d=4$, we have $N_4(n)=(n+1)^2$, and
        \begin{align*}
            \sum_{n=0}^{\infty}(2n+4)(n+1)^4\beta^n=\frac{4\left(1+18\beta+33\beta^2+8\beta^3\right)}{(1-\beta)^6}=\frac{240}{(1-\beta)^4}
        \end{align*}
        which yields $C_4=240$ and
        \begin{align*}
            \Lambda_4=2\sqrt{\frac{C_4}{\beta(4)}}=\frac{2\sqrt{120}}{\pi}=\frac{8\sqrt{30}}{\pi}.
        \end{align*}
    \end{proof}

    \begin{theorem}\label{pointwise_harmonic_Du}
        Let $m\geq 1$ and $d\geq 3$. For all $0<a<b<\infty$, let $\Omega=B_b\setminus\bar{B}_a(0)\subset \R^d$, let $u:\Omega\rightarrow \R^m$ be a harmonic map such that $\D u\in L^2(\Omega)$. Then, $\D u\in \mathrm{L}^{\infty}_{\mathrm{loc}}(\Omega)$ and for all $x\in \Omega$, we have
        \begin{align}
            |\D u(x)|\leq \frac{\left(1+\Gamma_d(1)\right)\Lambda_d}{\sqrt{1-\left(\frac{a}{b}\right)^{d-1}}}\left(\frac{1}{\left(1-\left(\frac{|x|}{b}\right)^2\right)^{d-1}}\left(\frac{|x|}{b}\right)^{\frac{d}{2}}+\frac{1}{\left(1-\left(\frac{a}{|x|}\right)^2\right)^{d-1}}\left(\frac{a}{|x|}\right)^{\frac{d-2}{2}}\right)\np{\D u}{2}{\Omega}
        \end{align}
        Furthermore, assuming that for some $a\leq r\leq b$
        \begin{align*}
            \int_{\partial B(0,r)}\partial_{\nu}u\,d\mathscr{H}^{d-1}=0,
        \end{align*}
        the following estimate holds
        \begin{align}
            |\D u(x)|\leq \frac{\left(1+\Gamma_d(1)\right)\Lambda_d}{\sqrt{1-\left(\frac{a}{b}\right)^{d}}}\left(\frac{1}{\left(1-\left(\frac{|x|}{b}\right)^2\right)^{d-1}}\left(\frac{|x|}{b}\right)^{\frac{d}{2}}+\frac{1}{\left(1-\left(\frac{a}{|x|}\right)^2\right)^{d-1}}\left(\frac{a}{|x|}\right)^{\frac{d}{2}}\right)\np{\D u}{2}{\Omega}.
        \end{align}
    \end{theorem}
    \begin{proof}
        Recall the identity
        \begin{align*}
            \int_{\Omega}|\D u|^2dx=\beta(d)\sum_{n=0}^{\infty}\sum_{k=1}^{N_d(n)}\left(n|a_{n,k}|^2b^{2n+d-2}+(n+d-2)|b_{n,k}|^2\frac{1}{a^{2n+d-2}}\right)\left(1-\left(\frac{a}{b}\right)^{2n+d-2}\right),
        \end{align*}
        while
        \small
        \begin{align*}
            &|\D u(x)|\leq \frac{1+\Gamma_d(1)}{|x|^{\frac{d}{2}}}\left(\sum_{n=1}^{\infty}\sum_{k=1}^{N_d(n)}n\sqrt{N_d(n)}\,|a_{n,k}||x|^{n+\frac{d}{2}-1}+\sum_{n=0}^{\infty}\sum_{k=1}^{N_d(n)}(n+d-2)\sqrt{N_d(n)}\,|b_{n,k}||x|^{-(x+\frac{d}{2}-1)}\right)\\
            &\leq \frac{1+\Gamma_d(1)}{|x|^{\frac{d}{2}}}\left(\left(\sum_{n=1}^{\infty}nN_d(n)^2\left(\frac{|x|}{b}\right)^{2n+d-2}\right)^{\frac{1}{2}}\left(\sum_{n=1}^{\infty}\sum_{k=1}^{N_d(n)}n|a_{n,k}|^2\,b^{2n+d-2}\right)^{\frac{1}{2}}\right.\\
            &\left.+\left(\sum_{n=0}^{\infty}\sum_{k=1}^{N_d(n)}(n+d-2)N_d(n)^2\left(\frac{a}{|x|}\right)^{2n+d-2}\right)^{\frac{1}{2}}\left(\sum_{n=0}^{\infty}\sum_{k=1}^{N_d(n)}(n+d-2)|b_{n,k}|^2\frac{1}{a^{2n+d-2}}\right)^{\frac{1}{2}}\right)\\
            &\leq \frac{1+\Gamma_d(1)}{|x|^{\frac{d}{2}}}\left(\frac{1}{\sqrt{1-\left(\frac{a}{b}\right)^{d}}}\frac{\Lambda_d}{2\sqrt{2}\left(1-\left(\frac{|x|}{b}\right)^{2}\right)^{d-1}}\left(\frac{|x|}{b}\right)^{\frac{d}{2}}\right.\\
            &\left.+\frac{1}{\sqrt{1-\left(\frac{a}{b}\right)^{d-1}}}\frac{\Lambda_d}{\left(1-\left(\frac{a}{|x|}\right)^2\right)^{d-1}}\left(\frac{a}{|x|}\right)^{\frac{d-2}{2}}\right)\np{\D u}{2}{\Omega}
        \end{align*}
        \normalsize
        since $2n+d\geq n+d-2$. The second inequality follows immediately.
    \end{proof}

    \begin{theorem}\label{pointwise_harmonic_D2u}
        Let $m\geq 1$ and $d\geq 3$, let $0<a<b<\infty$ and let $\Omega=B_b\setminus\bar{B}_a(0)\subset \R^d$. If $d\geq 7$, assume that
        \begin{align*}
            \log\left(\frac{b}{a}\right)\geq \frac{1}{d-2}\log\left(\frac{d}{4}\right).
        \end{align*}
        Then, for all harmonic map $u:\Omega\rightarrow \R^m$, if $\D^2u\in L^2(\Omega)$, we have $\D^2u\in L^{\infty}_{\mathrm{loc}}(\Omega)$ and for all $x\in \Omega$, we have 
        \begin{align*}
            |\D^2u(x)|\leq \frac{\Lambda_d'}{\sqrt{1-\left(\frac{a}{b}\right)^{d-2}}}\frac{1}{|x|^{\frac{d}{2}}}\left(\frac{1}{\left(1-\left(\frac{|x|}{b}\right)^2\right)^{d-1}}\left(\frac{|x|}{b}\right)^{\frac{d}{2}}+\frac{1}{\left(1-\left(\frac{a}{|x|}\right)^{d-1}\right)}\left(\frac{a}{|x|}\right)^{\frac{d}{2}}\right)\np{\D^2u}{2}{\Omega},
        \end{align*}
        where $\Lambda_d'=\left(1+2\,\Gamma_1(d)+\Gamma_d(2)\right)(d-1)\sqrt{d-2}\,\Lambda_d$.
    \end{theorem}
    \begin{proof}
    Recall that
    \small
    \begin{align*}
            \beta(d)\left(\sum_{n=2}^{\infty}\sum_{k=1}^{N_d(n)}\frac{2n^4}{2n+d-4}|a_{n,k}|^2b^{2n+d-4}+\sum_{n=0}^{\infty}\sum_{k=1}^{N_d(n)}\frac{(2n^3+1)(n+d-2)}{2n+d}\frac{1}{a^{2n+d}}\right)\leq \frac{1}{1-\left(\frac{a}{b}\right)^{d-2}}\np{\D^2u}{2}{\Omega}.
        \end{align*}
        \normalsize
        provided if $d\geq 7$ that 
        \begin{align}\label{conf_hessian2_ter}
            \log\left(\frac{b}{a}\right)\geq \frac{1}{d-2}\log\left(\frac{d}{4}\right).
        \end{align}
        Therefeore, we get
        \small
        \begin{align*}
            &|\D^2 u(x)|\leq \left(1+2\,\Gamma_d(1)+\Gamma_d(2)\right)\left(\sum_{n=2}^{\infty}n^2\sqrt{N_d(n)}\,|a_{n,k}||x|^{n-2}\right.\nonumber\\
            &\left.+\sum_{n=0}^{\infty}\sum_{k=1}^{N_d(n)}(n+d-1)(n+d-2)\sqrt{N_d(n)}\,|b_{n,k}||x|^{-(n+d)}\right)\\
            &\leq \frac{1+2\,\Gamma_d(1)+\Gamma_d(2)}{|x|^{\frac{d}{2}}}\left(\left(\sum_{n=2}^{\infty}(2n+d-4)N_d(n)^2\left(\frac{|x|}{b}\right)^{2n+d-4}\right)^{\frac{1}{2}}\left(\sum_{n=0}^{\infty}\sum_{k=1}^{N_d(n)}\frac{n^4}{2n+d-4}|a_{n,k}|^2b^{2n+d-4}\right)^{\frac{1}{2}}\right.\\
            &\left.+\left(\sum_{n=0}^{\infty}\frac{(2n+d)(n+d-1)^2(n+d-2)}{2n^3+1}N_d(n)^2\left(\frac{a}{|x|}\right)^{2n+d}\right)^{\frac{1}{2}}\left(\sum_{n=0}^{\infty}\sum_{k=1}^{N_d(n)}\frac{(2n^3+1)(n+d-2)}{2n+d}\frac{1}{a^{2n+d}}\right)^{\frac{1}{2}}\right)\\
            &\leq \frac{\left(1+2\,\Gamma_d(1)+\Gamma_d(2)\right)(d-1)\sqrt{d-2}\,\Lambda_d}{\sqrt{1-\left(\frac{a}{b}\right)^{d-2}}}\frac{1}{|x|^{\frac{d}{2}}}\left(\frac{1}{\left(1-\left(\frac{|x|}{b}\right)^2\right)^{d-1}}\left(\frac{|x|}{b}\right)^{\frac{d}{2}}\right.\\
            &\left.+\frac{1}{\left(1-\left(\frac{a}{|x|}\right)^{d-1}\right)}\left(\frac{a}{|x|}\right)^{\frac{d}{2}}\right)\np{\D^2u}{2}{\Omega}.
        \end{align*}
        \normalsize
    \end{proof}

    \begin{theorem}\label{pointwise_harmonic_u_Du}
        Let $m\geq 1$, and let $d\geq 3$. There exists a constant $\Lambda_d''<\infty$ with the following property. Let $0<a<b<\infty$ and let $\Omega=B_b\setminus\bar{B}_a(0)\subset \R^2$. Let $u:\Omega\rightarrow \R^m$ be a harmonic map such that $\D u\in L^2(\Omega)$. Then, we have $u\in L^{\infty}_{\mathrm{loc}}(\Omega)$ and there exists $c\in \R^m$ such that for all $x\in \Omega$, we have
        \begin{align*}
            |u(x)-c|\leq \frac{\Lambda_d''}{|x|^{\frac{d-2}{2}}}\left(\frac{1}{\left(1-\left(\frac{|x|}{b}\right)^2\right)^{d-2}}\left(\frac{|x|}{b}\right)^{\frac{d}{2}}+\frac{1}{\left(1-\left(\frac{a}{|x|}\right)^2\right)^{d-2}}\left(\frac{a}{|x|}\right)^{\frac{d}{2}}\right)\np{\D u}{2}{\Omega}.
        \end{align*}
    \end{theorem}
    \begin{proof}
    Recall that
    \begin{align*}
            \int_{\Omega}|\D u|^2dx=\beta(d)\sum_{n=0}^{\infty}\sum_{k=1}^{N_d(n)}\left(n|a_{n,k}|^2b^{2n+d-2}+(n+d-2)|b_{n,k}|^2\frac{1}{a^{2n+d-2}}\right)\left(1-\left(\frac{a}{b}\right)^{2n+d-2}\right).
        \end{align*}
        We have
        \begin{align*}
            &|u(x)-a_{0,1}|\leq \sum_{n=1}^{\infty}\sum_{k=1}^{N_d(n)}\sqrt{N_d(n)}\,|a_{n,k}||x|^{n}+\sum_{n=0}^{\infty}\sum_{k=1}^{N_d(n)}\sqrt{N_d(n)}\,|b_{n,k}||x|^{-(n+d-2)}\\
            &=\frac{1}{|x|^{\frac{d-2}{2}}}\left(\sum_{n=1}^{\infty}\sum_{k=1}^{N_d(n)}\sqrt{N_d(n)}\,|a_{n,k}|b^{n+\frac{d-2}{2}}+\sum_{n=0}^{\infty}\sum_{k=1}^{N_d(n)}\sqrt{N_d(n)}\,|b_{n,k}||x|^{-\left(n+\frac{d-2}{2}\right)}\right)\\
            &\leq \frac{1}{|x|^{\frac{d-2}{2}}}\left(\left(\sum_{n=1}^{\infty}\frac{N_d(n)^2}{n}\left(\frac{|x|}{b}\right)^{2n+d-2}\right)^{\frac{1}{2}}\left(\sum_{n=1}^{\infty}\sum_{k=1}^{N_d(n)}n|a_{n,k}|^2b^{2n+d-2}\right)^{\frac{1}{2}}\right.\\
            &\left.\left(\sum_{n=0}^{\infty}\frac{N_d(n)^2}{n+d-2}\left(\frac{a}{|x|}\right)^{2n+d-2}\right)^{\frac{1}{2}}\left(\sum_{n=0}^{\infty}\sum_{k=1}^{N_d(n)}(n+d-2)\frac{1}{a^{2n+d-2}}\right)^{\frac{1}{2}}\right)\\
            &\leq \frac{\Lambda_d''}{|x|^{\frac{d-2}{2}}}\left(\frac{1}{\left(1-\left(\frac{|x|}{b}\right)^2\right)^{d-2}}\left(\frac{|x|}{b}\right)^{\frac{d}{2}}+\frac{1}{\left(1-\left(\frac{a}{|x|}\right)^2\right)^{d-2}}\left(\frac{a}{|x|}\right)^{\frac{d}{2}}\right)\np{\D u}{2}{\Omega}.
        \end{align*}
    \end{proof}

    \begin{theorem}\label{pointwise_harmonic_Du_D2u}
        Let $m\geq 1$, and let $d\geq 3$. There exists a constant $\Lambda_d'''<\infty$ with the following property. Let $0<a<b<\infty$ and let $\Omega=B_b\setminus\bar{B}_a(0)\subset \R^2$. Provided that $d\geq 7$ Let $u:\Omega\rightarrow \R^m$ be a harmonic map such that $\D^2 u\in L^2(\Omega)$. Then, we have $\D u\in L^{\infty}_{\mathrm{loc}}(\Omega)$ and there exists $\lambda\in M_{m,d}(\R)$ such that for all $x\in \Omega$, we have
        \begin{align*}
            |\D \left(u(x)-\lambda\cdot x\right)|\leq \frac{\Lambda_d'''}{|x|^{\frac{d-2}{2}}}\left(\frac{1}{\left(1-\left(\frac{|x|}{b}\right)^2\right)^{d-2}}\left(\frac{|x|}{b}\right)^{\frac{d}{2}}+\frac{1}{\left(1-\left(\frac{a}{|x|}\right)^2\right)^{d-2}}\left(\frac{a}{|x|}\right)^{\frac{d}{2}}\right)\np{\D^2u}{2}{\Omega}.
        \end{align*}
    \end{theorem}
    \begin{proof}
        The proof is identical as the previous one and we omit it.
    \end{proof}

     \section{Energy Quantization Revisited}

     As we mentioned it in the introduction, the first estimate required in the proof of the Morse stability is the strong $L^{2,1}$ energy quantization. We first revisit the $L^{2,1}$ estimate of \cite{biharmonic_quanta}, that is needed for the weighted estimate and to prove the strong energy quantization. We treat the case $f=0$ since we will only be concerned by (intrinsic or extrinsic) biharmonic maps for which this term does not occur. However, it would be easy to add this extra term, or, as in \cite{biharmonic_quanta}, replace $E$ by 
     \begin{align*}
         \frac{1}{2}\int_{B(0,1)}|\Delta u|^2dx+\int_{B(0,1)}u^{\ast}\Omega\qquad\text{or}\qquad \frac{1}{2}\int_{B(0,1)}\left|\left(\Delta u\right)^{T}\right|^2dx+\int_{B(0,1)}u^{\ast}\Omega,
     \end{align*}
     where $\Omega$ is a smooth $4$-form on $\R^n$. Recall that the energy quantization is equivalent to the \emph{no-neck energy property,} which asserts that the energy of the \enquote{necks} that connect the main map to its bubbles vanishes in the limit. Explicitly, a neck region is conformally given by $\Omega_k(\alpha)=B_{\alpha}\setminus\bar{B}_{\alpha^{-1}\rho_k}(0)\subset \R^4$. If $\ens{u_k}_{k\in \N}$ is a sequence of biharmonic maps of bounded energy, we have
     \begin{align}\label{neck_region1}
         \lim_{\alpha\rightarrow 0}\limsup_{k\rightarrow \infty}\sup_{\alpha^{-1}\rho_k<r<\alpha\rho_k}\left(\np{\D^2u_k}{2}{B_{2r}\setminus\bar{B}_r(0)}+\np{\D u_k}{4}{B_{2r}\setminus\bar{B}_r(0)}\right)=0.
     \end{align}
     For technical reasons, we will replace this condition by the following:
     \begin{align}\label{neck_region2}
         \lim_{\alpha\rightarrow 0}\limsup_{k\rightarrow \infty}\sup_{\alpha^{-1}\rho_k<r<\alpha\rho_k}\left(\np{\D^2u_k}{2}{B_{2r}\setminus\bar{B}_r(0)}+\np{\frac{\D u_k}{|x|}}{2}{B_{2r}\setminus\bar{B}_r(0)}\right)=0.
     \end{align}
     Thanks to the improved Sobolev embedding $W^{1,2}(\R^4)\hooklongrightarrow L^{4,2}(\R^4)$, one could also characterise neck regions by the following property:
     \begin{align}\label{neck_region3}
         \lim_{\alpha\rightarrow 0}\limsup_{k\rightarrow \infty}\sup_{\alpha^{-1}\rho_k<r<\alpha\rho_k}\left(\np{\D^2u_k}{2}{B_{2r}\setminus\bar{B}_r(0)}+\np{\D u_k}{4,2}{B_{2r}\setminus\bar{B}_r(0)}\right)=0.
     \end{align}
     However, the first and third conditions are not suitable for dyadic arguments since $\np{\D u}{4}{\,\cdot\,}^2$ and $\np{\D u}{4,2}{\,\cdot\,}^2$ (or $\np{\D u}{4,2}{\,\cdot\,}$) have bad sumability properties. Furthermore, the quantity that ones controls thanks to the Pohozaev identity (be it for the $L^2$ energy quantization or the $L^{2,1}$ energy quantization) is the one appearing in \eqref{neck_region2}, which shows that this is the most adapated one to biharmonic map. Furthermore, our Whitney extension Theorem \ref{whitney_extension_dim4} shows that the three conditions are equivalent. Indeed, if $0<2\,a<b<\infty$, $\Omega=B_b\setminus\bar{B}_a(0)$, then, for all $u\in W^{2,2}(\Omega)$, there is an extension $\widetilde{u}\in W^{2,2}(\R^4)$ such that
     \begin{align}\label{whitney_applied}
         &\max\ens{\np{\D^2\widetilde{u}}{2}{\R^4},\np{\D \widetilde{u}}{4,2}{\R^4},\np{\D\widetilde{u}}{4}{\R^4},\np{\frac{\D\widetilde{u}}{|x|}}{2}{\R^4}}\nonumber\\
         &\leq 10^5\left(\np{\D^2u}{2}{\Omega}+\min\ens{\np{\D u}{4,2}{\R^4},\np{\D u}{4}{\R^4},\np{\frac{\D u}{|x|}}{2}{\R^4}}\right).
     \end{align}
    
    \begin{theorem}\label{l21_neck}
        Let $(M^m,h)$ be a closed Riemannian manifold isometrically embedded in $\R^n$, let $0<a<b<\infty$ be such that
        \begin{align*}
            \frac{b}{a}>\frac{9}{4},
        \end{align*}
        and $\Omega=B_b\setminus\bar{B}_a(0)\subset \R^4$. There exists a universal constant $\epsilon_0=\epsilon_0(M^m,h)>0$ such that for all (intrinsic or extrinsic) biharmonic map $u\in W^{2,2}(\Omega,M^m)$ such that
        \begin{align*}
            \sup_{2a<r<\frac{b}{2}}\int_{B_{2r}\setminus\bar{B}_{\frac{r}{2}}(0)}\left(|\D^2u|^2+\frac{|\D u|^2}{|x|^2}\right)dx\leq \epsilon_0,
        \end{align*}
        there exists $N\in \N$ such that 
        \begin{align*}
            N\leq C\int_{\Omega}\left(|\D^2u|^2+\frac{|\D u|^2}{|x|^2}\right)dx
        \end{align*}
        and a subdivision $2a=a_0<a_1<\cdots<a_N=\dfrac{b}{2}$ such that for all $0\leq i\leq N-1$, there exists    $\gamma_i\in \R^n$ such that
        \begin{align}
            &\np{A\,\D^3 u-\D^3\left(\frac{1}{2}\gamma_i\log|x|\right)}{2,1}{B_{a_{i+1}}\setminus\bar{B}_{a_i}(0)}
            +\np{A\D^2u-\D^2\left(\frac{1}{2}\gamma_i\log|x|\right)}{2,1}{B_{a_{i+1}}\setminus\bar{B}_{a_i}(0)}\nonumber\\
            &+\np{A\D u-\D\left(\frac{1}{2}\gamma_i\log|x|\right)}{4,1}{B_{a_{i+1}}\setminus\bar{B}_{a_i}(0)}
            \leq C\left(\int_{\Omega}\left(|\D^2u|^2+\frac{|\D u|^2}{|x|^2}\right)dx\right)^{\frac{1}{2}}.
        \end{align}
        In particular, we have
        \begin{align}
            &\np{A\frac{1}{r}\D_{\omega}u}{4,1}{\Omega_{\frac{1}{4}}}\leq C\int_{\Omega}\left(|\D^2u|^2+\frac{|\D u|^2}{|x|^2}\right)dx\nonumber\\
            &\np{A\D\left(\frac{1}{r}\D_{\omega}u\right)}{2,1}{\Omega_{\frac{1}{4}}}+\np{A\frac{1}{r}\D_{\omega}\D u}{2,1}{\Omega_{\frac{1}{4}}}\leq C\int_{\Omega}\left(|\D^2u|^2+\frac{|\D u|^2}{|x|^2}\right)dx\nonumber\\
            &\np{A\D^2\left(\frac{1}{r}\D_{\omega}u\right)}{\frac{4}{3},1}{\Omega_{\frac{1}{4}}}+\np{A\D\left(\frac{1}{r}\D_{\omega}\D u\right)}{\frac{4}{3},1}{\Omega_{\frac{1}{4}}}+\np{A\frac{1}{r}\D_{\omega}\D^2u}{\frac{4}{3},1}{\Omega_{\frac{1}{4}}}\nonumber\\
            &\leq C\int_{\Omega}\left(|\D^2u|^2+\frac{|\D u|^2}{|x|^2}\right)dx.
        \end{align}
    \end{theorem}
     \begin{proof}
     \textbf{Case 1:}
     \begin{align}
         \int_{\Omega}\left(|\D^2u|^2+\frac{|\D u|^2}{|x|^2}\right)dx\leq \epsilon_0.
     \end{align}
        Thanks to our Whitney extension theorem below (Theorem \ref{whitney_extension_dim4}; see the inequality \eqref{whitney_applied}), there exists an extension $\widetilde{u}\in W^{2,2}(\R^4)$ such that $\mathrm{supp}(\widetilde{u})\subset B_{2b}\setminus\bar{B}_{\frac{a}{2}}(0)$ and 
        \begin{align}\label{epsilon_bound}
            \np{\D^2 \widetilde{u}}{2}{\R^4}+\np{\D\widetilde{u}}{4,2}{\R^4}\leq 10^5\left(\np{\D^2u}{2}{\Omega}+\np{\frac{\D u}{|x|}}{2}{\Omega}\right)\leq 10^5\sqrt{\epsilon_0}.
        \end{align}
        More precisely, we have
        \begin{align}\label{epsilon_bound_precise}
            &\np{\D^2 \widetilde{u}}{2}{B_{2b}\setminus\bar{B}_b(0)}+\np{\D\widetilde{u}}{4,2}{B_{2b}\setminus\bar{B}_b(0)}\leq 10^5\left(\np{\D^2u}{2}{B_b\setminus\bar{B}_{\frac{b}{2}}(0)}+\np{\frac{\D u}{|x|}}{2}{B_b\setminus\bar{B}_{\frac{b}{2}}(0)}\right)\nonumber\\
            &\np{\D^2 \widetilde{u}}{2}{B_{a}\setminus\bar{B}_{\frac{a}{2}}(0)}+\np{\D\widetilde{u}}{4,2}{B_{a}\setminus\bar{B}_{\frac{a}{2}}(0)}\leq 10^5\left(\np{\D^2u}{2}{B_{2a}\setminus\bar{B}_{a}(0)}+\np{\frac{\D u}{|x|}}{2}{B_{2a}\setminus\bar{B}_{a}(0)}\right).
        \end{align}
        Therefore, using the explicit representations given in \cite[Proposition $2.1$]{lauriv1}, we deduce that  
        there exists extensions
        \begin{align*}
            \left\{\begin{alignedat}{2}
                &\widetilde{V}\in W^{1,2}(B(0,b),\mathrm{M}_n(\R)\otimes \Lambda^1\R^4)\qquad&& \widetilde{w}\in L^2(B(0,b),M_n(\R))\\
                &\widetilde{\omega}\in L^2(B(0,b),\mathfrak{so}(n))\qquad&& \widetilde{F}\in L^2\cdot W^{1,2}(B(0,b),\mathrm{M}_n(\R)\otimes \Lambda^1\R^4)
            \end{alignedat}\right.
        \end{align*}
        of respectively $V,w,\omega$, and $F$ on $\Omega$ such that
        \begin{align*}
            &\wp{\widetilde{V}}{1,2}{B(0,b)}+\np{\widetilde{w}}{2}{B(0,b)}+\np{\widetilde{\omega}}{2}{B(0,b)}+\ntimeswp{\widetilde{F}}{2}{1,2}{B(0,b)}\\
            &\leq \Gamma\left(\wp{{V}}{1,2}{B(0,b)}+\np{{w}}{2}{\Omega}+\np{{\omega}}{2}{\Omega}+\ntimeswp{{F}}{2}{1,2}{\Omega}\right)
        \end{align*}
        for some universal constant $\Gamma<\infty$ independent of $0<2\,a<b<\infty$. Then, by \cite[Theorem $1.5$]{riviere_lamm_biharmonic}, using \eqref{epsilon_bound}, provided that $\epsilon_0>0$ is small enough, there exists $A\in W^{2,2}\cap L^{\infty}(B(0,b),\mathrm{GL}_n(\R))$ and $B\in W^{1,(\frac{4}{3},1)}(B(0,b),\R^n)$ such that
        \begin{align}\label{A_est}
            &\mathrm{dist}(A,\mathrm{SO}(n))+\wp{A}{2,2}{B(0,b)}+\wp{B}{1,(\frac{4}{3},1)}{B(0,b)}\nonumber\\
            &\leq C\left(\wp{\widetilde{V}}{1,2}{B(0,b)}+\np{\widetilde{w}}{2}{B(0,b)}+\np{\widetilde{\omega}}{2}{B(0,b)}+\ntimeswp{\widetilde{F}}{2}{1,2}{B(0,b)}\right)\nonumber\\
            &\leq C\left(\wp{{V}}{1,2}{B(0,b)}+\np{{w}}{2}{\Omega}+\np{{\omega}}{2}{\Omega}+\ntimeswp{{F}}{2}{1,2}{\Omega}\right)\nonumber\\
            &\leq C\left(\np{\D^2u}{2}{\Omega}+\np{\frac{\D u}{|x|}}{2}{\Omega}\right)
        \end{align}
        and
        \begin{align*}
            \D\Delta A+(\Delta A)\,\widetilde{V}-(\D A)\widetilde{w}+A(\D\widetilde{\omega}+\widetilde{F})=\D^{\perp}B.
        \end{align*}
        Then, $A\Delta\widetilde{u}$ solves an equation of the form
        \begin{align*}
            \Delta(A\,\Delta\widetilde{u})=\dive(K)\qquad\text{in}\;\, B(0,b)
        \end{align*}
        where 
        \begin{align}\label{estimate_K}
            \np{K}{\frac{4}{3},1}{B(0,b)}\leq C\left(1+\np{\D^2u}{2}{\Omega}+\np{\frac{\D u}{|x|}}{2}{\Omega}\right)\left(\np{\D^2u}{2}{\Omega}+\np{\frac{\D u}{|x|}}{2}{\Omega}\right)^2.
        \end{align}
        Indeed, $K$ is given explicitly as
        \begin{align*}
            K=2\,\D A\Delta \widetilde{u}-\Delta A\D \widetilde{u}+A\,w\D \widetilde{u}+\D A(V\D \widetilde{u})-A\D(V\D \widetilde{u})-B\D \widetilde{u}.
        \end{align*}
        Then, all open $U\subset \Omega$, we have by Hölder's inequality $L^{4,2}\cdot L^{2,2}\hookrightarrow L^{\frac{4}{3},1}$
        \begin{align*}
            \np{\D A\Delta \widetilde{u}}{\frac{4}{3},1}{U}&\leq C_{\mathrm{H}}\np{\D A}{4,2}{U}\np{\Delta \widetilde{u}}{2}{U}\leq C\np{A}{2,2}{B(0,b)}\np{\Delta \widetilde{u}}{2}{U}\\
            &\leq C\left(\np{\D^2u}{2}{\Omega}+\np{\frac{\D u}{|x|}}{2}{\Omega}\right)\np{\Delta\widetilde{u}}{2}{U}.
        \end{align*}
        Likewise, we have
        \begin{align*}
            \np{\Delta A\D\widetilde{u}}{\frac{4}{3},1}{U}\leq \np{\Delta A}{2}{U}\np{\D\widetilde{u}}{4,2}{U}\leq C\left(\np{\D^2u}{2}{\Omega}+\np{\frac{\D u}{|x|}}{2}{\Omega}\right)\np{\D\widetilde{u}}{4,2}{U}
        \end{align*}
        Notice that since $\mathrm{SO}(n)$ is compact, we have by \eqref{A_est}
        \begin{align*}
            \np{A}{\infty}{B(0,b)}\leq C\left(1+\mathrm{dist}(A,\mathrm{SO}(n))\right)\leq C\left(1+\np{\D^2u}{2}{\Omega}+\np{\frac{\D u}{|x|}}{2}{\Omega}\right).
        \end{align*}
        Therefore, we have
        \begin{align*}
            \np{A\,w\D\widetilde{u}}{\frac{4}{3},1}{U}\leq C\left(1+\np{\D^2u}{2}{\Omega}+\np{\frac{\D u}{|x|}}{2}{\Omega}\right)\left(\np{\D^2u}{2}{\Omega}+\np{\frac{\D u}{|x|}}{2}{\Omega}\right)\np{\D \widetilde{u}}{4,2}{U}.
        \end{align*}
        As $V\in W^{1,2}(\R^4)\hookrightarrow L^{4}(\R^4)$ (the improved $L^{4,2}$ embedding is not needed here), we get
        \begin{align*}
            \np{\D A(V\D\widetilde{u})}{\frac{4}{3},1}{U}&\leq \np{\D A}{4,2}{U}\np{V\D\widetilde{u}}{2}{U}\leq \np{\D A}{4,2}{U}\np{V}{4}{U}\np{\D\widetilde{u}}{4}{U}\\
            &\leq C\left(\np{\D^2u}{2}{\Omega}+\np{\frac{\D u}{|x|}}{2}{\Omega}\right)^2\np{\D\widetilde{u}}{4}{U}.
        \end{align*}
        Likewise, 
        \begin{align*}
            \np{A\D (V\D\widetilde{u})}{2}{U}&\leq C\left(1+\np{\D^2u}{2}{\Omega}+\np{\frac{\D u}{|x|}}{2}{\Omega}\right)\left(\np{\D^2u}{2}{\Omega}+\np{\frac{\D u}{|x|}}{2}{\Omega}\right)\\
            &\times \left(\np{\D^2\widetilde{u}}{2}{U}+\np{\D\widetilde{u}}{4,2}{U}\right),
        \end{align*}
        and the Sobolev embedding $W^{1,\frac{4}{3}}(\R^4)\hookrightarrow L^{2}(\R^4)$, we finally get
        \begin{align*}
            \np{B\D\widetilde{u}}{\frac{4}{3},1}{U}\leq C_{\mathrm{H}}\np{B}{2}{U}\np{\D\widetilde{u}}{4,2}{U}\leq C\left(\np{\D^2u}{2}{\Omega}+\np{\frac{\D u}{|x|}}{2}{\Omega}\right)\np{\D\widetilde{u}}{4,2}{U}.
        \end{align*}
        Finally, we get
        \small
        \begin{align}\label{ineq:localisation_K}
            \np{K}{\frac{4}{3},1}{U}&\leq C\left(1+\np{\D^2u}{2}{\Omega}+\np{\frac{\D u}{|x|}}{2}{\Omega}\right)\left(\np{\D^2u}{2}{\Omega}+\np{\frac{\D u}{|x|}}{2}{\Omega}\right)\left(\np{\D^2\widetilde{u}}{2}{U}+\np{\D\widetilde{u}}{4,2}{U}\right).
        \end{align}
        Therefore, \eqref{estimate_K} follows from \eqref{ineq:localisation_K} and \eqref{epsilon_bound}.
        \normalsize
        Now, make an expansion $A\Delta u=\varphi+\psi$, where 
        \begin{align*}
            \left\{\begin{alignedat}{2}
                \Delta\varphi&=\dive(K)\qquad&& \text{in}\;\, B(0,b)\\
                \varphi&=0\qquad&& \text{on}\;\,\partial B(0,b).
            \end{alignedat}\right.
        \end{align*}
        Thanks to Calder\'{o}n-Zygmund estimates and the Sobolev embedding $W^{1,(\frac{4}{3},1)}(\R^4)\hooklongrightarrow L^{2,1}(\R^4)$, we deduce that
        \begin{align*}
            \np{\D \varphi}{\frac{4}{3},1}{B(0,b)}+\np{\varphi}{2,1}{B(0,b)}\leq C\np{K}{\frac{4}{3},1}{B(0,b)}\leq C\left(\np{\D^2u}{2}{\Omega}+\np{\frac{\D u}{|x|}}{2}{\Omega}\right).
        \end{align*}
        As for $d=4$, we have
        \begin{align*}
            &\sum_{n=0}^{\infty}(2n+d)N_d(n)^2\alpha^{2n}=\sum_{n=0}^{\infty}(2n+4)(n+1)^4\alpha^{2n}=\frac{4(1+18\alpha^2+33\alpha^4+8\alpha^6)}{(1-\alpha^2)^6}\leq \frac{240}{(1-\alpha^2)^6}\\
            &\sum_{n=1}^{\infty}\frac{4(n+d-2)^2}{(2n+d-4)}N_d(n)^2\alpha^{2n+d-4}=\sum_{n=1}^{\infty}4(n+2)^2(n+1)^3\frac{n+1}{2n}\alpha^{2n}\leq \sum_{n=1}^{\infty}4(n+2)^2(n+1)^3\alpha^{2n}\\
            &=\frac{16\alpha^2\left(18+22\alpha^4-15\alpha^6+6\alpha^8-\alpha^{10}\right)}{(1-\alpha^2)^6}\leq \frac{640}{(1-\alpha^2)^6}\alpha^2,
        \end{align*}
        which yields the estimate $\Gamma_1(4)\leq 640$ and
        \begin{align*}
            C_4\leq 8\sqrt{2\,\Gamma_1(4)}=8\sqrt{2\cdot 2^7\cdot 5}=128\sqrt{5}.
        \end{align*}
        Since $\psi$ is a harmonic function, thanks to inequality \eqref{lorentz_l2_gen_d_ineq} of Theorem \ref{lorentz_l2_gen_d} we deduce that there exists $\gamma\in \R^n$ such that for all $0<\alpha<1$, we have
        \begin{align*}
            \np{\psi-\frac{\gamma}{|x|^2}}{2,1}{\Omega_{\alpha}}&\leq \frac{128\sqrt{5}}{\sqrt{1-\left(\frac{a}{b}\right)^2}}\frac{\alpha}{(1-\alpha^2)^3}\np{\psi}{2}{\Omega}\\
            &\leq \frac{C}{\sqrt{1-\left(\frac{a}{b}\right)^2}}\frac{\alpha}{(1-\alpha^2)^3}\left(\np{\D^2u}{2}{\Omega}+\np{\frac{\D u}{|x|}}{2}{\Omega}\right)
        \end{align*}
        Therefore, we deduce that 
        \begin{align*}
            \np{A\Delta u-\frac{\gamma}{|x|^2}}{2,1}{\Omega_{\alpha}}\leq C\left(1+\frac{1}{\sqrt{1-\left(\frac{a}{b}\right)^2}}\frac{\alpha}{(1-\alpha^2)^3}\right)\left(\np{\D^2u}{2}{\Omega}+\np{\frac{\D u}{|x|}}{2}{\Omega}\right).
        \end{align*}
        Then, we have
        \begin{align*}
            \dive(A\D\widetilde{u})=A\,\Delta\widetilde{u}+\D A\cdot \D\widetilde{u}\qquad\text{in}\;\, B(0,b).
        \end{align*}
        Thanks to the Sobolev embedding $W^{1,2}(B(0,b))\hookrightarrow L^{4,2}(B(0,b))$ and Hölder's inequality for Lorentz spaces, we deduce that 
        \begin{align*}
            \np{\D A\cdot \D \widetilde{u}}{2,1}{B(0,b)}&\leq C\np{\D A}{4,2}{B(0,b)}\np{\D \widetilde{u}}{4,2}{B(0,b)}\\
            &\leq C\left(\np{\D^2A}{2}{B(0,b)}+\frac{1}{b}\np{\D A}{2}{B(0,b)}\right)\left(\np{\D^2\widetilde{u}}{2}{B(0,b)}+\frac{1}{b}\np{\D \widetilde{u}}{2}{B(0,b)}\right).
        \end{align*}
        Now, we have by Gagliardo-Nirenberg inequality \cite[Exemple $1$ p. $194$]{brezis}
        \begin{align*}
            \np{\D A}{2}{B(0,b)}&\leq C_{\mathrm{GN}}\np{A}{2}{B(0,b)}^{\frac{1}{2}}\np{\D^2A}{2}{B(0,b)}^{\frac{1}{2}}\leq C_{\mathrm{GN}}\\
            &\leq (2\pi^2)^{\frac{1}{4}}b\,C_{\mathrm{GN}}\np{A}{\infty}{B(0,b)}^{\frac{1}{2}}\np{\D^2A}{2}{B(0,b)}\\
            &\leq C\,b\,\left(\np{\D^2u}{2}{\Omega}+\np{\frac{\D u}{|x|}}{2}{\Omega}\right)\\
            \np{\D \widetilde{u}}{2}{B(0,b)}^{\frac{1}{4}}&\leq C_{\mathrm{GN}}\np{\widetilde{u}}{2}{B(0,b)}^{\frac{1}{2}}\np{\D^2\widetilde{u}}{2}{B(0,b)}^{\frac{1}{2}}\leq C\,b\,\left(\np{\D^2u}{2}{\Omega}+\np{\frac{\D u}{|x|}}{2}{\Omega}\right)^{\frac{1}{2}},
        \end{align*}
        where we used that $u$ takes values into the compact manifold $M^n\subset \R^d$.
        Therefore, we finally get the inequality
        \begin{align}\label{l42_product}
            \np{\D A\cdot \D\widetilde{u}}{2,1}{B(0,b)}\leq C\left(\np{\D^2u}{2}{\Omega}+\np{\frac{\D u}{|x|}}{2}{\Omega}\right)^2.
        \end{align}
        Therefore, we deduce that for all $0<\alpha<1$
        \begin{align*}
            \dive(A\D\widetilde{u})=\frac{\gamma}{|x|^2}+\varphi+\left(\psi-\frac{\gamma}{|x|^2}\right)+\D A\cdot \D\widetilde{u}=\frac{\gamma}{|x|^2}+F\qquad \text{in}\;\,\Omega_{\alpha},
        \end{align*}
        where 
        \begin{align*}
            \np{F}{2,1}{\Omega_{\alpha}}\leq C\left(1+\frac{1}{\sqrt{1-\left(\frac{a}{b}\right)^2}}\frac{\alpha}{(1-\alpha^2)^3}\right)\left(\np{\D^2 u}{2}{\Omega}+\np{\frac{\D u}{|x|}}{2}{\Omega}\right).
        \end{align*}
        Therefore, making a Hodge decomposition as in \cite{biharmonic_quanta} (\cite[Corollary $10.5.1$]{iwaniec}) $A\, d\widetilde{u}=d\alpha+d^{\ast}\beta$, where
        \begin{align*}
            \left\{\begin{alignedat}{2}
                \Delta \alpha&=\frac{\gamma}{|x|^2}+F\qquad&&\text{in}\;\,\Omega_{\alpha}\\
            \end{alignedat}\right.
        \end{align*}
        and
        \begin{align*}
            \left\{\begin{alignedat}{2}
                \Delta\beta&=dA\wedge d\widetilde{u}\qquad&&\text{in}\;\, B(0,\alpha\,b)\\
                \beta&=0\qquad&&\text{on}\;\,\partial B(0,\alpha\,b).
            \end{alignedat}\right.
        \end{align*}
        By standard Caldr\'{o}n-Zygmund estimates and interpolation theory, we have
        \begin{align*}
            \np{\D^2\beta}{2,1}{B(0,\alpha\,b)}+\np{\D\beta}{4,1}{B(0,\alpha b)}\leq C\left(\np{\D^2u}{2}{\Omega}+\np{\frac{\D u}{|x|}}{2}{\Omega}\right)^2,
        \end{align*}
        Furthermore, we have
        \begin{align*}
            \Delta(\D \beta)=\D dA\wedge d\widetilde{u}+dA\wedge \D d\widetilde{u},
        \end{align*}
        and thanks to the $L^{4,2}$ estimate of $\D A$ and $\D \widetilde{u}$, we deduce by Hölder's inequality for Lorentz spaces
         \begin{align*}
            &\np{|\D^2A||\D \widetilde{u}}{\frac{4}{3},1}{B(0,b)}+\np{|\D A||\D^2\widetilde{u}|}{\frac{4}{3},1}{B(0,b)}\\
            &\leq \Gamma_{H}\left(\np{\D^2A}{2}{B(0,b)}\np{\D \widetilde{u}}{4,2}{B(0,b)}+\np{\D A}{4,2}{B(0,b)}\np{\D^2\widetilde{u}}{2}{B(0,b)}\right)\\
            &\leq C\left(\np{\D^2 u}{2}{\Omega}+\np{\frac{\D u}{|x|}}{2}{\Omega}\right)^2.
        \end{align*}
        Therefore, Calder\'{o}n-Zygmund estimates, we also get
        \begin{align*}
            \np{\D^3\beta}{\frac{4}{3},1}{B(0,\alpha\,b)}\leq C\left(\np{\D^2u}{2}{\Omega}+\np{\frac{\D u}{|x|}}{2}{\Omega}\right).
        \end{align*}
        Now, recall that 
        \begin{align*}
            F=\varphi+\left(\psi-\frac{\gamma}{|x|^2}\right)+\D A\cdot \D\widetilde{u}.
        \end{align*}
        Using the estimate \eqref{d2_l21_ineq} from Theorem \ref{d2_l21}, we deduce that 
        \begin{align*}
            \np{\D \left(\psi-\frac{\gamma}{|x|^2}\right)}{\frac{4}{3},1}{\Omega_{\alpha}}\leq \frac{\Gamma_d^{\ast}}{\sqrt{1-\left(\frac{a}{b}\right)^2}}\frac{\alpha}{\left(1-\alpha^2\right)^4}\np{\psi-\frac{\gamma}{|x|^2}}{2}{\Omega}
        \end{align*}
        Therefore, thanks to the previous estimates on $\varphi$ and $\D A\cdot \D\widetilde{u}$, we deduce that
        \begin{align*}
            \np{\D F}{\frac{4}{3},1}{\Omega_{\alpha}}\leq C\left(1+\frac{1}{\sqrt{1-\left(\frac{a}{b}\right)^2}}\frac{\alpha}{\left(1-\alpha^2\right)^3}\right)\left(\np{\D^2u}{2}{\Omega}+\np{\frac{\D u}{|x|}}{2}{\Omega}\right).
        \end{align*}
        On the other hand, let $\widetilde{F}$ is a controlled extension of $F$ and $\widetilde{\alpha}:B(0,\alpha\,b)\rightarrow \R^n$ be such that
        \begin{align*}
            \left\{\begin{alignedat}{2}
                \Delta\widetilde{\alpha}&=\widetilde{F}\qquad&&\text{in}\;\, B(0,\alpha\,b)\\
                \widetilde{\alpha}&=0\qquad&&\text{on}\;\, \partial B(0,\alpha\,b),
            \end{alignedat}\right.
        \end{align*}
        As previously, by Calder\'{o}n-Zygmund estimates, we have
        \begin{align*}
            &\np{\D^3\alpha}{\frac{4}{3},1}{B(0,\alpha\,b)}+\np{\D^2 \widetilde{\alpha}}{2,1}{B(0,\alpha\,b)}+\np{\D \widetilde{\alpha}}{4,1}{B(0,\alpha\,b)}+\np{\widetilde{\alpha}}{\infty}{B(0,\alpha\,b)}\leq C\np{\widetilde{F}}{2,1}{B(0,\alpha\,b)}\\
            &\leq C\np{F}{2,1}{\Omega}.
        \end{align*}
        Then, we get by the estimate \eqref{lorentz_l2_hessian_ineq} of Theorem \ref{lorentz_l2_hessian} that for all $0<\beta<1$
        \begin{align*}
            \np{\D^2\left(\alpha-\widetilde{\alpha}-\frac{1}{2}\gamma\log|x|\right)}{2,1}{\Omega_{\alpha\,\beta}}&\leq \frac{C_4''}{\sqrt{1-\left(\frac{\alpha^2a}{b}\right)^2}}\frac{\beta^2}{(1-\beta^2)^3}\np{\D^2\left(\alpha-\widetilde{\alpha}-\frac{1}{2}\gamma\log|z|\right)}{2}{\Omega_{\alpha}}\\
            &\leq \frac{C}{\sqrt{1-\left(\frac{\alpha^2a}{b}\right)^2}}\frac{\beta^2}{(1-\beta^2)^3}\left(\np{\D^2u}{2}{\Omega}+\np{\frac{\D u}{|x|}}{2}{\Omega}\right)^2.
        \end{align*}
        On the other hand, thanks to inequality \eqref{lorentz_l2_grad_hessian_ineq} of Theorem \ref{lorentz_l2_grad_hessian}, there exists $c_0\in \R$ such that 
        \begin{align*}
            \np{\D\left(\alpha-\widetilde{\alpha}-\frac{1}{2}\gamma\log|x|\right)-c_0}{4,1}{\Omega_{\alpha\,\beta}}&\leq \frac{\Gamma_4'}{\sqrt{\pi}\sqrt{1-\left(\frac{\alpha^2a}{b}\right)^2}}\frac{\beta^2}{(1-\beta^2)^2}\np{\D^2\left(\alpha-\widetilde{\alpha}-\frac{1}{2}\gamma\log|x|\right)}{2}{\Omega_{\alpha}}\\
            &\leq \frac{C}{\sqrt{1-\left(\frac{\alpha^2a}{b}\right)^2}}\frac{\beta^2}{(1-\beta^2)^2}\left(\np{\D^2u}{2}{\Omega}+\np{\frac{\D u}{|x|}}{2}{\Omega}\right)^2,
        \end{align*}
        but the previous $L^4$ bound shows that 
        \begin{align*}
            \np{\D\left(\alpha-\widetilde{\alpha}-\frac{1}{2}\gamma\log|x|\right)}{4,1}{\Omega_{\alpha\,\beta}}\leq C\left(1+\frac{1}{\sqrt{1-\left(\frac{\alpha^2a}{b}\right)^2}}\frac{\beta^2}{(1-\beta^2)^2}\right)\left(\np{\D^2u}{2}{\Omega}+\np{\frac{\D u}{|x|}}{2}{\Omega}\right)^2.
        \end{align*}
        Therefore, we deduce that 
        \begin{align*}
            \np{\D^2\left(\alpha-\frac{1}{2}\gamma\log|x|\right)}{2,1}{\Omega_{\frac{1}{4}}}+\np{\D\left(\alpha-\frac{1}{2}\gamma\log|x|\right)}{4,1}{\Omega_{\frac{1}{4}}}\leq C\left(\np{\D^2 u}{2}{\Omega}+\np{\frac{\D u}{|x|}}{2}{\Omega}\right).
        \end{align*}
        In particular, we deduce that 
        \begin{align*}
            &\np{\p{r}\left(\frac{1}{r}\D_{\omega}\alpha\right)}{2,1}{\Omega_{\frac{1}{4}}}+\np{\frac{1}{r}\D_{\omega}\left(\p{r}\alpha\right)}{2,1}{\Omega_{\frac{1}{4}}}+\np{\frac{1}{r^2}\D_{\omega}^2\alpha}{2,1}{\Omega_{\frac{1}{4}}}\leq C\left(\np{\D^2 u}{2}{\Omega}+\np{\frac{\D u}{|x|}}{2}{\Omega}\right)\\
            &\np{\frac{1}{r}\D_{\omega}\alpha}{4,1}{\Omega_{\frac{1}{4}}}\leq C\left(\np{\D^2u}{2}{\Omega}+\np{\frac{\D u}{|x|}}{2}{\Omega}\right).
        \end{align*}
        Finally, we have by equation \eqref{d3_l21_ineq} of Theorem \ref{d3_l21} the estimate 
        \begin{align*}
            \np{\D^3\left(\alpha-\widetilde{\alpha}-\frac{1}{2}\gamma\log|x|\right)}{\frac{4}{3},1}{\Omega_{\alpha\,\beta}}&\leq \frac{\Gamma_4^{\ast\ast\ast}}{\sqrt{1-\left(\frac{a}{b}\right)^2}}\frac{\beta^2}{\left(1-\beta^2\right)^4}\np{\D^2\left(\alpha-\widetilde{\alpha}-\frac{1}{2}\gamma\log|x|\right)}{2}{\Omega_{\alpha}}\\
            &\leq C_{\alpha,\beta}\left(\np{\D^2u}{2}{\Omega}+\np{\frac{\D u}{|x|}}{2}{\Omega}\right),
        \end{align*}
        which also gives 
        \begin{align*}
            &\np{\D^3\left(\alpha-\frac{1}{2}\gamma\log|x|\right)}{\frac{4}{3},1}{\Omega_{\alpha\,\beta}}
            \leq C_{\alpha}\left(\np{\D^2u}{2}{\Omega}+\np{\frac{\D u}{|x|}}{2}{\Omega}\right)
        \end{align*}
        and
        \begin{align*}
            &\np{\frac{1}{r}\D_{\omega}\left(\p{r}^2\alpha\right)}{\frac{4}{3},1}{\Omega_{\frac{1}{4}}}+\np{\p{r}\left(\frac{1}{r}\D_{\omega}\p{r}\alpha\right)}{\frac{4}{3},1}{\Omega_{\frac{1}{4}}}
            +\np{\frac{1}{r^2}\D_{\omega}^2\alpha}{\frac{4}{3},1}{\Omega_{\frac{1}{4}}}+\np{\p{r}^2\left(\frac{1}{r}\D_{\omega}\alpha\right)}{\frac{4}{3},1}{\Omega_{\frac{1}{4}}}\\
            &+\np{\p{r}\left(\frac{1}{r^2}\D_{\omega}^2\alpha\right)}{\frac{4}{3},1}{\Omega_{\frac{1}{4}}}+\np{\frac{1}{r^3}\D_{\omega}^3\alpha}{\frac{4}{3},1}{\Omega}\leq C\left(\np{\D^2u}{2}{\Omega}+\np{\frac{\D u}{|x|}}{2}{\Omega}\right).
        \end{align*}
        Thanks to the Hodge decomposition $A\,d\widetilde{u}=d\alpha+d^{\ast}\beta$, we deduce that 
        \begin{align}
            &\np{\D^2\left(A\D u-\D\left(\frac{1}{2}\gamma\log|x|\right)\right)}{\frac{4}{3},1}{\Omega_{\frac{1}{4}}}+\np{\D\left(A\D u-\D \left(\frac{1}{2}\gamma\log|x|\right)\right)}{2,1}{\Omega_{\frac{1}{4}}}\nonumber\\
            &+\np{A\D u-\D\left(\frac{1}{2}\gamma\log|x|\right)}{4,1}{\Omega_{\frac{1}{4}}}\leq C\left(\np{\D^2 u}{2}{\Omega}+\np{\frac{\D u}{|x|}}{2}{\Omega}\right).
        \end{align}
        Furthermore, the previous $L^{4,2}$ estimate \eqref{l42_product} shows that
        \begin{align}
            \np{A\D^2u-\D^2\left(\frac{1}{2}\gamma\log|x|\right)}{2,1}{\Omega_{\frac{1}{4}}}\leq C\left(\np{\D^2 u}{2}{\Omega}+\np{\frac{\D u}{|x|}}{2}{\Omega}\right).
        \end{align}
        Likewise, as $\D^2A\cdot \D u\in L^{2}\cdot L^{4,2}\hookrightarrow L^{\frac{4}{3},1}(\Omega_{\frac{1}{4}})$ by Hölder's inequality, we deduce that 
        \begin{align*}
            \np{A \D^3u-\D^3\left(\frac{1}{2}\gamma\log|x|\right)}{\frac{4}{3},1}{\Omega_{\frac{1}{4}}}\leq C\left(\np{\D^2u}{2}{\Omega}+\np{\frac{\D u}{|x|}}{2}{\Omega}\right).
        \end{align*}
        Since $\log|x|$ only appears in purely radial derivatives, we deduce the following estimates
        \begin{align*}
            &\np{A\frac{1}{r}\D_{\omega}u}{4,1}{\Omega_{\frac{1}{4}}}\leq C\left(\np{\D^2u}{2}{\Omega}+\np{\frac{\D u}{|x|}}{2}{\Omega}\right)\\
            &\np{A\D\left(\frac{1}{r}\D_{\omega}u\right)}{2,1}{\Omega_{\frac{1}{4}}}+\np{A\frac{1}{r}\D_{\omega}\D u}{2,1}{\Omega_{\frac{1}{4}}}\leq C\left(\np{\D^2u}{2}{\Omega}+\np{\frac{\D u}{|x|}}{2}{\Omega}\right)\\
            &\np{A\D^2\left(\frac{1}{r}\D_{\omega}u\right)}{\frac{4}{3},1}{\Omega_{\frac{1}{4}}}+\np{A\D\left(\frac{1}{r}\D_{\omega}\D u\right)}{\frac{4}{3},1}{\Omega_{\frac{1}{4}}}+\np{A\frac{1}{r}\D_{\omega}\D^2u}{\frac{4}{3},1}{\Omega_{\frac{1}{4}}}\\
            &\leq C\left(\np{\D^2u}{2}{\Omega}+\np{\frac{\D u}{|x|}}{2}{\Omega}\right),
        \end{align*}
        which concludes the proof of the theorem in the first case. \\
        
        \textbf{Case 2: General Case}. 
         Thanks to the non-concentration hypothesis in all dyadic annuli, we can follow the argument of \cite{angular} to find a sequence of radii $4a=4a_0<a_1=b_0<\cdots a_{i+1}=b_i< \cdots b_N=4b$, where        \begin{align*}
            \int_{B_{4b_i}\setminus\bar{B}_{\frac{a_i}{4}}(0)}\left(|\D^2u|+\frac{|\D u|^2}{|x|^2}\right)dx\leq \epsilon(M^m,h)\qquad\text{and}\;\, N\leq \frac{1}{\epsilon(M^m,h)}\int_{\Omega}\left(|\D^2u|+\frac{|\D u|^2}{|x|^2}\right)dx.
        \end{align*}
        Therefore, applying the previous step, we deduce that for all $0\leq i\leq N-1$, there exists $\gamma_i\in \R^n$ such that 
        \begin{align}
            &\np{A\,\D^3 u-\D^3\left(\frac{1}{2}\gamma_i\log|x|\right)}{2,1}{B_{a_{i+1}}\setminus\bar{B}_{a_i}(0)}
            +\np{A\D^2u-\D^2\left(\frac{1}{2}\gamma_i\log|x|\right)}{2,1}{B_{a_{i+1}}\setminus\bar{B}_{a_i}(0)}\nonumber\\
            &+\np{A\D u-\D\left(\frac{1}{2}\gamma\log|x|\right)}{4,1}{B_{a_{i+1}}\setminus\bar{B}_{a_i}(0)}
            \leq C\left(\np{\D^2u}{2}{\Omega}+\np{\frac{\D u}{|x|}}{2}{\Omega}\right).
        \end{align}
        In particular, we get 
        \begin{align*}
            &\np{A\frac{1}{r}\D_{\omega}u}{4,1}{\Omega_{\frac{1}{4}}}\leq \sum_{i=0}^{N-1}\np{A\frac{1}{r}\D_{\omega}u}{4,1}{B_{a_{i+1}}\setminus\bar{B}_{a_i}(0)}\leq C\sum_{i=1}^{N-1}\left(\int_{B_{a_{i+1}}\setminus\bar{B}_{a_i}(0)}\left(|\D^2u|^2+\frac{|\D u|^2}{|x|^2}\right)dx\right)^{\frac{1}{2}}\\
            &\leq C\sqrt{N}\left(\sum_{i=0}^{N-1}\int_{B_{a_{i+1}}\setminus\bar{B}_{a_i}(0)}\left(|\D^2u|^2+\frac{|\D u|^2}{|x|^2}\right)dx\right)^{\frac{1}{2}}=C\sqrt{N}\left(\int_{\Omega_{\frac{1}{4}}}\left(|\D^2u|^2+\frac{|\D u|^2}{|x|^2}\right)dx\right)^{\frac{1}{2}}\\
            &\leq C\int_{\Omega_{\frac{1}{4}}}\left(|\D^2u|^2+\frac{|\D u|^2}{|x|^2}\right)dx.
        \end{align*}
        thanks to the Cauchy-Schwarz inequality for series.
        \end{proof}

        Now, using the $\epsilon$-regularity proven in \cite[Theorem $3.3$]{biharmonic_quanta}, we deduce the following result.
        \begin{theorem}
            Let $0<a<b<\infty$, $\Omega=B_b\setminus\bar{B}_a(0)\subset \R^4$ and let $u\in W^{2,2}(B(0,1),M^m)$ be an (intrinsic or extrinsic) biharmonic map. Then, there exists $\epsilon_1=\epsilon_1(M^m,h)>0$ and $C_1=C_1(M^m,h)<\infty$ independent of $u$ and $\Omega$ such that for all $x\in \Omega_{\frac{1}{2}}=B_{\frac{b}{2}}\setminus\bar{B}_{2a}(0)$, the condition
            \begin{align*}
                \np{\D^2u}{2}{B_{2|x|}\setminus\bar{B}_{\frac{|x|}{2}}(0)}+\np{\D u}{4}{B_{2|x|}\setminus\bar{B}_{\frac{|x|}{2}}(0)}\leq \epsilon_1
            \end{align*}
            implies that 
            \begin{align*}
                |x|^2|\D^2u(x)|+|x||\D u(x)|\leq C_1\left(\np{\D^2u}{2}{B_{2|x|}\setminus\bar{B}_{\frac{|x|}{2}}(0)}+\np{\D u}{4}{B_{2|x|}\setminus\bar{B}_{\frac{|x|}{2}}(0)}\right).
            \end{align*}
        \end{theorem}
        \begin{proof}
            Using the $\epsilon$-regularity of \cite[Theorem $3.3$]{biharmonic_quanta}, we deduce that if $v\in W^{2,2}(B(0,2),M^m)$ is an (intrinsic or extrinsic) biharmonic map and bootstrapping on the $W^{2,p}(B(0,\frac{3}{2}))$ estimate (for all $p<\infty$), we have
            \begin{align*}
                |\D^2v(1)|+|\D v(1)|\leq C\left(\np{\D^2 v}{2}{B_2\setminus\bar{B}_{\frac{1}{2}}(0)}+\np{\D v}{2}{B_2\setminus\bar{B}_{\frac{1}{2}}(0)}\right).
            \end{align*}
            Therefore, applying this identity to $v(y)=u(|x|y)$, we deduce the announced estimate.
        \end{proof}
        An identical scaling consideration shows that the following estimate holds (see \cite{riviere_lamm_biharmonic} where this quantity first appeared).
        \begin{theorem}
            Let $0<a<b<\infty$, $\Omega=B_b\setminus\bar{B}_a(0)\subset \R^4$ and let $u\in W^{2,2}(B(0,1),M^m)$ be an (intrinsic or extrinsic) biharmonic map. Then, there exists $\epsilon_1=\epsilon_1(M^m,h)>0$ and $C_1(m)=C_1(M^m,h)<\infty$ independent of $u$ and $\Omega$ such that for all $x\in \Omega_{\frac{1}{2}}=B_{\frac{b}{2}}\setminus\bar{B}_{2a}(0)$, the condition
            \begin{align*}
                \np{\D^2u}{2}{B_{2|x|}\setminus\bar{B}_{\frac{|x|}{2}}(0)}+\np{\D u}{2}{B_{2|x|}\setminus\bar{B}_{\frac{|x|}{2}}(0)}\leq \epsilon_1
            \end{align*}
            implies that
            \begin{align*}
                |x|^2|\D^2u(x)|+|x||\D u(x)|\leq C_1\left(\np{\D^2u}{2}{B_{2|x|}\setminus\bar{B}_{\frac{|x|}{2}}(0)}+\np{\frac{\D u}{|x|}}{2}{B_{2|x|}\setminus\bar{B}_{\frac{|x|}{2}}(0)}\right).
            \end{align*}
        \end{theorem}        
        In particular, for a sequence $\ens{u_k}_{k\in \N}$ of biharmonic maps of bounded energy, thanks to the definition of neck regions, we deduce that 
        \begin{align*}
            \np{\D u_k}{2,\infty}{\Omega_k(\alpha)}+\np{\D u_k}{4,\infty}{\Omega_k(\alpha)}\leq C_1(m)\sup_{\alpha^{-1}\rho_k<r<\alpha}\left(\np{\D^2u_k}{2}{B_{2r}\setminus\bar{B}_{r}(0)}+\np{\frac{\D u}{|x|}}{2}{B_{2r}\setminus\bar{B}_r(0)}\right)
        \end{align*}
        and the last quantity converges to zero as $k\rightarrow \infty$ and $\alpha\rightarrow 0$. In other words, we deduce that 
        \begin{align*}
            \lim_{\alpha\rightarrow 0}\limsup_{k\rightarrow \infty}\left(\np{\D u_k}{2,\infty}{\Omega_k(\alpha)}+\np{\D u_k}{4,\infty}{\Omega_k(\alpha)}\right)=0.
        \end{align*}

        \begin{theorem}\label{l2_quantization}
            Let $(M^m,h)$ be a closed Riemannian manifold isometrically embedded in $\R^n$, let $0<a<b<\infty$ be such that
        \begin{align*}
            \frac{b}{a}>\frac{9}{4},
        \end{align*}
        and $\Omega=B_b\setminus\bar{B}_a(0)\subset \R^4$. There exists universal constant $\epsilon_2=\epsilon_2(M^m,h)>0$ and $C_2=C_2(M^m,h)$ such that for all (intrinsic or extrinsic) biharmonic map $u\in W^{2,2}(\Omega,M^m)$ such that
        \begin{align}\label{smallness_annulus}
            \sup_{2a<r<\frac{b}{2}}\left(\np{\D^2 u}{2}{B_{2r}\setminus\bar{B}_{\frac{r}{2}}(0)}+\np{\frac{\D u}{|x|}}{2}{B_{2r}\setminus\bar{B}_{\frac{r}{2}}(0)}\right)\leq \epsilon_2,
        \end{align}
        Then, the following estimate holds true
            \begin{align}\label{l2_quantization_ineq}
                \int_{\Omega_{\frac{1}{2}}}\left(|\D^2u|^2+\frac{|\D u|^2}{|x|^2}\right)dx\leq C_2\left(\np{\frac{1}{r^2}\Delta_{S^3} u}{2,\infty}{\Omega_{\frac{1}{2}}}+\np{\p{r}u}{4,\infty}{\Omega_{\frac{1}{2}}}\right)\sqrt{\int_{\Omega}\left(|\D^2u|^2+\frac{|\D u|^2}{|x|^2}\right)dx}.
            \end{align}
        \end{theorem}
        \begin{proof}
        Thanks to the final argument in the proof of Theorem 
        By the Pohozaev identity (derived below in \eqref{pohozaev_identity}) and the co-area formula, we deduce that 
        \begin{align}\label{pohozero}
            &\int_{\Omega_{\alpha}}\left(|\p{r}^2u|^2+\frac{3}{r^2}|\p{r}u|^2\right)dx=\int_{\alpha^{-1}a}^{\alpha\,b}\left(\int_{\partial B(0,r)}\left(|\p{r}^2u|^2+\frac{3}{r^2}|\p{r}u|^2\right)d\mathscr{H}^3\right)dr\nonumber\\
            &=\int_{\alpha^{-1}a}^{\alpha\, b}\left(\int_{\partial B(0,r)}\left(2\,\p{r}u\cdot\left(\p{r}^3u+\frac{2}{r}\p{r}^2u\right)\right)d\mathscr{H}^3\right)dr\nonumber\\
            &+\int_{\alpha^{-1}a}^{\alpha\,b}\left(\int_{\partial B(0,r)}\left(\frac{1}{r^4}|\Delta_{S^3}u|^2+\frac{2}{r^2}\p{r}u\cdot \p{r}\left(\Delta_{S^3}u\right)\right)d\mathscr{H}^3\right)dr\nonumber\\
            &=\int_{\Omega_{\alpha}}\left(2\,\p{r}u\cdot\left(\p{r}^3u+\frac{2}{r}\p{r}^2u\right)\right)dx+\int_{\Omega_{\alpha}}\left(\frac{1}{r^4}|\Delta_{S^3}u|^2+\frac{2}{r^2}\p{r}u\cdot \p{r}\left(\Delta_{S^3}u\right)\right)dx.
        \end{align}
        Now, thanks to the estimates of Theorem \ref{l21_neck}, we deduce that there exists $\gamma\in \R^n$ such that the following expansions hold true:
        \begin{align}\label{taylor_lorentz}
        \left\{\begin{alignedat}{1}
            A\,\p{r}u&=\p{r}\left(\frac{1}{2}\gamma \log|x|\right)+F=\frac{1}{2}\frac{\gamma}{|x|}\\
            A\,\p{r}^2u&=\p{r}^2\left(\frac{1}{2}\gamma\log|x|\right)+G=-\frac{1}{2}\frac{\gamma}{|x|^2}+G\\
            A\,\p{r}^3u&=\p{r}^3\left(\frac{1}{2}\gamma\log|x|\right)+H=\frac{\gamma}{|x|^3}+H,
            \end{alignedat}\right.
        \end{align}
        where $F\in L^{4,1}(\Omega_{\frac{1}{2}})$, $G\in L^{2,1}(\Omega_{\frac{1}{2}})$, $H\in L^{\frac{4}{3},1}(\Omega_{\frac{1}{2}})$, and 
        \begin{align*}
            \np{F}{4,1}{\Omega_{\frac{1}{2}}}+\np{G}{2,1}{\Omega_{\frac{1}{2}}}+\np{H}{\frac{4}{3},1}{\Omega_{\frac{1}{2}}}\leq C\left(\np{\D^2u}{2}{\Omega}+\np{\frac{\D u}{|x|}}{2}{\Omega}\right).
        \end{align*}
        In particular, we have
        \begin{align*}
            A\left(\p{r}^3u+\frac{2}{r}\p{r}^2u\right)=A\left(\frac{\gamma}{r^3}+H+\frac{2}{r}\left(-\frac{1}{2}\frac{\gamma}{r^2}+G\right)\right)=A\,H+A\left(\frac{2}{r}G\right).
        \end{align*}
        We have by Hölder's inequality
        \begin{align*}
            \np{\frac{1}{|x|}G}{\frac{4}{3},1}{\Omega_{\frac{1}{2}}}\leq \Gamma_0\np{G}{2,1}{\Omega_{\frac{1}{2}}}\np{\frac{1}{|x|}}{4,\infty}{\Omega_{\frac{1}{2}}}\leq \left(2\pi^2\right)^{\frac{1}{4}}\Gamma_0\np{G}{2,1}{\Omega_{\frac{1}{2}}},
        \end{align*}
        where $\Gamma_0<\infty$ is independent of $\Omega$. Therefore, since $A$ is uniformly bounded, we deduce that 
        \begin{align*}
            \np{\p{r}^3u+\frac{2}{r}\p{r}^2u}{\frac{4}{3},1}{\Omega_{\frac{1}{2}}}\leq C\left(\np{\D^2u}{2}{\Omega}+\np{\frac{\D u}{|x|}}{2}{\Omega}\right),
        \end{align*}
        which shows by the $L^{\frac{4}{3},1}/L^{4,\infty}$ duality that
        \begin{align}\label{pohozero2}
            \left|\int_{\Omega_{\frac{1}{2}}}2\,\p{r}u\cdot \left(\p{r}^3u+\frac{2}{r}\p{r}^2u\right)\right|&\leq \np{\p{r} u}{4,\infty}{\Omega_{\frac{1}{2}}}\np{\p{r}^3u+\frac{2}{r}\p{r}^2u}{\frac{4}{3},1}{\Omega_{\frac{1}{2}}}\nonumber\\
            &\leq  C\left(\np{\D^2u}{2}{\Omega}+\np{\frac{\D u}{|x|}}{2}{\Omega}\right)\np{\p{r} u}{4,\infty}{\Omega_{\frac{1}{2}}}.
        \end{align}
        On the other hand, since $\dfrac{1}{r^2}\p{r}\left(\Delta_{S^3}u\right)\in L^{\frac{4}{3},1}(\Omega_{\frac{1}{2}})$ and $\dfrac{1}{r^2}\Delta_{S^3}u\in L^{2,1}(\Omega_{\frac{1}{2}})$, we deduce that
        \begin{align}\label{pohozero3}
            &\left|\int_{\Omega_{\frac{1}{2}}}\left(\frac{1}{r^4}|\Delta_{S^3}u|^2+\frac{2}{r^2}\p{r}u\cdot \p{r}\left(\Delta_{S^3}u\right)\right)dx\right|\nonumber\\
            &\leq \np{\frac{1}{r^2}\Delta_{S^3}u}{2,1}{\Omega_{\frac{1}{2}}}\np{\frac{1}{r^2}\Delta_{S^3}u}{2,\infty}{\Omega_{\frac{1}{2}}}+\np{\D u}{4,\infty}{\Omega_{\frac{1}{2}}}\np{\frac{1}{r^2}\p{r}\left(\Delta_{S^3}u\right)}{\frac{4}{3},1}{\infty}\nonumber\\
            &\leq C\left(\np{\D^2u}{2}{\Omega}+\np{\frac{\D u}{|x|}}{2}{\Omega}\right)\left(\np{\frac{1}{r^2}\Delta_{S^3} u}{2,\infty}{\Omega_{\frac{1}{2}}}+\np{\p{r}u}{4,\infty}{\Omega_{\frac{1}{2}}}\right).
        \end{align}
        Using the Pohoazev identity \eqref{pohozero} and the estimates \eqref{pohozero2}, \eqref{pohozero3}, we deduce that 
        \begin{align*}
            \int_{\Omega_{\frac{1}{2}}}\left(|\p{r}^2u|^2+\frac{3}{r^2}|\p{r}u|^2\right)dx\leq C\left(\np{\D^2u}{2}{\Omega}+\np{\frac{\D u}{|x|}}{2}{\Omega}\right)\left(\np{\frac{1}{r^2}\Delta_{S^3} u}{2,\infty}{\Omega_{\frac{1}{2}}}+\np{\p{r}u}{4,\infty}{\Omega_{\frac{1}{2}}}\right),
        \end{align*}
        which concludes the proof of the theorem.
        \end{proof}

        However, as we have pointed out in the introduction, this estimate is not sufficient to obtain the Morse stability of biharmonic maps. Indeed, we could only bound the $L^2$ norm of $\D^2u$, but could not find an $L^{2,1}$ estimate for the Hessian matrix (only its non-radial part can be bounded). This is where we develop a new argument to show the energy quantization, that we apply to two other classical problems: harmonic maps and Yang-Mills connections.

    \section{Pohozaev Identity and Improved Energy Quantization}
 
    Since the non-biharmonic component of $u$ has its $W^{2,(2,1)}$ energy quantized, it is easy to see that the $L^{2,1}$ energy quantization is equivalent to the following condition 
    \begin{align*}
        \lim_{\alpha\rightarrow 0}\limsup_{k\rightarrow \infty}\int_{\Omega_k(\alpha)}\frac{|\Delta u_k|}{|x|^2}dx=0,
    \end{align*}
    which is the same condition that first appeared in \cite{riviere_morse_scs} (see also \cite{morse_willmore_I}), where the condition becomes
    \begin{align*}
        \lim_{\alpha\rightarrow 0}\limsup_{k\rightarrow \infty}\int_{\Omega_k(\alpha)}\frac{|\D u_k|}{|x|}dx=0,
    \end{align*}
    if $\ens{u_k}_{k\in\N}$ is a sequence of harmonic maps from a (possibly) degenerating sequence of Riemann surfaces—or more generally, a sequence of critical points from a fixed conformally invariant Lagrangian. In the case that the Riemann surface does not degenerate, this condition follows from the $L^{2,1}$ energy quantization, that is based on the holomorphy of the Hopf differential. However, in dimension $4$, there are no holomorphic objects attached to a biharmonic map, so we must find a new argument. The idea combines the Pohozaev identity and a stability lemma for Lorentz norms, proven in the appendix (Lemma \ref{averaging_l21}). For simplicity, let us first expose the argument in the case of conformally invariant problems in dimension $2$, which will give us a new proof of the $L^{2,1}$ energy quantization in this setting. In the case of Willmore surfaces, the $L^{2,1}$ energy quantization came from the $\epsilon$-regularity and the Jacobian system in $(\vec{R},S)$ which implied an $L^{\infty}$ on $(\vec{R},S)$ depending only on the $L^2$ norm of the second fundamental form. However, for biharmonic maps, there does not seem to be a way to control the quantity
    \begin{align*}
        \left|\int_{\partial B(0,b)}u\,d\mathscr{H}^3-\int_{\partial B(0,a)}u\,d\mathscr{H}^3\right|
    \end{align*}
    that is equivalent to the $L^{2,1}$ energy quantization in the case of biharmonic maps with values in the sphere.

    \subsection{Harmonic Maps in Dimension Two into Riemannian Manifolds}

    Thanks to \cite{angular}, if $u:B(0,b)\rightarrow M^m\subset \R^d$ is a harmonic map (or more generally, a critical point of a conformally invariant problem), the Pohozaev identity implies that for all $0<\rho<b$, 
    \begin{align}
        \int_{\partial B(0,\rho)}\left|\p{r}u\right|^2d\mathscr{H}^1=\int_{\partial B(0,\rho)}\left|\frac{1}{r}\p{\theta}u\right|^2d\mathscr{H}^1.
    \end{align}
    Now, let $0<a<b<\infty$, and $\Omega=B_b\setminus\bar{B}_a(0)$, 
    \begin{align}
        \Lambda_t=\frac{1}{2\pi}\int_{\Omega_t}\frac{|\D u|}{|x|}dx,
    \end{align}
    where $0<t\leq 1$ and $\Omega_t=B_{t\,b}\setminus\bar{B}_{t^{-1}a}(0)$.
    The main result of \cite{angular} shows that (provided that the natural hypothesis on the smallness of energy on dyadic annuli holds, see \cite{riviere_morse_scs} for more details)
    \begin{align}
        \np{\frac{1}{r}\p{\theta}u}{2,1}{\Omega_{\frac{1}{2}}}\leq C\np{\D u}{2}{\Omega}.
    \end{align}
    In particular, we have by the $L^{2,1}$/$L^{2,\infty}$ duality
    \begin{align}
        \int_{\Omega_{\frac{1}{2}}}\frac{\left|\frac{1}{r}\p{\theta}u\right|}{|x|}dx\leq \np{\frac{1}{r} \p{\theta}u}{2,1}{\Omega_{\frac{1}{2}}}\np{\frac{1}{|x|}}{2,\infty}{\Omega_{\frac{1}{2}}}=2\sqrt{\pi}\np{\frac{1}{r} \p{\theta}u}{2,1}{\Omega_{\frac{1}{2}}}\leq 2\sqrt{\pi}\,C\np{\D u}{2}{\Omega}.
    \end{align}
    Therefore, we need only estimate
    \begin{align}
        \Lambda_t'=\frac{1}{2\pi}\int_{\Omega_t}\frac{|\p{r}u|}{|x|}dx.
    \end{align}
    We have by the co-area formula and Hölder's inequality
    \begin{align}
        2\pi\Lambda_{\frac{1}{2}}'&=\int_{\Omega_{\frac{1}{2}}}\frac{|\p{r}u|}{|x|}dx=\int_{2a}^{\frac{b}{a}}\frac{1}{r}\left(\int_{\partial B(0,r)}|\p{r}u|d\mathscr{H}^1\right)dr
        \leq \sqrt{2\pi}\int_{2a}^{\frac{b}{2}}\frac{1}{\sqrt{r}}\left(\int_{\partial B(0,r)}|\p{r}u|^2d\mathscr{H}^1\right)^{\frac{1}{2}}dr\\
        &=\sqrt{2\pi}\int_{2a}^{\frac{b}{2}}\frac{1}{\sqrt{r}}\left(\int_{\partial B(0,r)}\left|\frac{1}{r}\p{\theta}u\right|^2d\mathscr{H}^1\right)^{\frac{1}{2}}dr.
    \end{align}
    Now, thanks to our Lemma \ref{averaging_l21} from the appendix, if $I_{\frac{1}{2}}=[2a,\frac{b}{2}]$ we have
    \begin{align}
        &\np{r\mapsto \np{\frac{1}{r}\p{\theta}u}{2}{\partial B(0,r)}}{2,1}{I_{\frac{1}{2}}}\leq 2\sqrt{\pi}\np{\frac{1}{r}\p{\theta}u}{2,1}{\Omega_{\frac{1}{2}}}\leq 4\pi\, C\,\np{\D u}{2}{\Omega}.
    \end{align}
    Therefore, using the $L^{2,1}/L^{2,\infty}$ energy quantization once more, we get
    \begin{align}
        2\pi\Lambda_{\frac{1}{2}}'\leq \sqrt{2\pi}\np{r\mapsto \np{\frac{1}{r}\p{\theta}u}{2}{\partial B(0,r)}}{2,1}{I_{\frac{1}{2}}}\np{\frac{1}{\sqrt{r}}}{2,\infty}{I_{\frac{1}{2}}}\leq 4\sqrt{2}\pi^{\frac{3}{2}}\,C\,\np{\D u}{2}{\Omega},
    \end{align}
    which yields
    \begin{align}
        \Lambda_{\frac{1}{2}}'\leq 2\sqrt{2\pi}\,C\,\np{\D u}{2}{\Omega}
    \end{align}
    and finally, 
    \begin{align}
        \Lambda_{\frac{1}{2}}\leq \left(2\sqrt{2\pi}+\frac{1}{\sqrt{\pi}}\right)C\np{\D u}{2}{\Omega},
    \end{align}
    which concludes the proof thanks to the classical $L^2$ energy quantization from \cite{angular}. 

    The principle is exactly the same in all dimension, with the difference that the Pohozaev identity is more involved.

    \subsection{Biharmonic Maps in Dimension Four into Manifolds}
    
    \begin{theorem}\label{th:improved_quantization}
        Let $(M^m,h)$ be a closed Riemannian manifold isometrically embedded in $\R^n$, let $0<a<b<\infty$ be such that
        \begin{align*}
            \frac{b}{a}>\frac{9}{4},
        \end{align*}
        and $\Omega=B_b\setminus\bar{B}_a(0)\subset \R^4$. There exists universal constant $\epsilon_3=\epsilon_3(M^m,h)>0$ and $C_3=C_3(M^m,h)$ such that for all (intrinsic or extrinsic) biharmonic map $u\in W^{2,2}(\Omega,M^m)$ such that
        \begin{align}\label{smallness_annulus_bis}
            \sup_{2a<r<\frac{b}{2}}\left(\np{\D^2 u}{2}{B_{2r}\setminus\bar{B}_{\frac{r}{2}}(0)}+\np{\frac{\D u}{|x|}}{2}{B_{2r}\setminus\bar{B}_{\frac{r}{2}}(0)}\right)\leq \epsilon_3,
        \end{align}
        Then, the following estimate holds true
            \begin{align}\label{l21_quantization_ineq}
                \np{\D^2u}{2,1}{\Omega_{\frac{1}{4}}}+\np{\D u}{4,1}{\Omega_{\frac{1}{4}}}\leq C_3\left(\np{\frac{1}{r^2}\Delta_{S^3} u}{2,\infty}{\Omega_{\frac{1}{2}}}+\np{\p{r}u}{4,\infty}{\Omega_{\frac{1}{2}}}\right)\sqrt{\int_{\Omega}\left(|\D^2u|^2+\frac{|\D u|^2}{|x|^2}\right)dx}
            \end{align}
    \end{theorem}
    \begin{proof}
    Let $u:B(0,1)\rightarrow (M^m,h)\subset \R^n$ be an extrinsic biharmonic map. A proof of the Pohozaev identity for biharmonic map can be found in \cite{biharmonic_quanta} and , but we can derive it again. Then, we have $\Delta^2u\perp T_uM^m$, which implies in particular that $(x\cdot \D u)\cdot \Delta^2u=0$ pointwise. Therefore, we have
    \begin{align*}
        0=\int_{B(0,R)}(x\cdot \D u)\cdot \Delta^2u\,dx&=\int_{\partial B(0,R)}\left(r\,\p{r}u\cdot \p{r}\left(\Delta u\right)-\p{r}\left(r\,\p{r}u\right)\cdot \Delta u\right)\,d\mathscr{H}^3\\
        &+\int_{B(0,R)}\Delta\left(x\cdot \D u\right)\cdot \Delta u\,dx.
    \end{align*}
    We have 
    \begin{align*}
        &\Delta\left(x_i\,\p{x_i}u\right)=x_i\,\p{x_i}(\Delta u)+2\,\p{x_i}^2u\\
        &\Delta(x\cdot \D u)=x\cdot \D(\Delta u)+2\Delta u,
    \end{align*}
    which gives
    \begin{align*}
        \int_{B(0,R)}\Delta(x\cdot \D u)\cdot \Delta u\,dx&=\int_{B(0,R)}\left(2|\Delta u|^2+(x\cdot \D(\Delta u))\cdot \Delta u\right)dx=\int_{B(0,R)}\dive\left(\frac{x}{2}|\Delta u|^2\right)dx\\
        &=\int_{\partial B(0,R)}\frac{r}{2}|\Delta u|^2\,d\mathscr{H}^3.
    \end{align*}
    Therefore, we get
    \begin{align*}
        \int_{\partial B(0,R)}\left(\frac{r}{2}|\Delta u|^2+r\,\p{r}u\cdot \p{r}(\Delta u)-r\,\p{r}^2u\cdot \Delta u-\p{r}u\cdot \Delta u\right)d\mathscr{H}^3=0.
    \end{align*}
    We have
    \begin{align*}
        \Delta u=\p{r}^2u+\frac{3}{r}\p{r}u+\frac{1}{r^2}\Delta_{S^3}u,
    \end{align*}
    which implies that
    \begin{align*}
        &|\Delta u|^2=|\p{r}^2u|^2+\frac{9}{r^2}|\p{r}u|^2+\frac{1}{r^4}|\Delta_{S^3}u|^2+\frac{6}{r}\p{r}^2u\cdot \p{r}u+\frac{2}{r^2}\p{r}^2u\cdot \Delta_{S^3}u+\frac{6}{r^3}\p{r}u\cdot \Delta_{S^3}u\\
        &\p{r}(\Delta u)=\p{r}^3u+\frac{3}{r}\p{r}^2u-\frac{3}{r^2}\p{r}u+\frac{1}{r^2}\p{r}(\Delta_{S^3}u)-\frac{2}{r^3}\Delta_{S^3}u\\
        &\p{r}u\cdot \p{r}(\Delta u)=\p{r}u\cdot \p{r}^3u+\frac{3}{r}\p{r}u\cdot\p{r}^2u-\frac{3}{r^2}|\p{r}u|^2+\frac{1}{r^2}\p{r}u\cdot \p{r}(\Delta_{S^3}u)-\frac{2}{r^3}\p{r}u\cdot \Delta_{S^3}u\\
        &\p{r}^2u\cdot \Delta u=|\p{r}^2u|^2+\frac{3}{r}\p{r}u\cdot\p{r}^2u+\frac{1}{r^2}\p{r}^2u\cdot \Delta_{S^3}u\\
        &\frac{1}{r}\p{r}u\cdot \Delta u=\frac{1}{r}\p{r}u\cdot \p{r}^2u+\frac{3}{r^2}|\p{r}u|^2+\frac{1}{r^3}\p{r}u\cdot \Delta_{S^3}u.
    \end{align*}
    Therefore, we have
    \begin{align*}
        &\frac{1}{2}|\Delta u|^2+\p{r}u\cdot \p{r}(\Delta u)-\p{r}u\cdot \Delta u-\frac{1}{r}\p{r}u\cdot \Delta u=\frac{1}{2}|\p{r}^2u|^2+\frac{9}{2\,r^2}|\p{r}u|^2+\frac{1}{2\,r^4}|\Delta_{S^3}u|^2+\frac{3}{r}\p{r}u\cdot \p{r}^2u\\
        &+\colorcancel{\frac{1}{r}\p{r}^2u\cdot \Delta_{S^3}u}{blue}+\colorcancel{\frac{3}{r^3}\p{r}u\cdot \Delta_{S^3}u}{red}+\p{r}u\cdot \p{r}^3u+\frac{3}{r}\p{r}u\cdot \p{r}^2u-\frac{3}{r^2}|\p{r}u|^2+\frac{1}{r^2}\p{r}u\cdot \p{r}(\Delta_{S^3}u)-\colorcancel{\frac{2}{r^3}\p{r}u\cdot \Delta_{S^3}u}{red}\\
        &-|\p{r}^2u|^2-\frac{3}{r}\p{r}u\cdot \p{r}^2u-\colorcancel{\frac{1}{r^2}\p{r}^2u\cdot \Delta_{S^3}u}{blue}-\frac{1}{r}\p{r}u\cdot\p{r}^2u-\frac{3}{r^2}|\p{r}u|^2-\colorcancel{\frac{1}{r^3}\p{r}u\cdot \Delta_{S^3}u}{red}\\
        &=-\frac{1}{2}|\p{r}^2u|^2-\frac{3}{2r^2}|\p{r}u|^2+\p{r}^3u\cdot\p{r}u+\frac{2}{r}\p{r}^2u\cdot\p{r}u+\frac{1}{2\,r^4}|\Delta_{S^3}u|^2+\frac{1}{r^2}\p{r}u\cdot\p{r}(\Delta_{S^3}u).
    \end{align*}
    Therefore, we get
    \small
    \begin{align}\label{pohozaev_identity}
        \int_{\partial B(0,R)}\left(|\p{r}^2u|^2+\frac{3}{r^2}|\p{r}u|^2-2\,\p{r}u\cdot \left(\p{r}^3u+\frac{2}{r}\p{r}^2u\right)\right)d\mathscr{H}^3=\int_{\partial B(0,R)}\left(\frac{1}{r^4}|\Delta_{S^3}u|^2+\frac{2}{r^2}\,\p{r}u\cdot\p{r}(\Delta_{S^3}u)\right)d\mathscr{H}^3,
    \end{align}
    \normalsize
    We now define
    \begin{align*}
        \Lambda=\int_{\Omega_{\frac{1}{2}}}\frac{|\p{r}^2u|}{|x|^2}dx.
    \end{align*}
    Now, recall the controlled Taylor expansion of \eqref{taylor_lorentz}:
    \begin{align}\label{taylor_lorentz2}
        \left\{\begin{alignedat}{1}
            A\,\p{r}u&=\p{r}\left(\frac{1}{2}\gamma \log|x|\right)+F=\frac{1}{2}\frac{\gamma}{|x|}\\
            A\,\p{r}^2u&=\p{r}^2\left(\frac{1}{2}\gamma\log|x|\right)+G=-\frac{1}{2}\frac{\gamma}{|x|^2}+G\\
            A\,\p{r}^3u&=\p{r}^3\left(\frac{1}{2}\gamma\log|x|\right)+H=\frac{\gamma}{|x|^3}+H,
            \end{alignedat}\right.
        \end{align}
        where $F\in L^{4,1}(\Omega_{\frac{1}{2}})$, $G\in L^{2,1}(\Omega_{\frac{1}{2}})$, $H\in L^{\frac{4}{3},1}(\Omega_{\frac{1}{2}})$, and 
        \begin{align}\label{controlled_lorentz2}
            \np{F}{4,1}{\Omega_{\frac{1}{2}}}+\np{G}{2,1}{\Omega_{\frac{1}{2}}}+\np{H}{\frac{4}{3},1}{\Omega_{\frac{1}{2}}}\leq C\left(\np{\D^2u}{2}{\Omega}+\np{\frac{\D u}{|x|}}{2}{\Omega}\right).
        \end{align}
    We deduce  by the $L^{2,1}/L^{2,\infty}$ duality that 
        \begin{align*}
            \int_{\Omega}\frac{|A\p{r}^2u -\p{r}^2\left(\frac{\gamma}{2}\log|x|\right)|}{|x|^2}dx\leq \np{G}{2,1}{\Omega}\np{\frac{1}{|x|^2}}{2,\infty}{\Omega}\leq C\left(\np{\D^2u}{2}{\Omega}+\np{\frac{\D u}{|x|}}{2}{\Omega}\right),
        \end{align*}
        while, using the estimate 
        \begin{align*}
            \mathrm{dist}(A,\mathrm{SO}(n))\leq C\left(\np{\D^2u}{2}{\Omega}+\np{\frac{\D u}{|x|}}{2}{\Omega}\right),
        \end{align*}
        which shows in other words that
        \begin{align*}
            \np{A}{\infty}{\Omega}+\np{A^{-1}}{\infty}{\Omega}\leq C\left(1+\np{\D^2u}{2}{\Omega}+\np{\frac{\D u}{|x|}}{2}{\Omega}\right),
        \end{align*}
        and 
        we deduce that (provided that the right-hand side is small enough) 
        \begin{align}\label{l21_lower_bound}
            \int_{\Omega_{\frac{1}{2}}}\frac{|A^{-1}\Delta\left(\frac{\gamma}{2}\log|x|\right)|}{|x|^2}dx\geq \frac{|\gamma|}{2}\int_{\Omega_{\frac{1}{2}}}\frac{dx}{|x|^4}=\pi^2|\gamma|\log\left(\frac{4\,b}{a}\right)\geq \frac{1}{2}\np{\frac{\gamma}{|x|^2}}{2,1}{\Omega_{\frac{1}{2}}}.
        \end{align}
        Now, we have by the co-area formula, Hölder's inequality, and the Pohozaev identity \eqref{pohozaev_identity} 
        \begin{align*}
            \Lambda&=\frac{1}{2\pi^2}\int_{a}^b\frac{1}{r^2}\left(\int_{\partial B(0,r)}|\p{r}^2u|d\mathscr{H}^3\right)dr\leq \frac{1}{2\pi^2}\int_{a}^b\frac{1}{r^2}\sqrt{2\pi^2r^3}\left(\int_{\partial B(0,r)}|\Delta u|^2d\mathscr{H}^3\right)^{\frac{1}{2}}dr\\
            &\leq \frac{1}{\pi\sqrt{2}}\int_{a}^b\frac{1}{\sqrt{r}}\left(\int_{\partial B(0,r)}\left|\frac{1}{r^2}\Delta_{S^3}u\right|^2+2\,\p{r}u\left(\p{r}^3u+\frac{2}{r}\p{r}^2u+\frac{1}{r^2}\p{r}\left(\Delta_{S^3}u\right)\right)\right)^{\frac{1}{2}}dr.
        \end{align*}
        By the previous estimates, we have
        \begin{align*}
            \frac{1}{r^2}\Delta_{S^3}u\in L^{2,1}(\Omega),
        \end{align*}
        and using our averaging Lemma \ref{averaging_l21}, we deduce that 
        \begin{align*}
            \frac{1}{\pi\sqrt{2}}\int_{a}^b\frac{1}{\sqrt{r}}\left(\int_{\partial B(0,r)}16\left|\frac{1}{r^2}\Delta_{S^3}u\right|^2\right)^{\frac{1}{2}}dr&\leq \frac{4}{\pi\sqrt{2}}\cdot 2\pi\sqrt{2}\np{\frac{1}{r^2}\Delta_{S^2}u}{2,1}{\Omega}\np{\frac{1}{\sqrt{r}}}{2,\infty}{[a,b]}\\
            &=8\np{\frac{1}{r^2}\Delta_{S^2}u}{2,1}{\Omega}\leq C\left(\np{\D^2u}{2}{\Omega}+\np{\frac{\D u}{|x|}}{2}{\Omega}\right).
        \end{align*}
        On the other hand, we have 
        \begin{align}\label{radial_1}
            A\,\p{r}u=\frac{\gamma}{2r}+F_r,
        \end{align}
        where $F_r\in W^{2,\left(\frac{4}{3},1\right)}\cap W^{1,(2,1)}\cap L^{4,1}(\Omega)$, and 
        \begin{align*}
            A\left(\p{r}^2u+\frac{3}{r}\p{r}u+\frac{1}{r^2}\Delta_{S^3}u\right)=\frac{\gamma}{r^2}+G,
        \end{align*}
        where $G\in W^{1,\left(\frac{4}{3},1\right)}\cap L^{2,1}$. Therefore, we have 
        \begin{align}\label{radial_2}
            A\,\p{r}^2u=-\frac{\gamma}{2r^2}+G-\frac{3}{r}F_r.
        \end{align}
        This implies that
        \begin{align*}
            A\,\left(\p{r}^3u+\p{r}\left(\frac{1}{r^2}\Delta_{S^3}u\right)\right)=\frac{\gamma}{r^3}+\p{r}G+\frac{3}{r^2}F_r-\frac{3}{r}\p{r}F_r-\p{r}A\,\p{r}^2u.
        \end{align*}
        Since $\p{r}A\in L^{4,2}(\Omega)$ and $\p{r}^2u\in L^2(\Omega)$, we have $\p{r}A\,\p{r}^2u\in L^{\frac{4}{3},1}(\Omega)$ and 
        \begin{align*}
            \np{\p{r}A\,\p{r}^2u}{\frac{4}{3},1}{\Omega}\leq C\left(\np{\D^2u}{2}{\Omega}+\np{\frac{\D u}{|x|}}{2}{\Omega}\right)^2.
        \end{align*}
        Since $F_r\in L^{4,1}(\Omega)$, and $\dfrac{1}{r^2}\in L^{2,\infty}(\Omega)$, we have $\frac{1}{r^2}F_r\in L^{\frac{4}{3},1}(\Omega)$ and
        \begin{align*}
            \np{\frac{1}{r^2}F_r}{\frac{4}{3},1}{\Omega}\leq C\left(\np{\D^2u}{2}{\Omega}+\np{\frac{\D u}{|x|}}{2}{\Omega}\right).
        \end{align*}
        Finally, since $\p{r}F_r\in L^{2,1}(\Omega)$ and $\frac{1}{r}\in L^{4,\infty}(\Omega)$, we have 
        \begin{align*}
            \np{\frac{1}{r}\p{r}F_r}{\frac{4}{3},1}{\Omega}\leq C\left(\np{\D^2u}{2}{\Omega}+\np{\frac{\D u}{|x|}}{2}{\Omega}\right).
        \end{align*}
        Finally, we deduce that there exists $H\in L^{\frac{4}{3},1}(\Omega)$ such that
        \begin{align*}
            \np{H}{\frac{4}{3},1}{\Omega}\leq C\left(\np{\D^2u}{2}{\Omega}+\np{\frac{\D u}{|x|}}{2}{\Omega}\right),
        \end{align*}
        and
        \begin{align}\label{radial_3}
            A\left(\p{r}^3u+\p{r}\left(\frac{1}{r^2}\Delta_{S^3}u\right)\right)=\frac{\gamma}{r^3}+H.
        \end{align}
        By \eqref{radial_2} and \eqref{radial_3}, we get
        \begin{align*}
            A\left(\p{r}^3u+\frac{2}{r}\p{r}^2u+\p{r}\left(\frac{1}{r^2}\Delta_{S^3}u\right)\right)=H+\frac{2}{r}F_r\in L^{\frac{4}{3},1}(\Omega).
        \end{align*}
        Finally, we have
        \begin{align*}
            \frac{1}{r}\p{r}u\cdot \frac{1}{r^2}\Delta_{S^3}u=\left(A^{-1}\frac{\gamma}{2r^2}+A^{-1}\frac{F_r}{r}\right)\cdot \frac{1}{r^2}\Delta_{S^3}u.
        \end{align*}
        In particular, we get by the generalised Hölder's inequality
        \begin{align*}
            &\np{\sqrt{\left|\p{r}u\cdot \left(\p{r}^3u+\frac{2}{r}\p{r}^2u+\p{r}\left(\frac{1}{r^2}\Delta_{S^3}u\right)\right)\right|}}{2,1}{\Omega}\\
            &\leq C\np{\sqrt{|\p{r}u|}}{8,2}{\Omega}\np{\sqrt{\left|\p{r}^3u+\frac{2}{r}\p{r}^2u+\p{r}\left(\frac{1}{r^2}\Delta_{S^3}u\right)\right|}}{\frac{8}{3},2}{\Omega}\\
            &\leq C\np{\p{r}u}{4,1}{\Omega}^{\frac{1}{2}}\np{\p{r}^3u+\frac{2}{r}\p{r}^2u+\p{r}\left(\frac{1}{r^2}\Delta_{S^3}u\right)}{\frac{4}{3},1}{\Omega}^{\frac{1}{2}}\\
            &\leq C\sqrt{\np{\frac{\gamma}{|x|^2}}{2,1}{\Omega}+\left(\np{\D^2u}{2}{\Omega}+\np{\frac{\D u}{|x|}}{2}{\Omega}\right)}
            \sqrt{\left(\np{\D^2u}{2}{\Omega}+\np{\frac{\D u}{|x|}}{2}{\Omega}\right)}
        \end{align*}
        and
        \begin{align*}
            \np{\sqrt{\left|\frac{1}{r}\p{r}u\cdot \frac{1}{r^2}\Delta_{S^3}u\right|}}{2,1}{\Omega}&\leq C\sqrt{\np{\frac{\gamma}{|x|^2}}{2,1}{\Omega}+\left(\np{\D^2u}{2}{\Omega}+\np{\frac{\D u}{|x|}}{2}{\Omega}\right)}\\
            &\times \sqrt{\left(\np{\D^2u}{2}{\Omega}+\np{\frac{\D u}{|x|}}{2}{\Omega}\right)}.
        \end{align*}
        Finally, using once more our averaging Lemma \ref{averaging_l21}, we deduce that 
        \begin{align*}
            \frac{1}{4\pi^2}\np{\frac{\gamma}{|x|^2}}{2,1}{\Omega}&\leq \Lambda\\
            &\leq C\sqrt{\np{\frac{\gamma}{|x|^2}}{2,1}{\Omega}+\left(\np{\D^2u}{2}{\Omega}+\np{\frac{\D u}{|x|}}{2}{\Omega}\right)}\sqrt{\left(\np{\D^2u}{2}{\Omega}+\np{\frac{\D u}{|x|}}{2}{\Omega}\right)},
        \end{align*}
        which implies by Cauchy's inequality that
        \begin{align*}
            \np{\frac{\gamma}{|x|^2}}{2,1}{\Omega}\leq C\left(\np{\D^2u}{2}{\Omega}+\np{\frac{\D u}{|x|}}{2}{\Omega}\right)
        \end{align*}
        and concludes the proof of the $L^{2,1}$ energy quantization theorem.
        \end{proof}
        \begin{cor}\label{cor_quanta}
        Let $\ens{u_k}_{k\in \N}:B(0,1)\rightarrow (M^n,h)$ be a sequence of (extrinsic or intrinsic) biharmonic map of uniformly bounded energy. Then, for all neck region $\Omega_k(\alpha)=B_{\alpha}\setminus\bar{B}_{\alpha^{-1}\rho_k}(0)$, we have
        \begin{align*}
            \lim_{\alpha\rightarrow 0}\limsup_{k\rightarrow \infty}\left(\np{\D^2u}{2,1}{\Omega_k(\alpha)}+\np{\D u}{4,1}{\Omega_k(\alpha)}\right)=0.
        \end{align*}
        \end{cor}

        \subsection{Yang-Mills Functional}

        Let $G$ be a compact Lie group (that we suppose to be isometrically embedded in $\R^n$) of Lie algebra $\mathfrak{g}$. Let $A\in \Omega^1(\mathfrak{g})\cap W^{1,2}(\Lambda^1B(0,1),\mathfrak{g})$ be a $1$-form with values into $\mathfrak{g}$. Its curvature is given by
        \begin{align*}
            F_A=dA+[A,A],
        \end{align*}
        where $[\,\cdot\,,\,\cdot\,]$ is the Lie bracket of $\mathfrak{g}$. This expression is also written by abuse of notation $F_A=dA+A\wedge A$. The Yang-Mills functional is defined by 
        \begin{align*}
            \mathrm{YM}(A)=\int_{B(0,1)}|F_A|^2dx.
        \end{align*}
        A stationary Yang-Mills connection satisfies the following equation (\cite[(2.1.6)]{tian_yang_mills}, \cite[(VIII.5)]{riviere-survey-yang-mills-tian}) for all compactly supported vector field $X\in C^{\infty}_c(B(0,1),\R^n)$:
        \begin{align*}
            \int_{B(0,1)}\left(|F_A|^2-4\sum_{i,j=1}^4\s{F_A(\D_{\e_i}X,\e_j)}{F_A(\e_i,\e_j)}\right)dx=0,
        \end{align*}
        where $(\e_1,\cdots,\e_4)$ is an orthonormal basis. Let $\varphi\in C^{\infty}(B(0,1))$ be a radial function. Choosing $X=\varphi(r)r\p{r}$, we get by \cite[(2.1.15)]{tian_yang_mills} (notice that $n=4$, $c(p)=0$ and $\phi=1$ in this formula)
        \begin{align*}
            \int_{B(0,1)}\varphi'(r)r\left(|F_A|^2-\left|\p{r}\antires F_A\right|^2\right)dx=0,
        \end{align*}
        where $\p{r}\antires F_A=F_A\left(\p{r},\,\cdot\,\right)$. Using the co-area formula, we deduce that
        \begin{align*}
            \int_{0}^1\varphi'(r)\left(\int_{\partial B(0,r)}r\left(|F_A|^2-4|\p{r}\antires F_A|^2\right)d\mathscr{H}^3\right)dr=0.
        \end{align*}
        Therefore, we deduce that the distribution
        \begin{align*}
            T(r)=r\int_{\partial B(0,r)}\left(|F_A|^2-4|\p{r}\antires A|^2\right)d\mathscr{H}^3
        \end{align*}
        satisfies the equation
        \begin{align*}
            T'=0\qquad\text{in}\;\, \mathscr{D}'([0,1]).
        \end{align*}
        Therefore, $T$ is a constant function, and since $T(0)=0$, we deduce that $T=0$ identically. In other words, the following Pohozaev identity holds for all $0<r<1$
        \begin{align*}
            \int_{\partial B(0,r)}|F_A|^2d\mathscr{H}^3=4\int_{\partial B(0,r)}|\p{r}\antires F_A|^2d\mathscr{H}^3.
        \end{align*}
        Since
        \begin{align*}
            |F_A|^2=|\p{r}\antires F_A|^2+\frac{1}{r^2}|\p{\theta}\antires F_A|^2+\frac{1}{r^2\sin^2(\theta)}|\p{\varphi}\antires F_A|^2+\frac{1}{r^2\sin^2(\theta)\sin^2(\varphi)}|\p{\psi}\antires F_A|^2
        \end{align*}
        using spherical coordinates $(r,\omega)\in (0,\infty)\times S^3$ and a slight abuse of notations, we deduce that
        \begin{align*}
            \int_{\partial B(0,r)}|\p{r}\antires F_A|^2d\mathscr{H}^3=\int_{\partial B(0,r)}\frac{1}{r^2}|\p{\omega}\antires F_A|^2d\mathscr{H}^3,
        \end{align*}
        which bears a striking resemblance with the Pohozaev formula for harmonic maps (or critical points of conformally invariant Lagrangians).
        Now, using Rivière's gauge construction for small $\np{F_A}{2,\infty}{B(0,1)}$ (\cite[Theorem IV.$4$]{riviere-survey-yang-mills-tian}) norm on a neck region $\Omega=B_b\setminus\bar{B}_a(0)$, using a controlled extension $\widetilde{A}$ of $A$ on $B(0,1)$ (see \cite{biharmonic_quanta} for example in the setting of biharmonic maps), we deduce the existence of a \emph{global gauge} $g\in W^{2,2}(B(0,1),G)$ such that 
        \begin{align*}
        \left\{\begin{alignedat}{1}
            \np{\D A^g}{2}{\Omega}&+\np{A^g}{2}{\Omega}\leq C_{G}\sqrt{\mathrm{YM}(A)}\\
            d^{\ast}A^g&=0\qquad \text{in}\;\, \Omega\\
            \iota^{\ast}_{S^3}\left(\ast A^g\right)&=0.
            \end{alignedat}\right.
        \end{align*}
        Writing $A=A^g$ for simplicity, the Coulomb condition implies that the Yang-Mills equation becomes
        \begin{align*}
            \Delta A=d^{\ast}\left(A\wedge A\right)+[A,\res dA]+[A,\res (A\wedge A)]\qquad\text{in}\;\,\Omega.
        \end{align*}
        If $\widetilde{A}$ is a controlled extension of $A$ on $B(0,1)$, we deduce by the Sobolev embedding $W^{1,2}(\R^4)\hooklongrightarrow L^{4,2}(\R^4)$ that 
        \begin{align*}
            \np{\D \widetilde{A}}{2}{B(0,1)}+\np{\widetilde{A}}{4,2}{B(0,1)}\leq C\sqrt{\mathrm{YM}(A)}.
        \end{align*}
        Then, let $\varphi:\Lambda^1B(0,1)\rightarrow \mathfrak{g}\subset \R^n$ be the solution of the following equation
        \begin{align*}
            \left\{\begin{alignedat}{2}
                \Delta\varphi&=d^{\ast}\left(\widetilde{A}\wedge \widetilde{A}\right)+\left[\widetilde{A},\res d\widetilde{A}\right]+\left[\widetilde{A},\res\left(\widetilde{A}\wedge \widetilde{A}\right)\right]\qquad&&\text{in}\;\, B(0,1)\\
                \varphi&=0\qquad\text{on}\;\, \partial B(0,1).
            \end{alignedat}\right.
        \end{align*}
        Thanks to the Hölder's inequality for Lorentz spaces, we deduce that $\widetilde{A}\wedge \widetilde{A}\in L^{4,2}\cdot L^{4,2}\hooklongrightarrow L^{2,1}(B(0,1))$, while
        \begin{align*}
            &\left[\widetilde{A},\res d\widetilde{A}\right]\in L^{4,2}\cdot L^{2,2}\hooklongrightarrow L^{\frac{4}{3},1}(B(0,1))\\
            &\left[\widetilde{A},\res\left(\widetilde{A}\wedge \widetilde{A}\right)\right]\in L^{4,2}\cdot L^{2,1}\hooklongrightarrow L^{4,2}\cdot L^{2,2}\hooklongrightarrow L^{\frac{4}{3},1}(B(0,1)).
        \end{align*}
        In other words, we have
        \begin{align*}
            \Delta \varphi\in d^{\ast}\left(L^{2,1}\right)+L^{\frac{4}{3},1}(B(0,1)).
        \end{align*}
        Therefore, thanks to the Sobolev embedding $W^{1,\left(\frac{4}{3},1\right)}(\R^4)\hooklongrightarrow L^{2,1}(\R^4)$, we deduce by Calder\'{o}n-Zygmund estimates that
        \begin{align*}
            \np{\D \varphi}{2,1}{B(0,1)}+\np{\varphi}{4,1}{B(0,1)}&\leq C\np{\widetilde{A}}{4,2}{B(0,1)}\left(\np{d\widetilde{A}}{2}{B(0,1)}+\np{\widetilde{A}}{4,2}{B(0,1)}+\np{\widetilde{A}}{4,2}{B(0,1)}^2\right)\\
            &\leq C\left(1+\mathrm{YM}(A)\right)\mathrm{YM}(A).
        \end{align*}
        Since $\psi=A-\varphi$ is a harmonic function on $\Omega$, there exists a constant $\gamma\in \mathfrak{g}$ such that by Theorem \ref{dirichlet_dim_arbitraire} and Theorem \ref{pre_dirichlet_arbitraire} we have for all $0<\alpha<1$ small enough
        \begin{align*}
            \np{\D\psi}{2,1}{\Omega_{\alpha}}+\np{\psi-\gamma}{4,1}{\Omega_{\alpha}}\leq \frac{C_4'+\Gamma_4}{\sqrt{1-\left(\frac{a}{b}\right)^2}}\frac{\alpha}{(1-\alpha^2)^3}\np{\D\psi}{2}{\Omega},
        \end{align*}
        which concludes the proof since
        \begin{align*}
            \np{\D\psi}{2}{\Omega}\leq \np{\D A}{2}{\Omega}+\np{\D \varphi}{2}{\Omega}\leq C\left(1+\mathrm{YM}(A)\right)\mathrm{YM}(A).
        \end{align*}
        Alternatively, we could use Lemma \ref{averaging_l21}. 
        Let us also point out that another proof using a dyadic decomposition of Rivière of the $L^{2,1}$ energy quantization can be found in M. Gauvrit's Master thesis (\cite{gauvrit_yang_mills}).

    \section{Wente-Type Inequalities in Dimension Four}

    \subsection{Gradient Estimate}

    In this section, we obtain the suitable estimates that will allow us to improve the pointwise estimate of biharmonic maps in neck regions. The difficulty here is that we need to obtain two different kinds of estimates. As in \cite{riviere_morse_scs}, we will use a dyadic argument. First, specialising the previous Lemma \ref{gen_d_dirichlet_comp2} and Lemma \ref{dirichlet_comp3} to $d=4$, we deduce the following result.

    \begin{lemme}\label{lemma:dyadic_dim4}
        Let $0<a<b<\infty$ and $\Omega=B_b\setminus\bar{B}_a(0)\subset \R^4$. Assume that $u:\Omega\rightarrow \R^m$ is a harmonic map on $\Omega$ such that $\D u\in L^2(\Omega)$ and that $u=0$ on $\partial B(0,b)$. Then, for all $a\leq r<s\leq b$, we have
        \begin{align}\label{eq:dyadic_dim4_annulus}
            \int_{B_s\setminus\bar{B}_r(0)}|\D u|^2dx\leq 2\left(\frac{a}{r}\right)^{2}\int_{\Omega}|\D u|^2dx
        \end{align}
        If $u:B(0,b)\rightarrow\R^m$ is a harmonic function, then 
        \begin{align}\label{eq:dyadic_dim4_ball}
            \int_{B(0,r)}|\D u|^2dx\leq \left(\frac{r}{b}\right)^4\int_{B(0,b)}|\D u|^2dx.
        \end{align}
    \end{lemme}
    Before proving the Wente-type estimate, let us recall how to estimate the first eigenvalue of the Laplacian on a ball.
    \begin{theorem}
        For all $d\geq 2$, let $B(0,1)=B_d(0,1)\subset \R^d$ be the unit ball. Then, for all $u\in W^{1,2}_0(B(0,1))$, we have
        \begin{align*}
            \np{u}{2}{B(0,1)}\leq \frac{1}{j_{\frac{d-2}{2},1}}\np{\D u}{2}{B(0,1)},
        \end{align*}
        where $j_{\frac{d-2}{2},1}\geq j_{0,1}=2.4048\cdots $ is the first positive zero of the Bessel function of the first kind $J_{\frac{d-2}{2}}$.
    \end{theorem}
    \begin{proof}
        If $\alpha  \in \R$, the Bessel function of the first kind $J_{\alpha}$ is one of the two solutions to the ordinary differential equation (\cite[(1) p.38]{watson_bessel})
        \begin{align}\label{bessel}
            J_{\alpha}''(r)+\frac{1}{r}J_{\alpha}(r)-\frac{\alpha^2}{r^2}J_{\alpha}(r)=-J_{\alpha}(r).
        \end{align}
        By the standard method of the calculus of variations, we easily see that there exists a minimiser to the Dirichlet energy
        \begin{align*}
            E(u)=\frac{1}{2}\int_{B(0,1)}|\D u|^2dx
        \end{align*}
        in $W^{1,2}_0(B(0,1))$, which (by standard elliptic regularity) is a smooth function on $B(0,1)$ satisfying the following system of equations:
        \begin{align}\label{vp}
            \left\{\begin{alignedat}{2}
                -\Delta u&=\lambda_1u\qquad&&\text{in}\;\, B(0,1)\\
                u&=0\qquad&&\text{on}\;\,\partial B(0,1).
            \end{alignedat}\right.
        \end{align}
        where
        \begin{align*}
            \lambda_1=\inf_{\substack{v\in W^{1,2}_0(B(0,1))\\
            v\neq 0}}\frac{\displaystyle\int_{B(0,1)}|\D v|^2dx}{\displaystyle\int_{B(0,1)}v^2dx}.
        \end{align*}
        One could use Gidas-Ni-Nirenberg's theorem (\cite{gidas}) to show that $u$ is radial, but one can proceed directly. Expand $u$ using spherical harmonics as 
        \begin{align*}
            u(r,\omega)=\sum_{n=0}^{\infty}\sum_{k=1}^{N_d(n)}u_{n,k}(r)Y_n^k(\omega),
        \end{align*}
        where $N_d(n)$ is a polynomial of degree $d-2$. Since $-\Delta_{S^{d-1}}Y_n^k=n(n+d-2)Y_n^k$, we deduce that $u=u_{n,k}$ solves the following ordinary differential equation with Dirichlet boundary conditions
        \begin{align}\label{vp_lambda_n}
        \left\{\begin{alignedat}{1}
            u''+\frac{d-1}{r}u'-\frac{n(n+d-2)}{r^2}u&=-\lambda_1\,u\\
            &u(1)=0.
            \end{alignedat}\right.
        \end{align}
        Let us show that $u$ can be expressed with respect to Bessel functions. We look for a solution of the form $f(r)=r^{\beta}J_{\alpha}(r)$ to the system of equations
        \begin{align*}
        \left\{\begin{alignedat}{1}
            f''+\frac{d-1}{r}f'-\frac{n(n+d-2)}{r^2}f&=-f\\
            f(1)&=0
            \end{alignedat}\right.
        \end{align*}
        Noticing that if $J_{\alpha,\lambda}=J_{\alpha}(\lambda\,\cdot\,)$ solves the 
        \begin{align*}
            J_{\alpha,\lambda}''(r)+\frac{1}{r}J'_{\alpha,\lambda}(r)-\frac{\alpha^2}{r^2}=-\lambda^2\,J_{\alpha,\lambda},
        \end{align*}
        the function $g(r)=f(\sqrt{\lambda_1}r)$ will give us the solution to \eqref{vp_lambda_n}. We compute by \eqref{bessel}
        \begin{align*}
            f'(r)&=r^{\beta}\left(J_{\alpha}'+\frac{\beta}{r}J_{\alpha}\right)\\
            f''(r)&=r^{\beta}\left(J_{\alpha}''+\frac{2\beta}{r}J_{\alpha}'-\frac{\beta(\beta-1)}{r^2}J_{\alpha}\right)=r^{\beta}\left(\frac{2\beta-1}{r}J_{\alpha}'+\frac{\alpha^2-\beta(\beta-1)}{r^2}J_{\alpha}-J_{\alpha}\right)\\
            &=r^{\beta}\left(\frac{2\beta-1}{r}J_{\alpha}'+\frac{\alpha^2-\beta(\beta-1)}{r^2}J_{\alpha}\right)-f(r)
        \end{align*}
        Therefore, we get
        \begin{align*}
            f''(r)+\frac{d-1}{r}f'(r)-\frac{n(n+d-2)}{r^2}f(r)+f(r)
            &=r^{\beta}\left(\frac{2\beta+d-2}{r}J_{\alpha}'+\frac{\alpha^2+\beta^2+(d-2)\beta-\lambda_n}{r^2}J_{\alpha}\right).
        \end{align*}
        We get the conditions
        \begin{align*}
            \left\{\begin{alignedat}{1}
                &2\beta+d-2=0\\
                &\alpha^2+\beta^2+(d-2)\beta-\lambda_n=0,
            \end{alignedat}\right.
        \end{align*}
        which yields
        \begin{align*}
            \left\{\begin{alignedat}{1}
                \beta&=-\frac{d-2}{2}\\
                \alpha&=\pm\sqrt{\frac{(d-2)^2}{4}+\lambda_n}.
            \end{alignedat}\right.
        \end{align*}
        Recall that as $r\rightarrow 0$, for all $\alpha\notin \Z_-^{\ast}=\Z\cap\ens{n:|n|=-n,\;\, n\neq 0}$
        \begin{align}\label{taylor_bessel}
            J_{\alpha}(r)=\frac{1}{\Gamma(\alpha+1)}\left(\frac{r}{2}\right)^{\alpha}+o(r^{\alpha}),
        \end{align}
        while for $\alpha=-n<0$, where $n\in \N$, we get
        \begin{align}\label{taylor_bessel2}
            J_{-n}(r)=\frac{(-1)^n}{\Gamma(n+1)}\left(\frac{2}{r}\right)^{n}+o(r^{-n}).
        \end{align}
        Therefore, we finally deduce that
        \begin{align*}
            \left\{\begin{alignedat}{1}
                \beta&=-\frac{d-2}{2}\\
                \alpha&=\sqrt{\frac{(d-2)^2}{4}+\lambda_n}.
            \end{alignedat}\right.
        \end{align*}
        Indeed, For $d\in 2\N$, since $J_{-n}=(-1)^nJ_n$, we do not find another solution, and for $d\in 2\N+1$ (which implies in particular that $d\geq 3$), we have
        \begin{align*}
            \frac{1}{r^{\frac{d-2}{2}}}J_{-\frac{d-2}{2}}(r)=\frac{1}{\Gamma(\frac{d}{2})}\frac{1}{2^{\frac{d-2}{2}}}\frac{1}{r^{d-2}}+o(r^{2-d})\qquad\text{as}\;\,r\rightarrow 0,
        \end{align*}
        and $\D(|x|^{2-d})\notin L^1_{\mathrm{loc}}(\R^d)$ for $d\geq 3$. Furthermore, the Bessel function of the second kind is unbounded at $0$ and $u$ is smooth on $B(0,1)$, so we finally deduce that $u_{n,k}$ is given by
        \begin{align*}
            u_{n,k}(r)=\frac{1}{r^{\frac{d-2}{2}}}J_{\sqrt{\frac{(d-2)^2}{4}+n(n+d-2)}}\left(\sqrt{\lambda_1}r\right).
        \end{align*}
        We see by \eqref{taylor_bessel} that this is a bounded, and even real-analytic function on $B(0,1)$. The boundary condition shows that $\sqrt{\lambda_1}$ is a positive zero of $J_{\sqrt{\frac{(d-2)^2}{4}+n(n+d-2)}}$. Since $\lambda_1$ is the smallest positive eigenvalue, we deduce that if $j_{\alpha,1}>0$ is the first positive zero of $J_{\alpha}$, then
        \begin{align}
            \lambda_1=\inf_{n\in \N}j_{\sqrt{\frac{(d-2)^2}{4}+n(n+d-2)},1}^2.
        \end{align}
        Now, using a formula due to Schläfli and Gegenbauer or Watson's generalisation (\cite[\textbf{15.6}; (2-3) p. 508]{watson_bessel}), we deduce that $\R_+\rightarrow \R_+,\alpha\rightarrow j_{\alpha,1}$ is a strictly increasing function. In particular, the sequence $\ens{j_{\sqrt{\frac{(d-2)^2}{4}+n(n+d-2)},1}}$ is strictly increasing, and since \eqref{vp_lambda_n} holds for all $n\in \N$ and $1\leq k\leq N_d(n)$ we deduce that $u_{n,k}=0$ for all $n\geq 1$ and $k\in \ens{1,\cdots,N_d(n)}$, and that $u=u_{0,1}$ is a radial function, and finally, that
        \begin{align*}
            \lambda_1=j_{\frac{d-2}{2},1}^2.
        \end{align*}
        For $d=2$, we find
        \begin{align*}
            u(r)=J_0\left(j_{0,1}\,r\right)
        \end{align*}
        $\lambda_1(B_2(0,1))=j_{0,1}^2=5.7831\cdots $, where $j_{0,1}=2.4048\cdots$ is the first positive zero of $J_0$ (see \cite[p. 670, p. 748]{watson_bessel}; notice that the last digit of the zeroes of Bessel function is not also accurate: \cite[20.2 p. 662]{watson_bessel}). For $d=3$, $J_{\frac{1}{2}}$ is simply equal to
        \begin{align*}
            J_{\frac{1}{2}}(x)=\sqrt{\frac{2}{\pi x}}\sin(x),
        \end{align*}
        which shows that
        \begin{align*}
            u(r)=\frac{1}{\sqrt{r}}J_{\frac{1}{2}}(\pi\,r)=\sqrt{\frac{2}{\pi}}\frac{\sin(\pi r)}{\pi r}.
        \end{align*}
        In particular, we have $j_{\frac{1}{2},1}=\pi$, and $\lambda_1(B_3(0,1))=\pi^2$. Finally, in the case of interest of this article, for $d=4$, we have 
        \begin{align*}
            u(r)=J_1\left(j_{1,1}r\right),
        \end{align*}
        and $j_{1,1}=3.83170\cdots$ (see \cite[p. 673, p.748]{watson_bessel}), which yields $\lambda_1=j_{1,1}^2=14.681970\cdots $, and the growth of $j_{\alpha,1}$ in $\alpha>0$ shows the last inequality mentioned in the theorem.   
    \end{proof}
    \begin{lemme}\label{lemma:dyadic_gradient}
        Let $f\in L^2(B(0,1))$ and assume that $u\in W^{1,2}_0(B(0,1))\subset W^{1,2}(\R^4)$ solves the equation
        \begin{align*}
            \left\{\begin{alignedat}{2}
                \Delta u&=f\qquad&&\text{in}\;\, B(0,1)\\
                u&=0\qquad&&\text{on}\;\, \partial B(0,1).
            \end{alignedat}\right.
        \end{align*}
        Then, we have
        \begin{align}\label{lemma:dyadic_gradient_ineq}
            \np{|x|\D u}{2}{B(0,1)}\leq \Gamma_1\np{|x|\left(1+\log_2\left(\frac{1}{|x|}\right)\right)\log\left(e+\log_2\left(\frac{1}{|x|}\right)\right)f}{2}{B(0,1)},
        \end{align}
        where 
        \begin{align}\label{def_Gamma1}
            \Gamma_1=\frac{1}{j_{1,1}}\sqrt{72+\frac{2^{14}}{3^3\cdot 7^2}+192\left(\frac{1}{\log(2)}+\frac{1}{2\log^2(2)}\right)}=6.1824966\cdots
        \end{align}
    \end{lemme}
    \begin{proof}
        Extend $f$ by $0$ on $\R^4\setminus\bar{B}(0,1)$. Let $\ens{\chi_k}_{k\in \N}$ be a partition of unity such that $\mathrm{supp}(\chi_k)\subset B_{2^{-(k-1)}}\setminus\bar{B}_{2^{-(k+1)}}(0)$ for all $k\in \N$. Explicitly, we have $\chi_k\in \mathscr{D}(\R^4)$ for all $k\in \N$, $0\leq \chi_k\leq 1$ and
        \begin{align*}
            1=\sum_{k=0}^{\infty}\chi_k\qquad\text{in}\;\, B(0,1).
        \end{align*}
        Then, we have the expansion
        \begin{align*}
            u_k=\sum_{k\in \N}u_k,
        \end{align*}
        where $u_k$ solves the equation
        \begin{align*}
            \left\{\begin{alignedat}{2}
                \Delta u_k&=\chi_kf=f_k\qquad&&\text{in}\;\, B(0,1)\\
                u_k&=0\qquad&&\text{on}\;\, \partial B(0,1).
            \end{alignedat}\right.
        \end{align*}
        Integrating by parts, we deduce that
        \begin{align*}
            \int_{B(0,1)}|\D u_k|^2dx=-\int_{B(0,1)}u_k\,\Delta u_k\,dx=-\int_{B(0,1)}u_k\,f_k\,dx\leq \left(\int_{B(0,1)}|u_k|^2dx\right)^{\frac{1}{2}}\left(\int_{B(0,1)}|f_k|^2dx\right)^{\frac{1}{2}}.
        \end{align*}
        Using the Poincaré inequality
        \begin{align*}
            \np{u_k}{2}{B(0,1)}\leq \frac{1}{j_{1,1}}\np{\D u_k}{2}{B(0,1)},
        \end{align*}
        we deduce that
        \begin{align}\label{elementary_gradient_estimate}
            \np{\D u_k}{2}{B(0,1)}\leq \frac{1}{j_{1,1}}\np{f_k}{2}{B(0,1)}.
        \end{align}
        Therefore, we have
        \begin{align*}
            \int_{A_k}|\D u|^2dx=\sum_{i,j\in \N}\int_{A_k}\D u_i\cdot\D u_j\,dx,
        \end{align*}
        Furthermore, we have the elementary identity
        \begin{align}\label{exp_grad_u}
            \int_{B(0,1)}|x|^{2p}|\D u|^2dx=\sum_{k\in \N}\int_{A_k}|x|^{2p}|\D u|^2dx\leq \sum_{i,j,k\in \N}2^{-2pk}\int_{A_k}\D u_i\cdot \D u_j\,dx
        \end{align}
        Since $f_i\subset \widetilde{A_i}=B_{2^{-(i-1)}}\setminus\bar{B}_{2^{-(i+1)}}(0)$, we have $A_k\subset B(0,2^{-k})\subset B(0,2^{-(i+1)})$  for all $0\leq i\leq k-1$, which yields by inequality \eqref{eq:dyadic_dim4_ball} of Lemma \ref{lemma:dyadic_dim4} and \eqref{elementary_gradient_estimate}
        \begin{align}\label{ineq:i_small}
            \int_{B(0,2^{-k})}|\D u_i|^2dx\leq 2^{4(i+1-k)}\int_{B(0,2^{-(i+1)})}|\D u_i|^2dx\leq \frac{1}{\lambda_1}2^{4(i+2-k)}\int_{B(0,1)}|f_i|^2dx.
        \end{align}
        Likewise, noticing that for all $i\geq k+2$, $A_k=B_{2^{-k}}\setminus\bar{B}_{2^{-(k+1)}}(0)\subset B_1\setminus\bar{B}_{2^{-(i-1)}(0)}$, we get by \eqref{eq:dyadic_dim4_annulus} of Lemma \ref{lemma:dyadic_dim4} \eqref{elementary_gradient_estimate}
        \begin{align}\label{ineq:i_large}
            \int_{A_k}|\D u_i|^2dx\leq 2\cdot 2^{2(k-i)}\int_{B_1\setminus\bar{B}_{2^{-(i-1)}}(0)}|\D u_i|^2dx\leq \frac{2}{\lambda_1}\,2^{2(k-i)}\int_{B(0,1)}|f_i|^2dx.
        \end{align}
        First, we estimate directly by Cauchy-Schwarz inequality
        \begin{align}\label{dyadic_frequency1}
            &\sum_{k=0}^{\infty}2^{-2pk}\sum_{i=k}^{k+1}\sum_{j=k}^{k+1}\int_{A_k}\D u_i\cdot\D u_j\,dx\leq \sum_{k=0}^{\infty}2^{-2pk}\sum_{i=k}^{k+1}\sum_{j={k}}^{k+1}\left(\int_{A_k}|\D u_i|^2dx\right)^{\frac{1}{2}}\left(\int_{A_k}|\D u_j|^2dx\right)^{\frac{1}{2}}\nonumber\\
            &\leq \sum_{k=0}^{\infty}2^{-2pk}\left(\sum_{i=k}^{k+1}\sum_{j=k}^{k+1}\int_{A_k}|\D u_i|^2dx\right)^{\frac{1}{2}}\left(\sum_{i=k}^{k+1}\sum_{j=k}^{k+1}\int_{A_k}|\D u_j|^2dx\right)^{\frac{1}{2}}\nonumber\\
            &=2\sum_{k=0}^{\infty}2^{-2pk}\sum_{i=k}^{k+1}\int_{A_k}|\D u_i|^2dx
            \leq \frac{2}{\lambda_1}\sum_{k=0}^{\infty}2^{-2pk}\sum_{i=k}^{k+1}\int_{B(0,1)}|f_i|^2dx\nonumber\\
            &\leq \frac{2}{\lambda_1}\sum_{k=0}^{\infty}2^{-2pk}\sum_{i=k}^{k+1}\int_{B(0,1)}\chi_i|f|^2dx
            \leq \frac{2}{\lambda_1}\sum_{k=0}^{\infty}2^{-2pk}\int_{B_{2^{-(k-2)}}\setminus\bar{B}_{2^{-(k+1)}}(0)}|f|^2dx\nonumber\\
            &\leq \frac{2^{2p+1}}{\lambda_1}\sum_{k=0}^{\infty}\int_{B_{2^{-(k-2)}}\setminus\bar{B}_{2^{-(k+1)}}(0)}|x|^{2p}|f(x)|^2dx
            \leq 3\frac{2^{2p+1}}{\lambda_1}\int_{B(0,1)}|x|^{2p}|f(x)|^2dx,
        \end{align}
        where $\lambda_1=j_{1,1}^2=14.681970$ is the first eigenvalue of the Dirichlet Laplacian on $B(0,1)\subset \R^4$. Then, we estimate by virtue of Hölder's inequality (for the Lebesgue and counting measures) for all $I,J\subset \N$
        \begin{align}\label{double_holder}
            &\sum_{k\in\N}2^{-2pk}\sum_{i\in I}\sum_{j\in J}\int_{A_k}\D u_i\cdot \D u_j\,dx=\sum_{k\in\N}2^{-2pk}\int_{A_k}\left(\sum_{i\in I}\D u_i\right)\cdot \left(\sum_{j\in J}u_j\right)dx\nonumber\\
            &\leq \sum_{k\in\N}2^{-2pk}\left(\int_{A_k}\left|\sum_{i\in I}\D u_i\right|^2dx\right)^{\frac{1}{2}}\left(\int_{A_k}\left|\sum_{j\in J}\D u_j\right|^2dx\right)^{\frac{1}{2}}\nonumber\\
            &\leq \left(\sum_{k\in\N}2^{-2pk}\int_{A_k}\left|\sum_{i\in I}\D u_i\right|^2dx\right)^{\frac{1}{2}}\left(\sum_{k\in\N}2^{-2pk}\int_{A_k}\left|\sum_{j\in J}\D u_j\right|^2dx\right)^{\frac{1}{2}}\nonumber\\
            &=\left(\sum_{k\in \N}2^{-2pk}\sum_{i_1,i_2\in I}\int_{A_k}\D u_{i_1}\cdot \D u_{i_2}\,dx\right)^{\frac{1}{2}}\left(\sum_{k\in \N}2^{-2pk}\sum_{j_1,j_2\in J}\int_{A_k}\D u_{j_1}\cdot \D u_{j_2}\,dx\right)^{\frac{1}{2}}.
        \end{align}        
        In particular, we need only estimate the \enquote{pure} terms of such expressions where $I=J$. For example, applying it to $I=\ens{k,k+1}$ and $J=\N\setminus\ens{1,\cdots,k+1}$, we get
        \small
        \begin{align*}
            &\sum_{k=0}^{\infty}2^{-2pk}\sum_{i=k}^{k+1}\sum_{j=k+2}^{\infty}\int_{A_k}\D u_i\cdot \D u_j\,dx
            &\leq \left(\sum_{k=0}^{\infty}2^{-2pk}\int_{A_k}\left|\sum_{i=k}^{k+1}\D u_i\right|^2dx\right)^{\frac{1}{2}}\left(\sum_{k=0}^{\infty}2^{-2pk}\int_{A_k}\left|\sum_{j=k+2}^{\infty}\D u_j\right|^2dx\right)^{\frac{1}{2}}.
        \end{align*}
        \normalsize
        Therefore, we only have two more integrals left to bound. 
        We first have
        \begin{align}\label{dyadic3}
            \int_{A_k}\left|\sum_{i=0}^{k-1}\D u_i\right|^2dx&=\sum_{i=0}^{k-1}\sum_{j=0}^{k-1}\int_{A_k}\D u_i\cdot \D u_j\,dx\nonumber\\
            &\leq \sum_{i=0}^{k-1}\sum_{j=0}^{k-1}\left(\int_{A_k}|\D u_i|^2dx\right)^{\frac{1}{2}}\left(\int_{A_k}|\D u_j|^2dx\right)^{\frac{1}{2}}\leq \left(\sum_{i=0}^{k-1}\left(\int_{A_k}|\D u_i|^2dx\right)^{\frac{1}{2}}\right)^2\nonumber\\
            &\leq k\sum_{i=0}^{k-1}\int_{A_k}|\D u_i|^2dx,
        \end{align}
        which yields by Fubini theorem
        \begin{align}\label{dyadic4}
            &\sum_{k=1}^{\infty}2^{-2pk}\sum_{i=0}^{k-1}\sum_{j=0}^{k-1}\int_{A_k}\D u_i\cdot \D u_j\,dx\leq \frac{1}{\lambda_1}\sum_{k=1}^{\infty}2^{-2pk}\sum_{i=0}^{k-1}k\,2^{4(i+1-k)}\int_{B(0,1)}|f_i|^2dx\nonumber\\
            &=\frac{1}{\lambda_1}\sum_{i=0}^{\infty}2^{4(i+1)}\int_{B(0,1)}|f_i|^2dx\sum_{k=1}^{\infty}\mathbf{1}_{\ens{0\leq i\leq k-1}}k\,2^{-2(p+2)k}\nonumber\\
            &=\frac{1}{\lambda_1}\sum_{i=0}^{\infty}2^{4(i+1)}\int_{B(0,1)}|f_i|^2dx\sum_{k=i+1}^{\infty}k\,2^{-2(p+2)k}.
        \end{align}
        For all $-1<x<1$, we have
        \begin{align}\label{dyadic5}
            \sum_{k=i+1}^{\infty}k\,x^k&=\sum_{j=0}^{\infty}(j+i+1)x^{j+i+1}=x^{i+2}\sum_{j=0}^{\infty}j\,x^{j-1}+\frac{i+1}{1-x}x^{i+1}=\frac{x^{i+2}}{(1-x)^2}+\frac{i+1}{1-x}x^{i+1}\nonumber\\
            &=\frac{i+1-i\,x}{(1-x)^2}x^{i+1}\leq \frac{i+1}{(1-x)^2}x^{i+1}.
        \end{align}
        Therefore, we deduce by \eqref{dyadic4} and \eqref{dyadic5} that
        \begin{align}\label{dyadic6}
            &\sum_{k=1}^{\infty}2^{-2pk}\sum_{i=0}^{k-1}\sum_{j=0}^{k-1}\int_{A_k}\D u_i\cdot \D u_j\,dx\leq \frac{1}{\lambda_1(1-2^{-2(p+2)})^2}\sum_{i=0}^{\infty}(i+1)2^{4(i+1)}2^{-2(p+2)(i+1)}\int_{B(0,1)}|f_i|^2dx\nonumber\\
            &=\frac{1}{\lambda_1(1-2^{-2(p+2)})^2}\sum_{i=0}^{\infty}(i+1)2^{-2p(i+1)}\int_{B(0,1)}|f_i|^2dx\nonumber\\
            &\leq \frac{1}{\lambda_1(1-2^{-2(p+2)})^2}\sum_{i=0}^{\infty}(i+1)2^{-2p(i+1)}\int_{B_{2^{-(i-1)}}\setminus\bar{B}_{2^{-(i+1)}}}|f|^2dx\nonumber\\
            &\leq \frac{1}{\lambda_1(1-2^{-2(p+2)})^2}\sum_{i=0}^{\infty}\int_{B_{2^{-(i-1)}}\setminus\bar{B}_{2^{-(i+1)}}(0)}\left(2+\log_2\left(\frac{1}{|x|}\right)\right)|x|^{2p}|f(x)|^2dx\nonumber\\
            &\leq \frac{2}{\lambda_1(1-2^{-2(p+2)})^2}\int_{B(0,1)}\left(2+\log_2\left(\frac{1}{|x|}\right)\right)|x|^{2p}|f(x)|^2dx.
        \end{align}
        Finally, we estimate by \eqref{ineq:i_large}
        \begin{align*}
            &\sum_{k=0}^{\infty}2^{-2pk}\sum_{i=k+2}^{\infty}\sum_{j=k+2}^{\infty}\int_{A_k}\D u_i\cdot \D u_j\,dx\nonumber\\
            &\leq \sum_{i=2}^{\infty}\sum_{j=2}^{\infty}\sum_{k=0}^{\mathrm{min}\ens{i-2,j-2}}2^{-2pk}\int_{A_k}|\D u_i||\D u_j|dx\nonumber\\
            &\leq \sum_{i=2}^{\infty}\sum_{j=2}^{\infty}\sum_{k=0}^{\mathrm{min}\ens{i-2,j-2}}2^{-2pk}\left(\int_{A_k}|\D u_i|^2dx\right)^{\frac{1}{2}}\left(\int_{A_k}|\D u_j|^2dx\right)^{\frac{1}{2}}\nonumber\\
            &\leq \frac{2}{\lambda_1}\sum_{i=2}^{\infty}\sum_{j=2}^{\infty}\sum_{k=0}^{\mathrm{min}\ens{i-2,j-2}}2^{-2pk}2^{2(k+1)-(i+j)}\left(\int_{A_{i-1}\cup A_i}|f|^2dx\right)^{\frac{1}{2}}\left(\int_{A_{j-1}\cup A_j}|f|^2dx\right)^{\frac{1}{2}}\nonumber\\
            &=\frac{8}{\lambda_1}\sum_{i=2}^{\infty}\sum_{j=2}^{\infty}2^{-(i+j)}\np{f}{2}{A_{i-1}\cup A_i}\np{f}{2}{A_{j-1}\cup A_j}\sum_{k=0}^{\mathrm{min}\ens{i-2,j-2}}2^{-2(p-1)k}.
        \end{align*}
        For $p=1$, we get by Cauchy-Schwarz inequality
        \begin{align}\label{dyadic7}
            &\sum_{k=0}^{\infty}2^{-2pk}\sum_{i=k+2}^{\infty}\sum_{j=k+2}^{\infty}\int_{A_k}\D u_i\cdot \D u_j\,dx\nonumber\\
            &\leq \frac{8}{\lambda_1}\sum_{i=2}^{\infty}\sum_{j=2}^{\infty}\mathrm{min}\ens{i-1,j-1}2^{-(i+j)}\np{f}{2}{A_{i-1}\cup A_i}\np{f}{2}{A_{j-1}\cup A_j}\nonumber\\
            &\leq \frac{8}{\lambda_1}\left(\sum_{i=2}^{\infty}i\,2^{-i}\np{f}{2}{A_{i-1}\cup A_i}\right)^2\leq \frac{8}{\lambda_1}\left(\sum_{i=2}^{\infty}\frac{1}{i\,\log^2(i)}\right)\left(\sum_{i=2}^{\infty}i^2\log^2(i)2^{-2i}\int_{B_{2^{-(i-1)}}\setminus\bar{B}_{2^{-(i+1)}}(0)}|f|^2dx\right)\nonumber\\
            &\leq \frac{32}{\lambda_1}\left(\frac{1}{\log(2)}+\frac{1}{2\log^2(2)}\right)\sum_{i=2}^{\infty}\int_{B_{2^{-(i-1)}}\setminus\bar{B}_{2^{-(i+1)}}(0)}|x|^2\log_2^2\left(\frac{1}{|x|}\right)\log\left(1+\log_2\left(\frac{1}{|x|}\right)\right)|f(x)|^2dx\nonumber\\
            &\leq \frac{64}{\lambda_1}\left(\frac{1}{\log(2)}+\frac{1}{2\log^2(2)}\right)\int_{B_{(0,1/2)}}|x|^2\left(1+\log_2\left(\frac{1}{|x|}\right)\right)^2\log^2\left(1+\log_2\left(\frac{1}{|x|}\right)\right)|f(x)|^2dx
        \end{align}
        where we used the series-integral comparison
        \begin{align*}
            \sum_{i=2}^{\infty}\frac{1}{i\log^2(i)}\leq \frac{1}{2\log^2(2)}+\int_{2}^{\infty}\frac{dx}{x\log^2(x)}=\frac{1}{\log(2)}+\frac{1}{2\log^2(2)},
        \end{align*}
        and the inequality
        \begin{align*}
            i\log(i)2^{-i}\leq 2|x|\left(1+\log_2\left(\frac{1}{|x|}\right)\right)\log\left(1+\log_2\left(\frac{1}{|x|}\right)\right)\qquad \text{for all}\;\, 2^{-(i+1)}\leq |x|\leq 2^{-i}.
        \end{align*}
        For $p>1$, we get 
        \begin{align}
            &\sum_{k=0}^{\infty}2^{-2pk}\sum_{i=k+2}^{\infty}\sum_{j=k+2}^{\infty}\int_{A_k}\D u_i\cdot \D u_j\,dx\nonumber\\
            &\leq \frac{8}{\lambda_1}\frac{2^{2(p-1)}}{2^{2(p-1)}-1}\left(\sum_{i=2}^{\infty}\np{f}{2}{A_{i-1}\cup A_i}\right)^2\nonumber\\
            &\leq \frac{8}{\lambda_1}\frac{2^{2(p-1)}}{2^{2(p-1)}-1}\left(\sum_{i=2}^{\infty}\frac{1}{i\log^2(i)}\right)\left(\sum_{i=2}^{\infty}i\,\log^2(i)2^{-2i}\int_{B_{2^{-(i-1)}}\setminus\bar{B}_{2^{-(i+1)}}(0)}|f|^2dx\right)\nonumber\\
            &\leq \frac{32}{\lambda_1}\frac{2^{2(p-1)}}{2^{2(p-1)}-1}\left(\frac{1}{\log(2)}+\frac{1}{2\log^2(2)}\right)\int_{B(0,1/2)}|x|^2\left(1+\log_2\left(\frac{1}{|x|}\right)\right)\log^2\left(1+\log_2\left(\frac{1}{|x|}\right)\right)|f(x)|^2dx,
        \end{align}
        which barely improves the previous inequality and loses the dependence on $p$. We expect that this dyadic argument is severely sub-optimal for $p>1$ for this range of frequencies.
        Now, we have by \eqref{exp_grad_u} and \eqref{double_holder}
        \begin{align}\label{dyadic8}
            &\int_{B(0,1)}|x|^{2p}|\D u|^2dx=\sum_{i,j,k\in \N}2^{-2pk}\int_{A_k}\D u_i\cdot \D u_j\,dx
            =\sum_{k=0}^{\infty}2^{-2pk}\sum_{i=0}^{k-1}\sum_{j=0}^{k-1}\int_{A_k}\D u_i\cdot \D u_j\,dx\nonumber\\
            &+\sum_{k=0}^{\infty}2^{-2pk}\sum_{i=k}^{k+1}\sum_{j=k}^{k+1}\int_{A_k}\D u_i\cdot \D u_j\,dx
            +\sum_{k=0}^{\infty}2^{-2pk}\sum_{i=k+2}^{\infty}\sum_{j=k+2}^{\infty}\int_{A_k}\D u_i\cdot \D u_j\,dx\nonumber\\
            &+2\sum_{k=0}^{\infty}2^{-2pk}\sum_{i=0}^{k-1}\sum_{j=k}^{k+1}\int_{A_k}\D u_i\cdot \D u_j\,dx
            +2\sum_{k=0}^{\infty}2^{-2pk}\sum_{i=0}^{k-1}\sum_{j=k+2}^{\infty}\int_{A_k}\D u_i\cdot \D u_j\,dx\nonumber\\
            &+2\sum_{k=0}^{\infty}2^{-2pk}\sum_{i=k}^{k+1}\sum_{j=k+2}^{\infty}\int_{A_k}\D u_i\cdot \D u_j\,dx
            \leq \sum_{k=0}^{\infty}2^{-2pk}\sum_{i=0}^{k-1}\sum_{j=0}^{k-1}\int_{A_k}\D u_i\cdot \D u_j\,dx\nonumber\\
            &+\sum_{k=0}^{\infty}2^{-2pk}\sum_{i=k}^{k+1}\sum_{j=k}^{k+1}\int_{A_k}\D u_i\cdot \D u_j\,dx
            +\sum_{k=0}^{\infty}2^{-2pk}\sum_{i=k+2}^{\infty}\sum_{j=k+2}^{\infty}\int_{A_k}\D u_i\cdot \D u_j\,dx\nonumber\\
            &+2\left(\sum_{k=0}^{\infty}2^{-2pk}\sum_{i=0}^{k-1}\sum_{j=0}^{k-1}\int_{A_k}\D u_i\cdot \D u_j\,dx\right)^{\frac{1}{2}}\left(\sum_{k=0}^{\infty}2^{-2pk}\sum_{i=k}^{k+1}\sum_{j=k}^{k+1}\int_{A_k}\D u_i\cdot \D u_j\,dx\right)^{\frac{1}{2}}\nonumber\\
            &+2\left(\sum_{k=0}^{\infty}2^{-2pk}\sum_{i=0}^{k-1}\sum_{j=0}^{k-1}\int_{A_k}\D u_i\cdot \D u_j\,dx\right)^{\frac{1}{2}}\left(\sum_{k=0}^{\infty}2^{-2pk}\sum_{i=k+2}^{\infty}\sum_{j=k+2}^{\infty}\int_{A_k}\D u_i\cdot \D u_j\,dx\right)^{\frac{1}{2}}\nonumber\\
            &+2\left(\sum_{k=0}^{\infty}2^{-2pk}\sum_{i=k}^{k+1}\sum_{j=k}^{k+1}\int_{A_k}\D u_i\cdot \D u_j\,dx\right)^{\frac{1}{2}}\left(\sum_{k=0}^{\infty}2^{-2pk}\sum_{i=k+2}^{\infty}\sum_{j=k+2}^{\infty}\int_{A_k}\D u_i\cdot \D u_j\,dx\right)^{\frac{1}{2}}\nonumber\\
            &\leq 3\sum_{k=0}^{\infty}2^{-2pk}\sum_{i=0}^{k-1}\sum_{j=0}^{k-1}\int_{A_k}\D u_i\cdot \D u_j\,dx\nonumber\\
            &+3\sum_{k=0}^{\infty}2^{-2pk}\sum_{i=k}^{k+1}\sum_{j=k}^{k+1}\int_{A_k}\D u_i\cdot \D u_j\,dx
            +3\sum_{k=0}^{\infty}2^{-2pk}\sum_{i=k+2}^{\infty}\sum_{j=k+2}^{\infty}\int_{A_k}\D u_i\cdot \D u_j\,dx
        \end{align}
        Therefore, using \eqref{dyadic_frequency1}, \eqref{dyadic6}, and \eqref{dyadic7}, we finally get
        \begin{align*}
            &\int_{B(0,1)}|x|^{2}|\D u|^2dx\leq \frac{72}{\lambda_1}\int_{B(0,1)}|x|^{2}|f(x)|^2dx\\
            &+\frac{2^{13}}{3^3\cdot 7^2\cdot \lambda_1}\int_{B(0,1)}\left(2+\log_2\left(\frac{1}{|x|}\right)\right)|x|^2|f(x)|^2dx\\
            &+\frac{192}{\lambda_1}\left(\frac{1}{\log(2)}+\frac{1}{2\log^2(2)}\right)\int_{B(0,1)}|x|^2\left(1+\log_2\left(\frac{1}{|x|}\right)\right)^2\log^2\left(1+\log_2\left(\frac{1}{|x|}\right)\right)|f(x)|^2dx\\
            &\leq \Gamma_1^2\int_{B(0,1)}|x|^2\left(1+\log_2\left(\frac{1}{|x|}\right)\right)^2\log^2\left(e+\log_2\left(\frac{1}{|x|}\right)\right)|f(x)|^2dx,
        \end{align*}
        where 
        \begin{align*}
            \Gamma_1^2&=\frac{1}{\lambda_1}\left(72+\frac{2^{14}}{3^3\cdot 7^2}+192\left(\frac{1}{\log(2)}+\frac{1}{2\log^2(2)}\right)\right)\\
            &=\frac{1}{j_{1,1}^2}\left(72+\frac{2^{14}}{3^3\cdot 7^2}+192\left(\frac{1}{\log(2)}+\frac{1}{2\log^2(2)}\right)\right)=38.223\cdots
        \end{align*}
    \end{proof}
    Notice that up to changing the constant, the proof above works in any dimension. However, for technical reasons, we need to prove weighted estimate for 
    \begin{align*}
        \int_{B(0,1)}\frac{|\D u|^2}{|x|^2}dx.
    \end{align*} 
    Using \eqref{dirichlet_weighted_typeI} and Lemma \ref{dirichlet_comp_weighted}, we get the following estimate.
    \begin{lemme}
    If $0<2\,a<b<\infty$ and $\Omega=B_b\setminus\bar{B}_a(0)\subset \R^4$. Assume that $u:\Omega\rightarrow \R^m$ is a harmonic map such that $\D u\in L^2(\Omega)$ and that $u=0$ on $\partial B(0,b)$. Then, for all $a\leq r<2r\leq b$, we have
    \begin{align}\label{dirichlet_weighted_typeI_bis}
        \int_{B_{2r}\setminus\bar{B}_r(0)}\frac{|\D u|^2}{|x|^2}dx&\leq 2\left(\frac{a}{r}\right)^4\int_{\Omega}\frac{|\D u|^2}{|x|^2}dx.
    \end{align}
    If $u:B(0,1)\rightarrow \R^m$ is a harmonic function, we have
    \begin{align}\label{dirichlet_weighted_typeII_bis}
        \int_{B(0,r)}\frac{|\D u|^2}{|x|^2}dx\leq \left(\frac{r}{b}\right)^{2}\int_{B(0,b)}\frac{|\D u|^2}{|x|^2}dx.
    \end{align}
    \end{lemme}
    \begin{lemme}\label{weighted_modified_dirichlet}
        Let $f\in L^2(B(0,1))$ and $u:B(0,1)\rightarrow \R$ such that 
        \begin{align*}
            \left\{\begin{alignedat}{2}
                \Delta u&=f\qquad&&\text{in}\;\, B(0,1)\\
                u&=0\qquad&&\text{on}\;\, \partial B(0,1).
            \end{alignedat}\right.
        \end{align*}
        Then, there exists a universal constant $\Gamma_3<\infty$ such that
        \begin{align*}
            \int_{B(0,1)}|\D u|^2dx\leq \Gamma_3^2\int_{B(0,1)}\left(1+\log_2\left(\frac{1}{|x|}\right)\right)^2\log^2\left(e+\log_2\left(\frac{1}{|x|}\right)\right)|x|^2|f(x)|^2dx.
        \end{align*}
    \end{lemme}
    \begin{rem}
        It shoulds be seen as an $|x|^2$ weighted version of the direct estimate
        \begin{align*}
            \int_{B(0,1)}\frac{|\D u|^2}{|x|^2}dx\leq C\int_{B(0,1)}|f|^2dx.
        \end{align*}
    \end{rem}
    \begin{proof} 
    Let $\ens{\chi_k}_{k\in \N}$ be a partition of unity on $B(0,2)$ as in the previous theorem. By Calder\'{o}n-Zygmund estimates and the Sobolev embedding $W^{1,2}(\R^4)\hooklongrightarrow L^{4,2}(\R^4)$, we have  
        \begin{align*}
            \np{\D u_k}{4,2}{B(0,1)}\leq \Gamma_3\np{f_k}{2}{B(0,1)}.
        \end{align*}
        Therefore, we deduce that by \eqref{square_lorentz} and \eqref{norm_lorentz_infinity_dim4}
        \begin{align*}
            \int_{B(0,1)}\frac{|\D u_k|^2}{|x|^2}dx&\leq \np{|\D u_k|^2}{2,1}{B(0,1)}\np{\frac{1}{|x|^2}}{2,\infty}{B(0,1)}\leq 4\pi\sqrt{2}\np{\D u_k}{4,2}{B(0,1)}^2\\
            &\leq 4\pi\sqrt{2}\,\Gamma_3\int_{B(0,1)}|f_k|^2dx.
        \end{align*}
        Once more, we get by \eqref{exp_grad_u}
        \begin{align}\label{exp_grad_u_bis}
            \int_{B(0,1)}|x|^{2p}|\D u|^2dx=\sum_{k\in \N}\int_{A_k}|x|^{2p}|\D u|^2dx\leq \sum_{i,j,k\in \N}2^{-2pk}\int_{A_k}\D u_i\cdot \D u_j\,dx
        \end{align}
        but this time, we will apply the result to $p=-1$, so a discussion is in order to show that the previous proof applies. Since $f_i\subset \widetilde{A}_i=B_{2^{-(i-1)}}\setminus\bar{B}_{2^{-(i+1)}}$, we have $A_k\subset B(0,2^{-k})\subset B(0,2^{-(i+1)})$, we get by \eqref{dirichlet_weighted_typeII_bis}
        \begin{align*}
            \int_{B(0,2^{-k})}\frac{|\D u_i|^2}{|x|^2}dx\leq 2^{2(i+1-k)}\int_{B(0,2^{-(i+1)})}\frac{|\D u_i|^2}{|x|^2}dx\leq 4\pi\sqrt{2}\Gamma_3\,2^{2(i+1-k)}\int_{B(0,1)}|f_i|^2dx.
        \end{align*}
        Likewise, if $i\geq k+2$, we have $A_k=B_{2^{-k}}\setminus\bar{B}_{2^{-(k+1)}}(0)\subset B_1\setminus\bar{B}_{2^{-(i-1)}}(0)$, which implies by \eqref{dirichlet_weighted_typeI_bis}
        \begin{align*}
            \int_{A_k}\frac{|\D u_i|^2}{|x|^2}dx\leq 2\cdot 2^{4(k-i)}\int_{B_1\setminus\bar{B}_{2^{-(i-1)}}(0)}\frac{|\D u_i|^2}{|x|^2}dx\leq 8\pi\sqrt{2}\,\Gamma_3\,2^{4(k-i)}\int_{B(0,1)}|f_i|^2dx.
        \end{align*}
        Now, we first estimate
        \begin{align*}
            &\sum_{k=0}^{\infty}2^{-2pk}\sum_{i=k}^{k+1}\sum_{j=k}^{k+1}\int_{A_k}\D u_i\cdot \D u_j\,dx\leq 2\sum_{k=0}^{\infty}2^{-2pk}\sum_{i=k}^{k+1}\int_{A_k}|\D u_i|^2dx\leq 2\sum_{k=0}^{\infty}2^{-2(p+1)k}\sum_{i=k}^{k+1}\int_{A_k}\frac{|\D u_i|^2}{|x|^2}dx\\
            &\leq \Gamma_3'\sum_{k=0}^{\N}2^{-2(p+1)k}\sum_{i=k}^{k+1}\int_{B(0,1)}|f_i|^2dx
            \leq 3\,\Gamma_3'\int_{B(0,1)}|x|^{2(p+1)}|f(x)|^2dx.
        \end{align*}
        Then, we have
        \begin{align*}
            \int_{A_k}\left|\sum_{i=0}^{k-1}\D u_i\right|^2dx\leq k\sum_{i=0}^{k-1}\int_{A_k}|\D u_i|^2dx,
        \end{align*}
        which yields by Fubini's theorem and redoing the steps in \eqref{dyadic4}, \eqref{dyadic5}, \eqref{dyadic6}, we get
        \begin{align*}
            &\sum_{k=1}^{\infty}2^{-2pk}\sum_{i=0}^{k-1}\sum_{j=0}^{k-1}\int_{A_k}\D u_i\cdot \D u_j\,dx\leq \sum_{k=1}^{\infty}2^{-2(p+1)k}k\sum_{i=0}^{k-1}\int_{A_k}\frac{|\D u_i|^2}{|x|^2}dx\\
            &\leq \Gamma_3'\sum_{k=1}^{\infty}2^{-2pk}\sum_{i=0}^{k-1}2^{2(i+1-k)}\int_{A_k}|f_i|^2dx
            \leq \Gamma_3''\int_{B(0,1)}\left(2+\log_2\left(\frac{1}{|x|}\right)\right)|x|^{2(p+1)}|f(x)|^2dx.
        \end{align*}
        Finally, let us treat the last frequencies. We have
        \begin{align*}
            &\sum_{k=0}^{\infty}2^{-2pk}\sum_{i=k+2}^{\infty}\sum_{j=k+2}^{\infty}\int_{A_k}\D u_i\cdot \D u_j\,dx\\
            &\leq \sum_{i=2}^{\infty}\sum_{j=2}^{\infty}\sum_{k=0}^{\mathrm{min}\ens{i-2,j-2}}2^{-2(p+1)k}\left(\int_{A_k}|\D u_i|^2dx\right)^{\frac{1}{2}}\left(\int_{A_k}|\D u_j|^2dx\right)^{\frac{1}{2}}\\
            &\leq \Gamma_3'\sum_{i=2}^{\infty}\sum_{j=2}^{\infty}\sum_{k=0}^{\mathrm{min}\ens{i-2,j-2}}2^{-2(p+1)k}2^{4(k+1)-2(i+j)}\left(\int_{B(0,1)}|f_i|^2dx\right)^{\frac{1}{2}}\left(\int_{B(0,1)}|f_j|^2dx\right)^{\frac{1}{2}}\\
            &\leq \Gamma_3'\sum_{i=2}^{\infty}\sum_{j=2}^{\infty}2^{-2(i+j)}\np{f}{2}{A_{i-1}\cup A_i}\np{f}{2}{A_{j-1}\cup A_j}\sum_{k=0}^{\mathrm{min}\ens{i-2,j-2}}2^{2(1-p)k}.
        \end{align*}
        Taking $p=-1$, we get 
        \begin{align*}
            &\sum_{k=0}^{\infty}2^{-2pk}\sum_{i=k+2}^{\infty}\sum_{j=k+2}^{\infty}\int_{A_k}\D u_i\cdot \D u_j\,dx\leq \Gamma_3''\sum_{i=2}^{\infty}\sum_{j=2}^{\infty}2^{4\mathrm{min}{i-1,j-1}-2(i+j)}\np{f}{2}{A_{i-1}\cup A_i}\np{f}{2}{A_{j-1}\cup A_j}\\
            &\leq \Gamma_3'\left(\sum_{i=2}^{\infty}\np{f}{2}{A_{i-1}\cup A_i}\right)^2\leq \Gamma_3''\left(\sum_{i=2}^{\infty}\frac{1}{i\log^2(i)}\right)\left(\sum_{i=2}^{\infty}i^2\log^2(i)\int_{B_{2^{-(i-1)}}\setminus\bar{B}_{2^{-(i+1)}}(0)}|f|^2dx\right)\\
            &\leq \Gamma_3'''\int_{B(0,1/2)}\left(1+\log_2\left(\frac{1}{|x|}\right)\right)^2\log^2\left(1+\log_2\left(\frac{1}{|x|}\right)\right)|f(x)|^2dx,
        \end{align*}
        which concludes the proof of the lemma. Taking instead $p=0$ yields 
        \begin{align*}
            &\sum_{k=0}^{\infty}\sum_{i=k+2}^{\infty}\sum_{j=k+2}^{\infty}\int_{A_k}\D u_i\cdot \D u_j\leq \Gamma_3'\sum_{i=2}^{\infty}\sum_{j=2}^{\infty}2^{2\mathrm{min}\ens{i-1,j-1}-2(i+j)}\np{f}{2}{A_{i-1}\cup A_{i}}\np{f}{2}{A_{j-1}\cup A_j}\\
            &\leq \Gamma_3'\sum_{i=2}^{\infty}\sum_{j=2}^{\infty}2^{-(i+j)}\np{f}{2}{A_{i-1}\cup A_{i}}\np{f}{2}{A_{j-1}\cup A_j}
        \end{align*}
        so the proof of the previous theorem applies. 
    \end{proof}

    \subsection{Lebesgue Estimate for Special Divergence Structure}

    As in \cite{riviere_morse_scs} and \cite{morse_willmore_I}, we need to refine the bound
    \begin{align*}
         |x|^2|\D^2u(x)|+|x||\D u(x)|\leq C\left(\np{\D u}{2}{\Omega}+\np{\frac{\D u}{|x|}}{2}{\Omega}\right).
    \end{align*}
    We follow the adaptation of the iteration method of \cite{riviere_morse_scs} from \cite{morse_willmore_I}. We first need a replacement of the weighted Wente lemma for systems of the form $\Delta u=\dive(K)$ with $K\in L^{\frac{4}{3},1}(\R^4)$.

    For all $k\in \N$, let $A_k=B_{2^{-k}}\setminus\bar{B}_{2^{-(k+1)}}(0)$, and $\widetilde{A}_k=A_k\cup \bar{A}_{k+1}\setminus\partial B(0,2^{-(k+2)})=B_{2^{-k}}\setminus \bar{B}_{2^{-(k+2)}}(0)$

    \begin{lemme}[Lemma E.$3$ \cite{riviere_morse_scs}]\label{lemmae3}
        Let $0<2\,a<b<\infty$, $j\in \N$ and $K\in L^{\frac{4}{3}}(B(0,1))$ such that $\mathrm{supp}(K)\subset B_{2^{-j}}(0)$. Let $u:B(0,1)\rightarrow\R^m$ be the solution of 
        \begin{align*}
            \left\{\begin{alignedat}{2}
                \Delta u&=\dive(K)\qquad&& \text{in}\;\, B(0,1)\\
                u&=0\qquad&& \text{on}\;\, B(0,1).
            \end{alignedat}\right.
        \end{align*}
        Then, for all $0\leq k<j$, there holds
        \begin{align}\label{e31}
            \int_{A_k}|u|^2dx&\leq \frac{8}{3}\int_{B_1\setminus\bar{B}_{2^{-j}}(0)}|u|^2dx\leq \frac{8}{3}C_{\mathrm{CZ}}^22^{2(k+1-j)}\np{K}{\frac{4}{3}}{B(0,1)}^2.
        \end{align}
        If we assume that $\mathrm{supp}(K)\subset A_j=B_{2^{-j}}\setminus\bar{B}_{2^{-(j+1)}}(0)$, for all $k>j$, we have
        \begin{align}\label{e32}
            \int_{B(0,2^{-k})}|u|^2dx&\leq \,2^{4(j+1-k)}\int_{B(0,2^{-j})}|u|^2dx\leq C_{\mathrm{CZ}}^2\,2^{4(j+1-k)}\np{K}{\frac{4}{3}}{B(0,1)}^2,
        \end{align}
        where $C_{\mathrm{CZ}}$ is the norm of the linear map $\Delta_{0}^{-1}\dive:L^{\frac{4}{3}}(B(0,1))\rightarrow L^2(B(0,1))$.
    \end{lemme}
    \begin{proof}
        By hypothesis, the function $u$ is harmonic on the annulus $B_1\setminus\bar{B}_{2^{-j}(0)}$. Therefore, we have on $B_1\setminus\bar{B}_{2^{-j}}(0)$ an expansion
        \begin{align*}
            u(r,\omega)=\sum_{n=0}^{\infty}\sum_{k=1}^{(n+1)^2}\left(a_{n,k}\,r^n+b_{n,k}\,r^{-(n+2)}\right)Y_n^k(\omega).
        \end{align*}
        Integrating $\Delta u$ on $B(0,r)$ for $r>2^{-j}$, we deduce that
        \begin{align*}
            \int_{\partial B(0,r)}\partial_{\nu}u\,d\mathscr{H}^3=\int_{B(0,r)}\Delta u\,dx=\int_{B(0,r)}\dive(K)dx=\int_{\partial B(0,r)}K\,d\mathscr{H}^3=0
        \end{align*}
        since $\mathrm{supp}(K)\subset B_{2^{-j}}(0)$. Therefore, we have $b_{0,1}=0$, and we can apply Lemma \ref{function_comp_dyadic}. Notice that $A_k=B_{2^{-k}}\setminus\bar{B}_{2^{-(k+1)}}(0)\subset B_1\setminus\bar{B}_{2^{-j}}(0)$ if and only if $2^{-(k+1)}\geq 2^{-j}$, or $0\leq k<j$. For all $0\leq k<j$, we deduce that 
        \begin{align*}
            \int_{A_k}|u|^2dx\leq \frac{2}{1-\frac{1}{2^2}}\left(\frac{2^{-j}}{2^{-(k+1)}}\right)^2\int_{\Omega}|u|^2=\frac{8}{3}2^{2(k+1-j)}\int_{\Omega}|u|^2dx
            \leq \frac{8}{3}C_{\mathrm{CZ}}^22^{2(k+1-j)}\np{K}{\frac{4}{3}}{B(0,1)}^2.
        \end{align*}
        The second estimate follows directly from Lemma \ref{function_comp_dyadic2}.
    \end{proof}
    It seems difficult to prove a weighted estimate for a general $K\in L^{\frac{4}{3}}$. However, the $K$ appearing for biharmonic maps has a special $L^2$ underlying structure, and we will prove the weighted estimate for this sub-class function.
    
    Recall that in the case of biharmonic functions, we have
    \begin{align*}
        K=2\D A\Delta \widetilde{u}+Aw\D \widetilde{u}+\D A(V\D\widetilde{u})-A\D(V\D\widetilde{u})-B\D\widetilde{u}
    \end{align*}
    where $\widetilde{u}$ is a controlled extension of $u$ in the annular region $\Omega=B_b\setminus\bar{B}_a(0)\subset \R^4$ furnished by Theorem \ref{whitney_extension_dim4}. Recall that by \eqref{ineq:localisation_K}, we have for all open subset $U\subset \Omega$
    \small
    \begin{align}\label{ineq:localisation_K2}
            \np{K}{\frac{4}{3},1}{U}&\leq C\left(1+\np{\D^2u}{2}{\Omega}+\np{\frac{\D u}{|x|}}{2}{\Omega}\right)\left(\np{\D^2u}{2}{\Omega}+\np{\frac{\D u}{|x|}}{2}{\Omega}\right)\left(\np{\D^2\widetilde{u}}{2}{U}+\np{\D\widetilde{u}}{4,2}{U}\right).
    \end{align}
    \normalsize
    In particular, applying it to $U=A_k$, and using the equivalence of the three norms on $W^{2,2}(A_k)/\R$ (that follows either by scaling invariance or from Theorem \ref{whitney_extension_dim4}), we deduce that 
    \small
    \begin{align*}
        \np{K}{\frac{4}{3}}{A_k}^2\leq C\left(1+\np{\D^2u}{2}{\Omega}+\np{\frac{\D u}{|x|}}{2}{\Omega}\right)^2\left(\np{\D^2u}{2}{\Omega}+\np{\frac{\D u}{|x|}}{2}{\Omega}\right)^2\int_{A_k}\left(|\D^2\widetilde{u}|^2+\frac{|\D\widetilde{u}|^2}{|x|^2}\right)dx
    \end{align*}
    \normalsize
    If $\ens{\chi_k}_{k\in\N}$ a partition of unity such that $\mathrm{supp}(\chi_k)\subset \widetilde{A}_k=B_{2^{-(k-1)}}\setminus\bar{B}_{2^{-(k+1)}}(0)$, since $0\leq \chi_k\leq 1$, we also have
    \begin{align}
        \np{\chi_kK}{\frac{4}{3}}{B(0,1)}&\leq C\left(1+\np{\D^2u}{2}{\Omega}+\np{\frac{\D u}{|x|}}{2}{\Omega}\right)^2\left(\np{\D^2u}{2}{\Omega}+\np{\frac{\D u}{|x|}}{2}{\Omega}\right)^2\nonumber\\
        &\times\left(\np{|\D^2u|+\frac{|\D\widetilde{u}|}{|x|}}{2}{A_k\cup A_{k+1}}\right).
    \end{align}
    Therefore, instead of treating a general case, we will assume that such a similar inequality holds in our theorem analogous to the weighted Wente inequality of \cite{riviere_morse_scs}.

    \begin{theorem}\label{wente_dim4}
        Let $K\in L^{\frac{4}{3}}(B(0,1))$ and assume that there exists $f\in L^2(B(0,1))$ such that for all $k\in \N$, 
        \begin{align*}
            \np{K}{\frac{4}{3}}{\widetilde{A}_k}\leq \np{f}{2}{\widetilde{A}_k}.
        \end{align*}
        Let $u:B(0,1)\rightarrow\R^d$ be the solution of 
        \begin{align*}
            \left\{\begin{alignedat}{2}
                \Delta u&=\dive(K)\qquad&& \text{in}\;\, B(0,1)\\
                u&=0\qquad&& \text{on}\;\, B(0,1).
            \end{alignedat}\right.
        \end{align*}
        Then, there exists a universal constant $\Gamma_2<\infty$ such that
        \begin{align}\label{wente_dim4_ineq}
            &\int_{B(0,1)}|x|^{2}|u(x)|^2dx\leq \Gamma_2^2\int_{B(0,1)}|x|^2\left(1+\log_2\left(\frac{1}{|x|}\right)\right)^2\log^2\left(e+\log_2\left(\frac{1}{|x|}\right)\right)|f(x)|^2dx
        \end{align}
    \end{theorem}
    \begin{proof}
        Let $\ens{\chi_k}_{k\in\N}$ be a partition of unity as in Theorem \ref{lemma:dyadic_gradient} such that $\mathrm{supp}(\chi_k)\subset \widetilde{A}_k$ for all $k\in\N$. Then, we have an expansion 
        \begin{align*}
            u_k=\sum_{k\in\N}u_k,
        \end{align*}
        where $u_k$ solves the equation
        \begin{align*}
            \left\{\begin{alignedat}{2}
                \Delta u_k&=\dive(\chi_kK)=\dive(K_k)\qquad&&\text{in}\;\, B(0,1)\\
                u_k&=0\qquad&&\text{in}\;\, B(0,1).
            \end{alignedat}\right.
        \end{align*}
        In particular, since $0\leq \chi_k\leq 1$, we have
        \begin{align*}
            \np{K_k}{\frac{4}{3}}{B(0,1)}\leq \np{f}{2}{\widetilde{A}_k}.
        \end{align*}
        Therefore, we have
        \begin{align*}
            \int_{A_k}|u|^2dx=\sum_{i,j\in \N}\int_{A_k}u_i\cdot u_j\,dx,
        \end{align*}
        and
        \begin{align*}
            \int_{B(0,1)}|x|^{2p}|u|^2dx=\sum_{k\in \N}\int_{A_k}|x|^{2p}|u|^2dx\leq \sum_{i,j,k\in \N}2^{-2pk}\int_{A_k}u_i\cdot u_j\,dx. 
        \end{align*}
        Thanks to the previous Lemma \ref{lemmae3}, we see that the proof of Theorem \ref{lemma:dyadic_gradient} applies word by word if one replaces $\D u_i$ by $u_i$. Therefore, the announced estimate follows. 
        \end{proof}

    Using the two previous Wente Lemma, we will now be able to prove the Hölder-type estimates that first appeared in \cite{riviere_morse_scs}.

    \section{Refined Pointwise Bounds on Biharmonic Functions in Annuli}

    \subsection{Estimate of the Second Order Derivatives}
    
    \begin{lemme}[Lemma F.$1$ \cite{riviere_morse_scs}, Lemma $3.12$ \cite{riviere_morse_scs}]
        Let $K\in L^{\frac{4}{3}}(B(0,1),\R^m)$ and assume that there exists $f\in L^2(B(0,1))$ such that for all $k\in \N$, 
        \begin{align*}
            \np{K}{\frac{4}{3}}{A_k}\leq \np{f}{2}{A_k}.
        \end{align*}
        If $u$ solve the equation
        \begin{align*}
            \Delta u=\dive(K)\qquad\text{in}\;\, B(0,1).
        \end{align*}
        Then, the following estimate holds
        \begin{align*}
            \np{u}{2}{B_{\frac{1}{2}}\setminus\bar{B}_{\frac{1}{4}}(0)}\leq \frac{1}{4}\np{u}{2}{B_1\setminus\bar{B}_{\frac{1}{2}}(0)}+\frac{5}{2}\Gamma_2\np{\omega\,f}{2}{B(0,1)},
        \end{align*}
        where $\omega(x)=|x|\left(1+\log_2\left(\frac{1}{|x|}\right)\right)\log\left(e+\log_2\left(\frac{1}{|x|}\right)\right)$, and $\Gamma_2<\infty$ is the universal constant of Theorem \ref{wente_dim4}.
    \end{lemme}
    \begin{proof}
        Make the decomposition $u=\varphi+\psi$, where 
        \begin{align*}
            \left\{\begin{alignedat}{2}
                \Delta\varphi&=\dive(K)\qquad&&\text{in}\;\, B(0,1)\\
                \varphi&=0\qquad&&\text{on}\;\, \partial B(0,1).
            \end{alignedat}\right.
        \end{align*}
        By Theorem \ref{wente_dim4}, we have
        \begin{align*}
            \np{|x|\varphi}{2}{B(0,1)}\leq C\np{\omega\,h}{2}{B(0,1)}
        \end{align*}
        Since $\psi$ is harmonic on $B(0,1)$, it admits the expansion
        \begin{align*}
            \psi(r,\omega)=\sum_{n=0}^{\infty}\sum_{k=1}^{(n+1)^2}a_{n,k}\,r^n\,Y_n^k(\omega),
        \end{align*}
        where $Y_n^k$ are the spherical harmonics and $a_{n,k}\in \R^d$. We have for all $0\leq a\leq b\leq 1$
        \begin{align}\label{dyadic_lemma1}
            \int_{B_b\setminus\bar{B}_a(0)}|\psi|^2dx=\pi^2\sum_{n=0}^{\infty}\frac{|a_{n,k}|^2}{n+2}b^{2(n+2)}\left(1-\left(\frac{a}{b}\right)^{2(n+2)}\right).
        \end{align}
        Taking successively $(a,b)=\left(\frac{1}{2},1\right)$ and $(a,b)=\left(\frac{1}{4},\frac{1}{2}\right)$, we get
        \begin{align*}
            \int_{B_1\setminus\bar{B}_{\frac{1}{2}}(0)}|\psi|^2dx=\pi^2\sum_{n=0}^{\infty}\frac{|a_{n,k}|^2}{n+2}\left(1-\frac{1}{2^{2(n+2)}}\right)
        \end{align*}
        and
        \begin{align}\label{dyadic_lemma2}
            \int_{B_{\frac{1}{2}}\setminus\bar{B}_{\frac{1}{4}}(0)}|\psi|^2dx=\pi^2\sum_{n=0}^{\infty}\frac{|a_{n,k}|^2}{n+2}\frac{1}{2^{2(n+2)}}\left(1-\frac{1}{2^{2(n+2)}}\right)\leq \frac{1}{16}\int_{B_1\setminus\bar{B}_{\frac{1}{2}}(0)}|\psi|^2dx.
        \end{align}
        By \eqref{dyadic_lemma1} and \eqref{dyadic_lemma2}, we have 
        \begin{align*}
            \np{u}{2}{B_{\frac{1}{2}}\setminus\bar{B}_{\frac{1}{4}}(0)}&\leq \np{\psi}{2}{B_{\frac{1}{2}}\setminus\bar{B}_{\frac{1}{4}}(0)}+\np{\varphi}{2}{B_{\frac{1}{2}}\setminus\bar{B}_{\frac{1}{4}}(0)}
            \leq \frac{1}{4}\np{\psi}{2}{B_1\setminus\bar{B}_{\frac{1}{2}}(0)}+\np{\varphi}{2}{B_1\setminus\bar{B}_{\frac{1}{2}}(0)}\\
            &\leq \frac{1}{4}\np{u}{2}{B_1\setminus\bar{B}_{\frac{1}{2}}(0)}+\frac{5}{4}\np{\varphi}{2}{B_{1}\setminus\bar{B}_{\frac{1}{2}}(0)}
            \leq \frac{1}{4}\np{u}{2}{B_1\setminus\bar{B}_{\frac{1}{2}}(0)}+\frac{5}{2}\np{|x|\varphi}{2}{B_1\setminus\bar{B}_{\frac{1}{2}}(0)}\\
            &\leq \frac{1}{4}\np{u}{2}{B_1\setminus\bar{B}_{\frac{1}{2}}(0)}+C\np{\omega f}{2}{B(0,1)},
        \end{align*}
        which concludes the proof of the lemma.
    \end{proof}
    Likewise, the following result holds.
    \begin{lemme}
        Let $f\in L^2(B(0,1),\R^m)$ and $u:B(0,1)\rightarrow \R^m$ be a solution of the equation
        \begin{align*}
            \Delta u=f\qquad\text{in}\;\, B(0,1).
        \end{align*}
        Then, the following estimate holds true
        \begin{align*}
            \np{\D u}{2}{B_{\frac{1}{2}}\setminus\bar{B}_{\frac{1}{4}}(0)}\leq \frac{1}{4}\np{\D u}{2}{B_1\setminus\bar{B}_{\frac{1}{2}}(0)}+\frac{5}{2}\Gamma_2\np{\omega\,f}{2}{B(0,1)}
        \end{align*}
        where $\omega(x)=|x|\left(1+\log_2\left(\frac{1}{|x|}\right)\right)\log\left(e+\log_2\left(\frac{1}{|x|}\right)\right)$. Likewise, we have
        \begin{align*}
            \np{\frac{\D u}{|x|}}{2}{B_{\frac{1}{2}}\setminus\bar{B}_{\frac{1}{4}}(0)}\leq \frac{1}{2}\np{\D u}{2}{B_1\setminus\bar{B}_{\frac{1}{2}}(0)}+\frac{5}{2}\Gamma_2\np{\omega\,f}{2}{B(0,1)}.
        \end{align*}
    \end{lemme}
    \begin{proof}
        Indeed, let $\varphi\in W^{1,2}_0(B(0,1))$ be such that $\Delta \varphi=f$ and $\psi=u-\varphi$. Then, $\psi$ is a harmonic function, which implies by \eqref{dirichlet_dim4} that for all $0<a<b<\infty$.
        Now, recall that by formula \eqref{dirichlet_dim4}
        \begin{align}\label{dirichlet_dim4_bis}
             \int_{B_b\setminus\bar{B}_a(0)}|\D \psi|^2dx&=2\pi^2\sum_{n=1}^{\infty}\sum_{k=1}^{(n+1)^2}n|a_{n,k}|^2b^{2(n+1)}\left(1-\left(\frac{a}{b}\right)^{2(n+1)}\right).
        \end{align}
        Therefore, we have
        \begin{align*}
             \int_{B_1\setminus\bar{B}_{\frac{1}{2}}(0)}|\D \psi|^2dx=2\pi^2\sum_{n=1}^{\infty}\sum_{k=1}^{(n+1)^2}n|a_{n,k}|^2\left(1-\left(1-\frac{1}{2^{2(n+1)}}\right)^{2(n+1)}\right)
        \end{align*}
        and
        \begin{align*}
            \int_{B_{\frac{1}{2}}\setminus\bar{B}_{\frac{1}{4}}(0)}|\D\psi|^2dx=2\pi^2\sum_{n=1}^{\infty}\sum_{k=1}^{(n+1)^2}n|a_{n,k}|^2\frac{1}{2^{2(n+1)}}\left(1-\frac{1}{2^{2(n+1)}}\right)\leq \frac{1}{16}\int_{B_1\setminus\bar{B}_{\frac{1}{2}}}|\D\psi|^2dx.
        \end{align*}
        On the other hand, using Lemma \ref{weighted_modified_dirichlet}, we deduce that
        \begin{align*}
            \np{\D\varphi}{2}{B(0,1)}\leq \Gamma_2\np{\omega\, f}{2}{B(0,1)}.
        \end{align*}
        Therefore, we finally get as in the previous lemma 
        \begin{align*}
            \np{\D u}{2}{B_{\frac{1}{2}}\setminus\bar{B}_{\frac{1}{4}}(0)}\leq \frac{1}{4}\np{\D u}{2}{B_1\setminus\bar{B}_{\frac{1}{2}}(0)}+\frac{5}{2}\Gamma_2\np{\omega\,f}{2}{B(0,1)}.
        \end{align*}
        Likewise, let us get the weighted estimate. We have for all $0\leq a<b\leq 1$
        \begin{align*}
            \int_{B_b\setminus\bar{B}_a(0)}\frac{|\D \psi|^2}{|x|^2}dx=2\pi^2\sum_{n=1}^{\infty}\sum_{k=1}^{(n+1)^2}(n+1)|a_{n,k}|^2b^{2n}\left(1-\left(\frac{a}{b}\right)^{2n}\right),
        \end{align*}
        which implies that 
        \begin{align*}
            \int_{B_1\setminus\bar{B}_{\frac{1}{2}}(0)}\frac{|\D\psi|^2}{|x|^2}dx=2\pi^2\sum_{n=1}^{\infty}\sum_{k=1}^{(n+1)^2}(n+1)|a_{n,k}|^2\left(1-\frac{1}{2^{2n}}\right)
        \end{align*}
        and
        \begin{align*}
            \int_{B_{\frac{1}{2}}\setminus\bar{B}_{\frac{1}{4}}(0)}\frac{|\D u|^2}{|x|^2}dx=2\pi^2\sum_{n=1}^{\infty}\sum_{k=1}^{(n+1)^2}(n+1)|a_{n,k}|^2\frac{1}{2^{2n}}\left(1-\frac{1}{2^{2n}}\right)\leq \frac{1}{4}\int_{B_1\setminus\bar{B}_{\frac{1}{2}}(0)}\frac{|\D u|^2}{|x|^2}dx,
        \end{align*}
        which concludes the proof of the lemma by applying the exact same proof as in the first case. 
    \end{proof}
    \begin{theorem}[Lemma F.$2$ \cite{riviere_morse_scs}, Lemma $3.13$ \cite{morse_willmore_I}]\label{dyadic_main_theorem}
    Let $K\in L^{\frac{4}{3}}(B_{\R^4}(0,1),\R^m)$ and assume that there exists $f\in L^2(B(0,1))$ such that for all $k\in \N$, 
        \begin{align*}
            \np{K}{\frac{4}{3}}{A_k}\leq \np{f}{2}{A_k}.
        \end{align*}
        If $u$ solve the equation
        \begin{align*}
            \Delta u=\dive(K)\qquad\text{in}\;\, B(0,1).
        \end{align*}        
    Then, there exists a universal constant $C<\infty$ such that for all $0<\alpha<1$ and for all $k\in \N$
    \begin{align}\label{ineq_F20}
        \np{u}{2}{A_k}\leq \frac{1}{4^k}\np{u}{2}{A_0}+\frac{C}{(1-\alpha)^3}\left(\sum_{l=0}^{\infty}\frac{1}{2^{2\alpha|l-k+1|}}\int_{A_l}|f|^2dx\right)^{\frac{1}{2}},
    \end{align}
    where $A_k=B_{2^{-k}}\setminus\bar{B}_{2^{-(k+1)}}(0)$ for all $k\in\N$.
    \end{theorem}
    \begin{proof}
        The proof follows exactly from the one of Theorem $3.13$ of \cite{morse_willmore_I} and we omit it.
    \end{proof}
    Likewise, we obtain the following theorem.

    \begin{theorem}[Lemma F.$2$ \cite{riviere_morse_scs}, Lemma $3.13$ \cite{morse_willmore_I}]\label{dyadic_main_theorem2}
    Let $f\in L^{2}(B_{\R^4}(0,1),\R^m)$ and assume that $u:B(0,1)\rightarrow \R^m$ solves the equation
        \begin{align*}
            \Delta u=f\qquad\text{in}\;\, B(0,1).
        \end{align*}        
    Then, there exists a universal constant $C<\infty$ such that for all $0<\alpha<1$ and for all $k\in \N$
    \begin{align}\label{ineq_F202}
        \np{\D u}{2}{A_k}\leq \frac{1}{4^k}\np{\D u}{2}{A_0}+\frac{C}{(1-\alpha)^3}\left(\sum_{l=0}^{\infty}\frac{1}{2^{2\alpha|l-k+1|}}\int_{A_l}|f|^2dx\right)^{\frac{1}{2}},
    \end{align}
    where $A_k=B_{2^{-k}}\setminus\bar{B}_{2^{-(k+1)}}(0)$ for all $k\in\N$. Furthermore, we also have 
    \begin{align}\label{ineq_F21}
        \np{\frac{\D u}{|x|}}{2}{A_k}\leq \frac{1}{2^k}\np{\frac{\D u}{|x|}}{2}{A_0}+\frac{C}{(1-\alpha)^3}\left(\sum_{l=0}^{\infty}\frac{1}{2^{2\alpha|l-k+1|}}\int_{A_l}|f|^2dx\right)^{\frac{1}{2}},
    \end{align}
    \end{theorem}

 \subsection{Final Estimates}

    In order to conclude the proof, we need to prove a Hölder-type estimate for biharmonic maps in dimension $4$.

    \begin{theorem}\label{pointwise_biharmonic}
        Let $0<2\,a<b\leq 1$, $\Omega=B_b\setminus\bar{B}_a(0)$, and for all $0<t\leq 1$ define $\Omega_t=B_{t\,b}\setminus\bar{B}_{t^{-1}a}(0)$. Let  $u:\Omega\rightarrow M^m\subset \R^d$ be an (extrinsic) biharmonic map and define
        \begin{align}
            \Lambda=\frac{1}{2\pi^2}\int_{\Omega_{\frac{1}{2}}}\frac{|A \Delta u|}{|x|^2}dx
        \end{align}
        Then, for all $0<\beta<1$, for all $0<\delta<1$ there exists $C_{\beta}<\infty$ and $0<\epsilon_{\beta}(\delta)<1$ such that provided that
        \begin{align*}
            \np{\D^2u}{2}{\Omega}+\np{\frac{\D u}{|x|}}{2}{\Omega}\leq \epsilon_{\beta}(\delta),
        \end{align*}
        then for all $x\in \Omega_{1/2}$, we have
        \begin{align}
            \np{\D^2 u}{2}{B_{\frac{3|x|}{2}}\setminus\bar{B}_{\frac{|x|}{2}}(0)}+\np{\frac{\D u}{|x|}}{2}{B_{\frac{3|x|}{2}}\setminus\bar{B}_{\frac{|x|}{2}}(0)}&\leq C\left(\left(\frac{|x|}{b}\right)^{\beta}+\left(\frac{a}{|x|}\right)^{\beta}\right)\left(\np{\D^2u}{2}{\Omega}+\np{\frac{\D u}{|x|}}{2}{\Omega}\right)\nonumber\\
            &+\frac{C}{\log\left(\frac{b}{a}\right)}\left(\Lambda+\np{\D^2u}{2}{\Omega}+\np{\frac{\D u}{|x|}}{2}{\Omega}\right).
        \end{align}
        More precisely, for all $x\in \Omega_{1/2}$, we have
        \begin{align}
            \np{A\Delta u}{2}{B_{\frac{3|x|}{2}}\setminus\bar{B}_{\frac{|x|}{2}}(0)}+\np{\frac{\D u}{|x|}}{2}{B_{\frac{3|x|}{2}}\setminus\bar{B}_{\frac{|x|}{2}}(0)}&\leq C\left(\left(\frac{|x|}{b}\right)^{\beta}+\left(\frac{a}{|x|}\right)^{\beta}\right)\left(\np{\D^2u}{2}{\Omega}+\np{\frac{\D u}{|x|}}{2}{\Omega}\right)\nonumber\\
            &+\frac{\pi\sqrt{2\log(2)}}{\log\left(\frac{b}{a}\right)}\left(\Lambda+C\left(\np{\D^2u}{2}{\Omega}+\np{\frac{\D u}{|x|}}{2}{\Omega}\right)\right).
        \end{align}
    \end{theorem}
    \begin{proof}
        Let $\widetilde{u}:B(0,1)\rightarrow \R^n$ be a controlled extension of $u$ given by Theorem \ref{whitney_extension_dim4}, and let $K$ be given as in the proof of Theorem \ref{l21_neck}. Make the decomposition $A\Delta \widetilde{u}=\varphi+\psi$ in $\Omega$, where 
        \begin{align}
            \left\{\begin{alignedat}{2}
                \Delta \varphi&=\dive(K)\qquad&&\text{in}\;\, B(0,1)\\
                \varphi&=0\qquad&&\text{on}\;\, \partial B(0,1).
            \end{alignedat}\right.
        \end{align}
        Notice that by Calder\'{o}n-Zygmund estimates, we have 
        \begin{align*}
            \np{\varphi}{2,1}{B(0,1)}+\np{\D\varphi}{\frac{4}{3},1}{B(0,1)}\leq C\np{{K}}{\frac{4}{3},1}{B(0,1)}\leq C\left(\np{\D^2 u}{2}{\Omega}+\np{\frac{\D u}{|x|}}{2}{\Omega}\right).
        \end{align*}
        Since $\psi$ is a harmonic function in $\Omega$, it admits the expansion 
        \begin{align*}
            \psi(r,\omega)=\sum_{n=0}^{\infty}\sum_{k=1}^{(n+1)^2}\left(a_{n,k}\,r^n+b_{n,k}\,r^{-(n+2)}\right)Y_n^k(\omega).
        \end{align*}
        Notice that we have
        \begin{align*}
            \int_{\partial B(0,r)}\partial_{\nu}\psi \,d\mathscr{H}^3=-4\pi^2\,b_{0,1}.
        \end{align*}
        Writing for convenience
        \begin{align*}
            \gamma=b_{0,1}=-\frac{1}{4\pi^2}\int_{\partial B(0,r)}\partial_{\nu}u\,d\mathscr{H}^3,
        \end{align*}
        defining
        \begin{align}\label{def_varphi0}
            \psi_0=\psi-\frac{\gamma}{|x|^2},
        \end{align}
        we deduce by Theorem \ref{pointwise_harmonic_u} that for all $x\in \Omega_{\frac{1}{2}}$,
        \begin{align}\label{psi0_pointwise}
            |\psi_0(x)|\leq \frac{\widetilde{\Lambda}_4}{|x|^2}\left(\frac{|x|}{b}+\frac{a}{|x|}\right)\np{\psi_0}{2}{\Omega},
        \end{align}
        while by Theorem \ref{lorentz_l2_hessian}, we have for all $0<\alpha<1$
        \begin{align}\label{psi0_21}
            \np{\psi_0}{2,1}{\Omega_{\alpha}}\leq \frac{C_4}{\sqrt{1-\left(\frac{a}{b}\right)^2}}\frac{\alpha}{(1-\alpha^2)^3}\np{\psi_0}{2}{\Omega}.
        \end{align}
        In particular, we have for all $k\in \N$
        \begin{align}\label{bound_psi0}
            \np{\psi_0}{2}{A_k}\leq C\left(\frac{2^{-k}}{b}+\frac{a}{2^{-k}}\right)\np{\psi_0}{2}{\Omega}
        \end{align}
        Furthermore, we have
        \begin{align}
            \sqrt{\np{\psi_0}{2}{\Omega}^2+\np{\psi-\psi_0}{2}{\Omega}^2}&=\np{\psi}{2}{\Omega}\leq \np{A\Delta u}{2}{\Omega}+\np{\varphi}{2}{\Omega}\leq C\left(\np{\Delta u}{2}{\Omega}+\np{K}{\frac{4}{3},1}{\Omega}\right)\nonumber\\
            &\leq C\left(\np{\D^2u}{2}{\Omega}+\np{\frac{\D u}{|x|}}{2}{\Omega}\right).
        \end{align}
        Then, by $L^{2,1}/L^{2,\infty}$ duality, we have
        \begin{align*}
            \int_{\Omega}\frac{|\varphi|}{|x|^2}dx\leq \np{\varphi}{2,1}{\Omega}\np{\frac{1}{|x|^2}}{2,\infty}{\Omega}\leq C\np{{K}}{\frac{4}{3},1}{\Omega}
        \end{align*}
        and
        \begin{align*}
            \int_{\Omega_{\frac{1}{2}}}\frac{|\psi_0|}{|x|^2}dx\leq C\left(\np{\Delta u}{2}{\Omega}+\np{K}{\frac{4}{3},1}{\Omega}\right).
        \end{align*}
        Then, notice that
        \begin{align*}
            \int_{\Omega_{\frac{1}{2}}}\frac{|\psi-\psi_0|}{|x|^2}dx=\int_{\Omega_{\frac{1}{2}}}\frac{|\gamma|}{|x|^4}dx=2\pi^2|\gamma|\log\left(\frac{4\,b}{a}\right).
        \end{align*}
        Therefore, we have
        \begin{align}\label{log_estimate_final}
            2\pi^2\left||\gamma|\log\left(\frac{4\,b}{a}\right)-\Lambda\right|&=\left|\int_{\Omega_{\frac{1}{2}}}\frac{|\psi-\psi_0|}{|x|^2}dx-\int_{\Omega_{\frac{1}{2}}}\frac{|u|}{|x|^2}dx\right|\leq \int_{\Omega_{\frac{1}{2}}}\frac{|u-(\psi-\psi_0)|}{|x|^2}dx=\int_{\Omega_{\frac{1}{2}}}\frac{|\varphi+\psi_0|}{|x|^2}\\
            &\leq C\left(\np{\Delta u}{2}{\Omega}+\np{K}{\frac{4}{3},1}{\Omega}\right)\leq C\left(\np{\D^2 u}{2}{\Omega}+\np{\frac{\D u}{|x|}}{2}{\Omega}\right)
        \end{align}
        Now, using Theorem \ref{dyadic_main_theorem}, we deduce that for all $k\in \N$ and $0<\alpha<1$, we have
        \begin{align}\label{dyadic_varphi}
            \np{\varphi}{2}{A_k}&\leq \frac{1}{4^k}\np{\varphi}{2}{A_0}+\frac{C}{(1-\alpha)^{3}}\left(\sum_{l=0}^{\infty}\frac{1}{2^{2\alpha|l-k+1|}}\int_{A_l}\left(|\D^2 \widetilde{u}|^2+\frac{|\D \widetilde{u}|^2}{|x|^2}\right)dx\right)^{\frac{1}{2}}\nonumber\\
            &\leq \frac{1}{4^k}\left(\np{\D^2 u}{2}{\Omega}+\np{\frac{\D u}{|x|}}{2}{\Omega}\right)+\frac{C}{(1-\alpha)^{3}}\left(\sum_{l=0}^{\infty}\frac{1}{2^{2\alpha|l-k+1|}}\int_{A_l}\left(|\D^2 \widetilde{u}|^2+\frac{|\D \widetilde{u}|^2}{|x|^2}\right)dx\right)^{\frac{1}{2}}.
        \end{align}
        Then, we have
        \begin{align}\label{log_estimate_final2}
            \int_{A_k}|\psi-\psi_0|^2dx=\int_{B_{2^{-k}}\setminus\bar{B}_{2^{-(k+1)}}(0)}\frac{|\gamma|^2}{|x|^4}dx=2\pi^2\log(2)|\gamma|^2.
        \end{align}
        Therefore, we finally get by \eqref{log_estimate_final}, \eqref{dyadic_varphi}, and \eqref{log_estimate_final2}
        \begin{align}\label{ende0}
            &\np{A\Delta\widetilde{u}}{2}{A_k}\leq \np{\varphi}{2}{A_k}+\np{\psi_0}{2}{A_k}\nonumber\\
            &\leq \pi\sqrt{2\log(2)}|\gamma|+C\left(\frac{2^{-k}}{b}+\frac{a}{2^{-k}}\right)\left(\np{\D^2 u}{2}{\Omega}+\np{\frac{\D u}{|x|}}{2}{\Omega}\right)\nonumber\\
            &+\frac{C}{(1-\alpha)^{3}}\left(\sum_{l=0}^{\infty}\frac{1}{2^{2\alpha|l-k+1|}}\int_{A_l}\left(|\D^2 \widetilde{u}|^2+\frac{|\D \widetilde{u}|^2}{|x|^2}\right)dx\right)^{\frac{1}{2}}\nonumber\\
            &\leq \pi\sqrt{2\log(2)}\left(\frac{\Lambda+\np{|\D^2u|+|x|^{-1}|\D u|}{2}{\Omega}}{\log\left(\frac{4b}{a}\right)}\right)+C\left(\frac{2^{-k}}{b}+\frac{a}{2^{-k}}\right)\left(\np{\D^2 u}{2}{\Omega}+\np{\frac{\D u}{|x|}}{2}{\Omega}\right)\nonumber\\
            &+\frac{C}{(1-\alpha)^{3}}\left(\sum_{l=0}^{\infty}\frac{1}{2^{2\alpha|l-k+1|}}\int_{A_l}\left(|\D^2 \widetilde{u}|^2+\frac{|\D \widetilde{u}|^2}{|x|^2}\right)dx\right)^{\frac{1}{2}}.
        \end{align}
        Then, using the equation
        \begin{align*}
            \dive\left(A\D\widetilde{u}\right)=\D A\D\widetilde{u}+A\Delta\widetilde{u}=F,
        \end{align*}
        we notice that for all open subset $U\subset B(0,1)$
        \begin{align*}
            \np{F}{2}{U}\leq \np{\D A}{4}{U}\np{\D\widetilde{u}}{4}{U}+\np{A}{\infty}{U}\np{\Delta\widetilde{u}}{2}{U}\leq C\left(\np{\D^2\widetilde{u}}{2}{U}+\np{\frac{\D \widetilde{u}}{|x|}}{2}{U}\right).
        \end{align*}
        Then, make a Hodge decomposition $A\D\widetilde{u}=d\chi+d^{\ast}\xi$ as in \cite{biharmonic_quanta} (\cite[Corollary $10.5.1$]{iwaniec}), where
        \begin{align}
            \left\{\begin{alignedat}{2}
                \Delta \chi&=F\qquad&&\text{in}\;\,\Omega\\
            \end{alignedat}\right.
        \end{align}
        and
        \begin{align}
            \left\{\begin{alignedat}{2}
                \Delta\xi&=dA\wedge d\widetilde{u}\qquad&&\text{in}\;\, B(0,b)\\
               \xi&=0\qquad&&\text{on}\;\,\partial B(0,b).
            \end{alignedat}\right.
        \end{align}
        By standard Caldr\'{o}n-Zygmund estimates and interpolation theory, we have 
        \begin{align*}
            \np{\D^2\xi}{2,1}{B(0,b)}+\np{\D\xi}{4,1}{B(0,b)}\leq C_{\alpha}\left(\np{\D^2u}{2}{\Omega}+\np{\frac{\D u}{|x|}}{2}{\Omega}\right).
        \end{align*}
        Thanks to Theorem \ref{dyadic_main_theorem2}, we deduce that for all $k\in \N$ and $0<\alpha<1$, we have
        \begin{align*}
            \np{\frac{\D\chi}{|x|}}{2}{A_k}\leq \frac{1}{2^k}\np{\frac{\D\chi}{|x|}}{2}{A_0}+\frac{C}{(1-\alpha)^3}\left(\sum_{l=0}^{\infty}\frac{1}{2^{2\alpha|l-k+1|}}\int_{A_l}\left(|\D^2\widetilde{u}|^2+\frac{|\D\widetilde{u}|^2}{|x|^2}\right)dx\right)^{\frac{1}{2}}.
        \end{align*}
        Now, applying Theorem \ref{dyadic_main_theorem2}, we deduce that for all $k\in \N$ and $0<\alpha<1$, we have
        \begin{align*}
            &\np{\frac{\D\xi}{|x|}}{2}{A_k}\leq \frac{1}{A_k}\np{\D\xi}{|x|^2}{A_0}+\frac{C}{(1-\alpha)^3}\left(\sum_{l=0}^{\infty}\frac{1}{2^{2\alpha|l-k+1|}}\int_{A_l}\left(|\D A\wedge \D\widetilde{u}|^2\right)dx\right)^{\frac{1}{2}}\\
            &\leq \frac{C}{2^k}\left(\np{\D^2u}{2}{\Omega}+\np{\frac{\D u}{|x|}}{2}{\Omega}\right)+\frac{C}{(1-\alpha)^3}\left(\sum_{l=0}^{\infty}\frac{1}{2^{2\alpha|l-k+1|}}\np{\D\widetilde{u}}{4,2}{A_l}\right)^{\frac{1}{2}}\\
            &\leq \frac{C}{2^k}\left(\np{\D^2u}{2}{\Omega}+\np{\frac{\D u}{|x|}}{2}{\Omega}\right)+\frac{C}{(1-\alpha)^3}\left(\sum_{l=0}^{\infty}\frac{1}{2^{2\alpha|l-k+1|}}\int_{A_l}\left(|\D^2\widetilde{u}|^2+\frac{|\D\widetilde{u}|^2}{|x|^2}\right)dx\right)^{\frac{1}{2}}.
        \end{align*}
        Finally, we have
        \begin{align}\label{ende1}
            &\np{\frac{\D \widetilde{u}}{|x|}}{2}{A_k}\leq C\np{A\D\widetilde{u}}{2}{A_k}\leq C\left(\np{\frac{\D\chi}{|x|}}{2}{A_k}+\np{\frac{\D\xi}{|x|}}{2}{A_k}\right)\nonumber\\
            &\leq \frac{C}{2^k}\left(\np{\D^2u}{2}{\Omega}+\np{\frac{\D u}{|x|}}{2}{\Omega}\right)+\frac{C}{(1-\alpha)^3}\left(\sum_{l=0}^{\infty}\frac{1}{2^{2\alpha|l-k+1|}}\int_{A_l}\left(|\D^2\widetilde{u}|^2+\frac{|\D\widetilde{u}|^2}{|x|^2}\right)dx\right)^{\frac{1}{2}}
        \end{align}
        Now, by elliptic regularity and a scaling argument, there exists a universal constant $\Gamma_0<\infty$ such that for all $r>0$, we have
        \begin{align}\label{elliptic_reg}
            \int_{B_{2r}\setminus\bar{B}_r(0)}|\D^2v|^2dx\leq \int_{B_{4r}\setminus\bar{B}_{\frac{r}{2}}(0)}\left(|\Delta v|^2+\frac{|\D v|^2}{|x|^2}\right)dx.
        \end{align}
        Therefore, using that
        \begin{align*}
            \np{A}{\infty}{B(0,b)}+\np{A^{-1}}{\infty}{B(0,b)}\leq C\left(1+\np{\D^2 u}{2}{\Omega}+\np{\frac{\D u}{|x|}}{2}{\Omega}\right),
        \end{align*}
        we can rewrite \eqref{ende1} as
        \begin{align}\label{ende2}
            &\np{\frac{\D \widetilde{u}}{|x|}}{2}{A_k}\leq C\np{A\D\widetilde{u}}{2}{A_k}\leq C\left(\np{\frac{\D\chi}{|x|}}{2}{A_k}+\np{\frac{\D\xi}{|x|}}{2}{A_k}\right)\nonumber\\
            &\leq \frac{C}{2^k}\left(\np{\D^2u}{2}{\Omega}+\np{\frac{\D u}{|x|}}{2}{\Omega}\right)+\frac{C}{(1-\alpha)^3}\left(\sum_{l=0}^{\infty}\frac{1}{2^{2\alpha|l-k+1|}}\int_{A_l}\left(|A\Delta {u}|^2+\frac{|\D\widetilde{u}|^2}{|x|^2}\right)dx\right)^{\frac{1}{2}},
        \end{align}
        while \eqref{ende0} becomes
        \begin{align}\label{ende3}
            \np{A\Delta\widetilde{u}}{2}{A_k}&\leq \pi\sqrt{2\log(2)}\left(\frac{\Lambda+\np{|\D^2u|+|x|^{-1}|\D u|}{2}{\Omega}}{\log\left(\frac{4b}{a}\right)}\right)\nonumber\\
            &+C\left(\frac{2^{-k}}{b}+\frac{a}{2^{-k}}\right)\left(\np{\D^2 u}{2}{\Omega}+\np{\frac{\D u}{|x|}}{2}{\Omega}\right)\nonumber\\
            &+\frac{C}{(1-\alpha)^{3}}\left(\sum_{l=0}^{\infty}\frac{1}{2^{2\alpha|l-k+1|}}\int_{A_l}\left(|A\Delta \widetilde{u}|^2+\frac{|\D \widetilde{u}|^2}{|x|^2}\right)dx\right)^{\frac{1}{2}}.
        \end{align}
        Putting together \eqref{ende2} and \eqref{ende3}, we finally see that 
        \small
        \begin{align}\label{ende4}
            &\np{|A\Delta\widetilde{u}|+\frac{|\D \widetilde{u}|}{|x|}}{2}{A_k}\leq \pi\sqrt{2\log(2)}\left(\frac{\Lambda+\np{|\D^2u|+|x|^{-1}|\D u|}{2}{\Omega}}{\log\left(\frac{4b}{a}\right)}\right)\nonumber\\
            &+C\left(\frac{2^{-k}}{b}+\frac{a}{2^{-k}}\right)\left(\np{\D^2 u}{2}{\Omega}+\np{\frac{\D u}{|x|}}{2}{\Omega}\right)+\frac{C}{(1-\alpha)^{3}}\left(\sum_{l=0}^{\infty}\frac{1}{2^{2\alpha|l-k+1|}}\int_{A_l}\left(|A\Delta \widetilde{u}|^2+\frac{|\D \widetilde{u}|^2}{|x|^2}\right)dx\right)^{\frac{1}{2}}.
        \end{align}
        \normalsize
        This is the inequality given in \cite[Proposition III.1]{riviere_morse_scs} and in \cite[(1.3.67)]{morse_willmore_I}. Therefore, using \cite[Lemma 1.3.18]{riviere_morse_scs}, we see that the rest of the proof is identical and we omit it. The final inequality follows from \eqref{elliptic_reg}.
    \end{proof}

    \section{Proof of the Main Theorem}   

    \subsection{Weighted Estimate in Neck Regions}

       For all $0<\beta<1$, using the $\epsilon$-regularity from \cite{riviere_lamm_biharmonic}, we deduce that for all $x\in \Omega_{\frac{1}{2}}$,
       \begin{align*}
           |x|^2|\Delta u(x)|+|x||\D u(x)|\leq C\left(\int_{B_{\frac{3|x|}{2}}\setminus\bar{B}_{\frac{|x|}{2}}(0)}\left(|\D^2u|^2+\frac{|\D u|^2}{|x|^2}\right)dx\right)^{\frac{1}{2}}.
       \end{align*}
       Therefore by Theorem \eqref{pointwise_biharmonic}, we have the following estimate for all $x\in \Omega_{\frac{1}{2}}$
   \begin{align*}
       &\np{\D^2 u}{2}{B_{\frac{3|x|}{2}}\setminus\bar{B}_{\frac{|x|}{2}}(0)}+\np{\frac{\D u}{|x|}}{2}{B_{\frac{3|x|}{2}}\setminus\bar{B}_{\frac{|x|}{2}}}\\
       &\leq C\left(\left(\frac{|x|}{b}\right)^{\beta}+\left(\frac{a}{|x|}\right)^{\beta}+\frac{1}{\log\left(\frac{b}{a}\right)}\right)\left(\np{\D^2u}{2,1}{\Omega}+\np{\D u}{4,1}{\Omega}\right).
   \end{align*}
   Combining this estimate with the $\epsilon$-regularity yields the pointwise bound for all $x\in \Omega_{\frac{1}{2}}$
   \begin{align*}
       |x|^2|\Delta u(x)|+|x||\D u(x)|\leq C\left(\left(\frac{|x|}{b}\right)^{\beta}+\left(\frac{a}{|x|}\right)^{\beta}+\frac{1}{\log\left(\frac{b}{4a}\right)}\right)\left(\np{\D^2u}{2,1}{\Omega}+\np{\D u}{4,1}{\Omega}\right)
   \end{align*}
   Therefore, using \eqref{elementary_d2_lower_bound}, we deduce that there exists $C<\infty$ such that for all $v\in W^{1,2}_0(\Omega_{\frac{1}{2}})$
   \begin{align}\label{neck_positive}
 	&Q_{u}(v)\geq \frac{3}{4}\int_{\Omega_{\frac{1}{2}}}|\Delta v|^2dx-C\int_{\Omega_{\frac{1}{2}}}\left(|\Delta u|+|\D u|^2\right)|\D u|^2-C\int_{\Omega_{\frac{1}{2}}}\left(|\Delta u|^2+|\D u|^4\right)|v|^2dx\nonumber\\
  &\geq \frac{3}{4}\int_{\Omega_{\frac{1}{2}}}|\Delta v|^2dx-C\left(\np{\D^2u}{2}{\Omega}+\np{\D u}{4}{\Omega}\right)\int_{\Omega_{\frac{1}{2}}}\frac{|\D u|^2}{|x|^2}dx\nonumber\\
  &-C\left(\np{\D^2u}{2,1}{\Omega}+\np{\D u}{4,1}{\Omega}\right)\int_{\Omega_{\frac{1}{2}}}\frac{|v|^2}{|x|^4}\left(\left(\frac{|x|}{b}\right)^{2\beta}+\left(\frac{a}{|x|}\right)^{2\beta}+\frac{1}{\log^2\left(\frac{b}{4a}\right)}\right)dx\nonumber\\
  &\geq \left(\frac{1}{2}-\lambda_1\left(\np{\D^2u}{2}{\Omega}+\np{\D u}{4}{\Omega}\right)\right)\int_{\Omega_{\frac{1}{2}}}\frac{|\D v|^2}{|x|^2}dx\nonumber\\
  &+\left(2\pi^2-\lambda_1\left(\np{\D^2u}{2,1}{\Omega}+\np{\D u}{4,1}{\Omega}\right)\right)\frac{1}{\log^2\left(\frac{b}{a}\right)}\int_{\Omega_{\frac{1}{2}}}\frac{|v|^2}{|x|^4}dx\nonumber\\
  &+\left(\frac{C_{2\beta}}{12}-\lambda_1\left(\np{\D^2u}{2,1}{\Omega}+\np{\D u}{4,1}{\Omega}\right)\right)\int_{\Omega_{\frac{1}{2}}}\frac{|v|^2}{|x|^4}\left(\left(\frac{|x|}{b}\right)^{2\beta}+\left(\frac{a}{|x|}\right)^{2\beta}\right)dx
 \end{align}
 for some $\lambda_1<\infty$, where we wrote thanks to inequality \eqref{th:bound_biharmonic1:ineq} of Theorem  \ref{th:bound_biharmonic1}, inequality \eqref{th:bound_biharmonic0:ineq} of Theorem  \ref{th:bound_biharmonic0}, and inequality \eqref{th:bound_biharmonic2:ineq} of Theorem  \ref{th:bound_biharmonic2}, 
 \begin{align*}
     &\frac{3}{4}\int_{\Omega_{\frac{1}{2}}}|\Delta v|^2dx=\frac{1}{6}\int_{\Omega_{\frac{1}{2}}}|\Delta v|^2dx+\frac{1}{2}\int_{\Omega_{\frac{1}{2}}}|\Delta v|^2dx+\frac{1}{12}\int_{\Omega_{\frac{1}{2}}}|\Delta v|^2dx\\
      &\geq \int_{\Omega_{\frac{1}{2}}}\frac{|\D v|^2}{|x|^2}dx+\frac{2\pi^2}{\log^2\left(\frac{b}{4 a}\right)}\int_{\Omega}\frac{|v|^2}{|x|^4}dx+\frac{C_{2\beta}}{12}\int_{\Omega_{\frac{1}{2}}}\frac{|v|^2}{|x|^4}\left(\left(\frac{|x|}{b}\right)^{2\beta}+\left(\frac{a}{|x|}\right)^{2\beta}\right)dx.
 \end{align*}
 In particular, the following result holds true.
 \begin{theorem}
     Let $\ens{u_k}_{k\in \N}$ be a sequence of extrinsic biharmonic maps $u_k:B(0,1)\rightarrow (M^m,h)\subset \R^n$. Assume that 
     \begin{align*}
         \limsup_{k\rightarrow \infty}\left(\np{\D^2u_k}{2}{B(0,1)}+\np{\D u_k}{4}{B(0,1)}\right)<\infty.
     \end{align*}
     Then, if $\rho_k\conv{k\rightarrow \infty}0$ and $\Omega_k(\alpha)=B_{\alpha}\setminus\bar{B}_{\alpha^{-1}\rho_k}(0)$ is a neck region of $\ens{u_k}_{k\in \N}$, there exists $\alpha_0>0$ such that for all $0<\alpha<\alpha_0$ and $k\in \N$ large enough, for all $v\in W^{2,2}_0(\Omega_k(\alpha))$ 
     \begin{align}\label{ineq:stability_ineq}
         Q_{u_k}(v)&\geq \frac{1}{4}\int_{\Omega_k(\alpha)}\frac{|\D v|^2}{|x|^2}dx
  +\frac{\pi^2}{\log^2\left(\frac{\alpha^2}{\rho_k}\right)}\int_{\Omega_k(\alpha)}\frac{|v|^2}{|x|^4}dx
  +\frac{C_{2\beta}}{12}\int_{\Omega_k(\alpha)}\frac{|v|^2}{|x|^4}\left(\left(\frac{|x|}{b}\right)^{2\beta}+\left(\frac{a}{|x|}\right)^{2\beta}\right)dx.
     \end{align}
 \end{theorem}
 \begin{proof}
     Indeed, thanks to the strong energy quantization of Theorem \ref{th:improved_quantization}, we need only impose
     \begin{align*}
         &\np{\D^2u_k}{2}{\Omega_k(2\alpha_0)}+\np{\D u_k}{4}{\Omega_k(2\alpha_0)}\leq \frac{1}{4\lambda_1}\\
         &\np{\D^2u_k}{2,1}{\Omega_k(2\alpha_0)}+\np{\D u_k}{4,1}{\Omega_k(2\alpha_0)}\leq \min\ens{\pi^2,\frac{C_{2\beta}}{24}}
     \end{align*}
     and the inequality follows from \eqref{neck_positive}.
 \end{proof}
 In particular, if 
 \begin{align*}
     \omega_{a,b}(x)=\left\{\begin{alignedat}{2}
         &\frac{1}{|x|^4}\left(\left(\frac{|x|}{b}\right)^{2\beta}+\left(\frac{a}{|x|}\right)^{2\beta}+\frac{1}{\log^2\left(\frac{b}{a}\right)}\right)\qquad&&\text{for all}\;\, x\in B_b\setminus\bar{B}_a(0)\\
         &\frac{1}{b^4}\left(1+\left(\frac{a}{b}\right)^{2\beta}+\frac{1}{\log^2\left(\frac{b}{a}\right)}\right)\qquad&&\text{for all}\;\, x\in B(0,1)\setminus \bar{B}_{b}(0)\\
         &\frac{1}{a^4}\left(1+\left(\frac{a}{b}\right)^{2\beta}+\frac{1}{\log^2\left(\frac{b}{a}\right)}\right)\qquad&&\text{for all}\;\, x\in B(0,a)\\
     \end{alignedat}\right.,
 \end{align*}
 then for all $v\in W^{2,2}_0(\Omega_k(\alpha))$, we have
 \begin{align*}
     Q_{u_k}(v)\geq \lambda_0\int_{\Omega_k(\alpha)}\left(|\D v|^2+|v|^2\right)\omega_{\alpha^{-1}\rho_k,\alpha}dx,
 \end{align*}
 where
 \begin{align*}
     \lambda_0=\min\ens{\frac{1}{12},\pi^2,\frac{C_{2\beta}}{12}}.
 \end{align*}

     \subsection{Weighted Diagonalisation of the Second Derivative}

     Recall that by \eqref{der2_biharmonique0}, we have
     \small
     \begin{align}\label{der2_biharmonique1}
        &Q_u(w)=\int_{B(0,1)}\bigg\{|\Delta w|^2+\s{A_u(\D u,\D u)}{(\Delta A_u)(w,w)+2(\D A_u)(\D w,w)+2\,A_u(\D w,\D w)+2\,A_u(w,\Delta w)}\nonumber\\
        &-2\s{\s{\Delta u}{\D P_u}}{(\D A_u)(w,w)+2\,A_u(\D w,w)}
        +\s{\s{\Delta(P_u)}{\Delta u}}{A_u(w,w)}\bigg\}dx.
    \end{align}
    \normalsize
    Let $S:\Sigma\rightarrow \mathrm{Sym}(\R^n)$ be the shape operator such that for all $x\in \Sigma$ and $X,Y,Z\in T_x\Sigma$,
    \begin{align*}
        \s{A_x(X,Y)}{Z}=\s{S_x(Z) X}{Y}.
    \end{align*}
    Likewise, for all $1\leq i, j\leq 2$ there exists shape operators $S^i:\Sigma\rightarrow \mathrm{Sym}(\R^n)$ and $S^{i,j}:\Sigma\rightarrow \mathrm{Sym}(\R^n)$ such that for all $x\in \Sigma$ and $X,Y,Z\in T_x\Sigma$,
    \begin{align*}
        \s{(\p{x_i}A_x)(X,Y)}{Z}=\s{S_x^i(Z) X}{Y}\qquad\text{and}\qquad \s{(\p{x_i,x_j}^2A_x)(X,Y)}{Z}=\s{S_x^{i,j}(Z) X}{Y}.
    \end{align*}
    Therefore, we can rewrite the second derivative as
    \begin{align*}
        Q_u(w)=\int_{B(0,1)}\s{\mathscr{L}_u(w)}{w}\,dx,
    \end{align*}
    where $\mathscr{L}_u$ is the elliptic, fourth-order, \emph{self-adjoint operator}
    \begin{align*}
        \mathscr{L}_u&=\Delta^2+\sum_{i=1}^2\,S_u^{i,i}(A_u(\D u,\D u))+2\sum_{i=1}^2S_u^i(A_u(\D u,\D u))\p{x_i}(\,\cdot\,)-2\dive\left(S_u(A_u(\D u,\D u))\D(\,\cdot\,)\right)\\
        &+S_u(A_u(\D u,\D u))\Delta(\,\cdot\,)+\dive\left(S_u(A_u(\D u,\D u)\D(\,\cdot\,)\right)\\
        &-2\sum_{i=1}^2S_u^i(\s{\Delta u}{\p{x_i}P_u}-4\sum_{i=1}^2S_u\left(\s{\s{\Delta u}{\p{x_i}P_u}}\right)\p{x_i}(\,\cdot\,)+S_u(\s{\Delta u}{\Delta (P_u)}.
    \end{align*}
    In particular, if 
    \begin{align*}
        \mathscr{L}_{\alpha,k}=\omega_{\alpha^{-1}\rho_k,\alpha}^{-1}\mathscr{L}_{u_k},
    \end{align*}
    we have
    \begin{align*}
        Q_{u_k}(w)=\int_{B(0,1)}\s{w}{\mathscr{L}_{\alpha,k}w}\omega_{\alpha^{-1}\rho_k,\alpha}\,dx.
    \end{align*}
    We can therefore define for all $\lambda\in \R$ the weighted associated eigenspace as
    \begin{align*}
        \mathscr{E}_{\alpha,k}(\lambda)=W^{2,2}_0(B(0,1))\cap\ens{u:\mathscr{L}_{\alpha,k}u=\lambda\,u}.
    \end{align*}
    We take function that vanish on the boundary since in applications, one would consider closed (compact with boundary) manifolds, on which one can integrate by parts without creating boundary components. Alternatively, one could directly work on a closed manifold, but that would not bring us anything new from the analysis viewpoint for we would consider a sequence from a \emph{fixed} $4$-dimensional manifold (there is no equivalent to Deligne-Mumford compactification in dimension $4$).
    Then, following the proof of \cite[Lemma IV.3]{riviere_morse_scs}, we obtain the following result.
    \begin{lemme}\label{lemme_IV.3}
        For all $k\in \N$ and $0<\alpha<1$, we have 
        \begin{align}
            \mathrm{Ind}_{E}(u_k)=\dim\bigoplus_{\lambda<0}\mathscr{E}_{\alpha,k}(\lambda).
        \end{align}
    \end{lemme}
    The fundamental lemma that allows us to take weak limits of an appropriately normalised sequence of eigenfunctions is the following (see \cite[Lemma IV.4]{riviere_morse_scs}).
    \begin{lemme}
        There exists $\alpha_0>0$ and $0<\mu_0<\infty$, and for all $0<\alpha<\alpha_0$, there exists a sequence $\ens{\mu_{\alpha,k}}_{k\in \N}\subset (0,\mu_0)$ such that
        \begin{align*}
            \lim_{\alpha\rightarrow 0}\limsup_{k\rightarrow \infty}\mu_{\alpha,k}=0.
        \end{align*}
        Furthermore, for all $\lambda\in \R$,
        \begin{align*}
            \mathrm{dim}\mathscr{E}_{\alpha,k}(\lambda)>0\implies \lambda\geq -\mu_{\alpha,k}\geq -\mu_0.
        \end{align*}        
    \end{lemme}
    \begin{proof}
        Thanks to the pointwise bound of Theorem \ref{pointwise_biharmonic}, we have for all $0<\alpha<\alpha_0$ and $k\in \N$ large enough 
        \begin{align*}
            &|\D^2u_k(x)|^2\leq C\left(\np{\D^2u_k}{2,1}{\Omega_k(2\alpha)}+\np{\D u_k}{4,1}{\Omega_k(\alpha)}\right)\,\omega_{\alpha^{-1}\rho_k,\alpha}\qquad\text{for all}\;\,x\in \Omega_k(\alpha)\\
            &|\D u_k(x)|^4\leq C\left(\np{\D^2u_k}{2,1}{\Omega_k(2\alpha)}+\np{\D u_k}{4,1}{\Omega_k(\alpha)}\right)\,\omega_{\alpha^{-1}\rho_k,\alpha}\qquad\text{for all}\;\,x\in \Omega_k(\alpha)
        \end{align*}
        In particular, we have
        \begin{align*}
            \lim_{\alpha\rightarrow 0}\limsup_{k\rightarrow\infty}\left(\np{\frac{|\D^2u_k|^2}{\omega_{\alpha^{-1}\rho_k,\alpha}}}{\infty}{\Omega_k(\alpha)}+\np{\frac{|\D u_k|^4}{\omega_{\alpha^{-1}\rho_k,\alpha}}}{\infty}{\Omega_k(\alpha)}\right)=0.
        \end{align*}
        Likewise, using the strong convergence of $\ens{u_k}_{k\in \N}$ towards $u_{\infty}$ in $B_1\setminus\bar{B}_{\alpha}(0)$ and the strong convergence in bubble domains, we show as in \cite{riviere_morse_scs} that 
        \begin{align*}
            \lim_{\alpha\rightarrow 0}\limsup_{k\rightarrow\infty}\left(\np{\frac{|\D^2u_k|^2}{\omega_{\alpha^{-1}\rho_k,\alpha}}}{\infty}{B(0,1)}+\np{\frac{|\D u_k|^4}{\omega_{\alpha^{-1}\rho_k,\alpha}}}{\infty}{B(0,1)}\right)=0.
        \end{align*}
        Therefore, letting
        \begin{align*}
            \mu_{\alpha,k}=\np{\frac{|\D^2u_k|^2}{\omega_{\alpha^{-1}\rho_k,\alpha}}}{\infty}{B(0,1)}+\np{\frac{|\D u_k|^4}{\omega_{\alpha^{-1}\rho_k,\alpha}}}{\infty}{B(0,1)},
        \end{align*}
        we have in particular
        \begin{align*}
            |\D^2 u_k|^2+|\D u_k|^4\leq \mu_{\alpha,k}\,\omega_{\alpha^{-1}\rho_k,\alpha} 
        \end{align*}
        Therefore, if $w\in \mathscr{E}_{\alpha,k}(\lambda)$ and $\lambda<0$, we get
        \begin{align*}
            Q_{u_k}(w)=\int_{B(0,1)}\s{w}{\mathscr{L}_{\alpha,k}w}\omega_{\alpha^{-1}\rho_k,\alpha}=\lambda\int_{B(0,1)}|w|^2\omega_{\alpha^{-1}\rho_k,\alpha}dx,
        \end{align*}
        while by \eqref{elementary_d2_lower_bound} and \eqref{neck_positive} (applied to $B(0,1)$ instead, using the previous upper bound)
        \begin{align*}
            Q_{u_k}(w)\geq \int_{B(0,1)}\frac{|\D w|^2}{|x|^2}dx-\mu_{\alpha,k}\int_{B(0,1)}|w|^2\omega_{\alpha^{-1}\rho_k,\alpha}\,dx\geq -\mu_{\alpha,k}\int_{B(0,1)}|w|^2\omega_{\alpha^{-1}\rho_k,\alpha}\,dx,
        \end{align*}
        which shows that $\lambda\geq -\mu_{\alpha,k}$.
    \end{proof}
    We also need the following elementary generalisation of \cite[Lemma 1.4.5]{morse_willmore_I} (see also \cite[Lemma B.1]{riviere_morse_scs}). 
    \begin{lemme}
        Let $u:B(0,1)\rightarrow (M^m,h)\subset \R^n$ be a smooth extrinsic biharmonic map, and let $\omega\in C^{\infty}(B(0,1)\setminus\ens{0})$ such that for some $0<\alpha<1$
        \begin{align*}
            0<\omega_0\leq \omega(x)\leq \frac{C_0}{|x|^{4(1-\beta)}}\quad \text{and}\quad |\D^l\omega(x)|\leq \frac{C_l}{|x|^{4(1-\beta)}+l}\qquad\text{for all}\;\, l\in \N,
        \end{align*}
        where $\omega_0,\ens{C_l}_{l\in \N}\subset (0,\infty)$. Denote
        \begin{align*}
            L^2_{\omega}=L^2(B(0,1))\cap \ens{u:\int_{B(0,1)}u^2\omega\,dx<\infty}.
        \end{align*}
        Consider the operator $\mathscr{L}_{\omega}=\omega\mathscr{L}_u$, where $\mathscr{L}_u=\Delta^2+\cdots$ is the fourth-order elliptic operator such that
        \begin{align*}
            Q_u(w)=\int_{B(0,1)}\s{w}{\mathscr{L}_uw}dx
        \end{align*}
        for all $W^{2,2}$ variation $w:B(0,1)\rightarrow \R^n$. Then, there exists a Hilbertian base of $L^2_{\omega}(B(0,1))$ made of eigenvectors of $\mathscr{L}_{\omega}$ whose eigenvalues satisfy
        \begin{align*}
            \lambda_1\leq \lambda_2\leq \cdots \lambda_k\conv{k\rightarrow \infty}\infty.
        \end{align*}
    \end{lemme}
    Finally, if 
    \begin{align*}
        \mathscr{E}_{\alpha,\infty}^0=\bigoplus_{\lambda\leq 0}\mathscr{E}_{\alpha,\infty}(\lambda),
    \end{align*}
    where 
    \begin{align*}
        \mathscr{E}_{\alpha,\infty}(\lambda)=W^{2,2}_0(B(0,1))\cap\ens{u:\omega_{0,\alpha}^{-1}\mathscr{L}_{u_{\infty}}w=\lambda\,w},
    \end{align*}
    we have
    \begin{align*}
        \mathrm{dim}\left(\mathscr{E}_{\alpha,\infty}\right)\leq \mathrm{Ind}_{E}(u_{\infty})+\mathrm{Null}_{E}(u_{\infty}).
    \end{align*}
    We can finally move to the proof of the main theorem \ref{th:Main}. Let us recall it for the convenience of the reader.
    \begin{theorem}\label{th:Main2}
        Let $\ens{u_k}_{k\in \N}:B(0,1)\rightarrow (M^m,h)$ be a sequence of biharmonic maps such that
        \begin{align*}
            \limsup_{k\rightarrow \infty}E(u_k)<\infty.
        \end{align*}
        Then, if $\ens{u_k}_{k\in\N}$ bubble converges towards $(u_{\infty},v_1,\cdots,v_N)$, we have
        \begin{align*}
            \mathrm{Ind}_E(u_{\infty})+\sum_{i=1}^N\mathrm{Ind}_E(v_i)\leq \liminf_{k\rightarrow \infty}\mathrm{Ind}_E(u_k)\leq \limsup_{k\rightarrow \infty}\mathrm{Ind}_E^0(u_k)\leq \mathrm{Ind}_E^0(u_{\infty})+\sum_{i=1}^N\mathrm{Ind}_E^0(v_i),
        \end{align*}
        where $\mathrm{Ind}_E^0=\mathrm{Ind}+\mathrm{Null}$.
    \end{theorem}
    \begin{proof}
    We only treat the case $N=1$ where there is a single bubble. Up to translation, we assume that $0\in B(0,1)$ is the point of concentration. As in \cite[Lemma IV.6]{riviere_morse_scs}, introduce the following infinite-dimensional sphere:
        \begin{align*}
            \mathscr{S}_{\alpha,k}=\bigoplus_{\lambda\leq 0}\mathscr{E}_{\alpha,k}(\lambda)\cap\ens{w:\int_{B(0,1)}|w|^2\omega_{\alpha^{-1}\rho_k,\alpha}dx=1}.
        \end{align*}
        Let $w_k\in \mathscr{S}_{\alpha,k}$ and $\lambda_k\in \R$ be such that
        \begin{align*}
            \leb_{\alpha,k}w_k=\lambda_k\,w_k.
        \end{align*}        
        For $k$ large enough, we have 
        \begin{align*}
            Q_{u_k}(w_k)&\geq \frac{1}{2}\int_{B(0,1)}\left(|\Delta w_k|^2+\frac{|\D w_k|^2}{|x|^2}\right)dx-\mu_{\alpha,k}\int_{B(0,1)}|w_k|^2\omega_{\alpha^{-1}\rho_k,\alpha}\,dx\\
            &\geq \frac{1}{2}\int_{B(0,1)}\left(|\Delta w_k|^2+\frac{|\D w_k|^2}{|x|^2}\right)dx-\mu_{\alpha,k}\int_{B(0,1)}|w_k|^2\omega_{\alpha^{-1}\rho_k,\alpha}\,dx.
        \end{align*}
        On the other hand, we have
        \begin{align*}
            Q_{u_k}(w_k)=\lambda_k\int_{B(0,1)}|w_k|^2\omega_{\alpha^{-1}\rho_k,\alpha}\,dx=\lambda_k\geq -\mu_0,
        \end{align*}
        which shows that 
        \begin{align*}
            \int_{B(0,1)}\left(|\Delta w_k|^2+\frac{|\D w_k|^2}{|x|^2}\right)dx\leq \mu_0.
        \end{align*}
        Therefore, $\ens{w_k}_{k\in \N}$ is bounded in $W^{2,2}_0(B(0,1))$, which implies that up to a subsequence, 
        \begin{align*}
            w_k\underset{k\rightarrow \infty}{\hooklongrightarrow}w_{\infty}\quad  \text{in}\;\, W^{2,2}_0(B(0,1))\qquad \text{and}\qquad u_k(\delta_kx) \underset{k\rightarrow \infty}{\hooklongrightarrow}v_{\infty}\quad  \text{in}\;\, W^{2,2}_0(B(0,1)).
        \end{align*}
        Thanks to our stability inequality \ref{ineq:stability_ineq}, as in \cite{morse_willmore_I}, the same cutoff argument shows that we have either $u_{\infty}\neq 0$ or $v_{\infty}\neq 0$. The rest of the proof is identical to the proofs of \cite[Lemma IV.6]{riviere_morse_scs} and \cite[p. 78]{morse_willmore_I} and we omit it. 
    \end{proof}

    \section{Appendix}

     \subsection{Basic Properties of Lorentz Spaces}

     Most of the presentation here overlaps with \cite{rivnotes}. Another useful reference is \cite{helein}.

     Let $(X,\mu)$ be a fixed measured space, and fix a measurable function $f:X\rightarrow\R$. Define
      \begin{align*}
     f_{\ast}(t)&=\inf\ens{\lambda>0: \mu(X\cap\ens{x:|f(x)|>\lambda}\leq t}\\
     f_{\ast\ast}(t)&=\frac{1}{t}\int_{0}^tf_{\ast}(s)ds
 \end{align*}
 and for all $1<p<\infty$ and $1\leq q<\infty$, we define
 \begin{align*}
     \snp{f}{p,q}{X}&=\left(\int_{0}^{\infty}t^{\frac{q}{p}}f_{\ast}^q(t)\frac{dt}{t}\right)^{\frac{1}{q}}\\
     \np{f}{p,q}{X}&=\left(\int_{0}^{\infty}t^{\frac{q}{p}}f_{\ast\ast}^q(t)\frac{dt}{t}\right)^{\frac{1}{q}},
 \end{align*}
 where for $q=\infty$, we define instead
 \begin{align*}
     &\snp{f}{p,\infty}{X}=\sup_{t>0}\left(t^pf_{\ast}(t)\right)^{\frac{1}{p}}\\
     &\np{f}{p,\infty}{X}=\sup_{t>0}\left(t^pf_{\ast\ast}(t)\right)^{\frac{1}{p}}.
 \end{align*}
 The basic property in the theory of Lorentz spaces is that $\np{\,\cdot\,}{p,q}{X}$ is a norm for all $1<p<\infty$ and for all $1\leq q\leq \infty$, while $\snp{\,\cdot\,}{p,q}{X}$ is only a semi-norm in general. However, it is usually easier to work with $\snp{\,\cdot\,}{p,q}{X}$, and that is why this semi-norm will keep appearing in the proofs. Let us recall that if $q=1$, then $\snp{\,\cdot\,}{p,1}{X}$ is a norm and we have (see the appendix of \cite{pointwise}) 
 \begin{align}
     \np{f}{p,1}{X}=\frac{p}{p-1}|f|_{\mathrm{L}^{p,1}(X)}=\frac{p^2}{p-1}\int_{0}^{\infty}\mu\left(X\cap\ens{x:|f(x)|>t}\right)^{\frac{1}{p}}dt.
 \end{align}
 Likewise, we have $L^{p,p}(X)=L^{p}(X)$, and $\snp{f}{p,p}{X}=\np{f}{p}{X}$.
 Furthermore, using Fubini's theorem, it is easy to show that for all $1<p<\infty$ and for all $1\leq q<\infty$, we have
 \begin{align*}
     \snp{f}{p,q}{X}=p^{\frac{1}{q}}\left(\int_{0}^{\infty}t^q\mu\left(X\cap\ens{x:|f(x)|>t}\right)^{\frac{q}{p}}\frac{dt}{t}\right)^{\frac{1}{q}}.
 \end{align*}
 Finally, the two quantities can be estimated as follows:
 \begin{align}\label{comp_lorentz_norms}
     \snp{f}{p,q}{X}\leq \np{f}{p,q}{X}\leq \frac{p}{p-1}\snp{f}{p,q}{X}.
 \end{align}
 \begin{prop}\label{comp_lorentz_norm}
     For all $1<p<\infty$ and for all $1\leq q\leq r\leq \infty$, we have a continuous embedding $L^{p,r}(X)\longhookrightarrow L^{p,q}(X)$, and for all $f\in L^{p,q}(X)$, the following inequality holds:
     \begin{align}\label{comp_lorentz_norm_ineq}
         \np{f}{p,r}{X}\leq \left(\frac{p}{q}\right)^{\frac{1}{q}-\frac{1}{r}}\np{f}{p,q}{X}.
     \end{align}
 \end{prop}
 \begin{proof}
     We have by \eqref{comp_lorentz_norms}
     \begin{align*}
         t^{\frac{1}{p}}f_{\ast}(t)=\left(\frac{q}{p}\int_{0}^ts^{\frac{q}{p}}f_{\ast}^q(t)\frac{ds}{s}\right)^{\frac{1}{q}}\leq \left(\frac{p}{q}\right)^{\frac{1}{q}}\left(\int_{0}^ts^{\frac{q}{p}}f_{\ast}^q(s)\frac{ds}{s}\right)^{\frac{1}{q}}\leq \left(\frac{p}{q}\right)^{\frac{1}{q}}\snp{f}{p,q}{X}\leq \left(\frac{p}{q}\right)^{\frac{1}{q}}\np{f}{p,q}{X},
     \end{align*}
     which shows that
     \begin{align*}
         \snp{f}{p,q}{X}\leq \left(\frac{p}{q}\right)^{\frac{1}{q}}\np{f}{p,q}{X},
     \end{align*}
     and
     \begin{align*}
         \np{f}{p,\infty}{X}\leq \frac{p}{p-1}\left(\frac{p}{q}\right)^{\frac{1}{q}}\np{f}{p,q}{X}.
     \end{align*}
     Now, assume that $1\leq q<\infty$. Then, we have by an immediate interpolation
     \begin{align*}
         \snp{f}{p,r}{X}\leq \snp{f}{p,q}{X}^{\frac{q}{r}}\snp{f}{p,\infty}{X}^{1-\frac{q}{r}}\leq \left(\frac{p}{q}\right)^{\frac{1}{q}-\frac{1}{r}}\snp{f}{p,q}{X}\leq  \left(\frac{p}{q}\right)^{\frac{1}{q}-\frac{1}{r}}\np{f}{p,q}{X},
     \end{align*}
     which concludes the proof of the Proposition by another application of \eqref{comp_lorentz_norms}.
 \end{proof}

    \subsection{An Averaging Lemma for Lorentz Spaces}

 Let $d\geq 2$ and $0<a<b<\infty$, and define $I=[a,b]$ and $\Omega=B_b\setminus\bar{B}_a(0)\subset \R^d$. For all $f\in L^2(\Omega)$ define $\bar{f}:[a,b]\rightarrow \R$ by
 \begin{align}
     f(r)=\np{f}{2}{\partial B(0,r)}=\left(\int_{\partial B(0,r)}|f|^2\,d\mathscr{H}^{d-1}\right)^{\frac{1}{2}}.
 \end{align}
 Thanks to the co-area formula, we have
 \begin{align*}
     \int_{a}^b|\bar{f}(r)|^2d\leb^1(r)=\int_{\Omega}|f(x)|^2d\leb^d(x).
 \end{align*}
 We goal of this section is to generalise this estimate to the $L^{2,1}$ norm. 

 \begin{lemme}\label{averaging_l21}
     Let $d\geq 2$ and $0< a<b<\infty$, and define $I=[a,b]$ and $\Omega=B_b\setminus\bar{B}_a(0)\subset \R^d$. For all $f\in L^2(\Omega)$ define $\bar{f}:[a,b]\rightarrow \R$ for all $a\leq r\leq b$ by
     \begin{align}
     \bar{f}(r)=\np{f}{2}{\partial B(0,r)}=\left(\int_{\partial B(0,r)}|f|^2\,d\mathscr{H}^{d-1}\right)^{\frac{1}{2}}.
     \end{align}
     Then, provided that $f\in L^{2,1}(\Omega)$, we have $\bar{f}\in L^{2,1}([a,b])$ and
     \begin{align}\label{l21_bound}
         \np{\bar{f}}{2,1}{I}\leq 2^{\frac{d}{4}}\sqrt{\beta(d)}\np{f}{2,1}{\Omega},
     \end{align}
     where $\beta(d)=\mathscr{H}^{d-1}(S^{d-1})=\dfrac{2\pi^{\frac{d}{2}}}{\Gamma\left(\frac{d}{2}\right)}$.
 \end{lemme}
 \begin{rem}
     If $c_d$ is the constant of the inequality \eqref{l21_bound}, we deduce in particular that
     \begin{align*}
     \left\{\begin{alignedat}{1}
         c_2&=2\sqrt{\pi}\\
         c_3&=2^{\frac{7}{4}}\sqrt{\pi}\\
         c_4&=2\pi\sqrt{2}.
         \end{alignedat}\right.
     \end{align*}
 \end{rem}
 \begin{proof}
 Assume that $f=c\,\mathbf{1}_A$ for a measurable set $A$. Then, we trivially have
 \begin{align*}
     f_{\ast}(t)=c\,\mathbf{1}_{[0,\leb^n(A)]}(t),
 \end{align*}
 which yields
 \begin{align*}
     \np{f}{2,1}{\Omega}=2\int_{0}^{\infty}\sqrt{t}f_{\ast}(t)\frac{dt}{t}=2\,c\int_{0}^{\leb^n(A)}\sqrt{t}\frac{dt}{t}=4\,c\sqrt{\leb^n(A)}.
 \end{align*}
 On the other hand, we have
 \begin{align*}
     \leb^n\left([a,b]\cap\ens{r:\bar{f}(r)>\lambda}\right)=\leb^n\left([a,b]\cap\ens{r:c\sqrt{\mathscr{H}^{d-1}\left(A\cap \partial B(0,r)\right)}>\lambda}\right).
 \end{align*}
 Therefore, $\leb^n\left([a,b]\cap\ens{r:\bar{f}(r)>\lambda}\right)>t$ implies by the co-area formula that
 \begin{align*}
     \leb^n(A)=\int_{a}^b\mathscr{H}^{d-1}\left(A\cap \partial B(0,r)\right)\,dr\geq \left(\frac{\lambda}{c}\right)^2\,t.
 \end{align*}
 Therefore, we deduce the crude estimate
 \begin{align}\label{ineq_1}
     \bar{f}_{\ast}(t)\leq c\sqrt{\frac{\leb^n(A)}{t}}\mathrm{1}_{[0,\leb^1(I)]}(t).
 \end{align}
 In order to refine this inequality, we restict to a simpler class of sets: balls. If $A=B(p,R)\subset \Omega=B_b\setminus\bar{B}_a(0)$, then 
 \begin{align*}
     \mathscr{H}^{d-1}(\partial B(0,r)\cap A)\leq \beta(d)(2R)^{d-1}\mathbf{1}_{[|p|-R,|p|+R]}(r)
 \end{align*}
 where $\beta(d)=\mathscr{H}^{d-1}(S^{d-1})$. Therefore, we have
 \begin{align*}
     \leb^1\left([a,b]\cap\ens{r:\bar{f}(r)>\lambda}\right)\leq \left\{\begin{alignedat}{2}
         &2R\qquad&& \text{for all}\;\, \lambda<c\sqrt{\beta(d)(2R)^{d-1}}\\
         &0\qquad&& \text{for all}\;\, \lambda\geq c\sqrt{\beta(d)(2R)^{d-1}}.
     \end{alignedat}\right.
 \end{align*}
 In other words, we have 
 \begin{align*}
     \bar{f}_{\ast}(t)\leq c\sqrt{\beta(d)(2R)^{d-1}}\mathbf{1}_{[0,2R]}(t).
 \end{align*}
 Therefore, we get
 \begin{align*}
     \np{\bar{f}}{2,1}{I}\leq 2\int_{0}^{2R}c\sqrt{\beta(d)(2R)^{d-1}}\sqrt{t}\frac{dt}{t}=4\,c\sqrt{\beta(d)(2R)^{d-1}}\sqrt{2R}=2^{\frac{d}{2}}\,4\,c\sqrt{\leb^n(A)}=2^{\frac{d}{2}}\np{f}{2}{\Omega}.
 \end{align*}
 Using both inequalities, we deduce that 
 \begin{align*}
     \bar{f}_{\ast}(t)&\leq \min\ens{c\sqrt{\beta(d)(2R)^{d-1}},c\sqrt{\frac{\beta(d)R^d}{t}}}\mathrm{1}_{[0,2R]}(t)\\
     &=\left\{\begin{alignedat}{2}
         &c\sqrt{\beta(d)(2R)^{d-1}}\qquad&& \text{for all}\;\, 0\leq t\leq \sqrt{\frac{R}{2^{d-1}}}\\
         &c\sqrt{\frac{\beta(d)R^d}{t}}\qquad&&\text{for all}\;\, \sqrt{\frac{R}{2^{d-1}}}<t<2R\\
         &0\qquad&&\text{for all}\;\, t\geq 2R
     \end{alignedat}\right.
 \end{align*}
 In general, write 
 \begin{align*}
     f=\sum_{i=1}^nc_i\,\mathbf{1}_{A_i},
 \end{align*}
 where without loss of generality, we assume that $0<c_1<c_2<\cdots<c_n<\infty$ and $A_1,\cdots,A_n$ are pairwise disjoint sets. If $c_0=0$, then for all $1\leq i\leq n$, we have
 \begin{align*}
     \lambda_f(t)=\leb^n(\Omega\cap\ens{x:|f(x)|>t})=\sum_{j=i}^{n}\leb^n(A_j)\qquad \text{for all}\;\, c_{i-1}\leq t<c_{i}.
 \end{align*}
 In particular, we have
 \begin{align}
     \np{f}{2,1}{\Omega}=4\int_{0}^{\infty}\sqrt{\lambda_f(t)}\,dt=4\sum_{i=1}^n(c_i-c_{i-1})\sqrt{\sum_{j=i}^n\leb^n(A_j)}.
 \end{align}
 On the other hand, we have
 \begin{align}\label{decreasing_rearrangement}
     f_{\ast}(t)=\left\{\begin{alignedat}{2}
         &c_n\qquad&&\text{for all}\;\, 0<t<\leb^n(A_n)\\
         &c_{n-1}\qquad&&\text{for all}\;\, \leb^n(A_n)\leq t<\leb^n(A_n)+\leb^n(A_{n-1})\\
         &\vdots\\
         &c_1\qquad&&\text{for all}\;\, \sum_{i=2}^n\leb^n(A_i)\leq t<\sum_{i=1}^n\leb^n(A_i)\\
         &0\qquad&&\text{for all}\;\, t\geq \sum_{i=1}^{n}\leb^n(A_i).
     \end{alignedat}\right.
 \end{align}
 Therefore, we recover
 \begin{align}\label{f_l21}
     \np{f}{2,1}{\Omega}&=2\int_{0}^{\infty}\sqrt{t}\,f_{\ast}(t)\frac{dt}{t}=4c_n\sqrt{\leb^n(A_n)}+4c_{n-1}\left(\sqrt{\leb^n(A_n)+\leb^n(A_{n-1})}-\sqrt{\leb^n(A_n)}\right)\nonumber\\
     &+\cdots+4c_1\left(\sqrt{\sum_{i=1}^{n}\leb^n(A_i)}-\sqrt{\sum_{i=2}^{n}\leb^n(A_i)}\right)\nonumber\\
     &=4\sum_{i=1}^nc_i\left(\sqrt{\sum_{j=i}^{n}\leb^n(A_j)}-\sqrt{\sum_{j=i+1}^n\leb^n(A_j)}\right)=4\sum_{i=1}^n(c_i-c_{i-1})\sqrt{\sum_{j=i}^n\leb^n(A_j)}.
 \end{align}
 On the other hand, the previous argument with the co-area formula shows that
 \begin{align*}
     \bar{f}_{\ast}(t)\leq \frac{1}{\sqrt{t}}\sqrt{\sum_{i=1}^nc_i^2\leb^n(A_i)}\,\mathbf{1}_{[0,\leb^1(I)]}(t).
 \end{align*}
  Now, let us show by induction that for all $n\geq 1$
 \begin{align}\label{ineq_fund}
     \sum_{i=1}^n(c_i-c_{i-1})\sqrt{\sum_{j=1}^n|A_j|}\geq \sqrt{\sum_{i=1}^nc_i^2|A_i|},
 \end{align}
 where we wrote for simplicity of notation $|A_i|=\leb^n(A_i)$. Notice that it is equivalent to
 \begin{align}
     \sum_{i=1}^nc_i\left(\sqrt{\sum_{j=i}^n|A_i|}-\sqrt{\sum_{j=i+1}^n|A_j|}\right)\geq \sqrt{\sum_{i=1}^nc_i^2|A_i|}.
 \end{align}
 
 First, notice that the elementary inequality
 \begin{align*}
     c_j=\sum_{i=1}^j(c_i-c_{i-1})\leq \sqrt{j}\left(\sum_{i=1}^{j}(c_{i}-c_{i-1})^2\right)^{\frac{1}{2}}
 \end{align*}
 shows that
 \begin{align*}
     \sum_{i=1}^n(c_i-c_{i-1})^2\sum_{j=1}^n|A_j|=\sum_{j=1}^n|A_j|\sum_{i=1}^j(c_{i}-c_{i-1})^2\geq \sum_{j=1}^n\frac{c_j^2}{j}|A_j|
 \end{align*}
 which does not suffice for our purpose. The case $n=1$ is trivial, so let us treat the cases $n=2$ and $n=3$. Let us simplify the notations further by writing $D_i=|A_i|$

 For $n=2$, we have
 \begin{align*}
     \left(\sum_{i=1}^2(c_i-c_{i-1})\sqrt{\sum_{j=i}^2D_j}\right)^2&=\left(c_1\sqrt{D_1+D_2}+(c_2-c_1)\sqrt{D_2}\right)^2\\
     &=c_1^2(D_1+D_2)+(c_1^2+c_2^2-2c_2c_2)D_2+2c_1(c_2-c_1)\sqrt{D_2}\sqrt{D_1+D_2}\\
     &= c_1^2(D_1+2D_2)+c_2^2D_2-2c_1c_2D_2+2c_1(c_2-c_1)\sqrt{D_2}\sqrt{D_1+D_2}\\
     &\geq c_1^2(D_1+\colorcancel{2D_2}{red})+c_2^2D_2-\colorcancel{2c_1c_2D_2}{blue}+2c_1(\colorcancel{c_2}{blue}-\colorcancel{c_1}{red})D_2\\
     &=c_1^2D_1+c_2^2D_2.
 \end{align*}
 For $n=3$, we get
 \begin{align*}
     &\left(\sum_{i=1}^3(c_i-c_{i-1})\sqrt{\sum_{j=i}^3D_j}\right)^2=\left(c_1\sqrt{D_1+D_2+D_3}+(c_2-c_1)\sqrt{D_2+D_3}+(c_3-c_2)\sqrt{D_3}\right)^2\\
     &= c_1^2(D_1+D_2+D_3)+(c_1^2+c_2^2-2c_2c_2)(D_2+D_3)+(c_2^2+c_3^2-2c_2c_3)D_3\\
     &+2c_1(c_2-c_1)\sqrt{D_2+D_3}\sqrt{D_1+D_2+D_3}+2c_1(c_3-c_2)\sqrt{D_3}\sqrt{D_1+D_2+D_3}\\
     &+2(c_2-c_1)(c_3-c_2)\sqrt{D_3}\sqrt{D_2+D_3}\\
     &=c_1^2(D_1+2D_2+2D_3)+c_2^2(D_2+2D_3)+c_3^2D_3-2c_1c_2(D_2+D_3)-2c_2c_3D_3\\
     &+2c_1(c_2-c_1)\sqrt{D_2+D_3}\sqrt{D_1+D_2+D_3}+2c_1(c_3-c_2)\sqrt{D_3}\sqrt{D_1+D_2+D_3}\\
     &+2(c_2-c_1)(c_3-c_2)\sqrt{D_3}\sqrt{D_2+D_3}\\
     &\geq c_1^2(D_1+\colorcancel{2D_2+2D_3}{blue})+c_2^2(D_2+\colorcancel{2D_3}{red})+c_3^2D_3-2c_1c_2(D_2+D_3)-2c_2c_3D_3\\
     &+2c_1(c_2-\colorcancel{c_1}{blue})(D_2+D_3)+\colorcancel{2c_1(c_3-c_2)D_3}{orange}+2(c_2-\colorcancel{c_1}{orange})(c_3-\colorcancel{c_2}{red})D_3=c_1^2D_1+c_2^2D_2+c_3^2D_3.
 \end{align*}
 In general, we have
 \begin{align*}
     &\left(\sum_{i=1}^n(c_i-c_{i-1})\sqrt{\sum_{j=i}^nD_j}\right)^2=\sum_{i=1}^n(c_{i-1}^2+c_{i}^2-2c_{i-1}c_{i})^2\sum_{j=i}^nD_j\\
     &+2\sum_{i=1}^n(c_{i}-c_{i-1})\sqrt{\sum_{j=i}^nD_j}\sum_{k=i+1}^n(c_{k}-c_{k-1})\sqrt{\sum_{l=k+1}^nD_l}\\
     &=\sum_{i=1}^nc_i^2\left(D_i+2\sum_{j=i+1}^nD_j\right)-2\sum_{i=2}^nc_{i-1}c_i\sum_{j=i}^nD_j
     +2\sum_{i=1}^n(c_{i}-c_{i-1})\sqrt{\sum_{j=i}^nD_j}\sum_{k=i+1}^n(c_{k}-c_{k-1})\sqrt{\sum_{l=k}^nD_l}\\
     &\geq \sum_{i=1}^nc_i^2\left(D_i+2\sum_{j=i+1}^nD_j\right)-2\sum_{i=2}^nc_{i-1}c_i\sum_{j=i}^nD_j
     +2\sum_{i=1}^n(c_i-c_{i-1})\sum_{k=i+1}^n(c_k-c_{k-1})\sum_{l=k}^nD_l\\
     &\geq \sum_{i=1}^nc_i^2\left(D_i+2\sum_{j=i+1}^nD_j\right)-2\sum_{i=2}^nc_{i-1}c_i\sum_{j=i}^nD_j
     +2\sum_{i=1}^n(c_i-c_{i-1})(c_{i+1}-c_{i})\sum_{l=i+1}^nD_l\\
     &=\sum_{i=1}^nc_i^2D_i,
 \end{align*}
 which concludes the proof of the inequality \eqref{ineq_fund}. Notice that using that $\dfrac{1}{\sqrt{t}}\in L^{2,\infty}([a,b])$, this bound only furnishes up the trivial inequality
 \begin{align*}
     \np{\bar{f}}{2,\infty}{I}\leq \np{f}{2,1}{\Omega},
 \end{align*}
 and we will have to refine our bound.
 
 To simplify calculations (the general case seems to be a combinatorial nightmare), we will make two further reductions reductions. Notice that for all measurable function $f\in L^1(\Omega)$, if for all $n\in \N$, $Q_1,\cdots,Q_{N_n}$ are the open, disjoint, dyadic cubes with sides parallel to the coordinate axes, congruent to $[-\frac{1}{2n},\frac{1}{2n}]^d$ such that $Q_i\subset \Omega$ chosen such that
 \begin{align*}
     \frac{n^d\leb^d(\Omega)}{N_n}\conv{n\rightarrow \infty}1,
 \end{align*}
 defining
 \begin{align*}
     f_n=\sum_{i=1}^{N_n}\left(\dashint{Q_i}f\,d\leb^d\right)\mathbf{1}_{Q_i},
 \end{align*}
 we have
 \begin{align*}
     f_n\conv{n\rightarrow \infty}f\qquad \leb^d\;\, \text{almost everywhere.}
 \end{align*}
 In particular, if $f$ is bounded (using that $\Omega$ has finite measure), by the dominated convergence theorem of Lebesgue, we deduce that for all $1<p<\infty$ and $1\leq q<\infty$
 \begin{align*}
     \np{f_n}{p,q}{\Omega}\conv{n\rightarrow \infty}\np{f}{p,q}{\Omega}.
 \end{align*}
 Furthermore, by the embedding $L^{p,q}(\Omega)\hookrightarrow L^{p,\infty}(\Omega)$, we deduce that
 \begin{align*}
     \leb^4(\Omega\cap \ens{x:|f(x)|>A})\leq c_{p,q}\frac{\np{f}{p,q}{\Omega}}{A^p}\conv{A\rightarrow \infty}0,
 \end{align*}
 which implies that for all $t>0$
 \begin{align*}
     (f-f\mathbf{1}_A)_{\ast}(t)\leq \leb^4(\Omega\cap \ens{x:|f(x)|>A})\conv{A\rightarrow \infty}0,
 \end{align*}
 and since $(f-f\mathbf{1}_A)_{\ast}(t)\leq f_{\ast}(t)$, another application of Lebesgue's dominated convergence theorem shows that $L^{p,q}\cap L^{\infty}(\Omega)$ is dense in $L^{p,q}(\Omega)$, and finally, that it suffices to show our inequality for $f_n$ as above by Fatou's lemma for Lorentz spaces. Notice that alternatively, since $f_n\conv{n\rightarrow \infty}f$ in all $L^p$ space for $1\leq p<\infty$ provided that $f_n\in L^p\cap L^{\infty}$, interpolation theorem (here, simply Hölder's inequality for Lorentz spaces) shows that $f_n\conv{n\rightarrow \infty}f$ in $L^{p,q}$ for all $1<p<\infty$ and $1\leq q\leq \infty$, where $f\in L^{p,q}\cap L^{\infty}$. 
 
 Therefore, the situation is reduced to the case where $A_1,\cdots,A_n$ are dyadic cubes of equal length $R>0$. Notice that for all $a<r<b$, since the cubes are disjoint, we have
 \begin{align*}
     \bar{f}(r)=\sqrt{\sum_{i=1}^{n}c_i^2\mathscr{H}^{d-1}(\partial B(0,r)\cap A_i)}
 \end{align*}
 In particular, for all $t>0$, the function
 \begin{align*}
     \lambda_{\bar{f}}(t)=\leb^1\left([a,b]\cap\ens{r:\bar{f}(r)>t}\right)
 \end{align*}
 is maximal if the supports of the functions $r\mapsto \mathscr{H}^{d-1}(\partial B(0,r)\cap A_i)$ are disjoint. Therefore, without loss of generality, we can make this assumption. Furthermore, if for all $1\leq i\leq n$, $C_i$ is the ball such that $\partial C_i$ is the circumscribed sphere of $A_i$, then $C_i$ is a ball of radius $\frac{R}{\sqrt{2}}$. In particular, we have by the previous estimate
 \small
 \begin{align}\label{averages_bound}
     \leb^1\left([a,b]\cap\ens{r:\bar{f}(r)>\lambda}\right)\leq \left\{\begin{alignedat}{2}
         &0\qquad&&\text{for all}\;\, \lambda \geq c_n\sqrt{\beta(d)\left(R\sqrt{2}\right)^{d-1}}\\
         &R\sqrt{2}\qquad&&\text{for all}\;\, c_{n-1}\sqrt{\beta(d)\left(R\sqrt{2}\right)^{d-1}}\leq  \lambda<c_n\sqrt{\beta(d)\left(R\sqrt{2}\right)^{d-1}}\\
         &2R\sqrt{2}\qquad&& \text{for all}\;\, c_{n-2}\sqrt{\beta(d)\left(R\sqrt{2}\right)^{d-1}}\leq \lambda<c_{n-1}\sqrt{\beta(d)\left(R\sqrt{2}\right)^{d-1}}\\
         &\vdots\\
         &(n-1)R\sqrt{2}\qquad &&\text{for all}\;\, c_1\sqrt{\beta(d)\left(R\sqrt{2}\right)^{d-1}}\leq \lambda<c_2\sqrt{\beta(d)\left(R\sqrt{2}\right)^{d-1}}\\
         &nR\sqrt{2}\qquad&& \text{for all}\;\, 0<\lambda<c_1\sqrt{\beta(d)\left(R\sqrt{2}\right)^{d-1}}. 
     \end{alignedat}\right.
 \end{align}
 \normalsize
 In particular, we have
 \begin{align}\label{f_bar_l21_final}
     \np{\bar{f}}{2,1}{I}&=4\int_{0}^{\infty}\sqrt{\leb^1\left([a,b]\cap\ens{r:\bar{f}(r)>\lambda}\right)}\,d\lambda
     \leq 4\sum_{i=1}^n\int_{c_{i-1}\sqrt{\beta(d)\left(R\sqrt{2}\right)^{d-1}}}^{c_{i}\sqrt{\beta(d)\left(R\sqrt{2}\right)^{d-1}}}\sqrt{(n-i+1)\sqrt{2}R}\,d\lambda\nonumber\\
     &=4\cdot 2^{\frac{d}{4}}\sqrt{\beta(d)}\sum_{i=1}^n(c_i-c_{i-1})\sqrt{(n-i-1)R^{d}},
 \end{align}
 while
 \begin{align*}
     \np{f}{2,1}{\Omega}=4\sum_{i=1}^n(c_i-c_{i-1})\sqrt{\sum_{j=i}^n\leb^n(A_j)}=4\sum_{i=1}^n(c_i-c_{i-1})\sqrt{(n-i+1)R^d}
 \end{align*}
 since $\leb^n(A_i)=R^d$ for all $1\leq i\leq n$. Therefore, we deduce that
 \begin{align}\label{f_bar_l21}
     \np{\bar{f}}{2,1}{I}\leq 2^{\frac{d}{4}}\sqrt{\beta(d)}\np{f}{2,1}{\Omega}.
 \end{align}
 Therefore, using the density of test functions in $L^{2,1}$ (\cite[Theorem 1.4.13]{grafakos_modern}) and \eqref{f_bar_l21}, the theorem is proven. Indeed, by what precedes, if $\ens{f_n}_{n\in\N}$ is the dyadic approximation of a general $f\in L^{2,1}$, we have
 \begin{align*}
     \np{f_n}{2,1}{\Omega}\conv{n\rightarrow \infty}\np{f}{2,1}{\Omega},
 \end{align*}
 while
 $\bar{f_n}\conv{n\rightarrow \infty}\bar{f}$ $\leb^4$ almost everywhere. Therefore, thanks to the Fatou lemma for Lorentz spaces (see \cite{grafakos_modern,interpolation_bennett,lorentz_memoir}), we deduce that
 \begin{align*}
     \np{\bar{f}}{2,1}{I}\leq\liminf_{n\rightarrow \infty}\np{\bar{f_n}}{2,1}{I}\leq \liminf_{n\rightarrow\infty}\,2^{\frac{d}{4}}\sqrt{\beta(d)}\np{f_n}{2,1}{\Omega}=2^{\frac{d}{4}}\sqrt{\beta(d)}\np{f}{2,1}{\Omega},
 \end{align*}
 which concludes the proof of the theorem.
 \end{proof}

 We can extend this result to the case of $L^{2,q}$ norms for all $1\leq q< \infty$.
 
  \begin{lemme}\label{averaging_l2q}
     Let $d\geq 2$, $1\leq q< \infty$, and $0< a<b<\infty$, and define $I=[a,b]$ and $\Omega=B_b\setminus\bar{B}_a(0)\subset \R^d$. For all $f\in L^2(\Omega)$ define $\bar{f}:[a,b]\rightarrow \R$ for all $a\leq r\leq b$ by
     \begin{align}
     \bar{f}(r)=\np{f}{2}{\partial B(0,r)}=\left(\int_{\partial B(0,r)}|f|^2\,d\mathscr{H}^{d-1}\right)^{\frac{1}{2}}.
     \end{align}
     Then, provided that $f\in L^{2,q}(\Omega)$, we have $\bar{f}\in L^{2,q}(I)$ and
     \begin{align}\label{l2q_bound}
         |\bar{f}|_{\mathrm{L}^{2,q}(I)}\leq 2^{\frac{d}{4}}\sqrt{\beta(d)}\,|f|_{\mathrm{L}^{2,q}(\Omega)},
     \end{align}
     where $\beta(d)=\mathscr{H}^{d-1}(S^{d-1})=\dfrac{2\pi^{\frac{d}{2}}}{\Gamma\left(\frac{d}{2}\right)}$.
 \end{lemme} 
 \begin{proof} 
 If $f$ is the previous indicator of Theorem \ref{averaging_l21}, we have for all $1<q<\infty$
 \begin{align*}
     |f|_{\mathrm{L}^{2,q}}^q&=\int_{0}^{\infty}t^{\frac{q}{2}}f_{\ast}^q(t)\frac{dt}{t}=\sum_{i=1}^n\int_{\sum_{j=i+1}^n\leb^n(A_j)}^{\sum_{j={i}}^{n}\leb^n(A_j)}t^{\frac{q}{2}}c_{i}^{q}\frac{dt}{t}\\
     &=\frac{2}{q}\sum_{i=1}^{n}c_{i}^{q}\left(\left(\sum_{j=i}^n\leb^n(A_j)\right)^{\frac{q}{2}}-\left(\sum_{j=i+1}^n\leb^n(A_j)\right)^{\frac{q}{2}}\right)\\
     &=\frac{2}{q}\sum_{i=1}^n\left(c_i^q-c_{i-1}^q\right)\left(\sum_{j=i}^n\leb^n(A_j)\right)^{\frac{q}{2}}.
 \end{align*}
 In particular, if $\leb^n(A_j)=R^d$ for all $1\leq j\leq n$, we get
 \begin{align*}
     \bar{f}_{\ast}(t)=\frac{2}{q}\sum_{i=1}^nc_i^q\left(\left((n-i+1)R^d\right)^{\frac{q}{2}}-\left((n-i)R^d\right)^{\frac{q}{2}}\right).
 \end{align*}
 On the other hand, we have by the previous computation
 \begin{align}\label{decreasing_rearrangement_average}
     \bar{f}_{\ast}(t)\leq \left\{\begin{alignedat}{2}
         &c_n\sqrt{\beta(d)(R\sqrt{2})^{d-1}}\qquad&&\text{for all}\;\, 0<t<R\sqrt{2}\\
         &c_{n-1}\sqrt{\beta(d)(R\sqrt{2})^{d-1}}\qquad&&\text{for all}\;\, R\sqrt{2}<t<2R\sqrt{2}\\
         \vdots\\
         &c_1\sqrt{\beta(d)(R\sqrt{2})^{d-1}}\qquad&&\text{for all}\;\, (n-1)R\sqrt{2}<t<nR\sqrt{2}\\
         &0\qquad&&\text{for all}\;\, t\geq nR\sqrt{2}.
     \end{alignedat}\right.
 \end{align}
 Therefore, we have
 \begin{align*}
     |\bar{f}|_{\mathrm{L}^{2,q}}^q&=\int_{0}^{\infty}t^{\frac{q}{2}}\bar{f}_{\ast}^q(t)\frac{dt}{t}\leq \sum_{i=1}^n\int_{(i-1)R\sqrt{2}}^{iR\sqrt{2}}t^{\frac{q}{2}}c_{n-i+1}^q\left(\beta(d)(R\sqrt{2})^{d-1}\right)^{\frac{q}{2}}\frac{dt}{t}\\
     &=2^{\frac{qd}{4}}\beta(d)^{\frac{q}{2}}\frac{2}{q}\sum_{i=1}^nc_{n-i+1}^q\left((iR\sqrt{2})^{\frac{q}{2}}-((i-1)R\sqrt{2})^{\frac{q}{2}}\right)=2^{\frac{qd}{4}}\beta(d)^{\frac{q}{2}}|f|_{\mathrm{L}^{p,q}(\Omega)}^q.
 \end{align*}
 \end{proof}
 \begin{rem}
 \begin{enumerate}
 \item 
    For all measured space $(X,\mu)$, for all $1<p<\infty$, for all $1\leq q\leq \infty$, and $\varphi\in L^{p,q}(X)$, we have (by \cite{interpolation_bennett}
    \begin{align*}
         |\varphi|_{\mathrm{L}^{p,q}(X)}\leq \np{\varphi}{p,q}{X}\leq \frac{p}{p-1}|\varphi|_{\mathrm{L}^{p,q}(X)},
     \end{align*}
     we get by \eqref{l2q_bound}
     \begin{align}
         \np{\bar{f}}{2,q}{I}\leq \frac{p}{p-1}2^{\frac{d}{4}}\sqrt{\beta(d)}\np{f}{2,q}{\Omega}.
     \end{align}
     \item  For $q=\infty$, we trivially get by \eqref{decreasing_rearrangement}
 \begin{align*}
     |f|_{\mathrm{L}^{2,\infty}(\Omega)}=\sup_{t>0}\sqrt{t}f_{\ast}(t)=\max\ens{\sqrt{R^d}c_n,\sqrt{2R^d}c_{n-1},\cdots,\sqrt{nR^d}c_1},
 \end{align*}
 while \eqref{decreasing_rearrangement_average} shows that
 \small
 \begin{align*}
     &|\bar{f}|_{\mathrm{L}^{2,\infty}(I)}=\sup_{t>0}\sqrt{t}\,\bar{f}_{\ast}(t)\\
     &\leq \max\ens{\sqrt{R\sqrt{2}}\,c_n\sqrt{\beta(d)(R\sqrt{2})^{d-1}},\sqrt{2R\sqrt{2}}\,c_{n-1}\sqrt{\beta(d)(R\sqrt{2})^{d-1}},\cdots,\sqrt{nR\sqrt{2}}\,c_1\sqrt{\beta(d)(R\sqrt{2})^{d-1}}}\\
     &=2^{\frac{d}{4}}\sqrt{\beta(d)}|f|_{\mathrm{L}^{2,\infty}(\Omega)},
 \end{align*}
 \normalsize
 which should allow us to generalise the theorem for $q=\infty$. However, for $q=\infty$, simple functions are \emph{not} dense, though \emph{countable} linear combinations of simple functions \emph{are} dense (\cite[Remark 1.4.14]{grafakos_modern}).
 \end{enumerate}
 \end{rem}

 In fact, the result generalises to $L^{p,q}$ for all $p>2$ and $1\leq q\leq \infty$ thanks to the Stein-Weiss interpolation theorem, as the map
 \begin{align*}
     T:f\mapsto \bar{f}:r\mapsto \np{f}{2}{\partial B(0,r)}
 \end{align*}
 is a sub-linear map that sends $L^p(\Omega)$ into $L^p(I)$ for all $2\leq p<\infty$, which implies that $T$ is a bounded map from $L^{p,q}(\Omega)$ into $L^{p,q}(I)$ for all $2<p<\infty$ and $1\leq q\leq \infty$. Indeed, we have by Hölder's inequality
 \begin{align*}
     \int_{a}^b\bar{f}(r)^pdr&=\int_{a}^b\left(\int_{\partial B(0,r)}|f|^2d\mathscr{H}^{d-1}\right)^{\frac{p}{2}}dr\leq \int_{a}^b\left(\beta(d)r^{d-1}\right)^{\frac{\frac{p}{2}}{\left(\frac{p}{2}\right)'}}\left(\int_{\partial B(0,r)}|f|^pd\mathscr{H}^{d-1}\right)dr\\
     &=\beta(d)^{\frac{p-2}{2}}\int_{a}^br^{\frac{(d-1)(p-2)}{2}}\left(\int_{\partial B(0,r)}|f|^pd\mathscr{H}^{d-1}\right)dr\\
     &\leq \beta(d)^{\frac{p-2}{2}}b^{\frac{(d-1)(p-2)}{2}}\int_{\Omega}|f|^pdx,
 \end{align*}
 where we used that $\left(\dfrac{p}{2}\right)'=\dfrac{p}{p-2}$.
 Our main theorem allowed us to treat the limiting case $(p,q)=(2,1)$.

    \subsection{Stability of Lorentz Spaces Under Exponentiation}

    We will also need an elementary result on the stability of Lorentz under squaring.
 \begin{lemme}\label{square_lorentz}
     Let $(X,\mu)$ be a measured space and $1<p<\infty$ and $1\leq q\leq \infty$. Then, for all $f\in L^{2p,2q}(X)$, we have $f^2\in L^{p,q}(X)$ and
     \begin{align*}
         \np{|f|^2}{p,q}{X}\leq \frac{p}{p-1}\np{f}{2p,2q}{X}^2.
     \end{align*}
 \end{lemme}
 \begin{proof}
      For all $0<t<\infty$, we have 
      \begin{align*}
          (|f|^2)_{\ast}(t)=\sup\ens{\lambda:\mu(X\cap\ens{x:|f(x)|^2>\lambda}\leq t}=\sup\ens{\lambda:\mu\left(X\cap\ens{x:|f(x)|>\sqrt{\lambda}}\right)\leq t}=f_{\ast}^2(t)
      \end{align*}
      by definition of $f_{\ast}(t)$. Therefore, we get for all $1\leq q<\infty$
      \begin{align*}
          \snp{\,|f|^2}{p,q}{X}=\left(\int_{0}^{\infty}t^{\frac{q}{p}}(|f|^2)^q_{\ast}(t)\frac{dt}{t}\right)^{\frac{1}{q}}=\left(\int_{0}^{\infty}t^{\frac{2q}{2p}}f_{\ast}^{2q}(t)\frac{dt}{t}\right)^{\frac{1}{q}}=\snp{f}{2p,2q}{X}^2
      \end{align*}
      and
      \begin{align*}
          \np{|f|^2}{p,q}{X}\leq \frac{p}{p-1}\snp{|f|^2}{p,q}{X}=\frac{p}{p-1}\snp{f}{2p,2q}{X}^2\leq \frac{p}{p-1}\np{f}{2p,2q}{X}^2,
      \end{align*}
      which concludes the proof of the lemma.
 \end{proof}

        Likewise, one proves the following stability result.
        \begin{lemme}\label{lorentz_stability_general}
            Let $1<p<\infty$, $1\leq q\leq \infty$, and $\alpha>0$ such that $\alpha p>1$ and $\alpha q\geq 1$. For all measured space $(X,\mu)$ and for all measurable function $f:X\rightarrow \R$ such that $f\in L^{\alpha p,\alpha q}(X)$. Then, $|f|^{\alpha}\in L^{p,q}(X)$ and
            \begin{align*}
                \np{|f|^{\alpha}}{p,q}{X}\leq \frac{p}{p-1}\np{f}{\alpha p,\alpha q}{X}^{\alpha}.
            \end{align*}
        \end{lemme}
        \begin{proof}
            As in previous lemma, we need only check the case $q<\infty$ since one immediately gets $(|f|^{\alpha})_{\ast}(t)=f_{\ast}^{\alpha}(t)$.
        Therefore, we have
        \begin{align*}
            \lnp{|f|^{\alpha}}{p,q}{X}=\left(\int_{0}^{t}t^{\frac{q}{p}}(|f|^{\alpha})_{\ast}^q(t)\frac{dt}{t}\right)^{\frac{1}{q}}=\left(\int_{0}^{\infty}t^{\frac{\alpha q}{\alpha p}}f_{\ast}^{\alpha q}(t)\frac{dt}{t}\right)^{\frac{1}{q}}=\lnp{f}{\alpha p,\alpha q}{X}^{\alpha},
        \end{align*}
        which concludes the proof of the lemma thanks to \eqref{comp_lorentz_norms}.
        \end{proof}

        \subsection{An Improved Sobolev Embedding}

    For all $0<\alpha<d$, we have
    \begin{align*}
        \left(\frac{1}{|x|^{\alpha}}\right)_{\ast}(t)&=\inf\left((0,\infty)\cap \ens{\lambda:\leb^d\left(\R^d\cap\ens{x:\frac{1}{|x|^{\alpha}}>\lambda}\right)\leq t}\right)\\
        &=\inf\left((0,\infty)\cap \ens{\lambda:\frac{\beta(d)}{\lambda^{\frac{d}{\alpha}}}\leq t}\right)=\left(\frac{\beta(d)}{t}\right)^{\frac{\alpha}{d}}.
    \end{align*}
    Therefore, we deduce that
    \begin{align}\label{lp_infty_weight}
        \np{\frac{1}{|x|^{\alpha}}}{\frac{d}{\alpha},\infty}{\R^d}=\beta(d)^{\frac{\alpha}{d}}\sup_{t>0}\frac{1}{t^{1-\frac{\alpha}{d}}}\int_{0}^{t}\frac{ds}{s^{\frac{\alpha}{d}}}=\frac{d}{d-\alpha}\beta(d)^{\frac{\alpha}{d}}.
    \end{align}
    In particular, we have
    \begin{align}\label{norm_lorentz_infinity_dim4}
    \left\{\begin{alignedat}{1}
        &\np{\frac{1}{|x|}}{4,\infty}{\R^4}=\frac{4}{3}\left(2\pi^2\right)^{\frac{1}{4}}\\
        &\np{\frac{1}{|x|^2}}{2,\infty}{\R^4}=2\sqrt{2\pi^2}=2\pi\sqrt{2}.
        \end{alignedat}\right.
    \end{align}
    Refer to \cite{rivnotes} for the next construction and theorems. First, we introduce a function $\varphi\in \mathscr{D}(\R)$ such that $\mathrm{supp}(\varphi)\subset [\frac{1}{2},2]$ and for all $x\in \R$, 
    \begin{align*}
        x=\sum_{j\in \Z}2^{j}\varphi(2^{-j}x).
    \end{align*}
    It suffices to take 
    \begin{align*}
        \varphi(x)=x\left(\psi\left(\frac{x}{2}\right)-\psi(x)\right),
    \end{align*}
    where $\psi\in \mathscr{D}(\R)$ is such that $\psi=1$ on $\left[-\dfrac{1}{2},\dfrac{1}{2}\right]$ and $\mathrm{supp}(\psi)\subset [-1,1]$.
    \begin{theorem}
        Let $(X,\mu)$ be a measured space. Let $1<p<\infty$ and $1\leq q\leq p$. Then, we have
        \begin{align*}
            \np{f}{p,q}{X}\leq \frac{p}{p-1}\left(\frac{p\,2^{3q}}{q}\right)^{\frac{1}{p}}\serieslnp{\ens{\np{f_k}{p}{X}}_{k\in \Z}}{q}{\Z},
        \end{align*}
        where $f_k=2^k\varphi(2^{-k}|f|)$.
    \end{theorem}
    \begin{proof}
        We have thanks to a direct integration by parts 
        \begin{align*}
            &|f|_{\mathrm{L}^{p,q}(X)}^q=p\int_{0}^{\infty}\lambda^{q}\left(\mu\left(X\cap\ens{x:|f(x)|\geq \lambda}\right)\right)^{\frac{q}{p}}\frac{d\lambda}{\lambda}
            \leq p\sum_{j\in \Z}\int_{2^{j}}^{2^{j+1}}\lambda^q\left(\mu\left(X\cap\ens{x:|f(x)|\geq \lambda}\right)\right)^{\frac{q}{p}}\frac{d\lambda}{\lambda}\\
            &\leq p\sum_{j\in \Z}\left(\mu\left(X\cap\ens{x:|f(x)|\geq 2^j}\right)\right)^{\frac{q}{p}}\int_{2^j}^{2^{j+1}}\lambda^{q-1}d\lambda
            =\frac{p(2^q-1)}{q}\sum_{j\in \Z}2^{jq}\left(\mu\left(X\cap\ens{x:|f(x)|\geq 2^j}\right)\right)^{\frac{q}{p}}\\
            &= \frac{p(2^q-1)}{q}\sum_{j\in \Z}2^{jq}\left(\sum_{k\geq j}\mu\left(X\cap\ens{x:2^k\leq |f(x)|< 2^{k+1}}\right)\right)^{\frac{q}{p}}\\
            &\leq \frac{p(2^q-1)}{q}\sum_{j\in \Z}2^{jq}\sum_{k\geq j}\left(\mu\left(X\cap\ens{x:2^k\leq |f(x)|< 2^{k+1}}\right)\right)^{\frac{q}{p}}\\
            &=\frac{p(2^q-1)}{q}\sum_{k\in \Z}\left(\mu\left(X\cap\ens{x:2^k\leq |f(x)|< 2^{k+1}}\right)\right)^{\frac{q}{p}}\sum_{j\in \Z}2^{jq}\mathbf{1}_{\ens{j\leq k}}\\
            &=\frac{p\,2^q}{q}\sum_{j\in \Z}2^{kq}\left(\mu\left(X\cap\ens{x:2^k\leq |f(x)|< 2^{k+1}}\right)\right)^{\frac{q}{p}}.
        \end{align*}
        If $k\in \Z$ such that $2^k\leq |f(x)|<2^{k+1}$. Then, we have
        \begin{align*}
            |f(x)|=\sum_{l=k-2}^{k+1}|f_l(x)|,
        \end{align*}
        which shows by Markov inequality and Hölder's inequality that
        \begin{align*}
            2^k\left(\mu(X\cap\ens{x:2^k\leq |f(x)|< 2^{k+1}}\right)^{\frac{1}{p}}&\leq \left(\int_{X}\left(\sum_{l=k-2}^{k+1}|f_l|\right)^pd\mu\right)^{\frac{1}{p}}\leq \sum_{l=k-2}^{k+1}\np{f_l}{p}{X}.
        \end{align*}
        Finally, we deduce by by Hölder's inequality that
        \begin{align*}
            |f|_{\mathrm{L}^{p,q}(X)}^p\leq \frac{p\,2^q}{q}\sum_{k\in \Z}\left(\sum_{l=k-2}^{k+1}\np{f_l}{p}{X}\right)^{q}\leq \frac{p\,2^{3q-2}}{q}\sum_{k\in \Z}\sum_{l=k-2}^{k+1}\np{f_l}{p}{X}^q\leq \frac{p\,2^{3q}}{q}\sum_{l\in \Z}\np{f_l}{p}{X}^q.
        \end{align*}
        which implies that
        \begin{align*}
            \np{f}{p,q}{X}\leq \frac{p}{p-1}|f|_{\mathrm{L}^{p,q}(X)}\leq \frac{p}{p-1}\left(\frac{p\,2^{3q}}{q}\right)^{\frac{1}{p}}\serieslnp{\ens{\np{f_k}{p}{X}}_{k\in \Z}}{q}{\Z}.
        \end{align*}
    \end{proof}
    \begin{theorem}\label{improved_sobolev}
        Let $1\leq p<d$. There is a continuous embedding $W^{1,p}(\R^d)\hookrightarrow L^{p^{\ast},p}(\R^d)$, and for all $u\in W^{1,p}(\R^d)$, we have
        \begin{align}\label{improved_sobolev_ineq}
            \np{u}{p^{\ast},p}{\R^d}\leq \frac{p(d-1)}{d-p}\frac{p^{\ast}}{p^{\ast}-1}\left(\frac{p^{\ast}\,2^{3p}}{p}\right)^{\frac{1}{p^{\ast}}}\np{\D u}{p}{\R^d},
        \end{align}
        where $p^{\ast}=\dfrac{dp}{d-p}$.
    \end{theorem}
    \begin{proof}
        The Sobolev inequality implies that for all $u\in W^{1,p}(\R^d)$ such that $1\leq p<d$ (to get this specific constant, one follows the standard proof exposed in \cite[Chapitre $8$]{brezis}), we have
        \begin{align*}
            \np{u}{p^{\ast}}{\R^d}\leq \frac{p(d-1)}{d-p}\np{\D u}{p}{\R^d}.
        \end{align*}
        Applying this inequality to $f_k$, we get
        \begin{align*}
            \serieslnp{\ens{\np{u_k}{p^{\ast}}{\R^d}}_{k\in \Z}}{p}{\Z}&\leq \frac{p(d-1)}{d-p}\serieslnp{\ens{\np{\D u_k}{p}{\R^d}}_{k\in \Z}}{p}{\Z}=\frac{p(d-1)}{d-p}\left(\sum_{k\in \Z}\int_{\R^d}|u_k|^p\,dx\right)^{\frac{1}{p}}\\
            &=\frac{p(d-1)}{d-p}\np{\D u}{p}{\R^d}
        \end{align*}
        as $\displaystyle u=\sum_{k\in \Z}u_k$. Therefore, We have
        \begin{align*}
            \np{u}{p^{\ast},p}{\R^d}\leq \frac{p(d-1)}{d-p}\frac{p^{\ast}}{p^{\ast}-1}\left(\frac{p^{\ast}\,2^{3p}}{p}\right)^{\frac{1}{p^{\ast}}}\np{\D u}{p}{\R^d}.
        \end{align*}
    \end{proof}
    \begin{rem}
        Notice that the the inequality can also be explicitly rewritten in terms of $p$ and $d$ as
        \begin{align*}
            \np{u}{p^{\ast},p}{\R^d}\leq \frac{d(d-1)p^2}{(d-p)((d+1)p-d)}\left(\frac{d\,2^{3p}}{d-p}\right)^{\frac{1}{p^{\ast}}}\np{\D u}{p}{\R^d}
        \end{align*}
    \end{rem}

    \begin{cor}\label{l42_sobolev}
        There is a continuous embedding $W^{1,2}(\R^4)\hookrightarrow L^{4,2}(\R^4)$, and for all $u\in W^{1,2}(\R^4)$, we have 
        \begin{align}\label{l42_sobolev_ineq}
            \np{u}{4,2}{\R^4}\leq 14\np{\D u}{2}{\R^4}.
        \end{align}
    \end{cor}
    \begin{proof}
        It follows immediately from Theorem \ref{improved_sobolev} and the estimate
        \begin{align*}
            \frac{p(d-1)}{d-p}\frac{p^{\ast}}{p^{\ast}-1}\left(\frac{p^{\ast}\,2^{3p}}{p}\right)^{\frac{1}{p^{\ast}}}=\frac{2\cdot 3}{4-2}\frac{4}{3}\left(\frac{4\cdot 2^6}{2}\right)^{\frac{1}{4}}=8\cdot 2^{\frac{3}{4}}<14,
        \end{align*}
        since $2^{11}=2048<2401=7^4$.
    \end{proof}
    
    \subsection{Whitney Extension Lemmas}

    Notice that if $f\in C^2(\R^n)$ and $\varphi:\R^n\rightarrow \R^n$, the chain rule implies that 
    \begin{align*}
        \p{x_i}f(\varphi(x))&=\sum_{k=1}^d\p{x_k}f(\varphi(x))\cdot \p{x_i}\varphi_k(x)\\
        \p{x_j,x_i}^2f(\varphi(x))&=\sum_{k=1}^n\left(\sum_{l=1}^n\p{x_l,x_k}^2f(\varphi(x))\p{x_j}\varphi_l(x)\right)\p{x_i}\varphi_k(x)+\sum_{k=1}^{n}\p{x_k}f(\varphi(x))\p{x_j,x_i}^2\varphi_k(x)\\
        &=\sum_{k,l=1}^{n}\p{x_k,x_l}^2f(\varphi(x))\p{x_i}\varphi_k(x)\p{x_j}\varphi_l(x)+\sum_{k=1}^n\p{x_k}f(\varphi(x))\p{x_i,x_j}^2\varphi_k(x).
    \end{align*}
    Therefore, we deduce that 
    \begin{align}\label{hessian_chain_rule}
        \D^2(f\circ \varphi)=(\D\varphi)^t\cdot \left((\D^2f)\circ \varphi\right)\cdot(\D\varphi)+\sum_{k=1}^n\left((\p{x_k}f)\circ \varphi\right)\,\D^2\varphi_k.
    \end{align}
    Let $\iota:\R^4\setminus\ens{0}\rightarrow \R^d\setminus\ens{0}$ be the inversion, given for all $x\in \R^4\setminus\ens{0}$ by
    \begin{align*}
        \iota(x)=\frac{x}{|x|^2}=\left(\frac{x_1}{|x|^2},\frac{x_2}{|x|^2},\frac{x_3}{|x|^2},\frac{x_4}{|x|^2}\right)
    \end{align*}
    We have
    \begin{align*}
        \D\iota&=\begin{pmatrix}
            \vspace{0.5em}
            \dfrac{1}{|x|^2}-\dfrac{2\,x_1^2}{|x|^4} & -\dfrac{2\,x_1x_2}{|x|^4} & -\dfrac{2\,x_1\,x_3}{|x|^4} & -\dfrac{2\,x_1\,x_4}{|x|^4}\\
            \vspace{0.5em}
            -\dfrac{2\,x_1\,x_2}{|x|^4} & \dfrac{1}{|x|^2}-\dfrac{2\,x_2^2}{|x|^4} & -\dfrac{2\,x_2\,x_3}{|x|^4} & -\dfrac{2\,x_2\,x_4}{|x|^4}\\
            \vspace{0.5em}
            -\dfrac{2\,x_1\,x_3}{|x|^4} & -\dfrac{2\,x_2\,x_3}{|x|^4} & \dfrac{1}{|x|^2}-\dfrac{2\,x_3^2}{|x|^4} & -\dfrac{2\,x_3\,x_4}{|x|^4} \\
            \vspace{0.5em}
            -\dfrac{2\,x_1\,x_4}{|x|^4}& -\dfrac{2\,x_2\,x_4}{|x|^4} & -\dfrac{2\,x_3x_4}{|x|^4} & \dfrac{1}{|x|^4}-\dfrac{2\,x_4^2}{|x|^6}
        \end{pmatrix}
        \\
        &=\frac{1}{|x|^4}\begin{pmatrix}
        -x_1^2+x_2^2+x_3^2+x_4^2 & -2\,x_1\,x_2 & -2\,x_1\,x_3 & -2\,x_1\,x_4\\
        -2\,x_1\,x_2 & x_1^2-x_2^2+x_3^2+x_4^2 & -2\,x_2\,x_3 & -2\,x_2\,x_4\\
        -2\,x_1\,x_3 & -2\,x_2\,x_3 & x_1^2+x_2^2-x_3^2+x_4^2 & -2\,x_3\,x_4\\
        -2\,x_1\,x_4 & -2\,x_2\,x_4 & -2\,x_3\,x_4 & x_1^2+x_2^2+x_3^2-x_4^2
        \end{pmatrix}
    \end{align*} 
    Therefore, if $v=u\circ \iota$ we get (writing by abuse of notation $\p{x_i}u$ for  $\left(\p{x_i}u\right)\circ \iota$)
    \begin{align*}
        |\D v|^2&=\frac{1}{|x|^8}\bigg(\left(-x_1^2+x_2^2+x_3^2+x_4^2\right)^2(\p{x_1}u)^2+4\,x_1^2\,x_2^2(\p{x_2}u)^2+4\,x_1^2\,x_3^2\,(\p{x_3}u)^2+4\,x_1^2\,x_4^2\,(\p{x_4}u)^2\\
        &+4\,x_1^2\,x_2^2\,(\p{x_1}u)^2+\left(x_1^2-x_2^2+x_3^2+x_4^2\right)^2(\p{x_2}u)^2+4\,x_1^2\,x_2^2\,(\p{x_3}u)^2+4\,x_2^2\,x_4^2\,(\p{x_4}u)^2\\
        &+4\,x_1^2\,x_3^2\,(\p{x_1}u)^2+4\,x_2^2\,x_3^2\,(\p{x_2}u)^2+\left(x_1^2+x_2^2-x_3^2+x_4^2\right)(\p{x_3}u)^2+4\,x_3^2\,x_4^2\,(\p{x_4}u)^2\\
        &+4\,x_1^2\,x_4^2\,(\p{x_1}u)^2+4\,x_2^2\,x_4^2\,(\p{x_2}u)^2+4\,x_3^2\,x_4^2\,(\p{x_3}u)^2+\left(x_1^2+x_2^2+x_3^2-x_4^2\right)^2(\p{x_4}u)^2\\
        &-4\,x_1\,x_2\,\p{x_1}u\,\p{x_2}u\big(\left(-x_1^2+x_2^2+x_3^2+x_4^2\right)+\left(x_1^2-x_2^2+x_3^2+x_4^2\right)-2\,x_3^2-2\,x_4^2\big)\bigg)\\
        &-4\,x_1\,x_3\,\p{x_1}u\,\p{x_3}u\big(\left(-x_1^2+x_2^2+x_3^2+x_4^2\right)-2\,x_2^2+\left(x_1^2+x_2^2-x_3^2+x_4^2\right)-2\,x_4^2\big)\bigg)\\
        &-4\,x_1\,x_4\,\p{x_1}u\,\p{x_4}u\big(\left(-x_1^2+x_2^2+x_3^2+x_4^2\right)-2\,x_2^2-2\,x_3^2+\left(x_1^2+x_2^2+x_3^2-x_4^2\right)\big)\bigg)\\
        &-4\,x_2\,x_3\,\p{x_2}u\,\p{x_3}u\big(-2\,x_1^2+\left(x_1^2-x_2^2+x_3^2+x_4^2\right)+\left(x_1^2+x_2^2-x_3^2+x_4^2\right)-2\,\,x_4^2\big)\bigg)\\
        &-4\,x_2\,x_4\,\p{x_3}u\,\p{x_4}u\big(-2\,x_1^2+\left(x_1^2-x_2^2+x_3^2+x_4^2\right)-2\,x_3^2+\left(x_1^2+x_2^2+x_3^2-x_4^2\right)\big)\bigg)\\
        &-4\,x_3\,x_4\,\p{x_3}u\,\p{x_4}u\big(-2\,x_1^2-2\,x_2^2+\left(+x_1^2+x_2^2-x_3^2+x_4^2\right)+\left(x_1^2+x_2^2+x_3^2-x_4^2\right)\big)\bigg)\\
        &=\frac{|\D u|^2}{|x|^4}.
    \end{align*}
    In particular, for all $u\in W^{1,4}(\R^4)$, we have
    \begin{align}\label{inv_conf_grad_dim4}
        \int_{\R^4}|\D v|^4dx=\int_{\R^4}|\D u|^4dx
    \end{align}
    and
    \begin{align*}
        \int_{\R^4}\frac{|\D v|^2}{|x|^2}dx=\int_{\R^4}\frac{|\D u|^2}{|x|^2}dx.
    \end{align*}

    Now, we compute
    \begin{align*}
        \D^2\iota_1=\begin{pmatrix}
            \vspace{0.5em}
            -\dfrac{6\,x_1}{|x|^4}+\dfrac{8\,x_1^3}{|x|^6} & -\dfrac{2\,x_2}{|x|^4}+\dfrac{8\,x_1^2\,x_2}{|x|^6} & -\dfrac{2\,x_3}{|x|^4}+\dfrac{8\,x_1^2\,x_3}{|x|^6} & -\dfrac{2\,x_4}{|x|^4}+\dfrac{8\,x_1^2\,x_4}{|x|^6}\\
            \vspace{0.5em}
            -\dfrac{2\,x_2}{|x|^4}+\dfrac{8\,x_1^2\,x_2}{|x|^6} &-\dfrac{2\,x_1}{|x|^4}+\dfrac{8\,x_1\,x_2^2}{|x|^6} & \dfrac{8\,x_1\,x_2\,x_3}{|x|^6} & \dfrac{8\,x_1\,x_2\,x_3}{|x|^6}\\
            \vspace{0.5em}
            -\dfrac{2\,x_3}{|x|^4}+\dfrac{8\,x_1^2\,x_3}{|x|^6} & \dfrac{8\,x_1\,x_2\,x_3}{|x|^6} & -\dfrac{2\,x_1}{|x|^4}+\dfrac{8\,x_1\,x_3^2}{|x|^6} & \dfrac{8\,x_1\,x_3\,x_4}{|x|^6}\\
            \vspace{0.5em}
            -\dfrac{2\,x_4}{|x|^4}+\dfrac{8\,x_1^2\,x_4}{|x|^6} & \dfrac{8\,x_1x_2\,x_3}{|x|^6} & \dfrac{8\,x_1\,x_3\,x_4}{|x|^6} & \dfrac{-2\,x_1}{|x|^4}+\dfrac{8\,x_1\,x_4^2}{|x|^6}
        \end{pmatrix},
    \end{align*}
    with circular formulae for $\D^2\iota_i$ for $i=2,3,4$. In particular, we find that 
    \begin{align*}
        |\p{x_i,x_j}^2\iota_k(x)|\leq \frac{8}{|x|^3}\qquad\text{for all}\;\, 1\leq i,j,k\leq 4.
    \end{align*}
    Therefore, \eqref{hessian_chain_rule} implies that 
    \begin{align*}
        |\D^2v|\leq \frac{4}{|x|^4}\left(|\D^2u|+2|x||\D u|^2\right).
    \end{align*}
    Otherwise, one can prove directly that 
    \begin{align*}
    &|\D^2v|^2=\frac{1}{|x|^8}\Big(|\D^2u|^2+8\left(x_1^2+|x|^2\right)|\p{x_1}u|^2+8\left(x_2^2+|x|^2\right)|\p{x_2}u|^2\\
    &+8\left(x_3^2+|x|^2\right)|\p{x_3}u|^2+8\left(x_4^2+|x|^2\right)|\p{x_4}u|^2
    +16\sum_{i,j=1}^{4}x_i\,x_j\,\p{x_i}u\,\p{x_j}u+8\sum_{\substack{i,j=1\\i\neq j}}^{4}\p{x_i,x_j}^2u\,x_i\,\p{x_j}u\\
    &+4\,\p{x_1}^2u\left(x_1\,\p{x_1}u-x_2\,\p{x_2}u-x_3\,\p{x_3}u-x_4\,\p{x_4}u\right)
    +4\,\p{x_2}^2u\left(-x_1\,\p{x_1}u+x_2\,\p{x_2}u-x_3\,\p{x_3}u-x_4\,\p{x_4}u\right)\\
    &+4\,\p{x_3}^2u\left(-x_1\,\p{x_1}u-x_2\,\p{x_2}u+x_3\,\p{x_3}u-x_4\,\p{x_4}u\right)
    +4\,\p{x_4}^2u\left(-x_1\,\p{x_1}u-x_2\,\p{x_2}u-x_3\,\p{x_3}u+x_4\,\p{x_4}u\right)\Big)\\
    &=\frac{1}{|x|^8}\Big(|\D^2u|^2+8|x|^2|\D u|^2+8|x\cdot \D u|^2+8\,x^t\cdot \D^2u\cdot (\D u)-4(x\cdot \D u)\Delta u\Big).
\end{align*}
In particular, we have
\begin{align*}
    |\D^2v|^2\leq \frac{1}{|x|^8}\left(5|\D^2u|^2+12|x|^2|\D u|^2+12|x\cdot \D u|^2+|\Delta u|^2\right)\leq \frac{3}{|x|^8}\left(3|\D^2u|^2+8|x|^2|\D u|^2\right).
\end{align*}
Before stating the extension lemma, we need an elementary estimate of the Poincaré-Wirtinger constant.
\begin{theorem}\label{dyadic_poincare_wirtinger}
    Let $d\geq 3$, $0<r<\infty$ and $\Omega_r=B_{2r}\setminus\bar{B}_r(0)$. Then, for all $u\in W^{1,2}(\Omega_r)$, we have
    \begin{align*}
        \int_{\Omega_r}\frac{\left|u-\bar{u}_{\Omega_r}\right|^2}{|x|^2}dx\leq \frac{4}{(d-2)^2}\int_{\Omega_r}|\D u|^2dx.
    \end{align*}
    In particular, we have
    \begin{align*}
        \int_{\Omega_r}\left|u-\bar{u}_{\Omega_r}\right|^2\leq \frac{16\,r^2}{(d-2)^2}\int_{\Omega_r}|\D u|^2dx
    \end{align*}
\end{theorem}
\begin{proof}
    By an immediate scaling argument, it suffices to check the case $r=1$. Consider the following minimisation problem
    \begin{align*}
        \mu_d=\inf\ens{\int_{\Omega_1}|\D u|^2dx: \int_{\Omega_1}u\,dx=0,\;\,\int_{\Omega_1}|u|^2dx=1}.
    \end{align*}
    It yields the optimal constant in the Poincaré-Wirtinger inequality, but we will consider another simpler problem that gives a near-optimal constant
    \begin{align*}
         \mu_d^{\ast}=\inf\ens{\int_{\Omega_1}|\D u|^2dx: \int_{\Omega_1}u\,dx=0,\;\,\int_{\Omega_1}\frac{|u|^2}{|x|^2}dx=1}.
    \end{align*}
    By standard methods of the calculus of variations, there exists a minimiser $u$ that satisfies the following system of equations:
    \begin{align*}
        \left\{\begin{alignedat}{2}
            -\Delta u&=\mu_d^{\ast}\frac{u}{|x|^2}\qquad&&\text{in}\;\, \Omega_1\\
            \partial_{\nu}u&=0\qquad&&\text{on}\;\,\partial\Omega_1\\
            \int_{\Omega_1}u\,dx&=0\\
            \int_{\Omega_1}\frac{u^2}{|x|^2}dx&=1.
        \end{alignedat}\right.
    \end{align*}
    Expand $u$ in spherical harmonics
    \begin{align*}
        u(r,\omega)=\sum_{n=0}^{\infty}\sum_{k=1}^{N_d(n)}u_{n,k}(r)Y_n^k(\omega).
    \end{align*}
    The Euler-Lagrange equation implies that for all $n\in \N$ and $1\leq k\leq N_d(n)$, we have
    \begin{align*}
        u_{n,k}''(r)+\frac{d-1}{r}u_{n,k}'(r)-\frac{n(n+d-2)}{r^2}u_{n,k}(r)=-\mu_d^{\ast}u_{n,k}.
    \end{align*}
    Making the change of variable $u_{n,k}(r)=Y_{n,k}(\log(r))$ and removing the indices for simplicity, we deduce that
    \begin{align*}
        Y''+(d-2)Y'-n(n+d-2)Y=-\mu_{d}^{\ast}Y.
    \end{align*}
    The discriminant of the characteristic polynomial $P(X)=X^2+(d-2)X-n(n+d-2)+\mu_{d}^{\ast}$ is given by
    \begin{align*}
        D=(d-2)^2+4n(n+d-2)-4\mu_{d}^{\ast}.
    \end{align*}
    If $D\geq 0$, we easily see that the boundary condition $Y'=0$ on $\partial[0,\log(2)]$ cannot be satisfied, which implies that $D<0$ or
    \begin{align*}
        \mu_{d}^{\ast}>n(n+d-2)+\frac{(d-2)^2}{4}.
    \end{align*}
    Therefore, we obtain the elementary estimate $\mu_d^{\ast}>\dfrac{(d-2)^2}{4}$, which implies that for all $u\in W^{1,2}(\Omega_1)$, 
    \begin{align*}
        \int_{\Omega_1}|\D u|^2dx\geq \frac{(d-2)^2}{4}\int_{\Omega_1}\frac{|u-\bar{u}_{\Omega_1}|^2}{|x|^2}dx,
    \end{align*}
    and finally, the second inequality follows from the elementary estimate
    \begin{align*}
        \int_{\Omega_1}\frac{|u-\bar{u}_{\Omega_1}|^2}{|x|^2}dx\geq \frac{1}{4}\int_{\Omega_1}\left|u-\bar{u}_{\Omega_1}\right|^2dx,
    \end{align*}
    which concludes the proof of the theorem.
\end{proof}
Now, let us state the first extension theorem, inspired from \cite[Lemma C.1]{riviere_morse_scs}.
\begin{theorem}\label{whitney_extension_dim4}
    There exists a universal constant $0<\Gamma_{\mathrm{W}}<\infty$ with the following property. 
    Let $0<2\,a<b<\infty$ and let $\Omega=B_b\setminus\bar{B}_a(0)\subset \R^4$. Let $u\in W^{2,2}(\Omega)$. Then, there exists an extension $\widetilde{u}\in W^{2,2}(\R^4)$ such that
    \begin{align}\label{whitney_extension_dim4_ineq}
    \left\{\begin{alignedat}{1}
        &\np{\D^2\widetilde{u}}{2}{\R^4}\leq 19\np{\D^2u}{2}{\Omega}+289\np{\frac{\D u}{|x|}}{2}{\Omega}\\
        &\np{\D^2\widetilde{u}}{2}{\R^4}\leq 7\np{\D^2u}{2}{\Omega}+\left(1+160\left(\sqrt[4]{2\log(2)}+4\sqrt[4]{30}\right)\right)\np{\D u}{4}{\Omega}\\
        &\np{\frac{\D u}{|x|}}{2}{\R^4}\leq \sqrt{51}\np{\frac{\D u}{|x|}}{2}{\Omega}\\
        &\np{\frac{\D u}{|x|}}{2}{\R^4}\leq 196\sqrt[4]{2}\sqrt{\pi}\np{\D^2u}{2}{\Omega}
        +28\sqrt[4]{2}\sqrt{\pi}\left(1+160\left(\sqrt[4]{2\log(2)}+4\sqrt[4]{30}\right)\right)\np{\D u}{4}{\Omega}\\
        &\np{\D\widetilde{u}}{4,2}{\R^4}\leq 261\np{\D^2u}{2}{\Omega}+4046\np{\frac{\D u}{|x|}}{2}{\Omega}\\
        &\np{\D\widetilde{u}}{4,2}{\R^4}\leq 98\np{\D^2u}{2}{\Omega}+14\left(1+160\left(\sqrt[4]{2\log(2)}+4\sqrt[4]{30}\right)\right)\np{\D u}{4}{\Omega}\\
        &\np{\D\widetilde{u}}{4}{\R^4}\leq 261\sqrt[4]{2}\np{\D^2u}{2}{\Omega}+4046\sqrt[4]{2}\np{\frac{\D u}{|x|}}{2}{\Omega}\\
        &\np{\D \widetilde{u}}{4}{\R^4}\leq 98\sqrt[4]{2}\np{\D^2u}{2}{\Omega}+14\sqrt[4]{2}\left(1+160\left(\sqrt[4]{2\log(2)}+4\sqrt[4]{30}\right)\right)\np{\D u}{4}{\Omega}.
        \end{alignedat}\right.
    \end{align}
    In particular, as $\np{\,\cdot\,}{4}{U}\leq \sqrt[4]{2}\np{\,\cdot\,}{4,2}{U}$ and $\np{\dfrac{\,\cdot\,}{|x|}}{2}{U}\leq 2\sqrt[4]{2}\sqrt{\pi}\np{\,\cdot\,}{4,2}{U}$ for all open subset $U\subset \R^4$, we deduce that the three norms
    \begin{align}\label{whitney_extension_dim4_3norms}
    \left\{\begin{alignedat}{1}
        N_{2,4}&=\np{\D^2(\,\cdot\,)}{2}{\Omega}+\np{\D (\,\cdot\,)}{4}{\Omega}\\
        N_{2,2}&=\np{\D^2(\,\cdot\,)}{2}{\Omega}+\np{\dfrac{\D(\,\cdot\,)}{|x|}}{2}{\Omega}\\
        N_{2,(4,2)}&=\np{\D^2(\,\cdot\,)}{2}{\Omega}+\np{\D(\,\cdot\,)}{4,2}{\Omega}
        \end{alignedat}\right.
    \end{align}
    are mutually equivalent on $W^{2,2}(\Omega)/\R$ and that
    \begin{align}\label{whitney_extension_dim4_equiv_norms}
    \left\{\begin{alignedat}{1}
        &\max\ens{N_{2,2}(u),N_{2,(4,2)}(u)}\leq \Gamma_{\mathrm{W}}\,N_{2,4}(u)\\
        &\max\ens{N_{2,4}(u),N_{2,(4,2)}(u)}\leq \Gamma_{\mathrm{W}}\,N_{2,2}(u).
        \end{alignedat}\right.
    \end{align}
\end{theorem}
\begin{rem}
    Explicitly, we can take
    \begin{align*}
        \Gamma_{\mathrm{W}}=28\sqrt[4]{2}\sqrt{\pi}\left(1+160\left(\sqrt[4]{2\log(2)}+4\sqrt[4]{30}\right)\right)=98705.182\cdots,
    \end{align*}
\end{rem}
\begin{proof}
    \textbf{Step 1: Weighted Gradient Estimate.}

    Let $\chi\in \mathscr{D}(\R)$ such that $0\leq \chi\leq 1$, $\chi=1$ on $[0,1]$ and $\mathrm{supp}(\chi)\subset [0,2]$. By smoothing the function
    \begin{align*}
        \eta(t)=\left\{\begin{alignedat}{2}
            &0\qquad&&\text{for all}\;\, t\leq -2\\
            &2+t\qquad&&\text{for all}\;\, -2\leq t\leq -1\\
            &1\qquad&&\text{for all}\;\, -1\leq t\leq 1\\
            &2-t\qquad&&\text{for all}\;\, 1\leq t\leq 2\\
            &0\qquad&&\text{for all}\;\, t\geq 2
        \end{alignedat}\right.
    \end{align*}
    that satisfies $|\eta'|\leq 1$ on $\R\setminus\ens{-2,-1,1,2}$, we can assume that 
    \begin{align}\label{sup_chi}
        \max\ens{\np{\chi'}{\infty}{\R},\np{\chi''}{\infty}{\R}}\leq 2.
    \end{align}
    Then, for all $r>0$ let $\chi_r\in \mathscr{D}(\R^4)$ be such that $\chi_r(x)=\eta\left(\dfrac{|x|}{r}\right)$. We have
    \begin{align*}
        \D\chi_r(x)&=\frac{1}{r}\frac{x}{|x|}\eta'\left(\frac{|x|}{r}\right)\\
        \D^2\chi_r(x)&=\frac{1}{r|x|^3}\begin{pmatrix}
            x_2^2+x_3^2+x_4^2 & -x_1\,x_2 & -x_1\,x_3 & -x_1\,x_4\\
            -x_1\,x_2 & x_1^2+x_3^2+x_4^2 & -x_2\,x_3 & -x_2\,x_4\\
            -x_1\,x_3 & -x_2\,x_3 & x_1^2+x_2^2+x_4^2 & -x_3\,x_4\\
            -x_1\,x_4 & -x_2\,x_4 & -x_3\,x_4 & x_1^2+x_2^2+x_3^2
        \end{pmatrix}\eta'\left(\frac{|x|}{r}\right)\\
        &+\frac{1}{r|x|^2}\begin{pmatrix}
            x_1^2 & x_1\,x_2 & x_1\,x_3 & x_1\,x_4 \\
            x_1\,x_2 & x_2^2 & x_2\,x_3 & x_2\,x_4\\
            x_1\,x_3 & x_2\,x_3 & x_3^2 & x_3\,x_4\\
            x_1\,x_4 & x_2\,x_4 & x_3\,x_4 & x_4^2
        \end{pmatrix}\eta''\left(\frac{|x|}{r}\right).
    \end{align*}
    Therefore, we get by \eqref{sup_chi}
    \begin{align}\label{estimate_chi}
    \left\{\begin{alignedat}{1}
        &|\D\chi_r|\leq \frac{2}{r}\mathbf{1}_{B_{2r}\setminus\bar{B}_r(0)}\\
        &|\D^2\chi_r|\leq \frac{4}{r|x|}\mathbf{1}_{B_{2r}\setminus\bar{B}_r(0)}.
        \end{alignedat}\right.
    \end{align}
    Define
    \begin{align*}
        \bar{u}(x)=\chi_a(x)u(x)+\left(1-\chi_r(x)\right)\dashint{B_{2a}\setminus\bar{B}_a(0)}u\,d\leb^4=\chi_a(x)u(x)+\left(1-\chi_r(x)\right)\bar{u}_{B_{2a}\setminus\bar{B}_a(0)}.
    \end{align*}
    We have
    \begin{align}\label{grad_bar_u_dev}
        \D\bar{u}=\chi_a\D u+\left(u-\bar{u}_{B_{2a}\setminus\bar{B}_a(0)}\right)\D\chi_a.
    \end{align}
    Thanks to Theorem \ref{dyadic_poincare_wirtinger} and \eqref{estimate_chi}, we deduce that
    \begin{align*}
        \int_{\R^4\setminus\bar{B}_a(0)}\frac{\left|\left(u-\bar{u}_{B_{2a}\setminus\bar{B}_a(0)}\right)\D\chi_a\right|^2}{|x|^2}dx&\leq \frac{4}{a^2}\int_{B_{2a}\setminus\bar{B}_a(0)}\frac{\left|u-\bar{u}_{B_{2a}\setminus\bar{B}_a(0)}\right|^2}{|x|^2}dx\leq \frac{16}{a^2}\int_{B_{2a}\setminus \bar{B}_{a}(0)}|\D u|^2dx\\
        &\leq 16\int_{B_{2a}\setminus\bar{B}_a(0)}\frac{|\D u|^2}{|x|^2}dx.
    \end{align*}
    Therefore, we deduce that
    \begin{align*}
        \np{\frac{\D\bar{u}}{|x|}}{2}{\R^4\setminus\bar{B}_a(0)}&\leq \np{\chi_a\frac{\D u}{|x|}}{2}{\R^4\setminus\bar{B}_a(0)}+\np{\frac{\left(u-\bar{u}_{B_{2a}\setminus\bar{B}_a(0)}\right)}{|x|}\D\chi_a}{2}{\R^4\setminus\bar{B}_a(0)}\\
        &\leq 5\np{\frac{\D u}{|x|}}{2}{B_{2a}\setminus\bar{B}_a(0)},
    \end{align*}
    or
    \begin{align}\label{whitney_lemma1}
        \int_{\R^4\setminus\bar{B}_a(0)}\frac{|\D \bar{u}|^2}{|x|^2}dx\leq 25\int_{B_{2a}\setminus\bar{B}_a(0)}\frac{|\D u|^2}{|x|^2}dx.
    \end{align}
    Now, consider the function 
    \begin{align*}
        \widehat{u}(x)&=\left\{\begin{alignedat}{2}
            &u(x)\qquad&&\text{for all}\;\, x\in B_{b}\setminus\bar{B}_a(0)\\
            &\bar{u}\left(a^2\frac{x}{|x|^2}\right)\qquad&&\text{for all}\;\, x\in B_a(0)
        \end{alignedat}\right.
    \end{align*}
    By conformal invariance \eqref{inv_conf_grad_dim4}, we deduce by \eqref{whitney_lemma1} that 
    \begin{align}\label{whitney_lemma2}
        \int_{B_a(0)}\frac{|\D \widehat{u}|^2}{|x|^2}dx=\int_{\C\setminus\bar{B}_a(0)}\frac{|\D \bar{u}|^2}{|x|^2}\leq 25\int_{B_{2a}\setminus\bar{B}_a(0)}\frac{|\D u|^2}{|x|^2}dx.
    \end{align}
    Likewise, using the same construction on $B_{b}\setminus\bar{B}_{\frac{b}{2}}(0)$ yields a function $\widetilde{u}\in W^{1,2}(\R^4)$ such that $\mathrm{supp}(\widetilde{u})\subset B_{2b}\setminus\bar{B}_{\frac{a}{2}}(0)$, and using \eqref{whitney_lemma2}
    \begin{align}\label{whitney_lemma3}
    \left\{\begin{alignedat}{2}
        &\int_{\R^4}\frac{|\D \widetilde{u}|^2}{|x|^2}dx\leq 51\int_{\Omega}\frac{|\D u|^2}{|x|^2}dx\\
        &\int_{B_a(0)}\frac{|\D \widetilde{u}|^2}{|x|^2}dx\leq 25\int_{B_{2a}\setminus\bar{B}_a(0)}\frac{|\D u|^2}{|x|^2}dx\\
        &\int_{\R^4\setminus\bar{B}_b(0)}\frac{|\D \widetilde{u}|^2}{|x|^2}dx\leq 25\int_{B_{b}\setminus\bar{B}_{\frac{b}{2}}(0)}\frac{|\D u|^2}{|x|^2}dx.
        \end{alignedat}\right.
    \end{align}
    \textbf{Step 2: Hessian Estimate.}
    
    Now, let us estimate the $L^2$ norm of $\D^2\widetilde{u}$. We have
    \begin{align}\label{hessian_identity}
        &\int_{B_a(0)}|\D^2\widehat{u}|^2dx\nonumber\\
        &=\int_{\R^4\setminus\bar{B}_a(0)}\left(|\D^2\bar{u}|^2+8\frac{|\D \bar{u}|^2}{|x|^2}+8\left|\frac{x}{|x|^2}\cdot\D\bar{u} \right|^2+8\,\left(\frac{x}{|x|^2}\right)^t\cdot \D^2u\cdot \D u-4\left(\frac{x}{|x|^2}\cdot \D u\right)\Delta u\right)dx\nonumber\\
        &\leq 9\int_{\R^4\setminus\bar{B}_a(0)}|\D^2\bar{u}|^2dx+24\int_{\R^4\setminus\bar{B}_a(0)}\frac{|\D \bar{u}|^2}{|x|^2}dx.
    \end{align}
    We have
    \begin{align}\label{hessian_identity2}
        \D^2\bar{u}=\chi_a\D^2u+2\D u\cdot \D\chi_a+\left(u-\bar{u}_{B_{2a}\setminus\bar{B}_a(0)}\right)\D^2\chi_a.
    \end{align}
    Using \eqref{estimate_chi}, we deduce that 
    \begin{align}\label{whitney_lemma4}
        \int_{B_{2a}\setminus\bar{B}_a(0)}|\D u\cdot \D\chi_a|^2dx\leq 4\int_{B_{2a}\setminus\bar{B}_a(0)}\frac{|\D u|^2}{|x|^2}dx
    \end{align}
    and Theorem \ref{dyadic_poincare_wirtinger} shows that
    \begin{align}\label{whitney_lemma5}
        \int_{B_{2a}\setminus\bar{B}_a(0)}\left|\left(u-\bar{u}_{B_{2a}\setminus\bar{B}_a(0)}\right)\D^2\chi_a\right|^2dx&\leq \frac{16}{a^2}\int_{B_{2a}\setminus\bar{B}_a(0)}\frac{\left|u-\bar{u}_{B_{2a}\setminus\bar{B}_a(0)}\right|^2}{|x|^2}dx\nonumber\\
        &\leq \frac{64}{a^2}\int_{B_{2a}\setminus\bar{B}_a(0)}|\D u|^2dx\leq 64\int_{B_{2a}\setminus\bar{B}_a(0)}\frac{|\D u|^2}{|x|^2}dx.
    \end{align}
    Therefore, we have by \eqref{hessian_identity2}, \eqref{whitney_lemma4}, and \eqref{whitney_lemma5}
    \begin{align}\label{hessian_estimate}
        \np{\D^2\bar{u}}{2}{\R^4\setminus\bar{B}_a(0)}&\leq \np{\D^2u}{2}{B_{2a}\setminus\bar{B}_a(0)}+8\np{\frac{\D u}{|x|}}{2}{B_{2a}\setminus\bar{B}_a(0)}+8\np{\frac{\D u}{|x|}}{2}{B_{2a}\setminus\bar{B}_a(0)}\nonumber\\
        &=\np{\D^2u}{2}{B_{2a}\setminus\bar{B}_a(0)}+16\np{\frac{\D u}{|x|}}{2}{B_{2a}\setminus\bar{B}_a(0)}
    \end{align}
    and by \eqref{hessian_identity} and \eqref{hessian_estimate}, we obtain
    \begin{align}\label{whitney_lemma6}
        \np{\D^2\widehat{u}}{2}{B_a(0)}\leq 3\np{\D^2u}{2}{B_{2a}\setminus\bar{B}_a(0)}+80\np{\frac{\D u}{|x|}}{2}{B_{2a}\setminus\bar{B}_a(0)}.
    \end{align}
    Likewise, we have
    \begin{align}\label{whitney_lemma7}
        \np{\D^2\widetilde{u}}{2}{\R^4\setminus\bar{B}_b(0)}\leq 3\np{\D^2u}{2}{B_b\setminus\bar{B}_{\frac{b}{2}}(0)}+80\np{\frac{\D u}{|x|}}{2}{B_b\setminus\bar{B}_{\frac{b}{2}}(0)}.
    \end{align}
    Therefore, we finally deduce by the inversion estimate \eqref{hessian_identity}, and the estimates \eqref{whitney_lemma6} and \eqref{whitney_lemma7} that
    \begin{align}\label{whitney_lemma8}
        \np{\D^2\widetilde{u}}{2}{\R^4}\leq 7\np{\D^2u}{2}{\Omega}+161\np{\frac{\D u}{|x|}}{2}{\Omega}.
    \end{align}
    
    \textbf{Step 3: Gradient Lorentz-Sobolev Estimate.}
   
    Now, we have by \eqref{whitney_lemma8} and the improved Sobolev inequality from Theorem \ref{l42_sobolev}
    \begin{align}\label{whitney_lemma10}
        \np{\D\widetilde{u}}{4,2}{\R^4}\leq 14\np{\D^2\widetilde{u}}{2}{\R^4}\leq 98\np{\D^2u}{2}{\Omega}+2254\np{\frac{\D u}{|x|}}{2}{\Omega},
    \end{align}
    which shows by the extension property of $\widetilde{u}$ that
    \begin{align}\label{whitney_lemma11}
        \np{\D u}{4,2}{\Omega}\leq 261\np{\D^2u}{2}{\Omega}+4046\np{\frac{\D u}{|x|}}{2}{\Omega}.
    \end{align}
    Notice also that thanks to Proposition \ref{comp_lorentz_norm}
    \begin{align*}
        \np{\D u}{4}{\Omega}\leq \left(\frac{4}{2}\right)^{\frac{1}{2}-\frac{1}{4}}\np{\D u}{4,2}{\Omega}\leq 261\sqrt[4]{2}\np{\D^2u}{2}{\Omega}+4046\sqrt[4]{2}\np{\frac{\D u}{|x|}}{2}{\Omega}.
    \end{align*}
    On the other hand, the $L^{2,1}/L^{2,\infty}$ duality, the doubling estimate of Lemma \ref{square_lorentz}, and \eqref{norm_lorentz_infinity_dim4} show that
    \begin{align}\label{whitney_l42_l2}
        \int_{\Omega}\frac{|\D u|^2}{|x|^2}dx\leq \np{|\D u|^2}{2,1}{\Omega}\np{\frac{1}{|x|^2}}{2,\infty}{\Omega}\leq 4\pi\sqrt{2}\np{\D u}{4,2}{\Omega}^2.
    \end{align}

    \textbf{Step 4: Biquadratic Estimate for the Gradient.}

    Now, let us estimate the $L^4$ norm of the extension. We first need an elementary variant of the Poincaré-Sobolev inequality.
    \begin{theorem}\label{poincare_sobolev}
        There exists a universal constant $\Gamma_{\mathrm{PS}}<\infty$ such that for all $r>0$ and $u\in W^{2,2}(B_{2r}\setminus\bar{B}_r(0)$, the following inequality holds
        \begin{align}\label{poincare_sobolev_ineq}
            \np{u-\bar{u}_{B_{2r}\setminus\bar{B}_r(0)}}{4}{B_{2r}\setminus\bar{B}_r(0)}\leq \Gamma_{\mathrm{PS}}\left(\np{\D u}{2}{B_{2r}\setminus\bar{B}_r(0)}+r\np{\D^2u}{2}{B_{2r}\setminus\bar{B}_r(0)}\right)
        \end{align}
    \end{theorem}
    \begin{proof}
        An immediate scaling argument shows that we need only establish the inequality for $r=1$. Write for simplicity $A_1=B_2\setminus\bar{B}_1(0)\subset \R^4$. We argue by contradiction, and let $\ens{u_k}_{k\in\N}\subset W^{2,2}(A)$ such that 
        \begin{align}\label{normalisation}
        \left\{\begin{alignedat}{1}
            \int_{A_1}|u_k|^4dx&=1\\
            \int_{A_1}u_k\,dx&=0,
            \end{alignedat}\right.
        \end{align}
        and
        \begin{align*}
            \lim_{k\rightarrow \infty}\left(\np{\D^2u_k}{2}{A_1}+\np{\D u_k}{2}{A_1}\right)=0.
        \end{align*}
        Thanks to the Sobolev embedding $\displaystyle W^{2,2}(A_1)\hooklongrightarrow W^{1,4}(A_1)\hooklongrightarrow \bigcap_{q<\infty}L^q(A_1)$, the sequence $\ens{u_k}_{k\in \N}$ is precompact in $L^{4}(A_1)$ and $L^1(A_1)$, and up to a subsequence, we deduce that $u_k\hookrightarrow u_{\infty}\in W^{2,2}(A_1)$ such that
        \begin{align*}
            \np{\D^2u_{\infty}}{2}{A_1}+\np{\D u_{\infty}}{2}{A_1}\leq \liminf_{k\rightarrow \infty}\left(\np{\D^2u_{k}}{2}{A_1}+\np{\D u_k}{2}{A_1}\right)=0.
        \end{align*}
        while (by compactness)
        \begin{align}\label{normalisation2}
        \left\{\begin{alignedat}{2}
            \int_{A_1}u_{\infty}\,dx&=0.
            \end{alignedat}\right.
        \end{align}
        Therefore, we deduce that $u_{\infty}$ is constant, and since $u_{\infty}$ has vanishing mean, this implies that $u_{\infty}$, which contradicts the first equation in \eqref{normalisation2}. 
    \end{proof}
    Using Theorem \eqref{poincare_sobolev} and the pointwise estimate \eqref{estimate_chi}, we deduce that 
    \begin{align}\label{whitney_l41}
        \np{\left(u-\bar{u}_{B_{2a}\setminus\bar{B}_a(0)}\right)\D\chi_a}{4}{\R^4\setminus\bar{B}_a(0)}&\leq \frac{2}{a}\np{u-\bar{u}_{B_{2a}\setminus\bar{B}_a(0)}}{4}{B_{2a}\setminus\bar{B}_a(0)}\nonumber\\
        &\leq 2\,\Gamma_{\mathrm{PS}}\left(\frac{1}{a}\np{\D u}{2}{B_{2a}\setminus\bar{B}_a(0)}+\np{\D^2u}{2}{B_{2a}\setminus\bar{B}_a(0)}\right).
    \end{align}
    By Cauchy-Schwarz inequality, we deduce that 
    \begin{align}\label{whitney_l42}
        \frac{1}{a^2}\int_{B_{2a}\setminus\bar{B}_a(0)}|\D u|^2\leq \frac{1}{a^2}\sqrt{2\pi^2\left((2a)^4-a^4\right)}\np{\D u}{4}{B_{2a}\setminus\bar{B}_a(0)}
        =\pi\sqrt{30}\np{\D u}{4}{B_{2a}\setminus\bar{B}_a(0)}.
    \end{align}
    By \eqref{whitney_l41} and \eqref{whitney_l42}, we deduce that 
    \begin{align}\label{whitney_l43}
        \np{\left(u-\bar{u}_{B_{2a}\setminus\bar{B}_a(0)}\right)\D\chi_a}{4}{\R^4\setminus\bar{B}_a(0)}\leq 2\,\Gamma_{\mathrm{PS}}\left(\np{\D^2u}{2}{B_{2a}\setminus\bar{B}_a(0)}+\sqrt[4]{30}\sqrt{\pi}\np{\D u}{4}{B_{2a}\setminus\bar{B}_a(0)}\right).
    \end{align}
    Then, we trivially have 
    \begin{align}\label{whitney_l44}
        \np{\chi_a\D u}{4}{\R^4\setminus\bar{B}_a(0)}\leq \np{\D u}{4}{B_{2a}\setminus\bar{B}_a(0)}.
    \end{align}
    Combining \eqref{grad_bar_u_dev}, \eqref{whitney_l43}, and \eqref{whitney_l44}, we deduce that
    \begin{align}\label{whitney_l45}
        \np{\D\bar{u}}{4}{\R^4\setminus\bar{B}_a(0)}\leq 2\,\Gamma_{\mathrm{PS}}\np{\D^2u}{2}{B_{2a}\setminus\bar{B}_a(0)}+\left(1+2\,\sqrt[4]{30}\sqrt{\pi}\Gamma_{\mathrm{PS}}\right)\np{\D u}{4}{B_{2a}\setminus\bar{B}_a(0)}.
    \end{align}
    Therefore, \eqref{inv_conf_grad_dim4} and \eqref{whitney_l45} show that 
    \begin{align}\label{whitney_l46}
        \np{\D \widehat{u}}{4}{B_a(0)}=\np{\D\bar{u}}{4}{\R^4\setminus\bar{B}_a(0)}\leq 2\,\Gamma_{\mathrm{PS}}\np{\D^2u}{2}{B_{2a}\setminus\bar{B}_a(0)}+\left(1+2\,\sqrt[4]{30}\sqrt{\pi}\,\Gamma_{\mathrm{PS}}\right)\np{\D u}{4}{B_{2a}\setminus\bar{B}_a(0)}.
    \end{align}
    Therefore, we finally get an extension $\widetilde{u}\in W^{2,2}(\R^4)$ such that
    \begin{align}\label{whitney_l47}
        \np{\D \widetilde{u}}{4}{\R^4}\leq \left(1+4\,\Gamma_{\mathrm{PS}}\right)\np{\D^2u}{2}{\Omega}+\left(3+4\sqrt[4]{30}\sqrt{\pi}\,\Gamma_{\mathrm{PS}}\right)\np{\D u}{4}{\Omega}.
    \end{align}
    Now, let us get a new bound for the Hessian (involving the $L^4$ norm). We have 
    \begin{align*}
        \np{\frac{\chi_a\D u}{|x|}}{2}{\R^4\setminus\bar{B}_a(0)}\leq \np{\D u}{4}{B_{2a}\setminus\bar{B}_a(0)}\np{\frac{1}{|x|^2}}{2}{B_{2a}\setminus\bar{B}_a(0)}^{\frac{1}{2}}\leq \left(2\pi^2\log(2)\right)^{\frac{1}{4}}\np{\D u}{4}{B_{2a}\setminus\bar{B}_a(0)},
    \end{align*}
    and likewise
    \begin{align*}
        \int_{\R^4\setminus\bar{B}_a(0)}\frac{\left|\left(u-\bar{u}_{B_{2a}\setminus\bar{B}_a(0)}\right)\D\chi_a\right|^2}{|x|^2}dx&\leq \frac{4}{a^2}\int_{B_{2a}\setminus\bar{B}_a(0)}\frac{\left|u-\bar{u}_{B_{2a}\setminus\bar{B}_a(0)}\right|^2}{|x|^2}dx\leq \frac{16}{a^2}\int_{B_{2a}\setminus B_{a}(0)}|\D u|^2dx\\
        &\leq \frac{16}{a^2}\sqrt{2\pi^2\left((2a)^4-a^4\right)}\np{\D u}{4}{B_{2a}\setminus\bar{B}_a(0)}\\
        &\leq 16\pi\sqrt{30}\np{\D u}{4}{B_{2a}\setminus\bar{B}_a(0)}.
    \end{align*}
    Therefore, we deduce that 
    \begin{align}\label{whitney_l48}
        \np{\frac{\D u}{|x|}}{2}{\R^4\setminus\bar{B}_a(0)}
        \leq \left(\sqrt[4]{2\log(2)}+4\sqrt[4]{30}\right)\sqrt{\pi}\np{\D u}{4}{B_{2a}\setminus\bar{B}_a(0)}.
    \end{align}
    Therefore, \eqref{hessian_identity} and \eqref{whitney_l48} show that
    \begin{align}
         \np{\D^2\widetilde{u}}{2}{B_a(0)}&\leq 3\np{\D^2u}{2}{B_{2a}\setminus\bar{B}_a(0)}+80\np{\frac{\D u}{|x|}}{2}{B_{2a}\setminus\bar{B}_a(0)}\nonumber\\
         &\leq 3\np{\D^2u}{2}{B_{2a}\setminus\bar{B}_a(0)}+80\left(\sqrt[4]{2\log(2)}+4\sqrt[4]{30}\right)\np{\D u}{4}{B_{2a}\setminus\bar{B}_a(0)}.
    \end{align}
    Likewise, we get
    \begin{align*}
        \np{\D^2\widetilde{u}}{2}{\R^4\setminus\bar{B}_{b}(0)}\leq 3\np{\D^2u}{2}{B_{b}\setminus\bar{B}_{\frac{b}{2}}(0)}+80\left(\sqrt[4]{2\log(2)}+4\sqrt[4]{30}\right)\np{\D u}{4}{B_{b}\setminus\bar{B}_{\frac{b}{2}}(0)},
    \end{align*}
    and finally, we get
    \begin{align}\label{whitney_l49}
        \np{\D^2\widetilde{u}}{2}{\R^4}\leq 7\np{\D^2u}{2}{\Omega}+\left(1+160\left(\sqrt[4]{2\log(2)}+4\sqrt[4]{30}\right)\right)\np{\D u}{4}{\Omega}.
    \end{align}
    Therefore, reusing the Sobolev inequality from Theorem \ref{l42_sobolev_ineq}, we deduce that
    \begin{align*}
        \np{\D\widetilde{u}}{4,2}{\R^4}\leq 14\np{\D^2\widetilde{u}}{2}{\R^4}\leq 98\np{\D^2u}{2}{\Omega}+14\left(1+160\left(\sqrt[4]{2\log(2)}+4\sqrt[4]{30}\right)\right)\np{\D u}{4}{\Omega}.
    \end{align*}
    In particular, we have
    \begin{align*}
        \np{\D u}{4,2}{\Omega}\leq 98\np{\D^2u}{2}{\Omega}+14\left(1+160\left(\sqrt[4]{2\log(2)}+4\sqrt[4]{30}\right)\right)\np{\D u}{4}{\Omega},
    \end{align*}
    and by Proposition \ref{comp_lorentz_norm}, we also have
    \begin{align}\label{whitney_l410}
        \np{\D \widetilde{u}}{4}{\R^4}\leq 98\sqrt[4]{2}\np{\D^2u}{2}{\Omega}+14\sqrt[4]{2}\left(1+160\left(\sqrt[4]{2\log(2)}+4\sqrt[4]{30}\right)\right)\np{\D u}{4}{\Omega},
    \end{align}
    while \eqref{whitney_l42_l2} shows that
    \begin{align}\label{whitney_l411}
        \left(\int_{\R^4}\frac{|\D \widetilde{u}|^2}{|x|^2}dx\right)^{\frac{1}{2}}&\leq 2\,\sqrt[4]{2}\sqrt{\pi}\np{\D \widetilde{u}}{4,2}{\R^4}\nonumber\\
        &\leq 196\sqrt[4]{2}\sqrt{\pi}\np{\D^2u}{2}{\Omega}
        +28\sqrt[4]{2}\sqrt{\pi}\left(1+160\left(\sqrt[4]{2\log(2)}+4\sqrt[4]{30}\right)\right)\np{\D u}{4}{\Omega},
    \end{align}
    which concludes the proof of the theorem. 
\end{proof}

We see that on a neck region, there are three equivalent norms on $W^{2,2}(\Omega)/\R$ (and that the stronger norm involving $L^{4,2}$ always controls the other two for any domain), but for technical reasons, it will be more convenient to choose an $L^2$ norm, that is 
\begin{align*}
    N_{2,2}(u)=\np{\D^2u}{2}{\Omega}+\np{\frac{\D u}{|x|}}{2}{\Omega}=\left(\int_{\Omega}|\D^2u|^2dx\right)^{\frac{1}{2}}+\left(\int_{\Omega}\frac{|\D u|^2}{|x|^2}dx\right)^{\frac{1}{2}},
\end{align*}
which corresponds to an \enquote{intermediate} norm between $N_{2,4}=\np{\D^2(\,\cdot\,)}{2}{\Omega}+\np{\D(\,\cdot\,)}{4}{\Omega}$ and $N_{2,(4,2)}=\np{\D^2(\,\cdot\,)}{2}{\Omega}+\np{\D(\,\cdot\,)}{4,2}{\Omega}$.

    \nocite{}
	 \bibliographystyle{plain}
	 \bibliography{biblio_full}

    \end{document}